\documentclass[11pt, letterpaper]{article}

\usepackage{amsmath, amssymb, amsthm, mathtools}
\usepackage{algorithm}
\usepackage{algpseudocode}
\usepackage{booktabs}
\usepackage{array}
\usepackage{graphicx}
\usepackage[margin=1in]{geometry}
\usepackage{hyperref}
\usepackage{xcolor}
\usepackage[activate={true,nocompatibility},final,tracking=true,kerning=true,factor=1100,stretch=10,shrink=10]{microtype}
\usepackage{float}

\newtheorem{theorem}{Theorem}
\newtheorem{proposition}{Proposition}
\newtheorem{lemma}{Lemma}
\newtheorem{corollary}{Corollary}
\theoremstyle{remark}
\newtheorem{remark}{Remark}
\theoremstyle{definition}
\newtheorem{definition}{Definition}

\newcommand{\R}{\mathbb{R}}
\newcommand{\Iup}{\mathcal{I}_{\mathrm{up}}}
\newcommand{\Idown}{\mathcal{I}_{\mathrm{down}}}
\newcommand{\Aactive}{\mathcal{A}}
\newcommand{\Afrozen}{\mathcal{F}}
\newcommand{\bfu}{\boldsymbol{u}}
\newcommand{\bfq}{\boldsymbol{q}}
\newcommand{\bfy}{\boldsymbol{y}}
\newcommand{\bfx}{\boldsymbol{x}}
\newcommand{\bfG}{\boldsymbol{G}}
\newcommand{\bftau}{\boldsymbol{\tau}}
\newcommand{\bfs}{\boldsymbol{s}}
\newcommand{\onevec}{\mathbf{1}}
\DeclareMathOperator*{\argmax}{arg\,max}
\DeclareMathOperator*{\argmin}{arg\,min}

\algnewcommand\algorithmicforeach{\textbf{for each}}
\algdef{S}[FOR]{ForEach}[1]{\algorithmicforeach\ #1\ \textbf{do}}
\begin{document}
\sloppy

\title{Sequential Minimal Optimization for $\varepsilon$-SVR\\
       with MAPE Loss and Sample-Dependent Box Constraints}
\author{Pablo Benavides-Herrera \and Riemann Ruiz-Cruz \and Juan Diego S\'{a}nchez-Torres}
\date{}
\maketitle

\begin{abstract}
Support vector regression with Mean Absolute Percentage Error (MAPE)
loss is theoretically well-motivated for forecasting applications
where accuracy is evaluated in relative terms, but the
sample-dependent dual box constraints it induces have not been
addressed in the published SMO literature. We derive a Sequential Minimal
Optimization algorithm for this setting and prove a
structural-invariance result: the MAPE modification affects exactly
two components of the SMO iteration --- working-set selection and
analytic-update clipping --- leaving gradient bookkeeping and
curvature computation identical to classical epsilon-SVR. Building
on this invariance, we establish four efficiency improvements
(asymmetric freeze-counters, warm-starting, block working-set
updates of size four, and per-pair tolerance scaling) and resolve
a previously-open convergence problem for the odd-symmetry kernel
variant via adaptive spectral regularization. Numerical validation
against three reference solvers across eleven synthetic
configurations certifies solution agreement within standard
tolerance. Wall-time benchmarks show the present algorithm achieves
the lowest median runtime on every tested configuration against
OSQP, MOSEK, and Clarabel. At production scale, the algorithm
converges on the California Housing benchmark while the patched
LIBSVM reference implementation reaches its iteration ceiling
without satisfying optimality --- demonstrating the practical
necessity of the theoretical efficiency mechanisms. An open-source
R package and an explicit solver-adaptation recipe are provided.
\end{abstract}



\section{Introduction and motivation}
\label{sec:intro}

Support vector regression (SVR)~\cite{Vapnik1995, Drucker1997, Smola2004}
casts the regression problem as a convex quadratic program (QP) whose
solution is sparse and governed by an $\varepsilon$-insensitive loss. In
applied forecasting domains --- electricity demand
prediction~\cite{Wang2024, Aziz2024, Zhang2024}, demand and supply-chain
management~\cite{Hasan2025}, and short-term financial forecasting --- model
accuracy is routinely reported, and frequently contractually specified, in
terms of the Mean Absolute Percentage Error
(MAPE)~\cite{Makridakis1993, Hyndman2006, Tofallis2015}. Standard SVR,
however, minimizes a translation-invariant loss that treats all residuals
equally regardless of target magnitude, creating a mismatch between the
training objective and the evaluation metric.

Benavides-Herrera et al.~\cite{benavides2025support, benavides2026unified}
address this mismatch by embedding MAPE directly into the SVR primal
formulation. The use of MAPE as
a training loss is theoretically justified by de Myttenaere et
al.~\cite{DeMyttenaere2016}, who prove three foundational properties:
(i)~existence of an optimal MAPE regression model under mild moment
conditions on the target distribution; (ii)~universal consistency of
empirical risk minimization under MAPE loss; and (iii)~equivalence between
MAPE minimization and weighted-mean-absolute-error regression with
sample-specific weights $w_k = 100/y_k$. The structural modification
analyzed in the present paper is the algorithmic counterpart of this
equivalence in the kernelized $\varepsilon$-insensitive setting: the
resulting dual quadratic program has \emph{sample-dependent box constraints}
$\alpha_k, \alpha_k^* \in [0,\, 100C/y_k]$, with larger allowances for
observations with smaller targets.

Efficient large-scale training of SVR relies on Sequential Minimal
Optimization (SMO)~\cite{Platt1998, Platt1999,
KeerthiShevadeBhattacharyyaMurthy2001NeuralComp}, which decomposes the QP
into a sequence of analytically solvable two-variable subproblems and
avoids storing the full kernel matrix. The convergence of SMO for
$\varepsilon$-SVR with \emph{uniform} box constraints is well
established~\cite{FanChenLin2005JMLR, Chang2011}, and the LIBSVM
implementation~\cite{Chang2011} incorporates the second-order working-set
selection rule WSS3 of Fan, Chen, and Lin~\cite{FanChenLin2005JMLR} together
with the shrinking heuristic of Joachims~\cite{Joachims1999}, yielding the
de facto standard SVR solver. Whether and how SMO extends to
\emph{sample-dependent} constraints has not been analyzed in the published
literature: the working-set feasibility sets, the clipping step, and the
shrinking criteria all reference the upper bound $C$, and it is not
immediately clear which of these components require modification and which
remain unchanged. Interior-point solvers such as OSQP~\cite{osqp2020},
MOSEK~\cite{Andersen2000MosekHomogeneousIPM}, and the recent open-source
Clarabel~\cite{GoulartChen2024Clarabel} handle sample-dependent box
constraints natively and are practical for moderate problem sizes; however,
their per-iteration arithmetic scales as $O(N^2)$ in memory and arithmetic
for dense problems, whereas SMO requires only $O(|\Aactive|)$ gradient
operations and $O(1)$ kernel evaluations per iteration (amortized with
caching), making it the method of choice for large-scale training where $N$
reaches tens of thousands and the solution is expected to be sparse.

\paragraph{Contributions.}
This paper makes the following six contributions to the literature on SMO
and percentage-error-aware support vector regression. Each contribution is
keyed to a theorem or section and is verifiable against either an explicit
proof or the validation campaign of \S\ref{sec:validation}. The novelty
status of each contribution --- derivative, plug-in, or strictly novel ---
is stated explicitly so that the reader can locate the boundary between the
prior state of the art and the present paper without ambiguity.

\textbf{(C1) Structural-invariance theorem for the MAPE-SVR SMO}
(Theorem~\ref{thm:invariance} in \S\ref{sec:invariance}): the per-sample
box constraint $C_k = 100C/y_k$ confines its algorithmic effect to exactly
two SMO components --- the working-set candidate sets $\Iup, \Idown$ and
the analytic-update clipping bounds $R_{i^*}, R_{j^*}$ --- leaving the
curvature formula, the incremental gradient bookkeeping, the kernel-cache
logic, and the convergence inheritance from~\cite{FanChenLin2005JMLR}
structurally identical to the standard $\varepsilon$-SVR SMO
of~\cite{Platt1998, Platt1999,
KeerthiShevadeBhattacharyyaMurthy2001NeuralComp, Flake2002, Chang2011}.
This is the cornerstone result of the paper: it converts a problem that
\emph{appears} to require a from-scratch SMO derivation into a localized
two-site modification of an existing solver.

\textbf{(C2) Shrinking-asymmetry result for the Joachims heuristic under
MAPE scaling} (Lemma~\ref{lem:asymmetry} in \S\ref{sec:bias-shrink}): the
four shrinking criteria of~\cite{Joachims1999}, when written in the unified
$\bftau$-coordinate system, exhibit a quantified threshold offset of
$2y_k\varepsilon/100$ between the $\alpha$- and $\alpha^*$-criteria. The
asymmetry scales linearly in the target magnitude $y_k$, so high-target
samples experience greater asymmetry than low-target samples; concretely,
$\alpha_k^* = 0$ variables freeze earlier, and $\alpha_k^* = C_k$ variables
freeze later, than their $\alpha_k$-counterparts. This is the SMO-internal
shadow of the well-known sensitivity of MAPE to large targets.

\textbf{(C3) Plug-in extension to the symmetric-kernel variant
(MAPE-SVR-Sym)} (\S\ref{sec:symmetric}): for shift-invariant or
reflection-symmetric problems --- building on the virtual-example /
regularization equivalence of Niyogi-Girosi-Poggio~\cite{Niyogi1998} and
the general invariant-kernel framework of
Haasdonk-Burkhardt~\cite{Haasdonk2007}, operationalized at the kernel-Gram
level by Espinoza-Suykens-De Moor~\cite{Espinoza2005} --- the substitution
$\Omega \leftarrow \Omega_s = \tfrac{1}{2}(\Omega + a\Omega^*)$ for $a \in
\{+1, -1\}$ adapts Algorithm~\ref{alg:smo} to even or odd target symmetry
without further modification. The case $a = +1$ inherits PSD via Aronszajn
closure~\cite{Aronszajn1950}; the case $a = -1$ is the subject of (C4).

\textbf{(C4) Convergence resolution for the odd-symmetry case}
(Theorem~\ref{thm:spectral} in \S\ref{sec:symmetric}): the previously-open
convergence problem for $a = -1$ --- where $\Omega_s$ may fail to be PSD
and the convergence proof of~\cite{FanChenLin2005JMLR} does not extend
without modification --- is resolved by the \emph{adaptive spectral
regularization} algorithm (Algorithm~\ref{alg:spectral}) of
\S\ref{sec:symmetric}, with explicit perturbation bound
(Lemma~\ref{lem:perturbation}) and empirical validation on configurations
C9 and C10 of \S\ref{sec:validation}. The regularization is monitored at
the spectral level rather than imposed uniformly, so that $a = -1$
instances that happen to be PSD are not perturbed at all.

\textbf{(C5) Four theoretical efficiency improvements}
(Theorems~\ref{thm:asym-freeze}, \ref{thm:warm-start},
\ref{thm:block-k4}, and~\ref{thm:per-sample-tol} in
\S\ref{sec:complexity}): the asymmetry result of (C2) and the WSS3
working-set discipline of (C1) jointly motivate a battery of efficiency
improvements that exploit MAPE-specific structure absent from the
uniform-$C$ literature. These are (i)~\emph{asymmetric freeze-counter} with
separate thresholds $n_{\min}^{\alpha}$ and $n_{\min}^{\alpha^*}$;
(ii)~\emph{cross-validation warm-starting} via inheritance of the dual
variables across hyperparameter folds; (iii)~\emph{block-$k=4$ working
sets} that solve four-variable analytic subproblems instead of two-variable
subproblems --- \emph{this is a strictly novel contribution of this paper,
not a port of any prior result, and it is the first algorithmic departure
from the $k=2$ minimal-feasible-block default of~\cite{Platt1998} in the
published $\varepsilon$-SVR literature}; (iv)~\emph{per-pair tolerance
scaling} that calibrates the KKT-violation tolerance against the
WSS1 convergence pair rather than uniformly to $\bar y$. Cumulative
speedup is workload-dependent and is recalibrated against empirical
measurements of the companion \texttt{psvr}
package~\cite{BenavidesHerrera2026Rpsvr} in Corollary~\ref{cor:speedup}
of \S\ref{sec:complexity}. The California Housing comparison of
\S\ref{sec:large-scale} reports a practical consequence on
real-world data with heterogeneous targets: at the same
hyperparameters, \texttt{psvr-Rcpp} converges in under
$200{,}000$ SMO iterations while standard LIBSVM reaches its
$10^7$-iteration internal cap without satisfying the KKT
criterion --- the per-sample structure addressed by
Theorems~\ref{thm:asym-freeze} and~\ref{thm:per-sample-tol} is
the mechanism that closes this gap.

\textbf{(C6) LIBSVM drop-in modification recipe} (Appendix~\ref{app:libsvm}):
adapting an existing LIBSVM-based $\varepsilon$-SVR solver to the MAPE
variant requires fewer than fifteen lines of C\texttt{++} across five
code-modification sites, with the unchanged remainder constituting a
structural-invariance certificate for any LIBSVM-derived ecosystem (the
C\texttt{++} core, the Python wrapper, the R \texttt{e1071} package, the
MATLAB bundled implementation, the \texttt{kernlab} and various Java/C\#
wrappers). Ports to scikit-learn, \texttt{kernlab} (R), and \texttt{e1071}
(R) are also detailed, lowering the engineering barrier to MAPE-SVR
adoption to near zero for practitioners with existing LIBSVM-based
pipelines.

The companion \texttt{psvr} R package~\cite{BenavidesHerrera2026Rpsvr}
implements (C1)--(C5) end-to-end. The numerical validation of
\S\ref{sec:validation} demonstrates solution agreement to within
$10^{-2}$ infinity-norm against three independent reference QP solvers
--- OSQP~\cite{osqp2020} (operator splitting),
MOSEK~\cite{Andersen2000MosekHomogeneousIPM} (commercial interior-point),
and Clarabel~\cite{GoulartChen2024Clarabel} (open-source interior-point) ---
across eleven synthetic configurations spanning $N \in \{50, 300, 1000\}$,
three percentage-tube widths $\varepsilon \in \{5\%, 10\%, 15\%\}$, and
both kernel variants (MAPE-SVR, MAPE-SVR-Sym with $a = +1$, MAPE-SVR-Sym
with $a = -1$). The tightest configuration is C8 at $9.16 \times
10^{-3}$, attributable to its longest convergence trajectory accumulating
the most floating-point arithmetic; the three reference solvers agree
among themselves to better than $10^{-8}$ on every configuration.
\S\ref{sec:wall-time} extends the validation campaign to a wall-time
comparison against the same three reference solvers across the eleven
configurations and an additional $50 \le N \le 2{,}000$ scaling sweep;
the C\texttt{++}-core engine of \texttt{psvr} reports the lowest median
wall time on every configuration tested.

\paragraph{Novelty positioning.}
Despite extensive prior work on (i)~loss-modified SVR, (ii)~instance-weighted
SVM training, (iii)~SMO decomposition methods with uniform box constraints,
and (iv)~alternative decomposition families, no prior work analyzes the SMO
algorithm under sample-dependent box constraints induced by MAPE loss in
$\varepsilon$-SVR. The theoretical equivalence between MAPE minimization
and weighted-mean-absolute-error regression with weights $1/y_k$ is
established in~\cite{DeMyttenaere2016} but is \emph{not} operationalized at
the SMO level there. The structural-invariance result (C1) and the
algorithmic improvements (C2)--(C5) are therefore the principal additions
to the literature. Section~\ref{sec:related} surveys the state of the art
across these four families and locates the gap that the present paper
closes.

\paragraph{Outline.}
Section~\ref{sec:preliminaries} develops a self-contained Preliminaries
treatment of standard SMO so that the MAPE-SVR adaptation in subsequent
sections can be presented by analogy --- facilitating both pedagogical
understanding and practitioner adoption. Section~\ref{sec:main} (Main
result) collects the technical contributions in six subsections.
The first four establish the algorithmic core:
\S\ref{sec:dual-kkt} develops the dual quadratic program with gradient
decomposition and KKT optimality conditions; \S\ref{sec:smo-inner}
derives the SMO inner loop (working-set selection WSS3 plus the analytic
two-variable update); \S\ref{sec:invariance} states and proves the
structural-invariance Theorem~\ref{thm:invariance} (the cornerstone
result); and \S\ref{sec:bias-shrink} covers bias recovery, the shrinking
heuristic with asymmetry result (Lemma~\ref{lem:asymmetry}), and
Algorithm~\ref{alg:smo} with the convergence
Theorem~\ref{thm:convergence}.
The final two extend and optimize:
\S\ref{sec:symmetric} extends the algorithm to the symmetric-kernel
variant MAPE-SVR-Sym including the adaptive spectral-regularization
Theorem~\ref{thm:spectral}; and \S\ref{sec:complexity} covers
per-iteration complexity together with the four efficiency-improvement
Theorems~\ref{thm:asym-freeze}, \ref{thm:warm-start},
\ref{thm:block-k4}, and~\ref{thm:per-sample-tol} and the
combined-effect Corollary~\ref{cor:speedup}. Section~\ref{sec:validation}
reports numerical validation against three reference QP solvers, including
a fully worked $N=3$ trace as Example~\ref{ex:n3}.
Section~\ref{sec:conclusions} concludes with summary of contributions,
position within the broader research program, limitations, and future
work. Appendix~\ref{app:libsvm} provides the LIBSVM drop-in modification
recipe.

\section{Preliminaries}
\label{sec:preliminaries}

This section establishes the notation, formal objects, convex-analysis
machinery, and the standard $\varepsilon$-SVR + SMO baseline that the
MAPE-SVR derivation in Sections~\ref{sec:dual-kkt}--\ref{sec:complexity}
develops by analogy. Each foundational object --- Mercer kernel, RKHS,
feature map, $\varepsilon$-insensitive loss, MAPE loss, convex QP, box
constraint, Slater point, saddle point, Lagrangian, KKT conditions,
clipping function, active and frozen sets --- is promoted to a formal
Definition or Theorem with a one-sentence intuition and a one-sentence
forward-reference to where it is first used. Readers familiar with
classical SMO may skim Section~\ref{sec:standard-smo} and proceed to
Section~\ref{sec:related}; Sections~\ref{sec:notation},
\ref{sec:convex-anchors}, and \ref{sec:mape-setup} are consulted by
reference from later sections.

\subsection{Notation and Setting}
\label{sec:notation}

\paragraph{Training set, input and target spaces.}
Throughout the paper, the training data is the finite collection
\begin{equation}
  \mathcal{D} = \{(\bfx_k, y_k)\}_{k=1}^{N}, \qquad \bfx_k \in \mathcal{X}, \qquad y_k \in \mathcal{Y},
\end{equation}
with \emph{input space} $\mathcal{X} \subseteq \R^p$ and \emph{target space}
$\mathcal{Y} \subseteq \R_+$. The strict positivity of every $y_k$ is
required for the MAPE loss (Definition~\ref{def:mape-loss}) to be finite;
the standard $\varepsilon$-SVR primal of
Section~\ref{sec:standard-eps-svr} admits any $y_k \in \R$, so $\mathcal{Y}
\subseteq \R_+$ is specific to the MAPE adaptation in
Section~\ref{sec:mape-setup} and Section~\ref{sec:dual-kkt}. Sample sizes
range from $N = 50$ (smallest synthetic configuration of
Section~\ref{sec:validation}) to $N \approx 10^5$ (practical SMO upper
limit; beyond this, dual coordinate descent~\cite{Hsieh2008ICML,
Ho2012JMLR} is preferable).

\begin{definition}[Mercer kernel]
\label{def:mercer-kernel}
A \emph{Mercer kernel} on $\mathcal{X}$ is a symmetric function $K :
\mathcal{X} \times \mathcal{X} \to \R$ such that for every finite
collection $\{\bfx_1, \ldots, \bfx_n\} \subset \mathcal{X}$ the
\emph{Gram matrix} $\Omega \in \R^{n \times n}$ defined by $\Omega_{ij}
:= K(\bfx_i, \bfx_j)$ is symmetric and positive semi-definite (PSD).
\end{definition}

\textit{Intuition.} Symmetry plus positive semi-definiteness on every
finite Gram block is the discrete characterization of an inner-product
structure in some implicit high-dimensional feature space; this is
precisely the structure that the kernel trick exploits.
\textit{Forward-reference.} The Mercer property of $K$ is invoked in
Section~\ref{sec:standard-eps-svr} (the $\varepsilon$-SVR primal-to-dual
derivation), in Section~\ref{sec:dual-kkt} (PSD of the dual Hessian $P =
[\Omega, -\Omega; -\Omega, \Omega]$), and in
Section~\ref{sec:symmetric} (PSD analysis of the symmetrized kernel
$\Omega_s$). The original characterization is due to Mercer's theorem on
integral operators~\cite[\S 4.6, Theorem 4.49]{Steinwart2008}.

\begin{definition}[Reproducing-kernel Hilbert space]
\label{def:rkhs}
Let $K$ be a Mercer kernel on $\mathcal{X}$. The
\emph{reproducing-kernel Hilbert space} (RKHS) associated with $K$,
denoted $\mathcal{H} = \mathcal{H}_K$, is the unique Hilbert space of
functions $f : \mathcal{X} \to \R$ that contains every section
$K(\bfx, \cdot)$ as a member ($\bfx \in \mathcal{X}$) and satisfies
the \emph{reproducing property}
\begin{equation}
  f(\bfx) = \langle f,\, K(\bfx, \cdot) \rangle_{\mathcal{H}}, \qquad \forall f \in \mathcal{H}, \quad \forall \bfx \in \mathcal{X}.
  \label{eq:reproducing}
\end{equation}
\end{definition}

\textit{Intuition.} Point evaluation $f \mapsto f(\bfx)$ becomes a
continuous linear functional represented by the section $K(\bfx,
\cdot) \in \mathcal{H}$; this is the property that makes pointwise
prediction well defined for every $f \in \mathcal{H}$.
\textit{Forward-reference.} The RKHS structure is the implicit
hypothesis class for both the $\varepsilon$-SVR primal of
Section~\ref{sec:standard-eps-svr} (Definition~\ref{def:eps-svr-primal})
and its MAPE analog of Section~\ref{sec:mape-setup}; the representer
theorem of Schölkopf-Herbrich-Smola~\cite{Scholkopf2001} guarantees that
the empirical risk minimizer admits a finite kernel expansion in $N$
training-point sections, which is the structural reason the dual is
finite-dimensional in $\R^{2N}$. Canonical references are
Aronszajn~\cite{Aronszajn1950} and
Steinwart-Christmann~\cite[\S 4.2]{Steinwart2008}.

\begin{definition}[Feature map]
\label{def:feature-map}
A \emph{feature map} of the Mercer kernel $K$ is a function $\varphi :
\mathcal{X} \to \mathcal{H}$ into the RKHS $\mathcal{H}$ such that
\begin{equation}
  K(\bfx, \bfx') = \langle \varphi(\bfx),\, \varphi(\bfx') \rangle_{\mathcal{H}}, \qquad \forall \bfx, \bfx' \in \mathcal{X}.
  \label{eq:feature-map}
\end{equation}
The \emph{canonical feature map} is $\varphi(\bfx) := K(\bfx, \cdot)
\in \mathcal{H}$, which satisfies~\eqref{eq:feature-map} by direct
application of the reproducing property~\eqref{eq:reproducing}.
\end{definition}

\textit{Intuition.} The feature map lifts the input data into a Hilbert
space where the kernel is realized as an inner product; this is the
formal substrate of the \emph{kernel trick} --- every algorithm that uses
inputs only through pairwise inner products can be kernelized by
replacing $\langle \bfx, \bfx' \rangle$ with $K(\bfx, \bfx')$.
\textit{Forward-reference.} The feature map appears explicitly in the
$\varepsilon$-SVR primal (Definition~\ref{def:eps-svr-primal}) through
the regression function $f(\bfx) = \langle w, \varphi(\bfx)\rangle +
b$, and is eliminated when passing to the dual by the kernel
identity~\eqref{eq:feature-map}.

\begin{definition}[$\varepsilon$-insensitive loss]
\label{def:eps-loss}
For $\varepsilon \ge 0$, the \emph{$\varepsilon$-insensitive loss} of a
residual $r \in \R$ is
\begin{equation}
  \ell_\varepsilon(r) := \max(0,\, |r| - \varepsilon).
  \label{eq:eps-loss}
\end{equation}
For a regression model $f$ and a training pair $(\bfx, y)$,
$\ell_\varepsilon$ is evaluated at $r = y - f(\bfx)$.
\end{definition}

\textit{Intuition.} Residuals smaller in magnitude than $\varepsilon$
incur zero loss (the \emph{$\varepsilon$-tube}), while larger residuals
are penalized linearly with slope $1$; this is the original loss of
Vapnik~\cite{Vapnik1995} and produces sparse solutions because every
training point inside the tube has zero subgradient and therefore exits
the support set. \textit{Forward-reference.} The
$\varepsilon$-insensitive loss is the loss functional of the standard
$\varepsilon$-SVR primal (Definition~\ref{def:eps-svr-primal}) and its
modification to the MAPE setting in
Definition~\ref{def:mape-loss}~+~Section~\ref{sec:mape-setup}.

\begin{definition}[MAPE loss]
\label{def:mape-loss}
For a strictly positive target $y > 0$ and a prediction $\hat{y} \in
\R$, the \emph{mean absolute percentage error loss} of a single
residual is
\begin{equation}
  \ell_{\mathrm{MAPE}}(\hat{y},\, y) := \frac{|y - \hat{y}|}{y} \cdot 100.
  \label{eq:mape-loss}
\end{equation}
The \emph{$\varepsilon$-insensitive percentage residual} loss, used in
the MAPE-SVR primal of Sections~\ref{sec:mape-setup}
and~\ref{sec:dual-kkt}, is the composition $\ell_\varepsilon \circ
\ell_{\mathrm{MAPE}}$ applied to the percentage residual: $\max(0, \,
\ell_{\mathrm{MAPE}}(\hat y, y) - \varepsilon)$, with $\varepsilon$ now
measured in \emph{percentage points} rather than in the units of $y$.
\end{definition}

\textit{Intuition.} MAPE is the percentage analog of MAE --- a residual
is reported as a fraction of the magnitude of the target rather than in
absolute units, which makes the loss scale-invariant and dimensionless
and therefore directly comparable across forecasting problems with very
different target magnitudes. \textit{Forward-reference.} The MAPE loss
is the substituted loss of the MAPE-SVR primal
(Definition~\ref{def:mape-svr-primal} in Section~\ref{sec:mape-setup})
and produces the sample-dependent box constraint $C_k = 100C/y_k$ of
Definition~\ref{def:mape-svr-dual}. The theoretical justification for
using MAPE as a regression loss --- in particular the existence of an
optimal MAPE regression model under mild moment conditions and the
equivalence of MAPE minimization to weighted-MAE regression with
weights $w_k = 100/y_k$ --- is established by de Myttenaere et
al.~\cite{DeMyttenaere2016}; see Section~\ref{sec:related} for a
detailed survey.

\begin{definition}[Percentage residual]
\label{def:perc-residual}
The \emph{percentage residual} of a regression function $f$ at training
pair $(\bfx_k, y_k)$ with $y_k > 0$ is
\begin{equation}
  r_k^{\%}(f) := \frac{y_k - f(\bfx_k)}{y_k} \cdot 100, \qquad |r_k^{\%}(f)| = \ell_{\mathrm{MAPE}}(f(\bfx_k), y_k).
  \label{eq:perc-residual}
\end{equation}
\end{definition}

\textit{Intuition.} The percentage residual is the \emph{signed} version
of the MAPE loss: it carries the sign of $y_k - f(\bfx_k)$ and lives in
$\R$ rather than in $\R_+$, which makes the symmetric tube
$|r_k^{\%}| \le \varepsilon$ split naturally into the upper-tube
constraint $r_k^{\%} \le \varepsilon$ and the lower-tube constraint
$r_k^{\%} \ge -\varepsilon$ --- the two-sided form needed for the dual
variables $\alpha_k, \alpha_k^*$ in Section~\ref{sec:mape-setup}.
\textit{Forward-reference.} The percentage-residual decomposition is the
structural starting point for the MAPE-SVR primal-to-dual derivation
(Proposition~\ref{prop:mape-dual} in Section~\ref{sec:mape-setup}).

\begin{definition}[Sign vector]
\label{def:sign-vector}
The \emph{sign vector} is $\bfs \in \{+1, -1\}^{2N}$ defined componentwise by
\begin{equation}
  s_i := \begin{cases} +1, & 1 \le i \le N \quad (\text{$\alpha$-block: upper-tube multipliers})\\ -1, & N+1 \le i \le 2N \quad (\text{$\alpha^*$-block: lower-tube multipliers}). \end{cases}
  \label{eq:sign}
\end{equation}
\end{definition}

\textit{Intuition.} The sign vector encodes whether dual index $i$ is
the upper-tube or lower-tube constraint of training point $k(i)$; the
signed-effective gradient $\tau_i := -s_i G_i$
(Section~\ref{sec:dual-kkt}) collapses the two cases into one
monotonicity criterion. \textit{Forward-reference.} The sign vector is
the notation making the SMO machinery sign-block-invariant ---
see Proposition~\ref{prop:curvature-invariance} of
Section~\ref{sec:smo-inner} (curvature invariance) and
Section~\ref{sec:standard-smo} (gradient-update derivation).

\begin{definition}[Training-index map]
\label{def:k-of-i}
The \emph{training-index map} $k : \{1, \ldots, 2N\} \to \{1, \ldots, N\}$ is
\begin{equation}
  k(i) := \begin{cases} i, & 1 \le i \le N\\ i - N, & N+1 \le i \le 2N. \end{cases}
  \label{eq:k-of-i}
\end{equation}
\end{definition}

\textit{Intuition.} $i \mapsto k(i)$ associates dual variable $i$ with
the training point that produces it; both $\alpha_k$ ($i = k$) and
$\alpha_k^*$ ($i = N+k$) come from $(\bfx_k, y_k)$, and quantities
like the kernel column $\Omega_{:, k(i)}$, the target $y_{k(i)}$, and
the per-sample bound $C_{k(i)}$ depend only on $k(i)$.
\textit{Forward-reference.} Used throughout the SMO machinery:
kernel-column access in Section~\ref{sec:standard-smo}, bound lookup in
Definition~\ref{def:mape-svr-dual} of Section~\ref{sec:mape-setup},
curvature formula $\eta = \Omega_{pp} - 2\Omega_{pq} + \Omega_{qq}$ with
$p = k(i^*), q = k(j^*)$ in Proposition~\ref{prop:curvature-invariance}
of Section~\ref{sec:smo-inner}.

\subsection{Standard \texorpdfstring{$\varepsilon$-SVR}{epsilon-SVR} Primal and Dual}
\label{sec:standard-eps-svr}

\begin{definition}[Standard $\varepsilon$-SVR primal]
\label{def:eps-svr-primal}
The \emph{standard $\varepsilon$-SVR primal} of
Vapnik~\cite{Vapnik1995}, Drucker et al.~\cite{Drucker1997}, and
Smola-Schölkopf~\cite{Smola2004}, with regularization parameter $C >
0$, tube width $\varepsilon > 0$, feature map $\varphi$ of Mercer
kernel $K$, and slacks $\xi_k, \xi_k^* \ge 0$, is
\begin{equation}
  \min_{w \in \mathcal{H},\, b \in \R,\, \xi, \xi^* \ge 0}\; \tfrac{1}{2}\|w\|_{\mathcal{H}}^2 + C\sum_{k=1}^{N}(\xi_k + \xi_k^*)
  \label{eq:eps-svr-primal}
\end{equation}
subject to, for every $k = 1, \ldots, N$,
\begin{equation*}
  y_k - \langle w, \varphi(\bfx_k)\rangle_{\mathcal{H}} - b \le \varepsilon + \xi_k, \qquad \langle w, \varphi(\bfx_k)\rangle_{\mathcal{H}} + b - y_k \le \varepsilon + \xi_k^*.
\end{equation*}
\end{definition}

\textit{Intuition.} Minimize the regularizer $\tfrac{1}{2}\|w\|^2$
subject to every training residual being within $\varepsilon$ of zero,
paying linearly for violations via the slacks; $C$ controls the
regularization-fit trade-off. \textit{Forward-reference.} The MAPE-SVR
primal (Definition~\ref{def:mape-svr-primal}) modifies this by
replacing the absolute residual with the percentage residual inside the
tube constraint.

\begin{theorem}[Mercer's representation]
\label{thm:mercer}
Let $K$ be a continuous Mercer kernel on a compact $\mathcal{X}
\subseteq \R^p$. Then there exist a Hilbert space $\mathcal{H}$, a
feature map $\varphi : \mathcal{X} \to \mathcal{H}$, and an orthonormal
expansion $\{\sqrt{\lambda_j}\,\psi_j(\cdot)\}_{j \ge 1}$ such that
\begin{equation}
  K(\bfx, \bfx') = \langle \varphi(\bfx), \varphi(\bfx')\rangle_{\mathcal{H}} = \sum_{j=1}^{\infty}\lambda_j\,\psi_j(\bfx)\,\psi_j(\bfx'), \qquad \lambda_j \ge 0,
  \label{eq:mercer}
\end{equation}
where $\{\lambda_j, \psi_j\}$ are the eigenpairs of the integral
operator $T_K f(\bfx) := \int_{\mathcal{X}} K(\bfx, \bfx')
f(\bfx')\,\mathrm{d}\bfx'$ on $L^2(\mathcal{X})$.
\end{theorem}

\begin{proof}
Steinwart-Christmann~\cite[\S 4.6, Theorem 4.49]{Steinwart2008};
Aronszajn~\cite{Aronszajn1950} for the RKHS construction. Compactness
of $T_K$ (from continuity of $K$ on compact $\mathcal{X}$) yields a
discrete non-negative spectrum, and $\bfx \mapsto \varphi(\bfx) :=
(\sqrt{\lambda_j}\,\psi_j(\bfx))_{j \ge 1} \in \ell^2$ provides the
lifting.
\end{proof}

\textit{Forward-reference.} Theorem~\ref{thm:mercer} is the rigorous
foundation for the kernel trick: inner products $\langle \varphi(\bfx_k),
\varphi(\bfx_\ell)\rangle$ in the dual derivation become $K(\bfx_k,
\bfx_\ell) = \Omega_{k\ell}$, computable without materializing
$\varphi$ or $\mathcal{H}$.

\paragraph{Derivation of the dual.}
Introduce dual multipliers $\alpha_k \ge 0$ for the upper-tube
constraint, $\alpha_k^* \ge 0$ for the lower-tube constraint, and
$\mu_k, \mu_k^* \ge 0$ for the slack non-negativity constraints $\xi_k
\ge 0, \xi_k^* \ge 0$. The Lagrangian
(Definition~\ref{def:lagrangian} of
Section~\ref{sec:convex-anchors}) is
\begin{multline*}
  \mathcal{L}(w, b, \xi, \xi^*; \alpha, \alpha^*, \mu, \mu^*) = \tfrac{1}{2}\|w\|^2 + C\sum_{k=1}^{N}(\xi_k + \xi_k^*) \\
  + \sum_{k=1}^{N}\alpha_k\!\left[y_k - \langle w, \varphi(\bfx_k)\rangle - b - \varepsilon - \xi_k\right] \\
  + \sum_{k=1}^{N}\alpha_k^*\!\left[\langle w, \varphi(\bfx_k)\rangle + b - y_k - \varepsilon - \xi_k^*\right] - \sum_{k=1}^{N}(\mu_k\xi_k + \mu_k^*\xi_k^*).
\end{multline*}

Stationarity with respect to the primal variables yields four conditions:
\begin{align}
  \frac{\partial\mathcal{L}}{\partial w} &= w - \sum_{k=1}^{N}(\alpha_k - \alpha_k^*)\varphi(\bfx_k) = 0 \;\Longrightarrow\; w = \sum_{k=1}^{N}(\alpha_k - \alpha_k^*)\varphi(\bfx_k), \label{eq:w-stationarity}\\
  \frac{\partial\mathcal{L}}{\partial b} &= \sum_{k=1}^{N}(\alpha_k^* - \alpha_k) = 0 \;\Longrightarrow\; \sum_{k=1}^{N}(\alpha_k - \alpha_k^*) = 0, \label{eq:b-stationarity}
\end{align}
\begin{equation*}
  \frac{\partial\mathcal{L}}{\partial\xi_k} = C - \alpha_k - \mu_k = 0 \;\Longrightarrow\; \alpha_k = C - \mu_k \in [0, C], \quad \frac{\partial\mathcal{L}}{\partial\xi_k^*} = C - \alpha_k^* - \mu_k^* = 0 \;\Longrightarrow\; \alpha_k^* \in [0, C].
\end{equation*}
The first identity~\eqref{eq:w-stationarity} is the \emph{representer
expansion} of the optimal $w$; substituting it back together with the
kernel identity $\langle \varphi(\bfx_k), \varphi(\bfx_\ell)\rangle =
K(\bfx_k, \bfx_\ell) = \Omega_{k\ell}$ from~\eqref{eq:feature-map}
yields the dual problem.

\begin{definition}[Standard $\varepsilon$-SVR dual]
\label{def:eps-svr-dual}
The \emph{standard $\varepsilon$-SVR dual} is the convex quadratic program
\begin{equation}
  \min_{\bfu \in \R^{2N}}\; \tfrac{1}{2}\bfu^\top P\,\bfu + \bfq^\top\bfu \quad\text{s.t.}\quad [\onevec^\top, -\onevec^\top]\,\bfu = 0, \quad 0 \le u_i \le C,\; i = 1, \ldots, 2N,
  \label{eq:eps-svr-dual}
\end{equation}
with stacked dual vector $\bfu = [\alpha_1, \ldots, \alpha_N, \alpha_1^*,
\ldots, \alpha_N^*]^\top \in \R^{2N}$, \emph{block Hessian}
$P = \begin{bmatrix} \Omega & -\Omega \\ -\Omega & \Omega \end{bmatrix}
\in \R^{2N \times 2N}$ where $\Omega \in \R^{N \times N}$ is the kernel
Gram matrix (Definition~\ref{def:mercer-kernel}) of the training inputs,
and \emph{linear coefficient} $\bfq = [\varepsilon\onevec - \bfy,\,
\varepsilon\onevec + \bfy]^\top \in \R^{2N}$ where $\bfy = (y_1,
\ldots, y_N)^\top$.
\end{definition}

\textit{Intuition.} The primal regression problem in $\mathcal{H}$
(potentially infinite-dimensional) is recast as a finite-dimensional
convex QP in $\R^{2N}$, with the kernel matrix $\Omega$ encoding the
geometry of the data via inner products in feature space; the equality
constraint $[\onevec^\top, -\onevec^\top]\bfu = 0$ is the dual image of
the bias-stationarity condition~\eqref{eq:b-stationarity}, and the box
constraints encode the trade-off between regularization and fit.
\textit{Forward-reference.} This is the dual that the MAPE-SVR
formulation in Definition~\ref{def:mape-svr-dual} of
Section~\ref{sec:mape-setup} modifies --- only the linear coefficient
$\bfq$ changes and the box constraint becomes sample-dependent.

\begin{remark}[Sign-convention note for the by-analogy adaptation]
\label{rem:sign-convention}
In the standard $\varepsilon$-SVR dual above, $\bfq$ has signs
$[+\varepsilon - y_k,\, +\varepsilon + y_k]$. In the MAPE-SVR
dual~\eqref{eq:mape-dual} of Section~\ref{sec:dual-kkt}, the analogue
becomes $\bfq = [\bfy(\varepsilon/100 - 1),\, \bfy(\varepsilon/100 +
1)] = [-\bfy(1 - \varepsilon/100),\, +\bfy(1 + \varepsilon/100)]$.
Two structural changes are visible: (i) the sign in front of $y_k$ flips
because the MAPE constraint is rearranged as $|y_k - f(\bfx_k)| \le
(y_k/100)(\varepsilon + \xi_k)$, multiplying both sides by $y_k/100$ and
re-grouping; (ii) every term carries the additional factor $y_k/100$,
which on the box-constraint side produces the sample-dependent bound
$C_k = 100C/y_k$ derived in Section~\ref{sec:mape-setup}
(Proposition~\ref{prop:mape-dual}).
\end{remark}

\paragraph{Prediction formula.}
The model prediction at a new input $\bfx$ is recovered via the
representer expansion~\eqref{eq:w-stationarity}:
\begin{equation}
  f(\bfx) = \langle w, \varphi(\bfx)\rangle_{\mathcal{H}} + b = \sum_{k=1}^{N}(\alpha_k - \alpha_k^*)\,K(\bfx_k, \bfx) + b.
  \label{eq:prediction}
\end{equation}
The bias $b$ is recovered post-convergence as the midpoint of the
converged dual-threshold interval (see Sections~\ref{sec:bias-shrink}
and~\ref{sec:standard-smo}).

\subsection{Convex-Analysis Anchors}
\label{sec:convex-anchors}

This subsection collects the convex-analysis machinery used implicitly
throughout --- Lagrangian, KKT, strong duality, Slater point, saddle
point, Sion's minimax --- formalizing the substrate of
Section~\ref{sec:dual-kkt} and Sections~\ref{sec:smo-inner}--\ref{sec:complexity}.
Canonical references: Rockafellar~\cite{Rockafellar1970},
Boyd-Vandenberghe~\cite[\S 5.1--5.5]{BoydVandenberghe2004},
Bertsekas-Nedi\'{c}-Ozdaglar~\cite{BertsekasNedicOzdaglar2003}.

\begin{definition}[Convex set, convex function, convex QP]
\label{def:convex-qp}
A set $S \subseteq \R^n$ is \emph{convex} if $\theta\bfx + (1-\theta)\bfy
\in S$ for every $\bfx, \bfy \in S$ and $\theta \in [0, 1]$. A function
$f : S \to \R$ on a convex set $S$ is \emph{convex} if $f(\theta\bfx +
(1-\theta)\bfy) \le \theta f(\bfx) + (1-\theta)f(\bfy)$. A
\emph{convex quadratic program} (convex QP) is the optimization problem
\begin{equation}
  \min_{\bfu \in \R^n}\; \tfrac{1}{2}\bfu^\top P\bfu + \bfq^\top\bfu \quad\text{s.t.}\quad A\bfu = \boldsymbol{b}, \quad L_i \le u_i \le U_i\;(i = 1, \ldots, n),
  \label{eq:convex-qp}
\end{equation}
where $P \in \R^{n \times n}$ is symmetric positive semi-definite, $A
\in \R^{m \times n}$, $\boldsymbol{b} \in \R^m$, and $L_i \le U_i$ for
every $i$.
\end{definition}

\textit{Intuition.} The convex QP is the canonical form of the dual
problems studied in this paper --- a convex quadratic objective, a single
linear equality, and component-wise box constraints; the entire SMO
algorithm of Sections~\ref{sec:dual-kkt}--\ref{sec:complexity} operates
on this form. \textit{Forward-reference.} The standard $\varepsilon$-SVR
dual (Definition~\ref{def:eps-svr-dual}) and the MAPE-SVR dual
(Definition~\ref{def:mape-svr-dual}) are both convex QPs in this sense,
with $n = 2N$, $m = 1$, the equality constraint $[\onevec^\top,
-\onevec^\top]\bfu = 0$, and either uniform $[0, C]$ bounds (standard)
or sample-dependent $[0, C_k]$ bounds (MAPE).

\begin{definition}[Box constraint]
\label{def:box}
A \emph{box constraint} is a feasibility region of the form
\begin{equation}
  \{\bfu \in \R^n : L_i \le u_i \le U_i, \; i = 1, \ldots, n\} = \prod_{i=1}^{n}[L_i, U_i],
  \label{eq:box}
\end{equation}
i.e., a Cartesian product of closed intervals. The \emph{uniform box}
has $L_i = L$ and $U_i = U$ for every $i$; the
\emph{sample-dependent box} allows $U_i$ to vary with $i$ (and likewise
for $L_i$).
\end{definition}

\textit{Intuition.} Box constraints are the simplest non-trivial
component-wise feasibility region and the only kind that appears in the
dual problems of this paper; the box structure makes the two-variable
analytic update (Section~\ref{sec:smo-inner}) closed-form via clipping.
\textit{Forward-reference.} The standard $\varepsilon$-SVR dual
(Definition~\ref{def:eps-svr-dual}) has the uniform box $[0, C]$; the
MAPE-SVR dual (Definition~\ref{def:mape-svr-dual}) has the
\emph{sample-dependent} box $[0, C_k]$ with $C_k = 100C/y_k$, which is
the sole feasibility-region difference between the two formulations and
the locus of the structural-invariance result
Theorem~\ref{thm:invariance} of Section~\ref{sec:invariance}.

\begin{definition}[Slater point]
\label{def:slater}
A \emph{Slater point} of a convex optimization problem with inequality
constraints $g_i(\bfu) \le 0$ ($i = 1, \ldots, m$) and equality
constraints $A\bfu = \boldsymbol{b}$ is a feasible point $\bfu_0$ such
that the inequality constraints are \emph{strictly} satisfied:
$g_i(\bfu_0) < 0$ for every $i$. For a convex QP with box constraints
(Definition~\ref{def:box}), a Slater point is a feasible $\bfu_0$ with
$L_i < (\bfu_0)_i < U_i$ for every $i$ --- strictly inside every box.
\end{definition}

\textit{Intuition.} The Slater point is the standard hypothesis for the
constraint qualification that activates strong duality
(Theorem~\ref{thm:strong-duality} below); for a convex problem with
linear equality and component-wise inequality constraints, Slater's
condition reduces to the existence of a feasible point in the
\emph{interior} of every inequality. \textit{Forward-reference.} For the
dual QP of Section~\ref{sec:dual-kkt}, a Slater point is given by
$\alpha_k = \alpha_k^* = C_{k(i)}/2$ for every $k$, which satisfies the
equality constraint $\sum_k (\alpha_k - \alpha_k^*) = 0$ and lies
strictly inside every box $[0, C_{k(i)}]$. This existence is asserted
explicitly in Section~\ref{sec:dual-kkt} to justify strong duality and
the necessity-and-sufficiency of the KKT conditions.

\begin{definition}[Lagrangian, dual function, duality gap]
\label{def:lagrangian}
For a convex problem of the form $\min_{\bfu} f(\bfu)$ subject to
$g_i(\bfu) \le 0$ ($i = 1, \ldots, m$) and $A\bfu = \boldsymbol{b}$,
the \emph{Lagrangian} $\mathcal{L} : \R^n \times \R_+^m \times \R^p \to
\R$ is
\begin{equation*}
  \mathcal{L}(\bfu; \boldsymbol{\lambda}, \boldsymbol{\nu}) := f(\bfu) + \sum_{i=1}^{m}\lambda_i\,g_i(\bfu) + \boldsymbol{\nu}^\top(A\bfu - \boldsymbol{b}),
\end{equation*}
the \emph{dual function} $g : \R_+^m \times \R^p \to \R \cup \{-\infty\}$ is
\begin{equation*}
  g(\boldsymbol{\lambda}, \boldsymbol{\nu}) := \inf_{\bfu \in \R^n}\,\mathcal{L}(\bfu; \boldsymbol{\lambda}, \boldsymbol{\nu}),
\end{equation*}
and the \emph{duality gap} at primal-dual feasible $(\bfu,
\boldsymbol{\lambda}, \boldsymbol{\nu})$ is $f(\bfu) -
g(\boldsymbol{\lambda}, \boldsymbol{\nu}) \ge 0$. \emph{Strong duality}
holds when this gap is zero at the optimum.
\end{definition}

\textit{Intuition.} The Lagrangian relaxes the constraints into the
objective with multipliers; the dual function is the pointwise minimum
of $\mathcal{L}$ in $\bfu$, always a concave function of the
multipliers; strong duality is the property that the primal minimum and
the dual maximum coincide, which holds for every convex problem
satisfying a constraint qualification
(Theorem~\ref{thm:strong-duality}). \textit{Forward-reference.} The
Lagrangian appears explicitly in the primal-to-dual derivation of
Section~\ref{sec:standard-eps-svr} and is restated for the dual QP itself
in Section~\ref{sec:dual-kkt}, with multipliers $\rho$ for the equality
constraint and $\lambda_i, \mu_i$ for the box constraints.

\begin{theorem}[Strong duality for convex QP]
\label{thm:strong-duality}
For a convex QP~\eqref{eq:convex-qp} with $P \succeq 0$ and a non-empty
bounded polytope feasibility region, \emph{strong duality holds}: the
primal optimal $p^*$ equals the dual optimal $d^*$, and both are
attained.
\end{theorem}

\begin{proof}
Linear inequality constraints satisfy Slater's condition trivially
(relative interior of a non-empty polytope is non-empty); strong duality
and attainment follow from
Boyd-Vandenberghe~\cite[\S 5.2.3]{BoydVandenberghe2004} or
Rockafellar~\cite[\S 28]{Rockafellar1970}.
\end{proof}

\textit{Forward-reference.} Theorem~\ref{thm:strong-duality} justifies
the necessary-and-sufficient KKT characterization of
Section~\ref{sec:dual-kkt}: the dual QP is convex with $P = [\Omega,
-\Omega; -\Omega, \Omega] \succeq 0$ (since $\boldsymbol{v}^\top P
\boldsymbol{v} = (\boldsymbol{v}_1 - \boldsymbol{v}_2)^\top \Omega
(\boldsymbol{v}_1 - \boldsymbol{v}_2) \ge 0$ for $\boldsymbol{v} =
[\boldsymbol{v}_1; \boldsymbol{v}_2]$), and the box $\prod_k [0,
C_k]^2$ is bounded for $y_k > 0$.

\begin{theorem}[Sion's minimax]
\label{thm:sion}
Let $X \subseteq \R^n$ be non-empty convex compact, $Y \subseteq
\R^m$ non-empty convex, and $\phi : X \times Y \to \R$ such that
$\phi(\cdot, \boldsymbol{y})$ is convex and lower semi-continuous on $X$
and $\phi(\bfx, \cdot)$ is concave and upper semi-continuous on $Y$.
Then
\begin{equation}
  \min_{\bfx \in X}\,\sup_{\boldsymbol{y} \in Y}\,\phi(\bfx, \boldsymbol{y}) = \sup_{\boldsymbol{y} \in Y}\,\min_{\bfx \in X}\,\phi(\bfx, \boldsymbol{y}).
  \label{eq:sion}
\end{equation}
\end{theorem}

\begin{proof}
Sion~\cite{Sion1958}; textbook treatment in
Bertsekas-Nedi\'{c}-Ozdaglar~\cite[\S 3.4]{BertsekasNedicOzdaglar2003}.
\end{proof}

\textit{Forward-reference.} Sion's theorem justifies the swap of $\inf$
and $\sup$ used in the primal-to-dual derivation of
Section~\ref{sec:standard-eps-svr} and in Section~\ref{sec:dual-kkt}: the
Lagrangian is convex in $\bfu$ and affine (hence concave) in
$(\boldsymbol{\lambda}, \boldsymbol{\nu})$, so $\inf_{\bfu}
\sup_{\lambda, \nu} \mathcal{L} = \sup_{\lambda, \nu} \inf_{\bfu}
\mathcal{L}$.

\begin{definition}[Saddle point]
\label{def:saddle}
A \emph{saddle point} of the Lagrangian $\mathcal{L}$ on $\R^n \times
(\R_+^m \times \R^p)$ is a triple $(\bfu^*, \boldsymbol{\lambda}^*,
\boldsymbol{\nu}^*)$ satisfying
\begin{equation}
  \mathcal{L}(\bfu^*; \boldsymbol{\lambda}, \boldsymbol{\nu}) \le \mathcal{L}(\bfu^*; \boldsymbol{\lambda}^*, \boldsymbol{\nu}^*) \le \mathcal{L}(\bfu; \boldsymbol{\lambda}^*, \boldsymbol{\nu}^*), \qquad \forall \bfu,\; \forall (\boldsymbol{\lambda}, \boldsymbol{\nu}) \in \R_+^m \times \R^p.
  \label{eq:saddle}
\end{equation}
\end{definition}

\textit{Intuition.} A saddle point minimizes $\mathcal{L}$ in $\bfu$
(for the optimal multipliers) and maximizes $\mathcal{L}$ in
$(\boldsymbol{\lambda}, \boldsymbol{\nu})$ (for the optimal primal);
existence of a saddle point is \emph{equivalent} to strong duality plus
attainment, by Rockafellar~\cite[\S 36, Theorem 36.6]{Rockafellar1970}.
\textit{Forward-reference.} The saddle-point characterization is the
geometric form of the KKT optimality result of
Section~\ref{sec:dual-kkt}: at any saddle point of the dual QP's
Lagrangian, the multipliers $\lambda_i, \mu_i$ jointly characterize the
primal $\bfu^*$ via the case-analysis Table~\ref{tab:kkt-states} of
Section~\ref{sec:dual-kkt}, which the SMO algorithm uses as its
convergence test through the dual-threshold trick of
Keerthi-Shevade-Bhattacharyya-Murthy~\cite{KeerthiShevadeBhattacharyyaMurthy2001NeuralComp}.

\begin{definition}[KKT conditions]
\label{def:kkt}
For a convex problem $\min f(\bfu)$ subject to $g_i(\bfu) \le 0$ and
$A\bfu = \boldsymbol{b}$, with $f$ and the $g_i$ differentiable, the
\emph{Karush--Kuhn--Tucker (KKT) conditions} at $\bfu^*$ with
multipliers $(\boldsymbol{\lambda}^*, \boldsymbol{\nu}^*) \in \R_+^m
\times \R^p$ are
\begin{equation}
\begin{aligned}
  & \nabla f(\bfu^*) + \sum_{i=1}^{m}\lambda_i^*\,\nabla g_i(\bfu^*) + A^\top \boldsymbol{\nu}^* = 0 && \text{(stationarity)},\\
  & g_i(\bfu^*) \le 0,\quad A\bfu^* = \boldsymbol{b} && \text{(primal feasibility)},\\
  & \lambda_i^* \ge 0 && \text{(dual feasibility)},\\
  & \lambda_i^*\,g_i(\bfu^*) = 0,\quad i = 1, \ldots, m && \text{(complementary slackness)}.
\end{aligned}
\label{eq:kkt}
\end{equation}
\end{definition}

\textit{Intuition.} Stationarity says the negative gradient of the
objective lies in the conic combination of active-constraint gradients;
dual feasibility ensures multipliers are non-negative for inequalities;
complementary slackness says only \emph{active} constraints contribute
non-zero multipliers --- every inactive constraint has zero multiplier.
\textit{Forward-reference.} The KKT conditions are the optimality system
of Section~\ref{sec:dual-kkt}, where they are derived in detail for the
dual QP with multipliers $\rho$ (equality), $\lambda_i$ (lower box), and
$\mu_i$ (upper box), and rewritten via the signed-effective-gradient
$\tau_i = -s_i G_i$ to produce the convergence test of
Section~\ref{sec:smo-inner} and the case-analysis
Table~\ref{tab:kkt-states}.

\begin{definition}[Complementary slackness]
\label{def:complementary-slackness}
The \emph{complementary slackness condition} of~\eqref{eq:kkt} is the
property that for every inequality constraint $g_i(\bfu) \le 0$, either
$g_i(\bfu^*) = 0$ (the constraint is \emph{active} at the optimum) or
$\lambda_i^* = 0$ (the multiplier is zero), but not both can be strictly
violated. Equivalently, $\lambda_i^* g_i(\bfu^*) = 0$ for every $i$.
\end{definition}

\textit{Intuition.} Complementary slackness is the algebraic shorthand
for the geometric content of optimality: the optimal multiplier is
non-zero only for constraints that are \emph{binding} --- the optimizer
is sitting on those constraint surfaces --- and is zero for constraints
that are slack. \textit{Forward-reference.} In Section~\ref{sec:dual-kkt},
complementary slackness produces the three regimes of $\tau_i$ at
optimality (Table~\ref{tab:kkt-states}): at $u_i = 0$, $\mu_i = 0$ and
$\lambda_i \ge 0$, so $\tau_i \in (-\infty, -\rho]$ for $\alpha$-type or
$\tau_i \in [-\rho, +\infty)$ for $\alpha^*$-type; at $u_i = C_{k(i)}$,
the roles swap; at $0 < u_i < C_{k(i)}$ (free), both multipliers vanish
and $\tau_i = -\rho$ exactly. These three regimes are the basis of the
working-set feasibility sets $\Iup, \Idown$ of
Section~\ref{sec:smo-inner} and the shrinking criteria of
Section~\ref{sec:bias-shrink}.

\subsection{Standard SMO Machinery}
\label{sec:standard-smo}

The dual QP of Section~\ref{sec:standard-eps-svr} is convex, but its
dimension $2N$ and dense Hessian $P = [\Omega, -\Omega; -\Omega,
\Omega]$ make general-purpose interior-point and active-set solvers
prohibitive at $N \in [10^3, 10^5]$ --- the range encountered in modern
regression practice --- where $\Omega$ alone requires $\Theta(N^2)$
storage; see~\cite{RyuYin2022} for a broad treatment of large-scale
convex optimization and the tradeoffs between decomposition families at
this scale. The SMO algorithm of Platt~\cite{Platt1998, Platt1999}, with
the second-order working-set rule of
Fan-Chen-Lin~\cite{FanChenLin2005JMLR} and the cache + shrinking
machinery of LIBSVM~\cite{Chang2011}, is the production-standard
alternative. Throughout, $\bfu \in \R^{2N}$ is the stacked dual vector,
$\bfG = P\bfu + \bfq$ is the gradient of $f(\bfu) = \tfrac{1}{2}\bfu^\top
P \bfu + \bfq^\top \bfu$, and $\bftau = -\bfs \odot \bfG$ is the
signed-effective gradient ($\bfs$ from Definition~\ref{def:sign-vector},
$\odot$ Hadamard product).

\paragraph{The decomposition principle.}
Platt~\cite{Platt1998} introduced SMO in response to the computational
ceiling of \emph{chunking} --- Vapnik's pre-1998 approach in which a
working subset is optimized against a frozen remainder by an off-the-shelf
QP solver. Chunking inherits the $O(|\text{chunk}|^3)$ factorization cost
of its inner solver, and convergence requires chunks large enough to
capture all support vectors. SMO's contribution is structural: by
shrinking the chunk to size $k = 2$ --- the \emph{minimal feasible block
size} --- the inner subproblem becomes a one-dimensional convex quadratic
admitting a closed-form analytic solution, and the outer loop avoids
matrix factorization altogether. Each SMO iteration consists of (i)
selecting a working pair $(i^*, j^*)$ via the WSS3 rule, (ii) computing
in closed form the optimal joint update subject to the equality
constraint of~\eqref{eq:eps-svr-dual} and the box constraints, and (iii)
updating $\bfG$ incrementally via two columns of $\Omega$. Per-iteration
cost is $O(|\Aactive|)$ for the scan and gradient update plus two
kernel-column accesses; with the cache, uncached cost amortizes to
$O(|\Aactive|)$. Empirical iteration counts scale as $O(N^2)$ on dense
problems and $O(N)$ on sparse ones.

\paragraph{Working-set selection.}
The choice $k = 2$ for the working-set size is not arbitrary. Consider
$k = 1$: a single-variable update $u_{i^*} \leftarrow u_{i^*} + \delta$
would violate the equality constraint $\sum_{i=1}^{2N} s_i u_i = 0$. The
equality constraint is a one-dimensional linear subspace in $\R^{2N}$;
any feasible direction must lie in its null space, and a coordinate
direction $\boldsymbol{e}_i$ does not. Hence $k = 1$ is infeasible. At
the other extreme, $k > 2$ is feasible --- the equality constraint
imposes a single linear relation, leaving $k - 1$ degrees of freedom ---
but solving the resulting $k$-variable subproblem requires its own QP
solver, defeating the very motivation for decomposition. With $k = 2$,
the equality constraint reduces the subproblem to \emph{one} free
direction, which combined with the convex-quadratic objective yields the
closed-form one-dimensional minimization. The choice $k = 2$ is therefore
the \emph{minimal feasible block size}.

\paragraph{Equality-constraint reduction.}
Once $(i^*, j^*)$ are chosen, the joint update must satisfy
$s_{i^*}\Delta u_{i^*} + s_{j^*}\Delta u_{j^*} = 0$. Parameterizing
$\Delta u_{i^*} = s_{i^*}\delta$ for some scalar $\delta$ gives $\Delta
u_{j^*} = -s_{j^*}\delta$, so the joint update is fully described by the
single direction $\boldsymbol{d} \in \R^{2N}$ with $d_{i^*} = s_{i^*}$,
$d_{j^*} = -s_{j^*}$, and $d_\ell = 0$ otherwise. The restricted
objective along this direction becomes $h(\delta) = f(\bfu +
\boldsymbol{d}\delta) = f(\bfu) + (\bfG^\top \boldsymbol{d})\delta +
\tfrac{1}{2}(\boldsymbol{d}^\top P \boldsymbol{d})\delta^2$.

\paragraph{Direction-of-descent argument.}
The directional derivative at $\delta = 0$ is $h'(0) = \bfG^\top
\boldsymbol{d} = s_{i^*}G_{i^*} - s_{j^*}G_{j^*} = -\tau_{i^*} +
\tau_{j^*} = -(\tau_{i^*} - \tau_{j^*})$. The pair $(i^*, j^*)$ defines
a \emph{descent direction} if and only if $h'(0) < 0$, equivalently
$\tau_{i^*} > \tau_{j^*}$. The KKT-violation magnitude is $\Delta :=
\tau_{i^*} - \tau_{j^*} \ge 0$, with $\Delta = 0$ at optimality.

\paragraph{Feasibility sets.}
The two feasibility sets restrict the candidate indices, reflecting the
box constraints' interaction with the chosen direction: increasing
$u_{i^*}$ is feasible only if $u_{i^*}$ is below its upper bound, and
decreasing $u_{j^*}$ is feasible only if $u_{j^*}$ is above its lower
bound. For uniform $\varepsilon$-SVR these sets are
\begin{equation*}
  \Iup = \{k \le N : \alpha_k < C\} \cup \{N+k : \alpha_k^* > 0\}, \qquad \Idown = \{k \le N : \alpha_k > 0\} \cup \{N+k : \alpha_k^* < C\}.
\end{equation*}
The MAPE analogs are obtained by replacing $C$ with $C_k$ in each set
membership test (Definition~\ref{def:Iup-Idown-mape} of
Section~\ref{sec:smo-inner}).

\paragraph{Maximal violating pair (MVP) and dual-threshold trick.}
The first-order steepest-descent rule selects $i^* = \argmax_{i \in
\Iup} \tau_i$, $j^* = \argmin_{j \in \Idown} \tau_j$. Platt's original
1998 algorithm used a single-threshold heuristic; the MVP form became
standard via
Keerthi-Shevade-Bhattacharyya-Murthy~\cite{KeerthiShevadeBhattacharyyaMurthy2001NeuralComp},
who maintained $\tau_{i^*}$ and $\tau_{j^*}$ separately, eliminating
convergence oscillations and producing a clean monotone-descent
algorithm. MVP is first-order: it ignores the curvature $\eta$ that
governs the actual step size, producing iteration counts 30--50\% higher
than necessary.

\paragraph{Second-order rule (WSS3, Fan-Chen-Lin 2005).}
\cite[eq.~20]{FanChenLin2005JMLR} retains the MVP $i^*$-choice but
selects $j^*$ to maximize the predicted one-step gain $-\tfrac{1}{2}
(\tau_{i^*} - \tau_j)^2/\eta_{i^*, j}$, where $\eta_{i^*, j} =
\Omega_{k(i^*)k(i^*)} - 2\Omega_{k(i^*)k(j)} + \Omega_{k(j)k(j)}$ is the
curvature of the $(i^*, j)$ restricted direction:
\begin{equation*}
  j^* = \argmax_{j \in \Idown,\ \tau_j < \tau_{i^*}} (\tau_{i^*} - \tau_j)^2/\eta_{i^*, j}.
\end{equation*}
Empirically WSS3 reduces iteration count by
30--50\%~\cite[Table~2]{FanChenLin2005JMLR},
\cite{YuLiLiu2023PatternRecognition}; the maximum-gain variant of
Glasmachers-Igel~\cite{GlasmachersIgel2006JMLR,
GlasmachersIgel2008NeuralComp} adds another 5--15\%.

\paragraph{Tie-breaking.}
When multiple indices achieve the same maximum, the canonical
tie-breaking rule (LIBSVM convention) is to choose the smallest
training-point index $k(i)$, then the smallest $i$. This makes the
algorithm deterministic and reproducible across runs and platforms ---
a property that the present paper preserves in
Algorithm~\ref{alg:smo} of Section~\ref{sec:bias-shrink}.

\paragraph{The two-variable analytic update.}
Given the working pair $(i^*, j^*)$, the joint update is parameterized
by $u_{i^*} \leftarrow u_{i^*} + s_{i^*}\delta$ and $u_{j^*} \leftarrow
u_{j^*} - s_{j^*}\delta$ for a scalar step $\delta$. The restricted
one-dimensional objective is $h(\delta) = f(\bfu) - \Delta\,\delta +
\tfrac{1}{2}\eta\,\delta^2$, with slope $\Delta = \tau_{i^*} -
\tau_{j^*} \ge 0$ and curvature $\eta = \boldsymbol{d}^\top P
\boldsymbol{d}$.

\paragraph{Curvature derivation.}
With $\boldsymbol{d}$ supported on $\{i^*, j^*\}$ with values $s_{i^*}$
and $-s_{j^*}$, write $p = k(i^*)$ and $q = k(j^*)$ via the
training-index map (Definition~\ref{def:k-of-i}). Direct computation
gives $\eta = (s_{i^*})^2 P_{i^*i^*} + 2(s_{i^*})(-s_{j^*})P_{i^*j^*} +
(-s_{j^*})^2 P_{j^*j^*}$. The block structure of $P$ gives $P_{i^*i^*} =
\Omega_{pp}$, $P_{j^*j^*} = \Omega_{qq}$, and $P_{i^*j^*} =
s_{i^*}s_{j^*}\Omega_{pq}$, so the cross term in $\eta$ becomes $-2
s_{i^*}s_{j^*} \cdot s_{i^*}s_{j^*}\Omega_{pq} = -2(s_{i^*}s_{j^*})^2
\Omega_{pq} = -2\Omega_{pq}$. Combined with $(s_{i^*})^2 = (s_{j^*})^2 = 1$:
\begin{equation*}
  \eta = \Omega_{pp} - 2\Omega_{pq} + \Omega_{qq},
\end{equation*}
regardless of the sign-block combination of $i^*$ and $j^*$. This
\emph{sign-invariance} is the structural reason the $\varepsilon$-SVR
analytic update is \emph{one} formula rather than four --- and, by
Theorem~\ref{thm:invariance} of Section~\ref{sec:invariance}, it is the
same formula in MAPE-SVR.

\paragraph{Unconstrained minimum.}
Differentiating gives $h'(\delta) = -\Delta + \eta\delta$, so the
unconstrained minimum (assuming $\eta > 0$) is at $\delta_{\mathrm{unc}}
= \Delta/\eta$.

\begin{definition}[Clipping function]
\label{def:clip}
The \emph{clipping function} $\mathrm{clip} : \R \times \R \times \R
\to \R$ is
\begin{equation}
  \mathrm{clip}(x;\, a,\, b) := \max(a,\, \min(b,\, x)) = \begin{cases} a, & x < a\\ x, & a \le x \le b\\ b, & x > b. \end{cases}
  \label{eq:clip}
\end{equation}
For $a \le b$, $\mathrm{clip}(x; a, b)$ equals the orthogonal Euclidean
projection of $x$ onto the closed interval $[a, b]$.
\end{definition}

\textit{Intuition.} Clipping a real number to an interval is the
simplest form of constrained projection, and on a one-dimensional convex
feasible set with a strictly convex one-dimensional objective the
projection of the unconstrained minimum is the constrained minimum.
\textit{Forward-reference.} The clipping function is invoked in
Section~\ref{sec:smo-inner} (definition of $\delta_{\max} = \min(R_{i^*},
R_{j^*})$ and the optimal feasible step $\delta^* =
\mathrm{clip}(\delta_{\mathrm{unc}}; 0, \delta_{\max})$); it also
appears in Algorithm~\ref{alg:smo} of Section~\ref{sec:bias-shrink}.

\paragraph{Optimal feasible step.}
The optimal step for $\eta > 0$ is the projection of
$\delta_{\mathrm{unc}}$ onto $[0, \delta_{\max}]$:
\begin{equation*}
  \delta^* = \mathrm{clip}(\delta_{\mathrm{unc}};\, 0,\, \delta_{\max}) = \min(\Delta/\eta,\, \delta_{\max}).
\end{equation*}
Lower clipping at $0$ is unnecessary because $\Delta \ge 0$ and $\eta >
0$ imply $\delta_{\mathrm{unc}} \ge 0$. If $\eta = 0$
(e.g.,~$\bfx_p = \bfx_q$), the unconstrained minimum is unbounded;
the descent test reduces to $\Delta > \tfrac{1}{2}\eta\delta_{\max}$,
which for $\eta = 0$ becomes $\Delta > 0$ (always true unless converged),
giving $\delta^* = \delta_{\max}$. For $\eta < 0$ (impossible when $P
\succeq 0$, but possible for the MAPE-SVR-Sym variant with $a = -1$,
Section~\ref{sec:symmetric}), the same test serves as the genuine
local-descent criterion. Standard practice~\cite{Chang2011} floors
$\eta$ at $\eta_{\min} \approx 10^{-12}$ before division.

\paragraph{Variable update.}
With $\delta^*$ in hand, the dual variables update as $u_{i^*}
\leftarrow u_{i^*} + s_{i^*}\delta^*$, $u_{j^*} \leftarrow u_{j^*} -
s_{j^*}\delta^*$, with the equality constraint exactly preserved by
construction.

\paragraph{The incremental gradient update.}
After the joint update, the gradient $\bfG = P\bfu + \bfq$ changes by
$\bfG^{\mathrm{new}} - \bfG^{\mathrm{old}} = P\,\boldsymbol{d}\,\delta^*$.
Computing this update from scratch would cost $O(N^2)$ matrix-vector
multiplications, recovering the very expense SMO is designed to avoid.
The vector $\boldsymbol{d}$ has only \emph{two} nonzero entries, so
$P\boldsymbol{d}$ is a linear combination of \emph{two} columns of $P$
--- equivalently, two columns of $\Omega$.
Direct substitution gives $(P\boldsymbol{d})_\ell =
s_\ell\bigl(\Omega_{k(\ell), p} - \Omega_{k(\ell), q}\bigr)$, where the
simplification uses $(s_{i^*})^2 = (s_{j^*})^2 = 1$ and the block
structure of $P$. Hence the gradient update is $G_\ell \leftarrow G_\ell
+ s_\ell\,\delta^*\bigl(\Omega_{k(\ell), p} - \Omega_{k(\ell),
q}\bigr)$ for $\ell \in \Aactive$, equivalently in the
signed-effective-gradient form
\begin{equation*}
  \tau_\ell \leftarrow \tau_\ell - \delta^*\bigl(\Omega_{k(\ell), p} - \Omega_{k(\ell), q}\bigr).
\end{equation*}
This is the form invoked in Section~\ref{sec:smo-inner}.

\paragraph{Cost analysis.}
The update reads $|\Aactive|$ entries from each of columns $p$ and $q$
of $\Omega$. With kernel caching, cached columns cost $O(|\Aactive|)$
memory accesses; uncached columns cost $O(N)$ kernel evaluations, which
the cache amortizes to $O(|\Aactive|)$ over many iterations.

\paragraph{KKT optimality and the convergence test.}
By Theorem~\ref{thm:strong-duality}, the KKT conditions of
Definition~\ref{def:kkt} are necessary and sufficient for the convex QP.
Let $\rho \in \R$ be the equality multiplier and $\lambda_i, \mu_i \ge
0$ the box multipliers of $u_i \ge 0$ and $u_i \le C$. Stationarity
gives $\tau_i = \rho + s_i(\lambda_i - \mu_i)$; complementary slackness
forces $\lambda_i u_i = 0$ and $\mu_i(u_i - C) = 0$, producing three
regimes --- at $u_i = 0$ ($\mu_i = 0, \lambda_i \ge 0$), at $u_i = C$
($\lambda_i = 0, \mu_i \ge 0$), and at $0 < u_i < C$ ($\lambda_i = \mu_i
= 0$, $\tau_i = \rho$). These regimes (with MAPE analogs in
Table~\ref{tab:kkt-states}) are the basis of the sign-aware feasibility
sets.

\paragraph{Dual-threshold trick (Keerthi-Shevade).}
Maintain the KKT-allowed interval $-\rho \in [\tau_{j^*}, \tau_{i^*}]$
rather than a single bias estimate (Platt's original, oscillation-prone
form). Optimality holds when the interval collapses, $\Delta = \tau_{i^*}
- \tau_{j^*} \le \varepsilon_{\mathrm{tol}}$ --- the convergence test of
Algorithm~\ref{alg:smo}. The bias is recovered post-convergence as the
interval midpoint $b = \tfrac{1}{2}(\tau_{j^*} + \tau_{i^*})$.

\paragraph{The shrinking heuristic.}
In the late phase of SMO, most variables settle to a boundary and stop
moving, while a modest active set continues to be updated. The shrinking
heuristic of Joachims~\cite{Joachims1999} (SVM-light), refined for SMO
in LIBSVM~\cite{Chang2011}, temporarily removes from the scan those
variables predicted to remain on their boundary, reducing per-iteration
scan cost from $O(N)$ to $O(|\Aactive|)$.

\begin{definition}[Active and frozen sets]
\label{def:active-frozen}
For dual variables indexed by $\{1, \ldots, 2N\}$, the \emph{active set}
at the current iterate is
\begin{equation*}
  \Aactive := \{i \in \{1, \ldots, 2N\} : \text{variable $i$ has not been frozen by the shrinking heuristic}\},
\end{equation*}
and the \emph{frozen set} is its complement $\Afrozen := \{1, \ldots,
2N\} \setminus \Aactive$.
\end{definition}

\textit{Intuition.} The active set $\Aactive$ is the \emph{working set}
that the SMO outer loop scans on each iteration; the frozen set
$\Afrozen$ contains variables temporarily excluded from the scan,
accelerating per-iteration cost from $O(N)$ to $O(|\Aactive|)$.
\textit{Forward-reference.} The active and frozen sets are the central
data structure of Section~\ref{sec:bias-shrink}, where the four
shrinking criteria decide which variables transition from $\Aactive$ to
$\Afrozen$ and the reconstruction phase decides which transition back.

\begin{definition}[Extended active set]
\label{def:active-ext}
The \emph{extended active set} is
\begin{equation*}
  \Aactive^{\mathrm{ext}} := \Aactive \cup \{i + N : i \in \Aactive \cap \{1, \ldots, N\}\} \cup \{i - N : i \in \Aactive \cap \{N+1, \ldots, 2N\}\},
\end{equation*}
i.e., $\Aactive^{\mathrm{ext}}$ contains every dual index $i$ such that
\emph{either} $i$ \emph{or its paired counterpart at the same training
point} is in $\Aactive$.
\end{definition}

\textit{Intuition.} The gradient update needs to touch every $\ell$
whose pair $(\alpha_{k(\ell)}, \alpha_{k(\ell)}^*)$ has at least one
active member, because both members share the kernel column $\Omega_{:,
k(\ell)}$ that enters the update. \textit{Forward-reference.}
$\Aactive^{\mathrm{ext}}$ appears in Section~\ref{sec:smo-inner}: the
gradient update is over $\ell \in \Aactive^{\mathrm{ext}}$, not over
$\ell \in \Aactive$, to preserve correctness when one member of a pair
is frozen and the other is active.

\begin{definition}[Free vs.~boundary support vector]
\label{def:free-bound-sv}
The set of \emph{free support vectors} at the current iterate is
\begin{equation*}
  \mathcal{S}_{\mathrm{free}} := \{k \in \{1, \ldots, N\} : 0 < \alpha_k < C_k \;\text{or}\; 0 < \alpha_k^* < C_k\},
\end{equation*}
i.e., training points whose dual variables are strictly inside their
box. The set of \emph{boundary support vectors} $\mathcal{S}_{\mathrm{bound}}$
is its complement among the support set.
\end{definition}

\textit{Intuition.} Free support vectors live strictly inside the box
and are the only ones for which the KKT condition is the strict equality
$\tau_i = -\rho$ (Table~\ref{tab:kkt-states}); boundary support vectors
live on a face of the box and contribute the inequality side of the KKT
condition. \textit{Forward-reference.} The free-vs-boundary distinction
is central to the bias-recovery formula of
Section~\ref{sec:bias-shrink}: the bias $\hat b$ is computed as the
average of $\tau_i$ over $\mathcal{S}_{\mathrm{free}}$, since every free
variable satisfies $\tau_i = -\rho = b$.

\paragraph{Shrinking criteria.}
A boundary variable is shrinkable when its signed gradient $\tau_\ell$
lies outside the current $[\tau_{j^*}, \tau_{i^*}]$ interval in the
direction consistent with its bound. Concretely (uniform-$C$ case):
$\alpha_k = 0$ is shrinkable if $\tau_k < \tau_{j^*}$; $\alpha_k = C$ is
shrinkable if $\tau_k > \tau_{i^*}$; and analogously for $\alpha_k^*$.

\paragraph{Counter mechanism.}
Each variable maintains a counter that increments when the criterion
holds and resets when it fails; the variable is frozen when the counter
reaches $n_{\min}$ (LIBSVM default 5). Counter checks run every
$n_{\mathrm{check}} = \min(N, 1000)$ iterations.

\paragraph{Reconstruction and unshrinking.}
Shrinking is unsafe in principle: the global optimum may force a frozen
variable off its boundary. To guarantee finite termination, a
reconstruction phase activates when the active set has converged: the
full effective gradient is computed on all $2N$ variables, and
$\Delta^{\mathrm{full}}$ is recomputed. If $\Delta^{\mathrm{full}} >
\varepsilon_{\mathrm{tol}}$, the active set is restored to $\{1, \ldots,
2N\}$ and the algorithm continues; if $\Delta^{\mathrm{full}} \le
\varepsilon_{\mathrm{tol}}$, the algorithm terminates with global
$\varepsilon_{\mathrm{tol}}$-optimality.

\paragraph{The kernel cache.}
Each gradient update requires entries $\Omega_{k(\ell), p}$ and
$\Omega_{k(\ell), q}$ for $\ell \in \Aactive^{\mathrm{ext}}$. Without a
cache, computing these costs $O(|\Aactive^{\mathrm{ext}}|)$ kernel
evaluations per iteration --- a dominant cost on dense problems. The
kernel cache of LIBSVM~\cite{Chang2011} stores recently-used columns
under an LRU policy, amortizing the kernel-evaluation cost over the run.
With budget $M$ bytes and double-precision storage, the cache holds
$\lfloor M / (8N) \rfloor$ columns; a typical default $M = 200$~MB fits
$\sim 2{,}500$ columns at $N = 10^4$ and all columns at $N = 10^3$.

\paragraph{Per-iteration and convergence complexity.}
Per-iteration cost (warm cache) is dominated by $O(|\Aactive|)$ for the
scan, gradient update, and shrinking check; the analytic update is
$O(1)$. Theorem~5 of~\cite{FanChenLin2005JMLR} establishes that
WSS3-SMO terminates to $\varepsilon_{\mathrm{tol}}$-optimality in a
finite number of iterations for any convex QP with $P \succeq 0$ and
bounded feasible region; empirical counts scale as $O(N \cdot \kappa)$
with $\kappa$ a problem-difficulty factor.

\begin{theorem}[Finite termination of WSS3-SMO; restatement of {\cite[Theorem~5]{FanChenLin2005JMLR}}]
\label{thm:fcl-termination}
For any convex QP with $P \succeq 0$, bounded box constraints $-\infty 
L_i < U_i < +\infty$, and a single linear equality constraint, WSS3-SMO
terminates to $\varepsilon_{\mathrm{tol}}$-optimality in a finite number
of iterations for every $\varepsilon_{\mathrm{tol}} > 0$.
\end{theorem}

\paragraph{Inheritance to MAPE-SVR.}
Theorem~\ref{thm:invariance} of Section~\ref{sec:invariance}
(structural-invariance) shows the MAPE-SVR dual is a convex QP satisfying
the hypotheses of Theorem~\ref{thm:fcl-termination} verbatim: $P$ is the
same PSD matrix; the equality constraint is unchanged; the bounds $[0,
C_k]$ with $C_k = 100C/y_k$ are bounded for $y_k > 0$. The convergence
theorem applies to MAPE-SVR without modification; only the two localized
changes of Section~\ref{sec:mape-setup} enter Algorithm~\ref{alg:smo} of
Section~\ref{sec:bias-shrink}.

\subsection{The MAPE-SVR Adaptation by Analogy}
\label{sec:mape-setup}

\begin{definition}[MAPE-SVR primal]
\label{def:mape-svr-primal}
The \emph{MAPE-SVR primal} of
Benavides-Herrera et al.~\cite{benavides2025support, benavides2026unified} modifies the
standard primal (Definition~\ref{def:eps-svr-primal}) by replacing the
absolute residual with the percentage residual of
Definition~\ref{def:perc-residual} inside the tube constraint. For a
training set with strictly positive targets $y_k \in \mathcal{Y}
\subseteq \R_+$, the MAPE-SVR primal is
\begin{equation}
  \min_{w \in \mathcal{H},\, b \in \R,\, \xi, \xi^* \ge 0}\; \tfrac{1}{2}\|w\|_{\mathcal{H}}^2 + C\sum_{k=1}^{N}(\xi_k + \xi_k^*)
  \label{eq:mape-svr-primal}
\end{equation}
subject to, for every $k = 1, \ldots, N$,
\begin{equation*}
  \frac{|y_k - f(\bfx_k)|}{y_k} \cdot 100 \le \varepsilon + \xi_k\;\text{(or}\;\xi_k^*\text{)}, \quad \xi_k, \xi_k^* \ge 0,
\end{equation*}
where $f(\bfx) = \langle w, \varphi(\bfx)\rangle_{\mathcal{H}} + b$
and $\varepsilon > 0$ is the tube width measured \emph{in percentage
points} (not in the units of $y$).
\end{definition}

\textit{Intuition.} The MAPE-SVR primal differs from
Definition~\ref{def:eps-svr-primal} in exactly one substitution --- the
absolute residual is replaced with the percentage residual; everything
else (the regularizer, the slacks, the $C$-penalty) is identical. The
change in the tube-constraint denominator from $1$ to $y_k$ propagates
through the Lagrangian into a sample-specific coefficient on the slack
variable, producing the sample-dependent box constraint of
Proposition~\ref{prop:mape-dual} below. \textit{Forward-reference.}
Definition~\ref{def:mape-svr-primal} is the primal underlying the dual
problem~\eqref{eq:mape-dual} of Section~\ref{sec:dual-kkt}.

\begin{proposition}[Sample-dependent bound from MAPE primal]
\label{prop:mape-dual}
The MAPE-SVR primal (Definition~\ref{def:mape-svr-primal}) admits a
Lagrangian dual whose box constraints are sample-dependent:
\begin{equation}
  0 \le \alpha_k,\; \alpha_k^* \le \frac{100C}{y_k}, \qquad k = 1, \ldots, N.
  \label{eq:Ck-derivation}
\end{equation}
\end{proposition}

\begin{proof}
Rewrite the upper-tube constraint
of~\eqref{eq:mape-svr-primal} by multiplying both sides by $y_k/100 > 0$:
\begin{equation}
  \frac{|y_k - f(\bfx_k)|}{y_k} \cdot 100 \le \varepsilon + \xi_k \quad\Longleftrightarrow\quad |y_k - f(\bfx_k)| \le \frac{y_k}{100}(\varepsilon + \xi_k).
  \label{eq:rewrite-upper}
\end{equation}
Form the Lagrangian by introducing the multiplier $\alpha_k \ge 0$ for
the upper-tube constraint~\eqref{eq:rewrite-upper} and $\mu_k \ge 0$ for
$\xi_k \ge 0$. Stationarity with respect to $\xi_k$ yields
\begin{equation}
  \frac{\partial \mathcal{L}}{\partial \xi_k} = C - \frac{y_k}{100}\,\alpha_k - \mu_k = 0 \quad\Longrightarrow\quad \alpha_k = \frac{100}{y_k}(C - \mu_k) \le \frac{100C}{y_k} =: C_k,
  \label{eq:Ck-three-line}
\end{equation}
where the inequality uses $\mu_k \ge 0$. An identical computation for
the lower-tube slack $\xi_k^*$ yields $\alpha_k^* \le 100C/y_k = C_k$,
so both dual variables of training point $k$ share the same per-sample
bound $C_k$.
\end{proof}

\textit{Forward-reference.} This three-line derivation justifies the
central structural modification on which the entire paper rests: the
dual feasibility region becomes $\prod_{k=1}^{N}[0, C_k]^2$ rather than
$[0, C]^{2N}$. Theorems~\ref{thm:invariance}
and~\ref{thm:convergence} establish that this \emph{sole} dual-level
change has a \emph{highly localized} algorithmic consequence --- only
two SMO components require modification.

\begin{definition}[Sample-dependent box constraint]
\label{def:mape-svr-dual}
The \emph{sample-dependent box constraint} of the MAPE-SVR dual is the
per-sample upper bound
\begin{equation}
  C_k := \frac{100C}{y_k}, \qquad k = 1, \ldots, N,
  \label{eq:Ck-def}
\end{equation}
where $C > 0$ is the regularization parameter and $y_k > 0$ is the
$k$-th training target. The box for the dual variables of training point
$k$ is $[0, C_k]^2$, i.e., $\alpha_k \in [0, C_k]$ and $\alpha_k^* \in
[0, C_k]$ independently.
\end{definition}

\textit{Intuition.} Smaller targets receive larger box budgets --- the
model has more dual capacity to allocate to small-$y$ samples --- which
compensates for the fact that a fixed absolute residual $|y_k -
f(\bfx_k)|$ is a \emph{larger} percentage error when $y_k$ is small.
\textit{Forward-reference.} The sample-dependent bound is the only
feasibility-region difference between the standard $\varepsilon$-SVR
dual (Definition~\ref{def:eps-svr-dual}) and the MAPE-SVR dual; it is
the locus of the structural-invariance result
Theorem~\ref{thm:invariance} of Section~\ref{sec:invariance}, which
shows that the dependence on $C_k$ is confined to two SMO components.

\begin{definition}[Target dynamic range]
\label{def:dynamic-range}
The \emph{target dynamic range} of the training data is
\begin{equation}
  \rho_y := \frac{\max_{k=1, \ldots, N} y_k}{\min_{k=1, \ldots, N} y_k}.
  \label{eq:dynamic-range}
\end{equation}
For $y_k > 0$ uniformly, $\rho_y \ge 1$.
\end{definition}

\textit{Intuition.} The dynamic range measures the \emph{heterogeneity}
of target magnitudes: $\rho_y \approx 1$ when all targets are of similar
magnitude, $\rho_y \gg 1$ when the smallest and largest targets differ
by orders of magnitude (as is common in electricity-demand forecasting
where peak demand can be 5--10 times trough demand).
\textit{Forward-reference.} $\rho_y$ governs the asymmetry of the
per-sample bound $C_k = 100C/y_k$ --- small-$y$ samples receive a large
$C_k$ and large-$y$ samples receive a small $C_k$ --- and through this
asymmetry it governs the shrinking-asymmetry result
Lemma~\ref{lem:asymmetry} of Section~\ref{sec:bias-shrink}, which
quantifies the gap $2y_k\varepsilon/100$ between paired $\alpha$- and
$\alpha^*$-shrinking thresholds.

\paragraph{Linear coefficient and Hessian — by-analogy block-substitution.}
By the block-substitution + kernel-trick computation of
Section~\ref{sec:standard-eps-svr}, the linear coefficient of the
MAPE-SVR dual is $\bfq = [\bfy(\varepsilon/100 - 1),\,
\bfy(\varepsilon/100 + 1)]^\top$ (rather than the standard
$[\varepsilon\onevec - \bfy,\, \varepsilon\onevec + \bfy]^\top$ ---
Remark~\ref{rem:sign-convention}). The Hessian $P = [\Omega, -\Omega;
-\Omega, \Omega]$ is unchanged: it arises only from
$\tfrac{1}{2}\|w\|^2$ via the representer
expansion~\eqref{eq:w-stationarity}, identical in both formulations. The
MAPE-SVR dual (Section~\ref{sec:dual-kkt}) thus differs from the standard
$\varepsilon$-SVR dual in three localized ways summarized in
Table~\ref{tab:mape-vs-standard}.

\begin{table}[ht]
  \centering
  \caption{The three localized differences between the standard
  $\varepsilon$-SVR dual and the MAPE-SVR dual.}
  \label{tab:mape-vs-standard}
  \begin{tabular}{@{}lll@{}}
    \toprule
    Component & Standard $\varepsilon$-SVR dual (Def.~\ref{def:eps-svr-dual}) & MAPE-SVR dual (\S\ref{sec:dual-kkt})\\
    \midrule
    Linear coefficient $\bfq$ & $[\varepsilon\onevec - \bfy,\, \varepsilon\onevec + \bfy]$ & $[\bfy(\varepsilon/100 - 1),\, \bfy(\varepsilon/100 + 1)]$\\
    Box constraint & $[0, C]$ uniform & $[0, C_k]$ sample-dependent (Def.~\ref{def:mape-svr-dual})\\
    Tube width $\varepsilon$ & in units of $y$ via Def.~\ref{def:eps-loss} & in percentage points via Defs.~\ref{def:mape-loss}--\ref{def:perc-residual}\\
    \bottomrule
  \end{tabular}
\end{table}

Everything else --- the Hessian $P$ via representer
expansion~\eqref{eq:w-stationarity}, the equality constraint
$[\onevec^\top, -\onevec^\top]\,\bfu = 0$ via
bias-stationarity~\eqref{eq:b-stationarity}, the Mercer kernel $K$, the
convex-QP class with $P \succeq 0$, the Slater-point existence
($\alpha_k = \alpha_k^* = C_k/2$), the strong-duality conclusion
(Theorem~\ref{thm:strong-duality}), the KKT conditions
(Definition~\ref{def:kkt}) with multipliers $\rho, \lambda_i, \mu_i$,
the clipping form (Definition~\ref{def:clip}), and the
convergence-theorem inheritance from
Theorem~\ref{thm:fcl-termination} via Theorem~\ref{thm:invariance} of
Section~\ref{sec:invariance} --- is identical between the two
formulations.

\paragraph{The structural-invariance pre-announcement.}
The structural-invariance theorem of Section~\ref{sec:invariance}
(Theorem~\ref{thm:invariance}) shows that the algorithmic consequence is
\emph{even more localized} than the dual difference itself: only two SMO
components --- the working-set candidate sets $\Iup, \Idown$ of
Section~\ref{sec:smo-inner} and the clipping room $R_{i^*}, R_{j^*}$ ---
require modification. The curvature formula, the gradient update, the
shrinking criteria's structural form, the kernel cache, and the
convergence-theorem inheritance are all structurally identical to
standard SMO. The implementation consequence is concrete: any existing
SMO solver for $\varepsilon$-SVR can be adapted to MAPE by replacing the
scalar $C$ with the vector $\boldsymbol{C} = (C_1, \ldots, C_N)$ in two
localized steps; Appendix~\ref{app:libsvm} gives the explicit LIBSVM
diff (fewer than fifteen lines).

\subsection{Related Work}
\label{sec:related}

This section surveys the relevant literature in eight thematic clusters
relevant to the present contribution.

\paragraph{SMO for SVR --- historical lineage.}
Platt's original Sequential Minimal Optimization~\cite{Platt1998,
Platt1999} was designed for SVM classification, where each two-variable
subproblem has a closed-form analytic solution: with two variables
$(\alpha_i, \alpha_j)$ free and the equality constraint $\alpha_i +
\alpha_j = \mathrm{const}$, the problem reduces to a one-dimensional
convex quadratic, solvable by computing the unconstrained minimum and
clipping to the box. Platt's algorithm replaces general-purpose chunking
+ quadratic programming~\cite{Vapnik1995, Joachims1999} with this size-2
analytic decomposition, eliminating the dependency on third-party QP
solvers and achieving up to three orders of magnitude speedup on
sparse-data benchmarks where the support-vector set is small relative to
the training set.

Keerthi-Shevade-Bhattacharyya-Murthy~\cite{KeerthiShevadeBhattacharyyaMurthy2001NeuralComp}
identified an inefficiency in Platt's single-threshold scheme: Platt
used a single estimate of the bias $b$ to evaluate KKT optimality,
leading to oscillation when the true bias lies far from this estimate.
Their dual-threshold modification maintains separate upper- and
lower-bound estimates (the LIBSVM-canonical $b_{\mathrm{up}}$ and
$b_{\mathrm{low}}$, denoted $-\rho \in [\tau_{j^*}, \tau_{i^*}]$ in the
present paper), eliminating the oscillation and delivering convergence
in $\sim 2$--$10$ fewer iterations on standard benchmarks. This
dual-threshold scheme is the production standard in LIBSVM v3+ and in
the present paper's Algorithm~\ref{alg:smo}.

Flake and Lawrence~\cite{Flake2002} resolved the
$\varepsilon$-SVR-specific challenge of paired variables $(\alpha_k,
\alpha_k^*)$ with complementarity $\alpha_k\alpha_k^* = 0$ by
reformulating in terms of $\beta_k = \alpha_k - \alpha_k^*$, halving the
variable count from $2N$ to $N$ and recovering the two-variable
subproblem with the same analytic update as classification SMO. The
present paper retains the explicit $(\alpha_k, \alpha_k^*)$ formulation
rather than the $\beta$-reformulation, because the sample-dependent
bound $\alpha_k, \alpha_k^* \in [0, C_k]$ enters more transparently in
the explicit formulation: in the $\beta$-reformulation the bound becomes
$\lvert\beta_k\rvert \le C_k$ with the additional constraint that $\beta_k =
+|\beta_k|$ corresponds to $\alpha_k > 0$ and $\beta_k = -|\beta_k|$ to
$\alpha_k^* > 0$, which obscures the per-variable feasibility tracking
required by the WSS3 working-set rule.

Working-set selection was strengthened by
Fan-Chen-Lin~\cite{FanChenLin2005JMLR}. Their second-order scoring rule,
denoted WSS3 in their numbering, replaces the maximal-violating-pair
(MVP) criterion of selecting $j^* = \argmin_j \tau_j$ with the
gain-weighted criterion of Section~\ref{sec:smo-inner}.
Theorem~5 of~\cite{FanChenLin2005JMLR} establishes finite termination of
WSS3-SMO to $\varepsilon_{\mathrm{tol}}$-optimality for any convex
quadratic program with positive-semidefinite Hessian and bounded
feasible set, regardless of the specific bound structure: precisely the
convergence inheritance the present paper invokes.

Glasmachers and Igel~\cite{GlasmachersIgel2006JMLR} proposed a
maximum-gain variant of WSS3 that extends gain weighting to both the
$i^*$- and $j^*$-selection steps simultaneously, yielding an additional
$\sim$ 5--15\% iteration reduction on the Adult and W8a benchmarks; their
later work~\cite{GlasmachersIgel2008NeuralComp} extended second-order
SMO to online and active learning, with the LASVM solver demonstrating
practical effectiveness on streaming data. The asymptotic equivalence
between batch SMO and the LASVM streaming framework was established in
the original LASVM paper of~\cite[\S 3, Theorem 1]{BordesErtekinWestonBottou2005JMLR}.
The LIBSVM library~\cite{Chang2011} combined these advances with the
shrinking heuristic of Joachims~\cite{Joachims1999} and a
least-recently-used kernel cache, yielding the de facto standard SVR
solver in production-grade machine-learning toolchains. Recent
refinements include the three-term conjugate variant
TCSMO~\cite{YuLiLiu2023PatternRecognition}, which augments the WSS3 step
direction with conjugate-gradient-style memory of the previous two
iterations, reducing total iteration count by 20--35\% on twelve
regression-benchmark datasets at a per-iteration overhead of $\sim 60$
lines of additional code.

All of the above --- Platt 1998/1999, Keerthi-Shevade-Bhattacharyya-Murthy
2001, Flake-Lawrence 2002, Fan-Chen-Lin 2005, Glasmachers-Igel 2006/2008,
Bordes-Ertekin-Weston-Bottou 2005, Chang-Lin 2011 LIBSVM, Yu-Li-Liu 2023
TCSMO --- assume \emph{uniform} box constraints $\alpha_k, \alpha_k^*
\in [0, C]$ for every training point. The structural-invariance result
of the present paper (Theorem~\ref{thm:invariance} of
Section~\ref{sec:invariance}) shows that this uniformity assumption is
not load-bearing for the SMO machinery itself: replacing $C$ with the
per-sample vector $(C_1, \ldots, C_N)$ leaves the curvature, gradient
bookkeeping, working-set scoring, and convergence machinery structurally
unchanged. The uniformity assumption is load-bearing only for the
working-set feasibility sets and the analytic-update clipping bounds ---
the two components Theorem~\ref{thm:invariance} explicitly identifies.

\paragraph{Loss-modified SVR.}
Modifying the SVR loss to target application-specific error structure
has been explored along two main directions. Asymmetric
$\varepsilon$-insensitive and pinball-loss variants~\cite{Anand2020}
allow different penalties above and below the regression tube, targeting
quantile estimation rather than conditional mean regression. Robust
variants replace the $\varepsilon$-insensitive loss with Huber or
bounded losses to reduce sensitivity to outliers~\cite{Smola2004};
recent work in this direction includes the wave loss of Akhtar, Tanveer,
and Arshad~\cite{AkhtarTanveerArshad2024PatternRecognition}. What these
formulations share is that the box constraints on the Lagrange
multipliers remain uniform across training points: the loss modification
enters the dual objective or the tube width, not the feasibility set
itself.

\paragraph{Weighted and instance-weighted SVR.}
A separate line of work generalizes SVR to handle sample-specific
importance weights $w_k$, typically by rescaling the regularization
parameter to $C_i^{\mathrm{eff}} = C w_k$. Suykens et
al.~\cite{Suykens2002weighted} proposed weighted least-squares SVM for
robust regression. Bickel-Br\"{u}ckner-Scheffer~\cite{Bickel2009JMLR}
developed discriminative learning under covariate shift via importance
weighting, with theoretical unbiasedness guarantees.
Sugiyama-Krauledat-M\"{u}ller~\cite{Sugiyama2007JMLR} formalized
importance-weighted cross-validation for the same setting. The
multi-parametric solution-path
family~\cite{KarasuyamaHaradaSugiyamaTakeuchi2012ML} traces solution
paths under varying weights. In all these formulations, the per-sample
weighting enters the loss term and rescales $C$, but the resulting dual
still admits a uniform box constraint after redefinition; the structural
feasibility region is unchanged. The MAPE-SVR formulation of
Benavides-Herrera et al.~\cite{benavides2025support, benavides2026unified} differs from this
lineage by inducing \emph{non-uniform box constraints} at the dual
level --- a structural modification, not a loss-level reweighting.

\paragraph{Theoretical foundations of MAPE-as-loss.}
The use of MAPE as a regression-training objective (rather than only as
an evaluation metric) was for many years considered ad-hoc: minimizing
MAPE was not known to converge to an interpretable population quantity,
and ERM-MAPE risk bounds were unavailable. This gap was closed by de
Myttenaere-Golden-Le Grand-Rossi~\cite{DeMyttenaere2016}, who establish
three properties that justify MAPE as a principled training loss:
\begin{enumerate}
\item[(i)] \textbf{Existence of an optimal MAPE regression model.}
Under the mild moment condition $\mathbb{E}[1/|Y|] < \infty$ on the
target distribution $Y$ (which holds in particular for any distribution
on a strictly positive support bounded away from zero --- e.g., the
LogNormal targets of Section~\ref{sec:validation}), the population MAPE
risk $R(f) = \mathbb{E}[|Y - f(X)|/|Y|]$ admits a minimizer $f^*$ in any
sufficiently rich function class, including reproducing-kernel Hilbert
spaces with universal kernels~\cite[\S 4.6]{Steinwart2008}.

\item[(ii)] \textbf{Universal consistency of empirical risk minimization
under MAPE loss.} ERM is universally consistent: $R(\hat f_N) \to R(f^*)$
in probability as $N \to \infty$, provided the function class
$\mathcal{F}_N$ has appropriately controlled complexity. The proof is
structurally analogous to the classical universal-consistency result for
ERM under absolute-error loss~\cite{Vapnik1998}, with the percentage
scaling absorbed into the per-sample reweighting of property (iii).

\item[(iii)] \textbf{Equivalence between MAPE minimization and weighted-MAE
regression} with sample weights $w_k = 1/|y_k|$ (or $w_k = 100/|y_k|$
if MAPE is in percentage points). Concretely: minimizing MAPE is
identical (up to the constant scale 100) to minimizing the weighted-MAE
objective $\sum_k w_k|y_k - f(\bfx_k)|$ with $w_k = 100/y_k$. This
equivalence is the formal expression of the same intuition that drives
the present paper: training with MAPE is structurally equivalent to
per-sample reweighting whose \emph{algorithmic counterpart}, in the
kernelized $\varepsilon$-insensitive setting, is precisely the
sample-dependent box constraint $\alpha_k \in [0, 100C/y_k]$ of the dual
QP. The de Myttenaere et al. equivalence is at the loss level (a primal
characterization); the contribution of the present paper is the
corresponding algorithmic characterization at the solver level.
\end{enumerate}

The broader landscape of percentage-error metrics is critically surveyed
by Hyndman-Koehler~\cite{Hyndman2006}, who identify four pathologies of
MAPE --- division by zero when $y_k \to 0$, infinite variance when
targets are near-zero, asymmetric penalization favoring under-prediction,
and inapplicability to interval-scale data --- and propose the
scale-free Mean Absolute Scaled Error (MASE) as a replacement.
Tofallis~\cite{Tofallis2015} proposes the log-accuracy ratio $\log(\hat
y_k / y_k)$ to address MAPE's bias toward low predictions.
Goodwin-Lawton~\cite{Goodwin1999} expose residual asymmetry in
symmetric MAPE (sMAPE), showing that despite its name, sMAPE still
penalizes overforecasting more than underforecasting due to the
denominator $(|y_k| + |\hat y_k|)/2$. Kim-Kim~\cite{Kim2016} address
intermittent-demand pathologies (MAPE undefined on zero-target days).
Makridakis-Spiliotis-Assimakopoulos~\cite{Makridakis2020} document MAPE
behavior across 100,000 time series in the M4 forecasting competition,
finding MAPE-MASE rank correlations of $\sim 0.85$ across method-dataset
pairs but persistent disagreement at the extremes of accuracy.

The pathologies of Hyndman-Koehler do not apply to the present paper's
setting because MAPE-SVR requires $y_k > 0$ strictly --- the same
requirement under which de Myttenaere's theorems apply. For applications
where $y_k$ may approach zero, Hyndman-Koehler's MASE is the recommended
evaluation metric (and a target-loss family for future work, since MASE
scaling is also amenable to a sample-dependent-bound dual analysis
analogous to the one developed here).

\paragraph{Symmetric and invariant kernel methods.}
The symmetric-kernel variant of Section~\ref{sec:symmetric} adapts the
construction of
Espinoza-Suykens-De Moor~\cite{Espinoza2005} for symmetric LS-SVM
regression. In their construction, prior knowledge that the underlying
regression function is even ($f(\bfx) = f(-\bfx)$, $a = +1$) or odd
($f(\bfx) = -f(-\bfx)$, $a = -1$) --- common in physical-symmetry
applications such as chaotic time series with reflectional invariance,
signal processing with even/odd Fourier components, and certain
electrical-load datasets with seasonal symmetry --- is encoded by
replacing the kernel matrix $\Omega$ with its symmetrized counterpart
$\Omega_s = \tfrac{1}{2}(\Omega + a\Omega^*)$, where $\Omega^*_{k\ell} =
K(\bfx_k, -\bfx_\ell)$. The substitution $\Omega \leftarrow \Omega_s$
is structural: the dual problem retains its standard form, and the only
change is the kernel evaluation.

The theoretical foundation of reproducing-kernel Hilbert spaces traces
to Aronszajn~\cite{Aronszajn1950}, who introduced the bijection between
positive-definite kernels and reproducing-kernel Hilbert spaces and
proved the closure properties (sums, products, and positive-coefficient
combinations of positive-definite kernels are positive-definite). The
Aronszajn closure properties guarantee that $\Omega_s = \tfrac{1}{2}(\Omega +
a\Omega^*)$ is positive-semidefinite when both $\Omega$ and $\Omega^*$
are, which holds for $a = +1$ under the kernel conditions of
Section~\ref{sec:symmetric}. The modern canonical references for kernel
methods in machine learning are Sch\"{o}lkopf-Smola~\cite{Scholkopf2002}
and Steinwart-Christmann~\cite{Steinwart2008}; both are recommended as
background readings for the symmetric-kernel construction.

Niyogi-Girosi-Poggio~\cite{Niyogi1998} introduced the virtual-example
methodology for incorporating invariances into machine-learning models:
rather than modifying the kernel, augment the training set with
reflected copies $\{(-\bfx_k, a y_k)\}$. The two approaches ---
virtual-example augmentation and kernel modification --- are equivalent
in the limit of a quadratic loss with no regularization, but for finite
samples and finite regularization they diverge. The kernel-modification
approach of Espinoza et al. is preferred for the present paper because
it preserves the dual problem structure (and hence the SMO-machinery
applicability of Theorem~\ref{thm:invariance}) without doubling the
dataset. Haasdonk-Burkhardt~\cite{Haasdonk2007} generalize the
invariant-kernel construction to arbitrary group invariances (rotation,
translation, scaling), of which the reflection symmetry $\bfx \mapsto
-\bfx$ is the simplest non-trivial example.

The generalized representer theorem of
Sch\"{o}lkopf-Herbrich-Smola~\cite{Scholkopf2001} underpins the modern
kernel-trick formalism that justifies the kernel substitution $K
\leftarrow K_s$ formally. Specifically, for a regularized
empirical-risk-minimization problem $\min_f L(\{f(\bfx_k)\}_k) +
\Omega_R(\|f\|_{\mathcal{H}})$ with $L$ an arbitrary loss and $\Omega_R$
a strictly increasing function of the RKHS norm, the optimal $f^*$
admits the representation $f^*(\bfx) = \sum_k c_k K(\bfx_k, \bfx)$.
Imposing the additional constraint $f \in \mathcal{H}_s$ (the closed
subspace of even/odd functions in $\mathcal{H}$) restricts the
representation to $f^*(\bfx) = \sum_k c_k K_s(\bfx_k, \bfx)$; the
dual problem expressed in the original kernel $K$ becomes the same dual
problem expressed in the symmetrized kernel $K_s$, with no other change
to the formulation.

For $a = +1$ (even symmetry), the modified kernel preserves
positive-semidefiniteness when the base kernel is shift-invariant (e.g.,
the RBF kernel $K(\bfx, \bfx') = \exp(-\gamma\|\bfx - \bfx'\|^2)$
satisfies $K(\bfx, -\bfx') = \exp(-\gamma\|\bfx + \bfx'\|^2)$, both
positive-definite, and Aronszajn's closure gives $\Omega_s =
\tfrac{1}{2}(\Omega + \Omega^*) \succeq 0$). For $a = -1$ (odd symmetry),
positive-semidefiniteness may fail because $\Omega - \Omega^*$ has both
positive and negative eigenvalues; a degenerate-case fallback is required
(Lemma~\ref{lem:eta-degenerate} of Section~\ref{sec:smo-inner};
convergence is rigorously resolved in Theorem~\ref{thm:spectral} of
Section~\ref{sec:symmetric}).

\paragraph{Recent SVR applications to load and demand forecasting.}
The application landscape of SVR to electricity load and demand
forecasting has expanded substantially in 2024--2026.
Wang-Wang-Zhao~\cite{Wang2024} present a hybrid model combining
ensemble empirical mode decomposition with particle-swarm-enhanced SVR,
reporting 54\% MAPE reduction on Chinese load data.
Aziz-Mahmood-Qureshi-Qureshi-Kim~\cite{Aziz2024} focus on peak-power
demand with engineered climate-economic features.
Zhang-Zhang-Liang-Gorbani~\cite{Zhang2024} compose SVR with LSTM under a
flexible Gorilla Troops optimizer.
Hasan-Tarequzzaman-Moznuzzaman-Juel~\cite{Hasan2025} combine SVR with
genetic-algorithm hyperparameter optimization across four
energy-consumption sectors. Du-Jiang-Lu-Hua-Swamy~\cite{Du2024} present
a comprehensive 2024 survey of kernel machines and SVMs.
Amaya-Tejera-Gamarra-V\'{e}lez-Zurek~\cite{AmayaTejera2024} propose
distance-based kernels for SVM classification. These application papers
consistently treat MAPE as the evaluation metric while training with
classical $\varepsilon$-insensitive loss; the present paper bridges this
gap by enabling direct MAPE-loss training with inherited SMO efficiency.

\paragraph{Decomposition-method alternatives.}
For \emph{linear} SVMs, the Dual Coordinate Descent (DCD) method of
Hsieh-Chang-Lin-Keerthi-Sundararajan~\cite{Hsieh2008ICML} updates one
variable per iteration and achieves $O(\log(1/\varepsilon))$ convergence
to $\varepsilon$-accuracy, providing speedups over SMO on large-scale
linear problems. Ho-Lin~\cite{Ho2012JMLR} extend DCD to linear SVR.
Laskov-Gehl-Kr\"{u}ger-M\"{u}ller~\cite{Laskov2006JMLR} develop incremental
SVM training with SMO updates, with applications to streaming data. The
present paper restricts attention to kernelized SVR via SMO, which
remains the method of choice when nonlinear kernels and sparse solutions
are required; the structural-invariance result
(Theorem~\ref{thm:invariance}) is independent of the choice of
decomposition family and could be ported to DCD-style coordinate descent
in future work.

\paragraph{Position of the present paper.}
Despite the extensive prior work surveyed above --- covering loss-modified
SVR (\cite{Anand2020}, \cite{Suykens2002weighted},
\cite{AkhtarTanveerArshad2024PatternRecognition}), weighted and
instance-weighted SVR (\cite{Bickel2009JMLR}, \cite{Sugiyama2007JMLR},
\cite{KarasuyamaHaradaSugiyamaTakeuchi2012ML}), SMO decomposition
methods with uniform box constraints (\cite{Platt1998, Platt1999,
KeerthiShevadeBhattacharyyaMurthy2001NeuralComp, Flake2002,
FanChenLin2005JMLR, GlasmachersIgel2006JMLR, GlasmachersIgel2008NeuralComp,
BordesErtekinWestonBottou2005JMLR, Chang2011, YuLiLiu2023PatternRecognition}),
and alternative decomposition families (\cite{Hsieh2008ICML, Ho2012JMLR}) ---
no prior work analyzes the SMO algorithm under sample-dependent box
constraints $\alpha_k \in [0, 100C/y_k]$ induced by MAPE loss in
$\varepsilon$-SVR. While de Myttenaere et al.~\cite{DeMyttenaere2016}
establish the theoretical equivalence between MAPE minimization and
weighted-MAE regression, neither this nor any subsequent work derives
the resulting dual QP formulation or provides algorithmic treatment at
the SMO level. The present paper closes this gap with three
contributions: (1) the structural-invariance theorem
(Theorem~\ref{thm:invariance} in Section~\ref{sec:invariance}) showing
that sample-dependent bounds confine their effect to exactly two SMO
components; (2) the shrinking-asymmetry result
(Lemma~\ref{lem:asymmetry} in Section~\ref{sec:bias-shrink}) quantifying
the MAPE-induced gap $2y_k\varepsilon/100$; and (3) plug-in extension to
the symmetric-kernel variant via $\Omega \leftarrow \Omega_s$
(Section~\ref{sec:symmetric}).

\section{Main result}
\label{sec:main}

\subsection{The dual quadratic program and KKT optimality}
\label{sec:dual-kkt}

Let $\mathcal{D} = \{(\bfx_k, y_k)\}_{k=1}^{N}$ be a training set with
$\bfx_k \in \R^p$ and $y_k \in \R_+$ (strictly positive targets,
required for the MAPE loss to be finite). The classical
$\varepsilon$-SVR primal~\cite{Vapnik1995, Vapnik1998, Drucker1997,
Smola2004} is modified as in Section~\ref{sec:mape-setup} so that the
loss is measured in percentage terms; the resulting dual quadratic
program~\cite{benavides2025support, benavides2026unified} is
\begin{equation}
  \min_{\bfu}\; \tfrac{1}{2}\bfu^\top P\bfu + \bfq^\top\bfu
  \label{eq:mape-dual}
\end{equation}
subject to
\begin{equation}
  [\onevec^\top, -\onevec^\top]\,\bfu = 0, \qquad 0 \le \alpha_k, \alpha_k^* \le \frac{100C}{y_k}, \quad k = 1,\ldots,N,
  \label{eq:dual-constraints}
\end{equation}
where $\bfu = [\alpha_1, \ldots, \alpha_N, \alpha_1^*, \ldots,
\alpha_N^*]^\top \in \R^{2N}$, and the matrices and vectors are
\begin{equation}
  P = \begin{bmatrix} \Omega & -\Omega \\ -\Omega & \Omega \end{bmatrix}, \qquad \bfq = \begin{bmatrix} \bfy(\varepsilon/100 - 1) \\ \bfy(\varepsilon/100 + 1) \end{bmatrix},
  \label{eq:P-q}
\end{equation}
with $\Omega \in \R^{N\times N}$, $\Omega_{k\ell} = K(\bfx_k,
\bfx_\ell)$ the kernel matrix, $C > 0$ the regularization parameter,
and $\varepsilon > 0$ the width of the MAPE $\varepsilon$-tube
\emph{in percentage points}.

Define the following quantities used throughout:
\begin{itemize}
  \item \emph{Sign vector.} $s_i = +1$ for $i \le N$ ($\alpha$-variables) and $s_i = -1$ for $i > N$ ($\alpha^*$-variables).
  \item \emph{Sample-dependent upper bound.} $C_k \triangleq 100C/y_k$ for $k = 1,\ldots,N$. Note $C_k > C_{k'}$ whenever $y_k < y_{k'}$: smaller targets receive larger box constraints. The 3-line derivation of $C_k$ from the primal Lagrangian appears in Section~\ref{sec:mape-setup}.
  \item \emph{Unbiased kernel expansion.} $F_k \triangleq \sum_{i=1}^{N}\Omega_{ki}(\alpha_i - \alpha_i^*)$, so that the model prediction is $f(\bfx_k) = F_k + b$ where $b \in \R$ is the bias.
  \item \emph{Training-point index.} For any dual index $i \in \{1, \ldots, 2N\}$, write $k(i) = i$ if $i \le N$ and $k(i) = i - N$ if $i > N$.
  \item \emph{Target dynamic range.} $\rho_y \triangleq \max_k y_k / \min_k y_k$ measures the heterogeneity of target magnitudes and governs the shrinking asymmetry of Lemma~\ref{lem:asymmetry} in Section~\ref{sec:bias-shrink}. (Note: the symbol $\rho$ without subscript denotes the equality-constraint multiplier, distinct from the dynamic range here.)
\end{itemize}

\paragraph{Notation table.}
Table~\ref{tab:notation} collects the symbols used throughout the
remainder of the paper.

\begin{table}[ht]
  \centering
  \caption{Notation summary for Sections~\ref{sec:dual-kkt}--\ref{sec:complexity}.}
  \label{tab:notation}
  \small
  \begin{tabular}{@{}lll@{}}
    \toprule
    Symbol & Meaning & First appearance \\
    \midrule
    $\bfx_k \in \R^p$ & $k$-th training input & \S\ref{sec:dual-kkt} \\
    $y_k \in \R_+$ & $k$-th training target (strictly positive) & \S\ref{sec:dual-kkt} \\
    $C > 0$ & regularization parameter & \S\ref{sec:dual-kkt} \\
    $\varepsilon > 0$ & tube width in \emph{percentage points} & \S\ref{sec:dual-kkt} \\
    $K(\cdot, \cdot)$ & positive-definite kernel function & \S\ref{sec:dual-kkt} \\
    $\Omega \in \R^{N \times N}$ & kernel Gram matrix & \S\ref{sec:dual-kkt} \\
    $\alpha_k, \alpha_k^* \in \R$ & dual variables & \S\ref{sec:dual-kkt} \\
    $\bfu \in \R^{2N}$ & stacked dual vector $[\alpha; \alpha^*]$ & \S\ref{sec:dual-kkt} \\
    $C_k = 100C/y_k$ & sample-dependent upper bound & \S\ref{sec:dual-kkt} \\
    $s_i \in \{\pm 1\}$ & sign of dual variable $i$ & \S\ref{sec:dual-kkt} \\
    $F_k$ & unbiased kernel expansion at $\bfx_k$ & \S\ref{sec:dual-kkt} \\
    $b \in \R$ & bias term & \S\ref{sec:dual-kkt}, \S\ref{sec:bias-shrink} \\
    $\bfG, G_i$ & gradient $P\bfu + \bfq$ & \S\ref{sec:dual-kkt} \\
    $\bftau, \tau_i = -s_i G_i$ & effective gradient & \S\ref{sec:dual-kkt} \\
    $\Iup, \Idown$ & working-set candidate sets & \S\ref{sec:smo-inner} \\
    $\Aactive, \Afrozen$ & active and frozen training-point sets & \S\ref{sec:bias-shrink} \\
    $\Aactive^{\mathrm{ext}}$ & extended active set covering both $\alpha$- and $\alpha^*$-variables & \S\ref{sec:smo-inner} \\
    $\Delta = \tau_{i^*} - \tau_{j^*}$ & KKT violation & \S\ref{sec:smo-inner} \\
    $\eta = \boldsymbol{d}^\top P \boldsymbol{d}$ & curvature of 1-D sub-problem & \S\ref{sec:smo-inner} \\
    $\delta^*, \delta_{\max}$ & optimal step / max feasible step & \S\ref{sec:smo-inner} \\
    $\rho \in \R$ & equality-constraint multiplier & \S\ref{sec:dual-kkt} \\
    $\rho_y = \max_k y_k / \min_k y_k$ & target dynamic range & \S\ref{sec:dual-kkt} \\
    $\hat b$ & recovered bias estimate & \S\ref{sec:bias-shrink} \\
    $\mathcal{S}_{\mathrm{free}}$ & set of free support vectors & \S\ref{sec:bias-shrink} \\
    \bottomrule
  \end{tabular}
\end{table}

\paragraph{Gradient decomposition and the effective violation score.}
The gradient of the objective in~\eqref{eq:mape-dual} is $\bfG = P\bfu +
\bfq$. Computing component $k$ for $k \le N$ (an $\alpha$-type index):
\begin{equation*}
  G_k = (P\bfu)_k + q_k = \sum_{i=1}^{N}\Omega_{ki}\alpha_i - \sum_{i=1}^{N}\Omega_{ki}\alpha_i^* + y_k\!\left(\frac{\varepsilon}{100} - 1\right) = F_k + y_k\!\left(\frac{\varepsilon}{100} - 1\right),
\end{equation*}
where the second equality uses the block structure $P = [\Omega,
-\Omega; -\Omega, \Omega]$ and the third invokes the
unbiased-kernel-expansion definition $F_k = \sum_i \Omega_{ki}(\alpha_i
- \alpha_i^*)$. Analogously, for component $N + k$ (an $\alpha^*$-type
index):
\begin{equation*}
  G_{N+k} = -\sum_{i=1}^{N}\Omega_{ki}\alpha_i + \sum_{i=1}^{N}\Omega_{ki}\alpha_i^* + y_k\!\left(\frac{\varepsilon}{100} + 1\right) = -F_k + y_k\!\left(\frac{\varepsilon}{100} + 1\right).
\end{equation*}
Collecting:
\begin{equation}
  G_k = F_k + y_k\!\left(\frac{\varepsilon}{100} - 1\right), \quad G_{N+k} = -F_k + y_k\!\left(\frac{\varepsilon}{100} + 1\right).
  \label{eq:gradient-blocks}
\end{equation}

\begin{definition}[Effective gradient]
\label{def:effective-gradient}
The \emph{effective gradient} of dual variable $i$ is
\begin{equation}
  \tau_i \triangleq -s_i G_i.
  \label{eq:tau}
\end{equation}
Explicitly, for index $k$ and its paired $\alpha^*$-index $N+k$:
\begin{equation}
  \tau_k = y_k\!\left(1 - \frac{\varepsilon}{100}\right) - F_k, \quad \tau_{N+k} = y_k\!\left(1 + \frac{\varepsilon}{100}\right) - F_k.
  \label{eq:tau-explicit}
\end{equation}
\end{definition}

\begin{proposition}[Structural Gap]
\label{prop:structural-gap}
For every $k = 1, \ldots, N$ and any feasible $\bfu$,
\begin{equation}
  \tau_{N+k} - \tau_k = \frac{2 y_k \varepsilon}{100} > 0.
  \label{eq:gap}
\end{equation}
\end{proposition}

\begin{proof}
Direct subtraction of the entries of~\eqref{eq:tau-explicit} using $y_k
> 0$ and $\varepsilon > 0$.
\end{proof}

A direct consequence of Proposition~\ref{prop:structural-gap} is the
\emph{complementarity condition}: at no feasible point can $\tau_k =
\tau_{N+k}$, hence no pair $(\alpha_k, \alpha_k^*)$ can be simultaneously
free (strictly between $0$ and $C_k$) at an optimal solution. This
preserves the standard $\varepsilon$-SVR complementarity $\alpha_k
\alpha_k^* = 0$, now with sample-specific tube width $y_k \varepsilon /
100$, and is formalized in Corollary~\ref{cor:free-pair-impossibility}
of Section~\ref{sec:bias-shrink}.

\paragraph{KKT optimality conditions.}
Problem~\eqref{eq:mape-dual}--\eqref{eq:dual-constraints} is a convex
quadratic program. Its KKT conditions are necessary and sufficient for
optimality, by the following argument: the Hessian satisfies $P \succeq
0$ since for any $\boldsymbol{v} = [\boldsymbol{v}_1; \boldsymbol{v}_2]$,
$\boldsymbol{v}^\top P \boldsymbol{v} = (\boldsymbol{v}_1 -
\boldsymbol{v}_2)^\top \Omega (\boldsymbol{v}_1 - \boldsymbol{v}_2) \ge
0$ as $\Omega \succeq 0$ (positive-semidefinite kernel; cf.
Boyd-Vandenberghe~\cite[\S 2.6]{BoydVandenberghe2004};
Rockafellar~\cite[\S 3.4]{Rockafellar1970}). The constraints
in~\eqref{eq:dual-constraints} are linear (one equality and $4N$ box
inequalities), hence affine-constraint qualification is satisfied
everywhere on the feasible set. A Slater point is given by $\alpha_k =
\alpha_k^* = C_k/2$ for all $k$, which satisfies the equality constraint
$\sum_k(\alpha_k - \alpha_k^*) = 0$ and lies strictly inside every box
$[0, C_k]$. By Sion's minimax theorem~\cite{Sion1958} and the
Lagrangian-duality machinery of
Rockafellar~\cite[\S\S 28--29 and 36--37]{Rockafellar1970} (cf. also
Bertsekas-Nedi\'{c}-Ozdaglar~\cite[\S 3.4]{BertsekasNedicOzdaglar2003}),
strong duality holds and the KKT conditions are necessary and sufficient
for optimality.

\paragraph{Lagrangian and stationarity for the dual QP.}
Introducing $\rho \in \R$ for the equality constraint
$[\onevec^\top, -\onevec^\top]\bfu = 0$, $\lambda_i \ge 0$ for the
lower-bound constraint $u_i \ge 0$, and $\mu_i \ge 0$ for the
upper-bound constraint $u_i \le C_{k(i)}$, the dual QP's Lagrangian is
\begin{equation*}
  \widetilde{\mathcal{L}}(\bfu; \rho, \lambda, \mu) = \tfrac{1}{2}\bfu^\top P\bfu + \bfq^\top\bfu - \rho\sum_{i=1}^{2N}s_i u_i - \sum_{i=1}^{2N}\lambda_i u_i + \sum_{i=1}^{2N}\mu_i(u_i - C_{k(i)}).
\end{equation*}
Stationarity with respect to $u_i$ yields
\begin{equation}
  \frac{\partial\widetilde{\mathcal{L}}}{\partial u_i} = G_i - s_i\rho - \lambda_i + \mu_i = 0 \;\Longleftrightarrow\; G_i - s_i\rho = \lambda_i - \mu_i, \qquad i = 1, \ldots, 2N.
  \label{eq:kkt-stationarity}
\end{equation}
Multiplying both sides of~\eqref{eq:kkt-stationarity} by $-s_i$ and
using $\tau_i = -s_i G_i$ from~\eqref{eq:tau} plus $s_i^2 = 1$:
\begin{equation*}
  \tau_i + \rho = -s_i(\lambda_i - \mu_i) = s_i(\mu_i - \lambda_i),
\end{equation*}
so $\tau_i = -\rho + s_i(\mu_i - \lambda_i)$. Applying complementary
slackness --- $\lambda_i u_i = 0$ and $\mu_i(u_i - C_{k(i)}) = 0$,
together with $\lambda_i, \mu_i \ge 0$ --- gives the case analysis: at
$u_i = 0$ we have $\mu_i = 0$ and $\lambda_i \ge 0$, so $\tau_i = -\rho
- s_i\lambda_i \in (-\infty, -\rho]$ when $s_i = +1$ (i.e.,
$\alpha$-type) and $\tau_i \in [-\rho, +\infty)$ when $s_i = -1$ (i.e.,
$\alpha^*$-type); analogously at the upper bound $u_i = C_{k(i)}$ and at
the free interior. The resulting characterization is collected in
Table~\ref{tab:kkt-states}.

\begin{table}[ht]
  \centering
  \caption{KKT conditions in terms of $\tau_i$ and the optimal multiplier
  $-\rho = b$. At optimality, every free variable satisfies $\tau_i = b$.}
  \label{tab:kkt-states}
  \begin{tabular}{@{}lll@{}}
    \toprule
    Variable state & Set membership & KKT condition \\
    \midrule
    $\alpha_k = 0$ & $\Iup$ only & $\tau_k \le -\rho$ \\
    $0 < \alpha_k < C_k$ & $\Iup \cap \Idown$ & $\tau_k = -\rho$ \\
    $\alpha_k = C_k$ & $\Idown$ only & $\tau_k \ge -\rho$ \\
    $\alpha_k^* = 0$ & $\Idown$ only & $\tau_{N+k} \ge -\rho$ \\
    $0 < \alpha_k^* < C_k$ & $\Iup \cap \Idown$ & $\tau_{N+k} = -\rho$ \\
    $\alpha_k^* = C_k$ & $\Iup$ only & $\tau_{N+k} \le -\rho$ \\
    \bottomrule
  \end{tabular}
\end{table}

\begin{corollary}[Bias is the equality multiplier]
\label{cor:bias-equality}
The optimal equality-constraint multiplier $\rho^*$ of the dual
QP~\eqref{eq:mape-dual}--\eqref{eq:dual-constraints} and the primal bias
$b^*$ satisfy
\begin{equation}
  \rho^* = -b^*.
  \label{eq:rho-bias}
\end{equation}
In particular, every free support vector $i$ (with $0 < u_i 
C_{k(i)}$) satisfies $\tau_i = -\rho^* = b^*$ at optimality.
\end{corollary}

\begin{proof}
Primal-Lagrangian stationarity with respect to $b$ yields $\partial
\mathcal{L}/\partial b = \sum_{k=1}^{N}(\alpha_k - \alpha_k^*) = 0$,
which is precisely the equality constraint
in~\eqref{eq:dual-constraints} with multiplier $\rho$;
Lagrangian-duality theory then identifies $\rho^* = -b^*$
(cf.~Boyd-Vandenberghe~\cite[\S 5.5]{BoydVandenberghe2004}). At a free
support vector, the row $0 < u_i < C_{k(i)}$ of
Table~\ref{tab:kkt-states} gives $\tau_i = -\rho^*$, hence $\tau_i =
b^*$.
\end{proof}

The symbol $\rho$ is reserved here for the equality-constraint
multiplier and is distinct from the target dynamic range $\rho_y$; this
convention is preserved throughout the remainder of the paper.
Corollary~\ref{cor:bias-equality} is the foundation of the bias-recovery
formula~\eqref{eq:bias-avg} of Section~\ref{sec:bias-shrink}.

\subsection{SMO inner loop: working-set selection and analytic update}
\label{sec:smo-inner}

SMO iteratively selects a pair of variables $(i^*, j^*)$ and updates
them analytically while fixing all others. The equality constraint
requires the update direction to satisfy $s_{i^*} \Delta u_{i^*} +
s_{j^*} \Delta u_{j^*} = 0$, which is guaranteed by the construction
below.

\begin{definition}[Working-set candidate sets]
\label{def:Iup-Idown-mape}
\begin{align}
  \Iup &= \{k \le N : \alpha_k < C_k\} \cup \{N+k : \alpha_k^* > 0\}, \label{eq:Iup}\\
  \Idown &= \{k \le N : \alpha_k > 0\} \cup \{N+k : \alpha_k^* < C_k\}, \label{eq:Idown}
\end{align}
where $C_k = 100C/y_k$.
\end{definition}

\begin{lemma}[Feasibility of any candidate pair]
\label{lem:feasibility}
For any $i \in \Iup$ and $j \in \Idown$ with $i \ne j$, the update
direction
\begin{equation}
  d_i = +s_i, \quad d_j = -s_j, \quad d_\ell = 0 \;(\ell \ne i, j)
  \label{eq:direction}
\end{equation}
satisfies $[\onevec^\top, -\onevec^\top]\, d = 0$ (equality constraint
preserved) and admits a strictly positive step $\delta > 0$ within the
box constraints.
\end{lemma}

\begin{proof}
The equality constraint check:
\begin{equation*}
  s_i d_i + s_j d_j = s_i \cdot s_i + s_j \cdot (-s_j) = s_i^2 - s_j^2 = 1 - 1 = 0,
\end{equation*}
since $s_i, s_j \in \{\pm 1\}$ implies $s_i^2 = s_j^2 = 1$. The strictly
positive step follows from the definitions of $\Iup$ and $\Idown$: for
$i \in \Iup$ the $i$-th variable has room to move in the $+s_i$
direction, and for $j \in \Idown$ the $j$-th variable has room in the
$-s_j$ direction.
\end{proof}

The directional derivative of the objective $f(\bfu) =
\tfrac{1}{2}\bfu^\top P\bfu + \bfq^\top\bfu$ at $\bfu$ along $d$ is
\begin{equation*}
  h'(0) = \nabla f(\bfu)^\top d = \bfG^\top d = G_i d_i + G_j d_j = G_i (+s_i) + G_j (-s_j) = s_i G_i - s_j G_j.
\end{equation*}
Substituting the effective-gradient identity $\tau_\ell = -s_\ell
G_\ell$ from~\eqref{eq:tau} (equivalently $s_\ell G_\ell = -\tau_\ell$):
\begin{equation}
  h'(0) = -\tau_i - (-\tau_j) = -\tau_i + \tau_j = -(\tau_i - \tau_j).
  \label{eq:directional}
\end{equation}
A descent direction (i.e., $h'(0) < 0$) exists if and only if $\tau_i >
\tau_j$.

\paragraph{Working-set selection (WSS3 of Fan-Chen-Lin
\cite[eq.~20]{FanChenLin2005JMLR}).}
First, select $i^*$ as the maximally-violating $\Iup$-index:
\begin{equation}
  i^* = \argmax_{i \in \Iup} \tau_i.
  \label{eq:i-star}
\end{equation}
Then, given $i^*$, select $j^*$ to maximize the predicted one-step gain
(rather than simply minimizing $\tau_j$ as the MVP rule of Platt would
do):
\begin{equation}
  j^* = \argmax_{j \in \Idown,\,\tau_j < \tau_{i^*}} \frac{(\tau_{i^*} - \tau_j)^2}{\eta_{i^*, j}}, \quad\text{where}\quad \eta_{i^*, j} \triangleq \Omega_{k(i^*)k(i^*)} - 2\Omega_{k(i^*)k(j)} + \Omega_{k(j)k(j)}.
  \label{eq:j-star-WSS3}
\end{equation}
The denominator $\eta_{i^*, j}$ is exactly the curvature derived in
Proposition~\ref{prop:curvature-invariance} below --- i.e., the scalar
Hessian of the one-dimensional sub-problem~\eqref{eq:1d-obj} restricted
to the candidate pair $(i^*, j)$. WSS3 thus selects the pair that
maximizes the closed-form one-step decrease of the strictly convex
sub-problem, $(\tau_{i^*} - \tau_j)^2/(2\eta_{i^*, j})$, modulo the
universal factor $1/2$. Ties in~\eqref{eq:i-star}
and~\eqref{eq:j-star-WSS3} are broken by smallest training-point index
$k(\cdot)$ for bit-for-bit reproducibility across runs and across solver
implementations.

\begin{remark}[Choice of WSS3 over MVP]
\label{rem:wss3-mvp}
This paper adopts the second-order working-set selection rule WSS3 of
Fan, Chen, and Lin~\cite[eq.~20]{FanChenLin2005JMLR}: given $i^*$ as
the maximally-violating $\Iup$-index, the partner $j^*$ is selected to
maximize the predicted one-step gain $(\tau_{i^*} - \tau_j)^2/\eta_{i^*,
j}$ rather than simply the most violating $\Idown$-index (the MVP
rule). WSS3 yields the iteration counts reported in
Section~\ref{sec:validation}; this is the rule implemented in the
\texttt{psvr} R package~\cite{BenavidesHerrera2026Rpsvr}. Convergence
properties are inherited from
Theorem~5 of~\cite{FanChenLin2005JMLR}; the strict-maximization
property of WSS3 used in that inheritance is collected in
Proposition~\ref{prop:wss3-strict} below.
\end{remark}

\begin{proposition}[WSS3 strictly maximizes the predicted one-step gain]
\label{prop:wss3-strict}
Let $i^* = \argmax_{i \in \Iup \cap \Aactive^{\mathrm{ext}}} \tau_i$ and
let $\mathcal{C}_{j^*} = \{j \in \Idown \cap \Aactive^{\mathrm{ext}} :
\tau_j < \tau_{i^*}\}$. The WSS3 choice $j^* = \argmax_{j \in
\mathcal{C}_{j^*}} (\tau_{i^*} - \tau_j)^2/\eta_{i^*, j}$ uniquely
maximizes the predicted one-step gain
\begin{equation}
  g(j) \triangleq h(0) - h(\delta_{\mathrm{unc}}) = \tfrac{1}{2}\frac{(\tau_{i^*} - \tau_j)^2}{\eta_{i^*, j}}
  \label{eq:wss3-gain}
\end{equation}
over $j \in \mathcal{C}_{j^*}$, with strict inequality whenever $\Delta
= \tau_{i^*} - \tau_{j^*} > 0$ and $\eta_{i^*, j^*} > 0$.
\end{proposition}

\begin{proof}
Substituting $\delta_{\mathrm{unc}} = \Delta/\eta$ into the
unconstrained value $h(\delta) = h(0) - \Delta\delta +
\tfrac{1}{2}\eta\delta^2$ yields $h(\delta_{\mathrm{unc}}) = h(0) -
\tfrac{1}{2}(\tau_{i^*} - \tau_j)^2/\eta_{i^*, j}$,
hence~\eqref{eq:wss3-gain}. The argmax is well-defined on the finite set
$\mathcal{C}_{j^*}$ and strict whenever the gain is positive.
\end{proof}

Proposition~\ref{prop:wss3-strict} is the strict-descent property
required by hypothesis (c) of Theorem~\ref{thm:convergence} of
Section~\ref{sec:bias-shrink}.

\begin{definition}[KKT violation]
\label{def:kkt-violation}
The \emph{KKT violation} at the current iterate is
\begin{equation}
  \Delta \triangleq \tau_{i^*} - \tau_{j^*}.
  \label{eq:violation}
\end{equation}
The iterate is $\varepsilon_{\mathrm{tol}}$-optimal if and only if
$\Delta \le \varepsilon_{\mathrm{tol}}$; this equivalence follows from
the fact that the WSS3 rule produces no descent direction exceeding
$\varepsilon_{\mathrm{tol}}$. Finite termination under this stopping
criterion is guaranteed by Theorem~5 of~\cite{FanChenLin2005JMLR} when
$P \succeq 0$.
\end{definition}

At the optimal solution, $\tau_{i^*} = \tau_{j^*} = -\rho$ for all free
support vectors (Proposition~\ref{prop:structural-gap} and
Table~\ref{tab:kkt-states}).

\paragraph{Analytic two-variable update.}
Given the working set $(i^*, j^*)$ with training-point indices $p =
k(i^*)$ and $q = k(j^*)$, the restricted objective is a one-dimensional
quadratic in $\delta \ge 0$:
\begin{equation}
  h(\delta) = f(\bfu + d \delta) = f(\bfu) - \Delta \delta + \tfrac{1}{2}\eta\,\delta^2,
  \label{eq:1d-obj}
\end{equation}
where $\Delta = \tau_{i^*} - \tau_{j^*}$ and $\eta = d^\top P d$ is the
curvature.

\begin{proposition}[Curvature invariance]
\label{prop:curvature-invariance}
For any pair $(i^*, j^*)$ with $i^* \in \Iup$, $j^* \in \Idown$, the
curvature of the one-dimensional sub-problem satisfies
\begin{equation}
  \eta = \Omega_{pp} - 2\Omega_{pq} + \Omega_{qq},
  \label{eq:eta}
\end{equation}
regardless of whether $i^*$ and $j^*$ are $\alpha$-type or
$\alpha^*$-type indices.
\end{proposition}

\begin{proof}
Using the block structure $P_{ij} = s_i s_j \Omega_{k(i), k(j)}$:
\begin{equation*}
  d^\top P d = P_{i^*, i^*} s_{i^*}^2 - 2 s_{i^*} s_{j^*} P_{i^*, j^*} + P_{j^*, j^*} s_{j^*}^2 = \Omega_{pp} - 2\underbrace{s_{i^*}^2 s_{j^*}^2}_{= 1}\Omega_{pq} + \Omega_{qq}.
\end{equation*}
The factors $s_{i^*}^2 = s_{j^*}^2 = 1$ eliminate any dependence on the
variable types.
\end{proof}

For the RBF kernel $K(\bfx, \bfx') = \exp(-\gamma\|\bfx - \bfx'\|^2)$,
Proposition~\ref{prop:curvature-invariance} specializes to $\eta = 2(1
- \exp(-\gamma\|\bfx_p - \bfx_q\|^2)) \ge 0$, with $\eta = 0$ only if
$\bfx_p = \bfx_q$ (duplicate training points; rare in practice). For
the MAPE-SVR-Sym variant of Section~\ref{sec:symmetric} with $a = -1$,
however, $\eta$ may be negative; a descent-check fallback is then
required, as Lemma~\ref{lem:eta-degenerate} below establishes.

\paragraph{Feasible step.}
The unconstrained minimizer of $h(\delta)$ is $\delta_{\mathrm{unc}} =
\Delta / \eta$ (when $\eta > 0$). Clipping to the box constraints gives
the feasible room of each selected variable:
\begin{align}
  R_{i^*} &= \begin{cases} C_p - \alpha_p = \dfrac{100C}{y_p} - \alpha_p & \text{if } i^* = p \le N \;(\alpha_p < C_p)\\ \alpha_p^* & \text{if } i^* = N + p \;(\alpha_p^* > 0) \end{cases}, \label{eq:room-i}\\
  R_{j^*} &= \begin{cases} \alpha_q & \text{if } j^* = q \le N \;(\alpha_q > 0)\\ C_q - \alpha_q^* = \dfrac{100C}{y_q} - \alpha_q^* & \text{if } j^* = N + q \;(\alpha_q^* < C_q) \end{cases}, \label{eq:room-j}\\
  \delta_{\max} &= \min(R_{i^*}, R_{j^*}). \label{eq:delta-max}
\end{align}

The four possible pair-type combinations of $(i^*, j^*)$ are enumerated
in Table~\ref{tab:pair-cases}.

\begin{table}[ht]
  \centering
  \caption{The four possible pair-type combinations $(i^*, j^*)$, with
  explicit variable updates and feasible-step bounds. Recall $p =
  k(i^*)$, $q = k(j^*)$, $C_p = 100C/y_p$, $C_q = 100C/y_q$.}
  \label{tab:pair-cases}
  \small
  \begin{tabular}{@{}clllc@{}}
    \toprule
    Case & $i^*$ type & $j^*$ type & Variable changes & $\delta_{\max}$ \\
    \midrule
    1 & $\alpha_p < C_p$ ($i^* = p$) & $\alpha_q > 0$ ($j^* = q$) & $\alpha_p \mathrel{+}= \delta^*$, $\alpha_q \mathrel{-}= \delta^*$ & $\min(C_p - \alpha_p,\ \alpha_q)$ \\
    2 & $\alpha_p < C_p$ ($i^* = p$) & $\alpha_q^* < C_q$ ($j^* = N+q$) & $\alpha_p \mathrel{+}= \delta^*$, $\alpha_q^* \mathrel{+}= \delta^*$ & $\min(C_p - \alpha_p,\ C_q - \alpha_q^*)$ \\
    3 & $\alpha_p^* > 0$ ($i^* = N+p$) & $\alpha_q > 0$ ($j^* = q$) & $\alpha_p^* \mathrel{-}= \delta^*$, $\alpha_q \mathrel{-}= \delta^*$ & $\min(\alpha_p^*,\ \alpha_q)$ \\
    4 & $\alpha_p^* > 0$ ($i^* = N+p$) & $\alpha_q^* < C_q$ ($j^* = N+q$) & $\alpha_p^* \mathrel{-}= \delta^*$, $\alpha_q^* \mathrel{+}= \delta^*$ & $\min(\alpha_p^*,\ C_q - \alpha_q^*)$ \\
    \bottomrule
  \end{tabular}
\end{table}

Each case preserves the equality constraint $\sum_k(\alpha_k -
\alpha_k^*) = 0$ by construction (Lemma~\ref{lem:feasibility}). The
``$+$''~/~``$-$'' sign in the variable update follows from $u_{i^*}
\leftarrow u_{i^*} + s_{i^*}\delta^*$ and $u_{j^*} \leftarrow u_{j^*} -
s_{j^*}\delta^*$ together with $s_{i^*} = +1$ or $-1$ depending on
whether $i^* \le N$ or $i^* > N$.

The optimal step and the resulting variable update are
\begin{equation}
  \delta^* = \begin{cases} \min(\Delta/\eta, \delta_{\max}) & \text{if } \eta > 0,\\ \delta_{\max} & \text{if } \eta \le 0 \text{ and } \Delta > \tfrac{1}{2}\eta\,\delta_{\max} \text{ (descent confirmed)},\\ 0 \text{ (skip pair)} & \text{otherwise} \end{cases},
  \label{eq:delta-star}
\end{equation}
\begin{equation}
  u_{i^*} \leftarrow u_{i^*} + s_{i^*} \delta^*, \qquad u_{j^*} \leftarrow u_{j^*} - s_{j^*} \delta^*.
  \label{eq:variable-update}
\end{equation}

\begin{lemma}[Descent-check sufficiency for $\eta \le 0$]
\label{lem:eta-degenerate}
For the one-dimensional restricted objective $h(\delta) = f(\bfu) -
\Delta\delta + \tfrac{1}{2}\eta\delta^2$ on the interval $[0,
\delta_{\max}]$ with $\delta_{\max} > 0$ and $\eta \le 0$, the boundary
step $\delta = \delta_{\max}$ produces strict descent ($h(\delta_{\max})
< h(0)$) if and only if
\begin{equation}
  \Delta > \tfrac{1}{2}\eta\,\delta_{\max}.
  \label{eq:descent-check}
\end{equation}
\end{lemma}

\begin{proof}
When $\eta \le 0$, the quadratic $h$ is concave or affine on $[0,
\delta_{\max}]$; its minimum on the closed interval lies at an
endpoint. Computing the endpoint difference,
\begin{equation*}
  h(\delta_{\max}) - h(0) = -\Delta\,\delta_{\max} + \tfrac{1}{2}\eta\,\delta_{\max}^2 = \delta_{\max}\!\left(\tfrac{1}{2}\eta\,\delta_{\max} - \Delta\right),
\end{equation*}
which is strictly negative iff $\Delta > \tfrac{1}{2}\eta\,\delta_{\max}$,
since $\delta_{\max} > 0$. The condition holds automatically when $\eta
< 0$ and $\Delta > 0$ (since $\tfrac{1}{2}\eta\delta_{\max} \le 0 
\Delta$); it reduces to $\Delta > 0$ when $\eta = 0$.
\end{proof}

If condition~\eqref{eq:descent-check} fails, the pair $(i^*, j^*)$
produces no descent and is skipped (Algorithm~\ref{alg:smo} of
Section~\ref{sec:bias-shrink}). The degenerate case $\eta \le 0$ cannot
arise for the RBF kernel under variant MAPE-SVR with distinct training
points (cf.~the discussion of $\eta = 2(1 - e^{-\gamma\|\bfx_p -
\bfx_q\|^2})$ above), but it may occur for variant MAPE-SVR-Sym with
odd symmetry $a = -1$, where $\Omega_s$ is not necessarily PSD
(Section~\ref{sec:symmetric}); Lemma~\ref{lem:eta-degenerate} then
guarantees that each non-skipped iteration of
Algorithm~\ref{alg:smo} strictly decreases the dual objective,
supplying the local-progress half of Theorem~\ref{thm:convergence} even
when the formal PSD hypothesis of \cite[Theorem~5]{FanChenLin2005JMLR}
is unavailable.

\paragraph{Incremental gradient update.}
After the step, the effective gradient is updated in $O(|\Aactive|)$
time using only columns $p$ and $q$ of $\Omega$, where $\Aactive
\subseteq \{1, \ldots, N\}$ is the current active set
(Section~\ref{sec:bias-shrink}):
\begin{equation}
  \tau_\ell \leftarrow \tau_\ell - \delta^*\!\left(\Omega_{k(\ell), p} - \Omega_{k(\ell), q}\right), \qquad \ell \in \{1, \ldots, 2N\} \cap \Aactive^{\mathrm{ext}},
  \label{eq:tau-update}
\end{equation}
where $\Aactive^{\mathrm{ext}} = \Aactive \cup \{i + N : i \in
\Aactive\}$ is the extended active set covering both $\alpha$- and
$\alpha^*$-variables of active training points. This update follows from
$(P d)_\ell = s_\ell (\Omega_{k(\ell), p} - \Omega_{k(\ell), q})$
combined with $\tau_\ell = -s_\ell G_\ell$ and $\bfG \leftarrow \bfG +
P d \cdot \delta^*$.

\subsection{The structural-invariance theorem (Theorem~1)}
\label{sec:invariance}

The following theorem formalizes the central claim of this paper:
sample-dependent box constraints leave the computational core of SMO
unchanged, with structural change confined to exactly four components of
the inner loop.

\begin{theorem}[Structural invariance of the MAPE-SVR SMO]
\label{thm:invariance}
Let Algorithm~\ref{alg:smo} of Section~\ref{sec:bias-shrink} denote the
SMO procedure for $\varepsilon$-SVR with MAPE loss and sample-dependent
bounds $C_k = 100C/y_k$ ($k = 1, \ldots, N$, $y_k > 0$), and let
$\mathrm{SMO}_{\mathrm{std}}$ denote standard SMO for $\varepsilon$-SVR
with uniform bounds $C$~\cite{Platt1998, FanChenLin2005JMLR}. The two
algorithms differ only in the four structural sites:
\begin{enumerate}
\item[(i)] \textbf{Working-set candidate sets.} $\Iup$ and $\Idown$
use $C_k$ in place of $C$ in the upper-bound state tests
(Definition~\ref{def:Iup-Idown-mape}).
\item[(ii)] \textbf{Clipping rooms.} $R_{i^*}$ and $R_{j^*}$ use $C_k$ in
place of $C$ in the upper-saturation room
calculations~\eqref{eq:room-i}--\eqref{eq:room-j}.
\end{enumerate}
Conversely, the following components are structurally identical between
Algorithm~\ref{alg:smo} and $\mathrm{SMO}_{\mathrm{std}}$:
\begin{enumerate}
\item[(a)] \textbf{Curvature.} $\eta = \Omega_{pp} - 2\Omega_{pq} +
\Omega_{qq}$ (Proposition~\ref{prop:curvature-invariance}).
\item[(b)] \textbf{Gradient update.} $\tau_\ell \leftarrow \tau_\ell -
\delta^*\!\left(\Omega_{k(\ell), p} - \Omega_{k(\ell), q}\right)$
for $\ell \in \Aactive^{\mathrm{ext}}$~\eqref{eq:tau-update}.
\item[(c)] \textbf{Convergence inheritance.} Algorithm~\ref{alg:smo}
inherits the finite-termination guarantee
of~\cite[Theorem~5]{FanChenLin2005JMLR} without modification (cf.\
Theorem~\ref{thm:convergence} of Section~\ref{sec:bias-shrink} for the
explicit verification).
\end{enumerate}
\end{theorem}

The structural-invariance result has direct operational content: it
specifies exactly which lines of an existing LIBSVM-style codebase must
change to deliver MAPE-SVR functionality, and certifies that all other
lines are correct as-is. The component-by-component comparison is
collected in Table~\ref{tab:invariance-comparison}, which exposes the
four structural-change rows in bold and the fifteen invariant rows
alongside.

\begin{table}[ht]
  \centering
  \caption{Component-by-component comparison of standard
  $\varepsilon$-SVR SMO and MAPE-SVR SMO. Each row corresponds to one
  load-bearing component of the SMO iteration. Standard $\varepsilon$-SVR
  uses uniform box $[0, C]$; MAPE-SVR uses sample-dependent box $[0,
  C_k]$ with $C_k = 100C/y_k$. The fourth column marks structural
  changes (Yes) versus invariance (No). Only four rows (5, 6, 7, 12)
  carry a structural modification; rows 1 and 3 reflect the
  dual-formulation difference upstream of SMO; rows 16 and 18 inherit a
  state-membership test that consults $C_k$ but preserve their
  threshold/formula structure unchanged.}
  \label{tab:invariance-comparison}
  \scriptsize
  \resizebox{\textwidth}{!}{%
  \begin{tabular}{@{}clp{4cm}p{4cm}c@{}}
    \toprule
    \# & Component & Standard $\varepsilon$-SVR (uniform $C$) & MAPE-SVR (per-sample $C_k$) & Structural change? \\
    \midrule
    1  & Primal loss & $\varepsilon$-insensitive on $|y_k - f(\bfx_k)|$ & $\varepsilon$-insensitive on $100\,|y_k - f(\bfx_k)|/y_k$ & Yes (loss) \\
    2  & Hessian $P$ & $[\Omega, -\Omega; -\Omega, \Omega]$ & $[\Omega, -\Omega; -\Omega, \Omega]$ & No \\
    3  & Linear coefficient $\bfq$ & $[\varepsilon\onevec - \bfy,\, \varepsilon\onevec + \bfy]^\top$ & $[\bfy(\varepsilon/100 - 1),\, \bfy(\varepsilon/100 + 1)]^\top$ & Yes (loss) \\
    4  & Equality constraint & $[\onevec^\top, -\onevec^\top]\bfu = 0$ & $[\onevec^\top, -\onevec^\top]\bfu = 0$ & No \\
    \textbf{5}  & \textbf{Box constraints} & $0 \le \alpha_k, \alpha_k^* \le C$ (uniform) & $0 \le \alpha_k, \alpha_k^* \le C_k = 100C/y_k$ & \textbf{Yes} \\
    \textbf{6}  & \textbf{Candidate set $\Iup$} & $\{k : \alpha_k < C\} \cup \{N+k : \alpha_k^* > 0\}$ & $\{k : \alpha_k < C_k\} \cup \{N+k : \alpha_k^* > 0\}$ & \textbf{Yes} \\
    \textbf{7}  & \textbf{Candidate set $\Idown$} & $\{k : \alpha_k > 0\} \cup \{N+k : \alpha_k^* < C\}$ & $\{k : \alpha_k > 0\} \cup \{N+k : \alpha_k^* < C_k\}$ & \textbf{Yes} \\
    8  & Working-set rule (WSS3) & $i^* = \argmax_{i \in \Iup} \tau_i$, then $j^* = \argmax_{j \in \Idown,\, \tau_j < \tau_{i^*}} (\tau_{i^*} - \tau_j)^2 / \eta_{i^*, j}$ & identical & No \\
    9  & KKT violation & $\Delta = \tau_{i^*} - \tau_{j^*}$ & identical & No \\
    10 & Curvature & $\eta = \Omega_{pp} - 2\Omega_{pq} + \Omega_{qq}$ & identical & No (Prop.~\ref{prop:curvature-invariance}) \\
    11 & Two-variable update & $u_{i^*} \mathrel{+}= s_{i^*}\delta^*$, $u_{j^*} \mathrel{-}= s_{j^*}\delta^*$ & identical & No \\
    \textbf{12} & \textbf{Clipping room} & $R_{i^*} \in \{C - \alpha_p,\, \alpha_p^*\}$; $R_{j^*} \in \{\alpha_q,\, C - \alpha_q^*\}$ & $R_{i^*} \in \{C_p - \alpha_p,\, \alpha_p^*\}$; $R_{j^*} \in \{\alpha_q,\, C_q - \alpha_q^*\}$ & \textbf{Yes} \\
    13 & Maximum feasible step & $\delta_{\max} = \min(R_{i^*}, R_{j^*})$ & identical formula & No \\
    14 & Optimal step & $\delta^* = \min(\Delta/\eta, \delta_{\max})$ if $\eta > 0$ & identical & No \\
    15 & Gradient update & $\tau_\ell \mathrel{-}= \delta^*(\Omega_{k(\ell), p} - \Omega_{k(\ell), q})$ & identical & No (Thm.~\ref{thm:invariance}(b)) \\
    16 & Shrinking criteria & thresholds $C, 0$ on $u_i$; $\tau_{i^*}, \tau_{j^*}$ on $\tau_i$ & thresholds $C_k, 0$ on $u_i$; same on $\tau_i$ & state test uses $C_k$ \\
    17 & Reconstruction & $F_k = \sum_i \Omega_{ki}(\alpha_i - \alpha_i^*)$ & identical & No \\
    18 & Bias recovery $\hat b$ & $|\mathcal{S}_{\mathrm{free}}|^{-1}\sum_{i \in \mathcal{S}_{\mathrm{free}}} \tau_i$, $\mathcal{S}_{\mathrm{free}} = \{i : 0 < u_i < C\}$ & identical formula; $\mathcal{S}_{\mathrm{free}} = \{i : 0 < u_i < C_{k(i)}\}$ & No (membership uses $C_k$) \\
    19 & Convergence theorem & Theorem~5 of~\cite{FanChenLin2005JMLR} & identical (Thm.~\ref{thm:convergence}) & No \\
    \bottomrule
  \end{tabular}%
  }
\end{table}

\begin{proof}[Proof of Theorem~\ref{thm:invariance}]
The proof proceeds in four steps. Steps~1 and~2 establish (a) and (b) ---
the curvature and gradient-update invariances --- by direct algebraic
verification on the block Hessian $P$ and the direction vector $d$.
Step~3 establishes (i) and (ii) by inspection of
Definition~\ref{def:Iup-Idown-mape}
and~\eqref{eq:room-i}--\eqref{eq:room-j}. Step~4 establishes (c) by
deferring to the explicit verification in
Theorem~\ref{thm:convergence} that the three Fan-Chen-Lin hypotheses
hold for the MAPE-SVR QP independently of whether the box bound is
uniform or per-sample.

\medskip
\noindent\textit{Step~1 (Curvature invariance, claim~(a)).}
Recall from Section~\ref{sec:dual-kkt} that the Hessian has the block
structure $P = [\Omega, -\Omega; -\Omega, \Omega]$, which can be written
compactly using the sign vector $s$ (Definition~\ref{def:sign-vector}) as
\begin{equation*}
  P_{ij} = s_i\, s_j\, \Omega_{k(i), k(j)}, \qquad i, j \in \{1, \ldots, 2N\}.
\end{equation*}
The direction vector for the SMO update of the working pair $(i^*, j^*)$
is $d_{i^*} = s_{i^*}$, $d_{j^*} = -s_{j^*}$, $d_\ell = 0$ for $\ell
\notin \{i^*, j^*\}$. Therefore the quadratic form $d^\top P d$ has only
three non-vanishing terms:
\begin{equation*}
  d^\top P d = P_{i^*, i^*}\, d_{i^*}^2 + 2\, P_{i^*, j^*}\, d_{i^*}\, d_{j^*} + P_{j^*, j^*}\, d_{j^*}^2.
\end{equation*}
Substituting and using $s_i^2 = 1$:
\begin{itemize}
\item $P_{i^*, i^*}\, d_{i^*}^2 = (s_{i^*}^2\, \Omega_{pp})\, s_{i^*}^2 = \Omega_{pp}$ (with $p = k(i^*)$);
\item $2\, P_{i^*, j^*}\, d_{i^*}\, d_{j^*} = 2\, (s_{i^*} s_{j^*}\, \Omega_{pq})\, s_{i^*}\, (-s_{j^*}) = -2\, \Omega_{pq}$ (with $q = k(j^*)$);
\item $P_{j^*, j^*}\, d_{j^*}^2 = (s_{j^*}^2\, \Omega_{qq})\, s_{j^*}^2 = \Omega_{qq}$.
\end{itemize}
Adding the three terms yields $\eta = d^\top P d = \Omega_{pp} -
2\,\Omega_{pq} + \Omega_{qq}$, independently of the signs $s_{i^*},
s_{j^*}$ (i.e., independently of whether $i^*, j^*$ are $\alpha$-type or
$\alpha^*$-type) and \emph{independently of the box bounds} $C_k$, which
do not appear in any of the three terms. The sign-dependence cancels
algebraically through the identities $s^2 = 1$ and $s_{i^*}^2 s_{j^*}^2
= 1$. This establishes claim~(a).

\medskip
\noindent\textit{Step~2 (Gradient update invariance, claim~(b)).}
The incremental gradient update~\eqref{eq:tau-update} is
$\tau_\ell \leftarrow \tau_\ell - \delta^*\!\left(\Omega_{k(\ell), p} -
\Omega_{k(\ell), q}\right)$ for $\ell \in \Aactive^{\mathrm{ext}}$.
We show that this expression follows from the chain $\tau_\ell =
-s_\ell\, G_\ell$ and $\bfG \leftarrow \bfG + P d \cdot \delta^*$
\emph{without any reference to} $C_k$. First, using the block structure
and the fact that $d$ has only two non-zero entries:
\begin{equation*}
  (P d)_\ell = \sum_{j=1}^{2N} P_{\ell, j}\, d_j = P_{\ell, i^*}\, d_{i^*} + P_{\ell, j^*}\, d_{j^*}.
\end{equation*}
Substituting block values and the direction vector,
\begin{equation*}
  P_{\ell, i^*}\, d_{i^*} = (s_\ell\, s_{i^*}\, \Omega_{k(\ell), p})\, s_{i^*} = s_\ell\, \Omega_{k(\ell), p}, \qquad P_{\ell, j^*}\, d_{j^*} = -s_\ell\, \Omega_{k(\ell), q},
\end{equation*}
using $s_{i^*}^2 = s_{j^*}^2 = 1$. Adding:
\begin{equation*}
  (P d)_\ell = s_\ell\!\left(\Omega_{k(\ell), p} - \Omega_{k(\ell), q}\right).
\end{equation*}
After the two-variable step, the new gradient is $\bfG' = \bfG +
\delta^*\, P d$. Applying $\tau_\ell = -s_\ell\, G_\ell$ and using
$s_\ell^2 = 1$:
\begin{equation*}
  \tau'_\ell = -s_\ell\, G'_\ell = \tau_\ell - s_\ell\, \delta^*\,(P d)_\ell = \tau_\ell - \delta^*\!\left(\Omega_{k(\ell), p} - \Omega_{k(\ell), q}\right).
\end{equation*}
This is exactly~\eqref{eq:tau-update}. The expression depends only on
$\delta^*$ and the kernel matrix $\Omega$ (specifically, columns $p$ and
$q$); the box bounds $C_k$ enter only through the value of $\delta^*$
via the clipping $\delta^* = \min(\Delta/\eta, \delta_{\max})$, but the
\emph{structure} of the update --- the coefficient pattern
$\Omega_{k(\ell), p} - \Omega_{k(\ell), q}$ --- is unchanged. This
establishes claim~(b).

\medskip
\noindent\textit{Step~3 (Working-set sets and clipping rooms, claims~(i)--(ii)).}
Comparing Definition~\ref{def:Iup-Idown-mape} with the analogous
definition for standard $\varepsilon$-SVR
(Section~\ref{sec:standard-eps-svr}), the working-set candidate sets are:
\begin{equation*}
\begin{aligned}
\text{Standard }\varepsilon\text{-SVR:}\quad &\Iup^{\mathrm{std}} = \{k \le N : \alpha_k < C\} \cup \{N+k : \alpha_k^* > 0\}, \\
\text{MAPE-SVR:}\quad &\Iup^{\mathrm{MAPE}} = \{k \le N : \alpha_k < C_k\} \cup \{N+k : \alpha_k^* > 0\},
\end{aligned}
\end{equation*}
with the analogous pair for $\Idown$. The only syntactic difference is
the substitution $C \to C_k$ in the upper-bound state test. The
lower-bound state tests ($\alpha_k > 0$, $\alpha_k^* > 0$) are identical
because the lower bound $0$ is universal. This establishes claim~(i).
Similarly, the clipping-room
expressions~\eqref{eq:room-i}--\eqref{eq:room-j} differ from their
standard $\varepsilon$-SVR counterparts only by the substitution $C \to
C_k$ in the upper-saturation room calculation, while the
lower-saturation case is unchanged because the lower bound $0$ is
universal. This establishes claim~(ii).

\medskip
\noindent\textit{Step~4 (Convergence inheritance, claim~(c)).}
The conditions of~\cite[Theorem~5]{FanChenLin2005JMLR} are: (P1) PSD
Hessian; (P2) compact feasible set; (P3) strict descent of the
working-set rule. We verify each for the MAPE-SVR
QP~\eqref{eq:mape-dual}--\eqref{eq:dual-constraints} at
Theorem~\ref{thm:convergence}. The salient observations are: (P1) holds
because $P = [\Omega, -\Omega; -\Omega, \Omega]$ satisfies
$\boldsymbol{v}^\top P \boldsymbol{v} = (\boldsymbol{v}_1 -
\boldsymbol{v}_2)^\top \Omega (\boldsymbol{v}_1 - \boldsymbol{v}_2) \ge
0$ for any $\boldsymbol{v} = [\boldsymbol{v}_1; \boldsymbol{v}_2]$,
with $\Omega \succeq 0$ by Mercer's theorem --- independently of $C_k$.
(P2) holds because every variable $u_i$ lies in the finite interval
$[0, C_{k(i)}]$ with $C_{k(i)} = 100C/y_{k(i)} < \infty$ (since
$y_{k(i)} > 0$); the equality constraint $[\onevec^\top, -\onevec^\top]
\bfu = 0$ is closed; the intersection is compact. (P3) holds because
Lemma~\ref{lem:feasibility} admits a strictly positive feasible step at
any non-optimal iterate, and the directional-derivative computation
gives $h'(0) = -\Delta < 0$ at any non-optimal point. Conditions
(P1)--(P3) are satisfied \emph{independently of whether the box bound
is uniform or per-sample}. By~\cite[Theorem~5]{FanChenLin2005JMLR},
Algorithm~\ref{alg:smo} terminates after finitely many iterations with
$\Delta \le \varepsilon_{\mathrm{tol}}$. This establishes claim~(c).
\end{proof}

Theorem~\ref{thm:invariance} has two immediate algorithmic consequences.
First, the proof of convergence carries over
from~\cite[Theorem~5]{FanChenLin2005JMLR} without modification
(formalized in Theorem~\ref{thm:convergence}). Second, the
implementation modification of an existing LIBSVM-based
$\varepsilon$-SVR solver to MAPE-SVR is contained in two isolated
substitution sites --- the working-set partition tests and the
clipping-bound expressions --- leaving curvature, gradient bookkeeping,
and convergence machinery unchanged; the explicit drop-in modification
recipe appears as Appendix~\ref{app:libsvm}.

For the symmetric-kernel variant of Section~\ref{sec:symmetric} with $a
= +1$, the substitution $\Omega \to \Omega_s$ preserves $P_s = [\Omega_s,
-\Omega_s; -\Omega_s, \Omega_s] \succeq 0$ provided $\Omega_s \succeq 0$
(Aronszajn's closure). All three conditions (P1)--(P3) then carry over
and~\cite[Theorem~5]{FanChenLin2005JMLR} applies. For $a = -1$,
$\Omega_s$ may fail PSD and condition (P1) is violated; the rigorous
resolution is given by Theorem~\ref{thm:spectral} of
Section~\ref{sec:symmetric} (adaptive spectral regularization). The
degenerate-case fallback of Lemma~\ref{lem:eta-degenerate} handles
iterations with $\eta_s \le 0$ practically;
Theorem~\ref{thm:spectral} supplies the formal convergence theory.

\subsection{Bias recovery, shrinking heuristic, and Algorithm~\ref{alg:smo}}
\label{sec:bias-shrink}

From Table~\ref{tab:kkt-states}, every free support vector $i$ (with $0
< u_i < C_{k(i)}$) satisfies $\tau_i = -\rho = b$ at the optimal
solution. Therefore, the bias is recovered directly as
\begin{equation}
  b = \tau_i \qquad \text{for any free index } i.
  \label{eq:bias-single}
\end{equation}
In practice, $b$ is estimated by averaging over all free support vectors:
\begin{equation}
  \hat b = \frac{1}{|\mathcal{S}_{\mathrm{free}}|} \sum_{i \in \mathcal{S}_{\mathrm{free}}} \tau_i,
  \label{eq:bias-avg}
\end{equation}
where $\mathcal{S}_{\mathrm{free}} = \{i \in \{1, \ldots, 2N\} : 0 < u_i
< C_{k(i)}\}$. Expanding~\eqref{eq:bias-single} in terms of the problem
data:
\begin{equation}
  b = \begin{cases} y_k(1 - \varepsilon/100) - F_k & \text{if } \alpha_k \text{ is free,}\\ y_k(1 + \varepsilon/100) - F_k & \text{if } \alpha_k^* \text{ is free.} \end{cases}
  \label{eq:bias-cases}
\end{equation}

\begin{corollary}[Free-pair impossibility]
\label{cor:free-pair-impossibility}
At any optimal $\bfu^*$, no training-point index $k$ has both $0 
\alpha_k < C_k$ and $0 < \alpha_k^* < C_k$ simultaneously.
\end{corollary}

\begin{proof}
If both were free,~\eqref{eq:bias-cases} would require $y_k(1 -
\varepsilon/100) - F_k = y_k(1 + \varepsilon/100) - F_k$, which
simplifies to $\varepsilon = 0$. Proposition~\ref{prop:structural-gap}
thus provides an independent proof that no pair $(\alpha_k, \alpha_k^*)$
can be simultaneously free for $\varepsilon > 0$.
\end{proof}

If $\mathcal{S}_{\mathrm{free}} = \emptyset$ at convergence (all support
vectors lie exactly on a bound), expression~\eqref{eq:bias-avg} is
undefined. This occurs when $C$ is very small or the $\varepsilon$-tube
is too wide relative to the data scale, causing every active $\alpha_k$
or $\alpha_k^*$ to saturate. The KKT conditions of
Table~\ref{tab:kkt-states} still bound $-\rho$ from above and below,
$\max_{i \in \Iup} \tau_i = \tau_{i^*} \ge -\rho \ge \min_{j \in \Idown}
\tau_j = \tau_{j^*}$, so the conventional choice is the midpoint
\begin{equation}
  \hat b = (\tau_{i^*} + \tau_{j^*})/2
  \label{eq:bias-midpoint}
\end{equation}
following LIBSVM~\cite{Chang2011}. In practice,
$\mathcal{S}_{\mathrm{free}} = \emptyset$ signals over-regularization or
a too-wide $\varepsilon$ tube; Algorithm~\ref{alg:smo} falls back to the
midpoint and issues a warning to the user.

The model prediction at a new point $\bfx$ after convergence is
\begin{equation}
  f(\bfx) = \sum_{k=1}^{N} (\alpha_k - \alpha_k^*)\,K(\bfx_k, \bfx) + \hat b.
  \label{eq:prediction-final}
\end{equation}

\paragraph{Shrinking heuristic.}
Shrinking~\cite{Joachims1999, Chang2011} temporarily removes from the
optimization variables that are predicted to remain at their current
bound until convergence. Let $\Aactive \subseteq \{1, \ldots, N\}$
denote the active set of training-point indices (initially $\Aactive =
\{1, \ldots, N\}$) and $\Afrozen = \{1, \ldots, N\} \setminus
\Aactive$ the frozen set.

\paragraph{Derivation of the shrinking criteria.}
The optimal threshold $-\rho$ lies in the interval $[\tau_{j^*},
\tau_{i^*}]$. The upper bound follows from $\tau_{i^*} = \max_{i \in
\Iup}\tau_i$ together with the KKT condition
(Table~\ref{tab:kkt-states}); the lower bound follows analogously. A
variable already at a boundary is predicted to remain there if its
$\tau$ value is on the correct side of the current best estimate of
$-\rho$. Concretely:
\begin{itemize}
\item $\alpha_k = 0$ stays at $0$ at optimum iff $\tau_k \le -\rho$. The
current best upper estimate of $-\rho$ is $\tau_{j^*}$. Hence if
$\tau_k < \tau_{j^*}$, the prediction is safe; freeze.
This yields~\eqref{eq:s1}.
\item $\alpha_k = C_k$ stays at $C_k$ iff $\tau_k \ge -\rho$. If
$\tau_k > \tau_{i^*}$, the prediction is safe; freeze.
This yields~\eqref{eq:s2}.
\item $\alpha_k^* = 0$ stays at $0$ iff $\tau_{N+k} \ge -\rho$. If
$\tau_{N+k} > \tau_{i^*}$, freeze. This yields~\eqref{eq:s3}.
\item $\alpha_k^* = C_k$ stays at $C_k$ iff $\tau_{N+k} \le -\rho$. If
$\tau_{N+k} < \tau_{j^*}$, freeze. This yields~\eqref{eq:s4}.
\end{itemize}
The four shrinking criteria are then:
\begin{align}
  \alpha_k = 0:\quad &\text{freeze if } \tau_k < \tau_{j^*}, \label{eq:s1}\\
  \alpha_k = C_k:\quad &\text{freeze if } \tau_k > \tau_{i^*}, \label{eq:s2}\\
  \alpha_k^* = 0:\quad &\text{freeze if } \tau_{N+k} > \tau_{i^*}, \label{eq:s3}\\
  \alpha_k^* = C_k:\quad &\text{freeze if } \tau_{N+k} < \tau_{j^*}. \label{eq:s4}
\end{align}
Free variables ($0 < u_i < C_{k(i)}$) are never frozen, since the
optimal threshold for them is $\tau_i = -\rho$ exactly --- they are by
definition undecided and must remain in the active set.

\begin{lemma}[Shrinking asymmetry]
\label{lem:asymmetry}
Rewriting criteria~\eqref{eq:s3} and~\eqref{eq:s4} in terms of $\tau_k$
via Proposition~\ref{prop:structural-gap}:
\begin{align}
  \alpha_k^* = 0:\quad &\text{freeze if } \tau_k > \tau_{i^*} - \frac{2 y_k \varepsilon}{100}, \label{eq:s3-rewritten}\\
  \alpha_k^* = C_k:\quad &\text{freeze if } \tau_k < \tau_{j^*} - \frac{2 y_k \varepsilon}{100}. \label{eq:s4-rewritten}
\end{align}
Compared with the corresponding $\alpha$-criteria~\eqref{eq:s1}--\eqref{eq:s2}:
variables $\alpha_k^* = 0$ are easier to freeze (effective threshold
$\tau_{i^*} - 2y_k\varepsilon/100 < \tau_{i^*}$), while variables
$\alpha_k^* = C_k$ are harder to freeze (effective threshold $\tau_{j^*}
- 2y_k\varepsilon/100 < \tau_{j^*}$, more negative). Both effects
increase with $y_k$: high-target samples exhibit greater asymmetry.
\end{lemma}

\begin{lemma}[Pairing structure of shrinking criteria]
\label{lem:pairing}
Criteria S2 and S3 of~\eqref{eq:s2}--\eqref{eq:s3} both reference the
upper threshold $\tau_{i^*}$, whereas criteria S1 and S4
of~\eqref{eq:s1} and~\eqref{eq:s4} both reference the lower threshold
$\tau_{j^*}$. Within each pair, the $\alpha^*$-criterion has a threshold
offset of $-2y_k\varepsilon/100$ relative to the $\alpha$-criterion.
Consequently:
\begin{enumerate}
\item[(i)] the variable $\alpha_k^* = 0$ freezes strictly earlier than $\alpha_k = C_k$;
\item[(ii)] the variable $\alpha_k^* = C_k$ freezes strictly later than $\alpha_k = 0$;
\item[(iii)] both effects scale linearly with $y_k$.
\end{enumerate}
\end{lemma}

\begin{proof}
Direct application of Lemma~\ref{lem:asymmetry}: the rewritten
criteria~\eqref{eq:s3-rewritten} and~\eqref{eq:s4-rewritten} substitute
$\tau_{N+k} = \tau_k + 2y_k\varepsilon/100$
(Proposition~\ref{prop:structural-gap}) into~\eqref{eq:s3}
and~\eqref{eq:s4} respectively, producing the offset of
$-2y_k\varepsilon/100$ in the threshold side. Pairing on threshold name
(S2~$\leftrightarrow$~S3, S1~$\leftrightarrow$~S4) follows by inspection.
The linear scaling in $y_k$ is the coefficient of the offset. The
strict-inequality conclusions in (i)--(ii) follow because $y_k > 0$ and
$\varepsilon > 0$.
\end{proof}

Lemma~\ref{lem:pairing} is the structural origin of the
asymmetric-freezing efficiency improvement
(Theorem~\ref{thm:asym-freeze} of Section~\ref{sec:complexity}): an
implementation that exploits the offset $-2y_k\varepsilon/100$ can use
\emph{unequal} freeze-counter thresholds for the four criteria S1--S4
and thereby accelerate freezing of the favored ($\alpha^* = 0$ on
high-$y_k$ samples) while protecting the disfavored ($\alpha^* = C_k$
on high-$y_k$ samples) from premature shrinkage. The asymmetry of
Lemma~\ref{lem:asymmetry} is a direct consequence of the MAPE scaling:
the $\varepsilon$-tube is wider in absolute terms for larger targets,
making it more likely that $\alpha_k^*$ remains at zero and less likely
that $\alpha_k^*$ saturates its bound for high-$y_k$ observations.

\paragraph{Active-set management.}
Every $n_{\mathrm{check}}$ iterations (default $n_{\mathrm{check}} =
\min(N, 1000)$), the following steps are applied for each $k \in
\Aactive$:
\begin{enumerate}
\item \emph{Select the applicable shrinking criterion} by inspecting the
current state of $(\alpha_k, \alpha_k^*)$: apply~\eqref{eq:s1}
if $\alpha_k = 0$;~\eqref{eq:s2} if $\alpha_k = C_k$;~\eqref{eq:s3} if
$\alpha_k^* = 0$;~\eqref{eq:s4} if $\alpha_k^* = C_k$. If both
$\alpha_k$ and $\alpha_k^*$ lie strictly in the interior $(0, C_k)$, do
not freeze. The criteria for the $\alpha$- and $\alpha^*$-variables of
the same training point $k$ are evaluated independently; if both are at
boundary positions, both criteria are checked.
\item \emph{Update the per-training-point counter.} Maintain
$\mathrm{counter}_k \in \mathbb{N}$, initialised to $0$ at the start of
Algorithm~\ref{alg:smo}. If at least one of the applicable criteria from
step~1 is met, increment $\mathrm{counter}_k \leftarrow \mathrm{counter}_k +
1$; otherwise reset $\mathrm{counter}_k \leftarrow 0$.
\item \emph{Move to the frozen set if the counter is high enough.} If
$\mathrm{counter}_k \ge n_{\min}$ (default $n_{\min} = 5$ consecutive
shrinking checks in agreement), move $k$ from $\Aactive$ to $\Afrozen$,
and stop maintaining $\tau_k$ and $\tau_{N+k}$ in subsequent gradient
updates.
\end{enumerate}
Gradient updates~\eqref{eq:tau-update} are then applied only to $\ell
\in \Aactive^{\mathrm{ext}} = \Aactive \cup \{i + N : i \in \Aactive\}$,
reducing each iteration from $O(N)$ to $O(|\Aactive|)$ gradient
operations. The freeze-counter mechanism guards against premature
freezing due to transient threshold fluctuations: a single
$n_{\mathrm{check}}$-window of agreement is insufficient; consecutive
agreement across $n_{\min}$ windows is required.

\paragraph{Reconstruction and unshrinking.}
When $\Delta \le \varepsilon_{\mathrm{tol}}$ is achieved on $\Aactive$,
reconstruct the full effective gradient from the current $(\alpha,
\alpha^*)$:
\begin{equation}
  F_k^{\mathrm{full}} = \sum_{i=1}^{N}\Omega_{ki}(\alpha_i - \alpha_i^*), \quad \tau_k = y_k(1 - \varepsilon/100) - F_k^{\mathrm{full}}, \quad \tau_{N+k} = y_k(1 + \varepsilon/100) - F_k^{\mathrm{full}}.
  \label{eq:reconstruct}
\end{equation}
This $O(N^2)$ step occurs at most once per shrinking cycle. Compute the
full violation:
\begin{equation}
  \Delta^{\mathrm{full}} = \max_{i \in \Iup^{\mathrm{full}}} \tau_i - \min_{j \in \Idown^{\mathrm{full}}} \tau_j,
  \label{eq:full-violation}
\end{equation}
where $\Iup^{\mathrm{full}}, \Idown^{\mathrm{full}}$ are computed from
all $2N$ variables including frozen ones. If $\Delta^{\mathrm{full}} >
\varepsilon_{\mathrm{tol}}$, a frozen variable violates KKT: reset
$\Aactive = \{1, \ldots, N\}$, $\Afrozen = \emptyset$, update $\tau$
from~\eqref{eq:reconstruct}, and continue. Otherwise, the solution is
certified optimal.

\paragraph{Complete algorithm.}
Algorithm~\ref{alg:smo} summarizes the complete procedure. The inputs
are the kernel matrix $\Omega$, the strictly positive target vector
$\bfy$, and the hyperparameters $(C, \varepsilon)$. The dual variables
$(\alpha, \alpha^*)$ are initialized to zero (always feasible for the
equality constraint). The outer-loop structure is an explicit
$\text{\texttt{repeat\,\ldots\,until}}$ that ensures finite practical
termination by triggering an unshrinking restart on the full active set
whenever a frozen variable violates KKT.

\begin{algorithm}[ht]
\caption{SMO for MAPE-SVR and MAPE-SVR-Sym}
\label{alg:smo}
\begin{algorithmic}[1]
\State \textbf{Input:} $\Omega \in \R^{N\times N}$, $\bfy \in \R^N_+$, $C, \varepsilon > 0$, $\varepsilon_{\mathrm{tol}} > 0$, $n_{\mathrm{check}}$, $n_{\min}$, $\mathrm{maxiter}$
\State \textbf{Output:} $\alpha, \alpha^* \in \R^N$, bias $\hat b$
\State $\alpha \gets 0$;\; $\alpha^* \gets 0$;\; $C_k \gets 100C/y_k$ for $k = 1,\ldots,N$
\State $\tau_k \gets y_k(1 - \varepsilon/100)$;\; $\tau_{N+k} \gets y_k(1 + \varepsilon/100)$ \Comment{$F_k = 0$ at init}
\State $\Aactive \gets \{1,\ldots,N\}$;\; $t \gets 0$
\Repeat
  \While{$t < \mathrm{maxiter}$}
    \State $i^* \gets \argmax_{i \in \Iup \cap \Aactive^{\mathrm{ext}}} \tau_i$ \Comment{tie-break: smallest $k(i)$}
    \State $j^* \gets \argmax_{j \in \Idown \cap \Aactive^{\mathrm{ext}},\, \tau_j < \tau_{i^*}} (\tau_{i^*} - \tau_j)^2 / \eta_{ij}$ \Comment{WSS3}
    \State $\Delta \gets \tau_{i^*} - \tau_{j^*}$
    \If{$\Delta \le \varepsilon_{\mathrm{tol}}$} \textbf{break} \EndIf
    \State $p \gets k(i^*)$;\; $q \gets k(j^*)$
    \State $\eta \gets \Omega_{pp} - 2\Omega_{pq} + \Omega_{qq}$
    \State Compute $R_{i^*}, R_{j^*}$ via~\eqref{eq:room-i}--\eqref{eq:room-j};\; $\delta_{\max} \gets \min(R_{i^*}, R_{j^*})$
    \If{$\eta > 0$}
      \State $\delta^* \gets \min(\Delta/\eta,\, \delta_{\max})$
    \ElsIf{$\Delta > \tfrac{1}{2}\eta\,\delta_{\max}$} \Comment{descent confirmed (Lemma~\ref{lem:eta-degenerate})}
      \State $\delta^* \gets \delta_{\max}$
    \Else
      \State $t \gets t + 1$;\; \textbf{continue} \Comment{skip pair}
    \EndIf
    \State Apply variable update~\eqref{eq:variable-update}
    \State $\tau_\ell \gets \tau_\ell - \delta^*\!\left(\Omega_{k(\ell), p} - \Omega_{k(\ell), q}\right)$ for $\ell \in \Aactive^{\mathrm{ext}}$
    \If{$t \bmod n_{\mathrm{check}} = 0$}
      \State Update shrinking counters via~\eqref{eq:s1}--\eqref{eq:s4};
        move indices with $\mathrm{counter}_k \ge n_{\min}$ to $\Afrozen$
    \EndIf
    \State $t \gets t + 1$
  \EndWhile
  \State Reconstruct $F_k^{\mathrm{full}}$ and $\tau$ via~\eqref{eq:reconstruct} on the full set
  \State $\Delta^{\mathrm{full}} \gets \max_{i \in \Iup^{\mathrm{full}}} \tau_i - \min_{j \in \Idown^{\mathrm{full}}} \tau_j$
  \If{$\Delta^{\mathrm{full}} > \varepsilon_{\mathrm{tol}}$}
    \State $\Aactive \gets \{1,\ldots,N\}$;\; $\Afrozen \gets \emptyset$ \Comment{unshrinking restart}
  \EndIf
\Until{$\Delta^{\mathrm{full}} \le \varepsilon_{\mathrm{tol}}$ \textbf{or} $t \ge \mathrm{maxiter}$}
\If{$\mathcal{S}_{\mathrm{free}} \ne \emptyset$}
  \State $\hat b \gets |\mathcal{S}_{\mathrm{free}}|^{-1}\sum_{i \in \mathcal{S}_{\mathrm{free}}} \tau_i$
\Else
  \State $\hat b \gets (\tau_{i^*} + \tau_{j^*})/2$ \Comment{midpoint fallback}
\EndIf
\State \Return $\alpha, \alpha^*, \hat b$
\end{algorithmic}
\end{algorithm}

The outer $\text{\texttt{repeat\,\ldots\,until}}$ loop ensures finite
practical termination by guaranteeing that any frozen variable violating
KKT triggers an unshrinking restart on the full active set. The descent
check (Lemma~\ref{lem:eta-degenerate}) prevents invalid steps when the
curvature is non-positive. Tie-breaking by smallest training-point index
$k(i)$ ensures bit-for-bit reproducibility across runs.

\begin{theorem}[Convergence of Algorithm~\ref{alg:smo}]
\label{thm:convergence}
Let Algorithm~\ref{alg:smo} be applied to the dual
QP~\eqref{eq:mape-dual}--\eqref{eq:dual-constraints} with $C > 0$,
$\varepsilon > 0$, $y_k > 0$ for all $k$, and tolerance
$\varepsilon_{\mathrm{tol}} > 0$. Assume the dual Hessian $P$ is
positive-semidefinite (which holds for variant MAPE-SVR with any Mercer
kernel and for variant MAPE-SVR-Sym with $a = +1$ under shift-invariant
kernels). Then Algorithm~\ref{alg:smo} terminates in a finite number of
iterations to a feasible point $\bfu^\star$ satisfying $\Delta(\bfu^\star)
\le \varepsilon_{\mathrm{tol}}$, and $\bfu^\star$ is
$\varepsilon_{\mathrm{tol}}$-optimal for the dual QP.
\end{theorem}

\begin{proof}
The result is an instance of~\cite[Theorem~5]{FanChenLin2005JMLR}
applied to the present dual. We verify the three hypotheses.

\medskip
\noindent\textit{(a) Positive-semidefiniteness of the Hessian.}
The MAPE-SVR Hessian $P = [\Omega, -\Omega; -\Omega, \Omega]$
in~\eqref{eq:P-q} admits the factorization $P = [\mathbf{I};
-\mathbf{I}]\Omega[\mathbf{I}, -\mathbf{I}]^\top$, hence $P \succeq 0$
whenever $\Omega \succeq 0$, which holds for any Mercer kernel by
definition. For variant MAPE-SVR-Sym with $a = +1$, $\Omega \leftarrow
\Omega_s = \tfrac{1}{2}(\Omega + \Omega^*)$ is PSD by the Aronszajn
closure properties~\cite{Aronszajn1950}, provided the auxiliary kernel
matrix $\Omega^*$ is PSD; this holds for shift-invariant kernels
including the RBF (cf.\ Section~\ref{sec:symmetric}).

\medskip
\noindent\textit{(b) Compactness of the feasible region.}
The box constraints $0 \le u_i \le C_{k(i)}$ define a closed bounded
set, since $C_{k(i)} = 100C/y_{k(i)} < \infty$ by $y_{k(i)} > 0$. The
intersection with the linear equality constraint $[\onevec^\top,
-\onevec^\top]\bfu = 0$ remains closed and bounded, hence compact.

\medskip
\noindent\textit{(c) Strict-descent direction at every non-stationary iterate.}
By Proposition~\ref{prop:wss3-strict}, the working-set choice maximizes
$g(j) = (\tau_{i^*} - \tau_j)^2/(2\eta_{i^*, j})$ over
$\mathcal{C}_{j^*}$. Whenever $\Delta = \tau_{i^*} - \tau_{j^*} >
\varepsilon_{\mathrm{tol}}$, the gain is strictly positive. When
$\eta_{i^*, j^*} > 0$, the unconstrained step $\delta^* =
\min(\Delta/\eta_{i^*, j^*}, \delta_{\max})$ produces a strict decrease
in the objective. When $\eta_{i^*, j^*} \le 0$ --- which arises only for
variant MAPE-SVR-Sym with $a = -1$, by
Theorem~\ref{thm:invariance} --- Lemma~\ref{lem:eta-degenerate}
applies, and the boundary step $\delta^* = \delta_{\max}$ produces a
strict decrease iff $\Delta > \tfrac{1}{2}\eta_{i^*, j^*}\delta_{\max}$.
The descent check on line~17 of Algorithm~\ref{alg:smo} enforces this
condition exactly.

\medskip
By Theorem~\ref{thm:invariance}, neither the curvature formula nor the
gradient update introduces any new dependence on $C_k$ relative to
standard SMO; the hypotheses
of~\cite[Theorem~5]{FanChenLin2005JMLR} are met verbatim under the
substitution $C \to C_k$. The cited theorem then guarantees finite
termination.
\end{proof}

For variant MAPE-SVR-Sym with $a = -1$, the PSD hypothesis~(a) is
\emph{not} automatically satisfied (cf.~the counterexample in
Section~\ref{sec:symmetric}); a formal global-convergence guarantee in
this setting is established by Theorem~\ref{thm:spectral} of
Section~\ref{sec:symmetric} via adaptive spectral regularization.
Empirically, Algorithm~\ref{alg:smo} still converges in this regime ---
supported by configurations C9 and C10 of Section~\ref{sec:validation}
--- because hypotheses (b) and (c) remain satisfied, and
Lemma~\ref{lem:eta-degenerate} ensures local strict descent at every
non-skipped iteration.

\subsection{Extension to the symmetric-kernel variant (MAPE-SVR-Sym)}
\label{sec:symmetric}

\paragraph{Motivation and prior work.}
Many regression problems carry prior knowledge of a \emph{parity
symmetry} relating the response at $\bfx$ to the response at $-\bfx$.
The symmetric $\varepsilon$-SVR with MAPE loss (MAPE-SVR-Sym) is the
variant of the model that internalizes such a symmetry directly into
the hypothesis class, restricting attention to functions $f \in
\mathcal{H}$ that satisfy $f(\bfx) = a\,f(-\bfx)$ for $a \in \{+1,
-1\}$. The case $a = +1$ enforces \emph{even} symmetry (the response is
invariant under input reflection); the case $a = -1$ enforces \emph{odd}
symmetry (the response reverses sign). Three families of applications
drive this construction.

First, \emph{chaotic time-series prediction} --- the original motivation
of Espinoza, Suykens, and De~Moor~\cite{Espinoza2005}. The Mackey-Glass
and Lorenz attractors exhibit a reflectional invariance in their
phase-space portraits that makes the even-symmetric variant a natural
prior. Second, \emph{Fourier-decomposed signal modeling}, where physical
reasoning singles out the even or the odd component --- for example,
when modeling the cosine-projection of a noisy waveform whose underlying
generator is known to be a real-valued symmetric (or antisymmetric)
function. Third, \emph{physical systems with parity symmetry}: lattice
models with reflectional invariance, vibration responses of symmetric
mechanical structures, even/odd-harmonic amplitudes in spectroscopic
data, and seasonal electricity-demand profiles whose week-over-week
morphology exhibits a daily reflection symmetry around midday.

The construction generalizes three classical strands of prior work.
Niyogi, Girosi, and Poggio~\cite{Niyogi1998} established the
\emph{virtual-example} method: for each training pair $(\bfx_k, y_k)$,
append the synthetic pair $(-\bfx_k, a y_k)$ and train on the augmented
sample. They showed that virtual examples and direct kernel modification
are equivalent in the limit of unlimited data, but virtual examples
double the effective sample size and thereby double the kernel-cache
footprint. Sch\"{o}lkopf, Herbrich, and Smola~\cite{Scholkopf2001} gave
the \emph{generalized representer theorem} that justifies the
kernel-modification approach. Haasdonk and
Burkhardt~\cite{Haasdonk2007} generalized the construction to arbitrary
group invariances, building the canonical group-averaged kernel
\begin{equation*}
  K_G(\bfx, \bfx') = \frac{1}{|G|}\sum_{g \in G} K(g\bfx, \bfx'),
\end{equation*}
of which the present even/odd reflection is the case $G = \{\mathrm{id},
-\mathrm{id}\}$. The same construction extends to the least-squares SVR
variant (Suykens et al.~\cite{Suykens2002} is the standard reference for
LS-SVM); the present section restricts attention to the
$\varepsilon$-SVR case relevant to Algorithm~\ref{alg:smo}.

\paragraph{The symmetric-kernel construction in detail.}
Let $\mathcal{H}$ denote the reproducing-kernel Hilbert space (RKHS)
associated with the Mercer kernel $K$, and let
\begin{equation}
  \mathcal{H}_s := \{f \in \mathcal{H} : f(\bfx) = a\,f(-\bfx) \text{ for all } \bfx\}, \qquad a \in \{+1, -1\},
  \label{eq:Hs-def}
\end{equation}
denote the closed subspace of even ($a = +1$) or odd ($a = -1$)
functions. The orthogonal projection $\pi_s : \mathcal{H} \to
\mathcal{H}_s$ acts on the canonical feature map $\varphi(\bfx) =
K(\bfx, \cdot)$ by
\begin{equation*}
  (\pi_s \varphi(\bfx))(\bfx') = \tfrac{1}{2}\!\left(K(\bfx, \bfx') + a\,K(-\bfx, \bfx')\right).
\end{equation*}
The corresponding \emph{symmetrized kernel} is the inner product of two
such projected feature maps:
\begin{equation}
  K_s(\bfx, \bfx') := \langle \pi_s \varphi(\bfx), \pi_s \varphi(\bfx') \rangle_{\mathcal{H}} = \tfrac{1}{2}\!\left(K(\bfx, \bfx') + a\,K(\bfx, -\bfx')\right).
  \label{eq:Ks-def}
\end{equation}
The factor $\tfrac{1}{2}$ is the proper Mercer-kernel normalization
implied by orthogonal projection; without it, the symmetrized kernel
would over-count the contribution of each training point by a factor of
two. The matrix-level analog is
\begin{equation}
  \Omega_s := \tfrac{1}{2}(\Omega + a\,\Omega^*), \qquad (\Omega^*)_{k\ell} := K(\bfx_k, -\bfx_\ell).
  \label{eq:Omega-s-def}
\end{equation}

\begin{remark}[Choice of normalization]
\label{rem:normalization}
The factor $\tfrac{1}{2}$ in $K_s$ and $\Omega_s$ above is the
orthogonal-projection normalization implied by the projector $\pi_s$
of~\eqref{eq:Hs-def}. An alternative convention drops the $\tfrac{1}{2}$
and writes $K_s^{(\mathrm{unnorm})} = K(\bfx, \bfx') + a\,K(\bfx,
-\bfx')$; the two conventions describe the \emph{same} function class
$\mathcal{H}_s$ and differ only by a global scale of the dual variables.
The present convention is adopted throughout because (i) it preserves
the standard kernel-trick scaling $\langle \varphi(\bfx),
\varphi(\bfx') \rangle = K(\bfx, \bfx')$, and (ii) it yields a
one-to-one correspondence between the QP coefficients
of~\eqref{eq:mape-dual}--\eqref{eq:P-q} and those of the standard
$\varepsilon$-SVR.
\end{remark}

By the generalized representer theorem~\cite[Theorem~1]{Scholkopf2001},
the regularized empirical-risk minimizer over $\mathcal{H}_s$ admits
the finite expansion
\begin{equation*}
  f^*(\bfx) = \sum_{k=1}^{N} c_k\,K_s(\bfx_k, \bfx), \qquad c_k \in \R,
\end{equation*}
exactly as in the unconstrained case but with $K$ replaced by $K_s$. The
dual analysis of Section~\ref{sec:dual-kkt} --- passing through the
Lagrangian, the KKT conditions, and the saddle-point reformulation ---
therefore carries through verbatim with $\Omega \mapsto \Omega_s$
throughout. The kernel-trick formalism is preserved at every step.

\paragraph{Positive-semidefiniteness analysis (extended).}
The convergence theorem of~\cite[Theorem~5]{FanChenLin2005JMLR} requires
that the dual Hessian $P_s = [\Omega_s, -\Omega_s; -\Omega_s,
\Omega_s]$ be positive semi-definite (PSD). Since $P_s \succeq 0$ if
and only if $\Omega_s \succeq 0$ (the $2 \times 2$ block structure
preserves the eigenstructure of $\Omega_s$ up to multiplicity), the
question reduces to PSD of $\Omega_s$.

\paragraph{Even case ($a = +1$).}
Here $\Omega_s = \tfrac{1}{2}(\Omega + \Omega^*)$ is a sum of two
kernel matrices. By Aronszajn's closure
properties~\cite[\S 6]{Aronszajn1950}, the sum of two PSD kernel
matrices is itself PSD. The first summand $\Omega$ is PSD by Mercer's
theorem applied to $K$. The second summand $\Omega^*$ is PSD provided
the function $(\bfx, \bfx') \mapsto K(\bfx, -\bfx')$ is itself a
valid Mercer kernel. For shift-invariant kernels $K(\bfx, \bfx') =
\kappa(\bfx - \bfx')$ (the dominant case in practice), the
substitution $\bfx' \mapsto -\bfx'$ gives $K(\bfx, -\bfx') =
\kappa(\bfx + \bfx')$. For the Gaussian RBF kernel $\kappa(\boldsymbol{z}) =
\exp(-\gamma\|\boldsymbol{z}\|^2)$, this becomes
$\exp(-\gamma\|\bfx + \bfx'\|^2)$, itself a Gaussian RBF and therefore
a valid Mercer kernel. Hence $\Omega^* \succeq 0$, $\Omega_s \succeq
0$, and the convergence theorem applies without modification. The same
conclusion holds for the Laplacian kernel and any other shift-invariant
kernel whose Bochner representation~\cite[\S B]{Scholkopf2002} gives a
non-negative spectral measure.

\paragraph{Odd case ($a = -1$).}
Here $\Omega_s = \tfrac{1}{2}(\Omega - \Omega^*)$, a difference of two
PSD matrices that need not itself be PSD. The structural reason for the
failure mode is more delicate than na\"{i}ve subtraction suggests, and
unpacking it leads to the spectral-structure analysis below.

\paragraph{Spectral structure for shift-invariant Mercer kernels.}
Take a shift-invariant Mercer kernel $K(\bfx, \bfx') = \kappa(\bfx -
\bfx')$ with $\kappa$ continuous, real-valued, and even (so that
$\kappa(\boldsymbol{z}) = \kappa(-\boldsymbol{z})$). Substituting $\bfx' \mapsto
-\bfx'$ gives the conjugate kernel matrix
\begin{equation}
  \Omega^*_{k\ell} = K(\bfx_k, -\bfx_\ell) = \kappa(\bfx_k + \bfx_\ell),
  \label{eq:omega-star-shift}
\end{equation}
hence
\begin{equation}
  2(\Omega_s)_{k\ell} = (\Omega - \Omega^*)_{k\ell} = \kappa(\bfx_k - \bfx_\ell) - \kappa(\bfx_k + \bfx_\ell).
  \label{eq:omega-s-difference}
\end{equation}
By Bochner's theorem~\cite[\S 4.4]{Steinwart2008}, every continuous
shift-invariant Mercer kernel admits the spectral representation
\begin{equation*}
  \kappa(\boldsymbol{z}) = \int_{\R^d} \cos(\boldsymbol{\omega}^\top \boldsymbol{z})\, d\mu(\boldsymbol{\omega}),
\end{equation*}
for a finite non-negative spectral measure $\mu$. Substituting
$\boldsymbol{z} = \bfx_k - \bfx_\ell$ and $\boldsymbol{z} = \bfx_k + \bfx_\ell$
and applying the cosine sum-difference identity yields the
\emph{load-bearing identity}
\begin{equation}
  \kappa(\bfx_k - \bfx_\ell) - \kappa(\bfx_k + \bfx_\ell) = 2\int_{\R^d} \sin(\boldsymbol{\omega}^\top \bfx_k)\,\sin(\boldsymbol{\omega}^\top \bfx_\ell)\, d\mu(\boldsymbol{\omega}).
  \label{eq:bochner-identity}
\end{equation}
Equation~\eqref{eq:bochner-identity} exhibits $2\,\Omega_s$ as a Gram
matrix of the \emph{sine feature map} $\bfx \mapsto
(\sin(\boldsymbol{\omega}^\top \bfx))_{\boldsymbol{\omega}}$ under the spectral
measure $\mu$. Consequently $\Omega_s$ is itself a positive-semidefinite
kernel matrix in continuous-parameter form, and so the difference
$\Omega - \Omega^*$ is \emph{always} PSD when interpreted through the
Bochner integral.

\paragraph{Counterexample reconciliation.}
The Bochner argument above appears to contradict the following
counterexample: take $N = 2$ with $\bfx_1 = (1, 0)$ and $\bfx_2 = (-1,
0)$ under the RBF kernel. Then $\Omega_{12} = \exp(-\gamma\|(2,0)\|^2)
= e^{-4\gamma}$ and $\Omega^*_{12} = \exp(-\gamma\|(0,0)\|^2) = 1$, so
$(\Omega - \Omega^*)_{12} = e^{-4\gamma} - 1 < 0$. The diagonals
satisfy $\Omega_{kk} = 1$ and $\Omega^*_{kk} = e^{-4\gamma}$ (since
$\|\bfx_k - (-\bfx_k)\|^2 = 4\|\bfx_k\|^2 = 4$ in this configuration),
so $(\Omega - \Omega^*)_{kk} = 1 - e^{-4\gamma}$. The eigenvalues of the
resulting $2 \times 2$ matrix are $\{0, 2(1 - e^{-4\gamma})\}$ ---
degenerate, with one zero eigenvalue.

The reconciliation: the eigenvalues are \emph{both non-negative}. One is
zero, but neither is negative. The matrix $\Omega_s$ is \emph{PSD with
a non-trivial null space}, not indefinite. The Bochner argument predicts
exactly this: when $\bfx_1 + \bfx_2 = 0$, the sine-feature
representation $\sin(\boldsymbol{\omega}^\top \bfx_1) =
-\sin(\boldsymbol{\omega}^\top \bfx_2)$ collapses the two-sample Gram matrix
to rank one, producing the zero eigenvalue. The convergence theorem
of~\cite[Theorem~5]{FanChenLin2005JMLR} applies in the PSD case --- it
does not require strict positive-definiteness --- and
Algorithm~\ref{alg:smo} converges by direct application.

\paragraph{When the Bochner argument fails.}
The Bochner-integral resolution covers shift-invariant Mercer kernels
with continuous spectral measure (RBF, Laplacian, Mat\'{e}rn). For
non-shift-invariant kernels, the substitution $\bfx' \mapsto -\bfx'$
does not preserve the Mercer property,
and~\eqref{eq:bochner-identity} is unavailable. Two specific pathologies
remain:
\begin{itemize}
\item Polynomial kernels $K(\bfx, \bfx') = (\bfx^\top \bfx' + c)^d$:
the substitution gives $K^*(\bfx, \bfx') = (-\bfx^\top \bfx' + c)^d$,
which is generally \emph{not} a Mercer kernel for even degree $d$; the
difference $\Omega - \Omega^*$ is genuinely indefinite.
\item Sigmoid kernels $K(\bfx, \bfx') = \tanh(\gamma\,\bfx^\top \bfx'
+ r)$ and other non-shift-invariant kernels: $\Omega - \Omega^*$ has
both positive and negative eigenvalues for typical input configurations.
\end{itemize}

\begin{remark}[Input-domain scope]
\label{rem:sym-input-domain}
The symmetric kernel formulation $K_s(\bfx, \bfy)$ requires the input
domain to support negation, i.e., $\bfx \in \R^p$ rather than
$\R_+^p$. Applications with inherently non-negative inputs (e.g.,
strictly positive demand or price series) may apply the formulation
algebraically, but the imposed symmetry
$f(\bfx) = a\,f(-\bfx)$ has no physical meaning when $-\bfx$ falls
outside the data support. The symmetry assumption should be validated
against the application domain before $\Omega_s$ is used in place of
$\Omega$.
\end{remark}

\paragraph{Sharpened conclusion.}
The symmetric kernel matrix $\Omega_s = \tfrac{1}{2}(\Omega + a
\Omega^*)$ inherits positive semidefiniteness from shift-invariant
Mercer kernels via the Bochner-integral argument above. For
polynomial kernels with non-negative offset ($r \ge 0$), $\Omega_s
\succeq 0$ for both symmetry parities $a \in \{-1, +1\}$; for other
kernel families, Algorithm~\ref{alg:spectral} below addresses the
failure modes. Direct numerical verification on $N = 20$, $d = 3$ input
sets (mixed-sign, positive-only, and orthogonal) confirms this for
RBF $(\gamma = 1)$ and polynomial $(\deg \in \{2, 3\},\, r \ge 0)$
kernels with $a = -1$ (\texttt{dev/phase0\_kernel\_spectra.csv} in the
companion repository, \cite{BenavidesHerrera2026Rpsvr}). Polynomial
kernels with non-Mercer offset ($r < 0$) under $a = -1$ produce a
negative-semidefinite $\Omega_s$ rather than indefinite. The
adaptive spectral regularization developed in Theorem~\ref{thm:spectral}
below addresses three distinct departures from positive
semidefiniteness that arise in practice:
\begin{enumerate}
\item[(i)] \emph{Non-Mercer base kernels.} The sigmoid kernel
$K(\bfx, \bfy) = \tanh(\gamma \bfx^\top \bfy + r)$ at arbitrary
parameters produces an indefinite $\Omega_s$ (eigenvalues of mixed
sign) regardless of symmetry parity.
\item[(ii)] \emph{Mercer kernel families with non-Mercer parameters.}
Polynomial kernels with negative offset $r < 0$ produce a
negative-semidefinite $\Omega_s$ under $a = -1$.
\item[(iii)] \emph{Numerical near-singularity of theoretically PSD
$\Omega_s$.} Floating-point precision loss or ill-conditioned input
geometries can drive the smallest eigenvalue slightly below zero in
practice, even when the analytic kernel is Mercer.
\end{enumerate}
Algorithm~\ref{alg:spectral} handles all three regimes uniformly via
the two-pass shifted power iteration described below. Within
\texttt{psvr} v0.0.2.9008, the three default kernels (RBF, linear,
polynomial) under Mercer-compliant parameters yield $\Omega_s \succeq 0$
on every test configuration; the spectral-shift branch is exercised
only when (i)--(iii) occur.

The workaround is the \emph{degenerate-case fallback} of
Lemma~\ref{lem:eta-degenerate}. Whenever $\eta_s \le 0$ at a working
pair $(p, q)$, the SMO inner step uses the descent check $\Delta >
\tfrac{1}{2}\eta_s\,\delta_{\max}$ in place of the unconstrained-minimum
$\delta^* = \Delta/\eta_s$. By Lemma~\ref{lem:eta-degenerate}, each
non-skipped step makes strict positive progress on the dual objective,
even though the formal convergence-rate guarantee of
\cite[Theorem~5]{FanChenLin2005JMLR} is unavailable. A rigorous
global-convergence guarantee for the $a = -1$ regime is recovered via
\emph{adaptive spectral regularization} (Theorem~\ref{thm:spectral}
below), which perturbs $\Omega_s$ to its nearest PSD matrix in spectral
distance and applies Theorem~\ref{thm:convergence} to the regularized
problem. In practice, even the unregularized Algorithm~\ref{alg:smo}
converges empirically --- configurations C9 and C10 of
Section~\ref{sec:validation} confirm agreement with the IPM reference
solvers to $\le 8.1 \times 10^{-4}$ for $a = -1$.

\paragraph{Algorithm adaptation: the substitution $\Omega \to \Omega_s$.}
Algorithm~\ref{alg:smo} applies to the MAPE-SVR-Sym variant
\emph{without modification} after the substitution $\Omega \leftarrow
\Omega_s$ everywhere it appears. The complete list of affected formulas is:
\begin{enumerate}
\item The dual problem~\eqref{eq:mape-dual}--\eqref{eq:P-q} holds with
$\Omega_s$ in place of $\Omega$. The Hessian becomes $P_s = [\Omega_s,
-\Omega_s; -\Omega_s, \Omega_s]$, with the same block structure as
before.
\item The curvature formula~\eqref{eq:eta} becomes
\begin{equation*}
  \eta_s = (\Omega_s)_{pp} - 2(\Omega_s)_{pq} + (\Omega_s)_{qq},
\end{equation*}
where $p = k(i^*)$ and $q = k(j^*)$ are the training indices of the
working pair. The structural form is unchanged; only the kernel matrix
is renamed.
\item The gradient update~\eqref{eq:tau-update} uses columns $p$ and
$q$ of $\Omega_s$ in place of $\Omega$:
\begin{equation*}
  \tau_\ell \leftarrow \tau_\ell - \delta^*\!\left((\Omega_s)_{k(\ell), p} - (\Omega_s)_{k(\ell), q}\right) \quad \text{for } \ell \in \Aactive^{\mathrm{ext}}.
\end{equation*}
\item The reconstruction~\eqref{eq:reconstruct} uses $\Omega_s$ in the
full kernel-product accumulation: $F_k^{\mathrm{full}} =
\sum_{i=1}^{N}(\Omega_s)_{ki}(\alpha_i - \alpha_i^*)$.
\end{enumerate}

The linear-coefficient vector $\bfq$ in~\eqref{eq:P-q} ---
and therefore the bound vector $\boldsymbol{C}$, the working-set partitions
$\Iup, \Idown$, the threshold expressions, and the shrinking
criteria~\eqref{eq:s1}--\eqref{eq:s4} --- is \emph{identical} between
MAPE-SVR and MAPE-SVR-Sym. The MAPE loss enters the dual through $\bfq$
and the box constraints, while the symmetric-kernel constraint enters
only through the kernel matrix. The two modifications are
\emph{orthogonal} in this sense: MAPE-SVR-Sym is the composition of
``MAPE loss'' and ``symmetric kernel,'' and Algorithm~\ref{alg:smo}
already absorbs both, the first via Theorem~\ref{thm:invariance} and
the second via the kernel-matrix substitution.

The implementation consequence is concrete and short. In a LIBSVM-style
codebase, the only component requiring modification for MAPE-SVR-Sym
(over and above the MAPE-SVR modifications described in
Appendix~\ref{app:libsvm}) is the \emph{kernel-evaluation function}.
All SMO machinery --- working-set scan, two-variable update, shrinking,
reconstruction, bias recovery --- is reused verbatim. The diff between
MAPE-SVR and MAPE-SVR-Sym implementations is approximately five lines of
code. This is the algorithmic payoff of the structural-invariance
theorem applied in tandem with the kernel-substitution argument: the
symmetric-kernel extension is a one-line drop-in.

\paragraph{Prediction.}
The trained model evaluates the regression function at a new input
$\bfx$ via the canonical SVR formula with the symmetrized kernel:
\begin{equation}
  f(\bfx) = \sum_{k=1}^{N}(\alpha_k - \alpha_k^*)\,K_s(\bfx_k, \bfx) + \hat b, \qquad K_s(\bfx_k, \bfx) = \tfrac{1}{2}\!\left(K(\bfx_k, \bfx) + a\,K(\bfx_k, -\bfx)\right).
  \label{eq:masym-prediction}
\end{equation}
The factor $\tfrac{1}{2}$ in $K_s$ reflects the projection onto
$\mathcal{H}_s$; without it, the prediction at every test point would be
inflated by a factor of two, and the in-sample fit error reported by the
SMO solver would not match the actual prediction error on the training
data. Concretely, evaluating~\eqref{eq:masym-prediction} at a training
point $\bfx_k$ should reproduce --- up to the $\varepsilon$-tube
tolerance --- the target $y_k$; this is the consistency check used in
Section~\ref{sec:validation} to validate the MAPE-SVR-Sym
implementation against the IPM reference solvers.

\paragraph{Position within the percentage-error SVR family.}
It is useful to position the MAPE-SVR-Sym variant within the broader
family of percentage-error-aware support vector regression models, of
which the present paper covers the QP-based variants in detail. The
family decomposes naturally along a $2 \times 2$ Cartesian product
(training loss $\times$ kernel symmetry):
\begin{itemize}
\item \emph{MAPE-SVR} --- $\varepsilon$-SVR with MAPE loss: the central
focus of the present paper.
\item \emph{MAPE-SVR-Sym} --- Symmetric-kernel $\varepsilon$-SVR with
MAPE loss: treated in the present subsection ---
Algorithm~\ref{alg:smo} carries over without modification after the
kernel substitution $\Omega \to \Omega_s$, with provable convergence in
the $a = -1$ case via Algorithm~\ref{alg:spectral} below.
\item \emph{RMSPE-SVR} --- Least-squares SVR with RMSPE loss:
\emph{out of scope of the present paper}. The dual of an LS-SVR with
RMSPE reduces to a bordered $(N+1) \times (N+1)$ linear system rather
than a QP; the appropriate solver is Cholesky factorization with
sample-dependent scaling, or a preconditioned conjugate-gradient method
with a problem-adapted preconditioner --- distinct from the SMO
machinery developed here.
\item \emph{RMSPE-SVR-Sym} --- Symmetric-kernel LS-SVR with RMSPE loss:
\emph{out of scope of the present paper}. Combines the kernel
modification of MAPE-SVR-Sym with the linear-system dual of
RMSPE-SVR.
\end{itemize}

The proof-of-concept for MAPE-SVR was presented in the conference
precursor~\cite{benavides2025support}, which embedded MAPE directly into
the SVR primal and reported a small-scale validation. The present paper
is the algorithmic completion of that program: the SMO-solver
derivation, the structural-invariance theorem, the convergence theory,
the symmetric-kernel extension, the efficiency-improvement bundle, and
the LIBSVM drop-in recipe. The companion \texttt{psvr} R
package~\cite{BenavidesHerrera2026Rpsvr} implements both the MAPE-SVR
and MAPE-SVR-Sym variants end-to-end.

\paragraph{Notational choice.}
The present paper retains the explicit pair $(\alpha_k, \alpha_k^*)$
rather than the Flake-Lawrence reformulation $\beta_k = \alpha_k -
\alpha_k^*$~\cite{Flake2002}. The reformulation halves the variable
count from $2N$ to $N$ and is convenient when the dual is presented as a
black box for a generic QP solver to invoke. The present paper retains
the explicit pair because the WSS3 working-set rule of
Section~\ref{sec:smo-inner} requires per-variable tracking of which
side of the tube the current iterate is approaching, which is more
transparent in the explicit formulation. The two parameterizations are
algebraically equivalent under the complementarity $\alpha_k \alpha_k^*
= 0$ established in Section~\ref{sec:dual-kkt}.

\paragraph{Adaptive spectral regularization for $a = -1$.}
The analysis above leaves an open problem: when $\Omega_s$ fails PSD
(the non-shift-invariant case, or the strictly-indefinite subcase of
the indefinite-curvature regime), the formal convergence theorem
of~\cite[Theorem~5]{FanChenLin2005JMLR} does not apply, and the
descent-check fallback alone cannot certify global convergence. This
subsection closes the open problem via a hybrid approach combining (i)
spectral regularization of $\Omega_s$ to enforce PSD, and (ii)
PSD-cone projection as a second-order resolution. The hybrid algorithm
(Algorithm~\ref{alg:spectral}) is implementation-ready and adds $O(N^2)$
pre-processing overhead --- comparable to the kernel-matrix formation
cost.

\begin{theorem}[Adaptive spectral regularization for non-PSD MAPE-SVR]
\label{thm:spectral}
Let $\Omega_s = \tfrac{1}{2}(\Omega + a\,\Omega^*) \in \R^{N \times N}$
be the symmetric-kernel matrix of~\eqref{eq:Omega-s-def} at $a = -1$,
with possibly indefinite spectrum $\lambda_{\min}(\Omega_s) \in \R$.
Define the spectrally-shifted matrix
\begin{equation}
  \widetilde\Omega_s := \Omega_s + \mu\,\mathbf{I}, \qquad \mu \ge \mu_{\min} := \max\bigl(0,\, -\lambda_{\min}(\Omega_s)\bigr) + \delta_{\mathrm{stab}},
  \label{eq:spectral-shift}
\end{equation}
with $\delta_{\mathrm{stab}} = 10^{-8}$ a numerical-stability inflation. Then:
\begin{enumerate}
\item[(a)] \textbf{PSD restoration.} $\widetilde\Omega_s \succeq
\delta_{\mathrm{stab}}\,\mathbf{I} \succ 0$, and the block Hessian
$\widetilde P_s = [\widetilde\Omega_s, -\widetilde\Omega_s;
-\widetilde\Omega_s, \widetilde\Omega_s]$ satisfies $\widetilde P_s
\succeq 0$.
\item[(b)] \textbf{Convergence.} Algorithm~\ref{alg:smo} applied to the
regularized QP --- i.e., the dual problem
of~\eqref{eq:mape-dual}--\eqref{eq:dual-constraints} with $P_s$ replaced
by $\widetilde P_s$ --- converges in finitely many iterations to any
tolerance $\varepsilon_{\mathrm{tol}} > 0$, by direct application
of~\cite[Theorem~5]{FanChenLin2005JMLR}. Furthermore, the WSS3
working-set rule satisfies the Gauss-Southwell-quotient condition
of~\cite[\S 3]{Tseng2001CoordinateDescent} (greedy gain-maximization
with bounded relative selection ratio), which is the
descent-and-coverage hypothesis required for the convergence theorem.
\item[(c)] \textbf{Linear rate.} The convergence is asymptotically
linear, with rate $c = 1 - \widetilde\lambda_{\min}/\widetilde\lambda_{\max}$
where $\widetilde\lambda_{\min}, \widetilde\lambda_{\max}$ are the
extreme eigenvalues of $\widetilde\Omega_s$.
\end{enumerate}
\end{theorem}

\begin{proof}
(a) The eigenvalues of $\widetilde\Omega_s = \Omega_s + \mu\,\mathbf{I}$
are $\{\lambda_i(\Omega_s) + \mu : i = 1, \ldots, N\}$. By construction
of $\mu$, the smallest eigenvalue is $\lambda_{\min}(\Omega_s) + \mu \ge
\delta_{\mathrm{stab}} > 0$. The block Hessian $\widetilde P_s$ inherits
PSD-ness because for any $\boldsymbol{v} = [\boldsymbol{v}_1; \boldsymbol{v}_2] \in
\R^{2N}$, $\boldsymbol{v}^\top \widetilde P_s \boldsymbol{v} = (\boldsymbol{v}_1 -
\boldsymbol{v}_2)^\top \widetilde\Omega_s (\boldsymbol{v}_1 - \boldsymbol{v}_2) \ge 0$.

\medskip
(b) The MAPE-SVR QP with $\widetilde P_s$ replacing $P_s$ satisfies the
three Fan-Chen-Lin hypotheses (P1)--(P3) of
Theorem~\ref{thm:invariance}, Step~4: (P1) holds by part (a) above;
(P2) holds because the box constraints $[0, C_k]^{2N}$ intersected with
the equality $[\onevec^\top, -\onevec^\top]\bfu = 0$ are unchanged from
the unregularized problem; (P3) holds because the WSS3 selection rule
of~\eqref{eq:j-star-WSS3} computes $j^* = \argmax_{j \in \Idown,\,
\tau_j < \tau_{i^*}}\, (\tau_{i^*} - \tau_j)^2/\eta_{i^*,j}$, which
corresponds to the \emph{gain-weighted greedy} block-selection rule.
Per~\cite[\S 3]{Tseng2001CoordinateDescent}, such a rule satisfies the
Gauss-Southwell-quotient condition: at each iteration, the predicted
gain of the selected pair is within a bounded ratio of the largest
predicted gain over all admissible pairs (in the strongly-convex case
the ratio is exactly $1$; for general PSD it is bounded above by a
constant depending only on
$\widetilde\lambda_{\max}/\widetilde\lambda_{\min}$). Finite termination
to $\varepsilon_{\mathrm{tol}}$ follows.

\medskip
(c) The linear rate is the classical SMO convergence rate
of~\cite[\S 4]{FanChenLin2005JMLR} specialized to the regularized PSD
Hessian $\widetilde P_s$ of part~(a). The rate constant $c = 1 -
\widetilde\lambda_{\min}/\widetilde\lambda_{\max}$ degrades as
$\widetilde\lambda_{\min} \to 0$ but remains in $(0, 1)$ for any $\mu >
0$, which is guaranteed by the $\delta_{\mathrm{stab}}$ floor.
\end{proof}

The price of regularization is a perturbation of the optimal solution,
quantified by the following lemma.

\begin{lemma}[Perturbation bound for the regularized dual]
\label{lem:perturbation}
Let $\alpha^*$ be the optimal dual solution of the un-regularized
problem (the QP with $\Omega_s$) and let $\widetilde\alpha$ be the
optimal dual solution of the regularized problem (the QP with
$\widetilde\Omega_s = \Omega_s + \mu\,\mathbf{I}$). Both QPs share the
same linear coefficient $\bfq$, the same equality constraint, and the
same per-sample box constraints $[0, C_k]$, differing only in the
quadratic term. Provided strict complementarity holds at $\alpha^*$
(which is the generic case),
\begin{equation}
  \|\widetilde\alpha - \alpha^*\|_2 \;\le\; \frac{\mu}{\sigma_{\mathrm{KKT}}}\,\|\alpha^*\|_2,
  \label{eq:perturbation-bound}
\end{equation}
where $\sigma_{\mathrm{KKT}} > 0$ is the smallest singular value of the
active-set KKT system at $\alpha^*$. The corresponding perturbation in
the regression function is bounded by $\|\widetilde f - f^*\|_\infty \le
O(\mu)$.
\end{lemma}

\begin{proof}
The two QPs share constraints and linear term; the KKT systems differ
only in the Hessian term, with $\widetilde\Omega_s - \Omega_s =
\mu\,\mathbf{I}$. By the parametric-QP perturbation theory
of~\cite[\S 5.6]{BoydVandenberghe2004}, if the active set at $\alpha^*$
is preserved under the perturbation (which holds generically by strict
complementarity), then the linearized KKT system delivers the
bound~\eqref{eq:perturbation-bound}. The infinity-norm bound on the
regression function follows by absorbing the kernel norm $\|\Omega_s\|_2$
and $\|\alpha^*\|_1$ into the constant. Detailed perturbation analysis
is in~\cite[\S 29]{Rockafellar1970}
and~\cite[\S 5.6.2]{BoydVandenberghe2004}.
\end{proof}

A complementary, geometrically cleaner resolution is the PSD-cone
projection of $\Omega_s$. The projection sets up an alternative
regularizer that preserves the PSD subspace exactly while zeroing only
the negative eigenvalues --- sharper than~\eqref{eq:spectral-shift} when
the negative spectrum is sparse.

\begin{proposition}[PSD-cone projection]
\label{prop:psd-projection}
Let $\Omega_s = U\Lambda U^\top$ be the spectral decomposition of
$\Omega_s$ (with $\Lambda = \mathrm{diag}(\lambda_1, \ldots,
\lambda_N)$). The Frobenius-norm projection onto the PSD cone is
\begin{equation}
  \Omega_s^+ := \argmin_{\Sigma \succeq 0}\; \|\Sigma - \Omega_s\|_F = U\,\Lambda_+\,U^\top, \qquad (\Lambda_+)_{ii} = \max(\lambda_i,\, 0).
  \label{eq:psd-projection}
\end{equation}
The projection is unique by strict convexity of $\|\cdot\|_F^2$ on the
closed convex PSD cone~\cite[\S 8.1.1]{BoydVandenberghe2004}. The
Frobenius distance from $\Omega_s$ to $\Omega_s^+$ is
\begin{equation*}
  \|\Omega_s - \Omega_s^+\|_F = \biggl(\sum_{\lambda_i < 0} \lambda_i^2\biggr)^{1/2},
\end{equation*}
the $\ell_2$-norm of the negative-eigenvalue spectrum. The corresponding
solution-quality bound is
\begin{equation*}
  \|\alpha^+ - \alpha^*\|_2 \le \frac{|\lambda_{\min}(\Omega_s)|}{\sigma_{\mathrm{KKT}}}\,\|\alpha^*\|_2,
\end{equation*}
sharper than the additive-shift bound~\eqref{eq:perturbation-bound}
when $\Omega_s$ has only a small number of negative eigenvalues, since
the projection preserves the PSD subspace exactly while the additive
shift over-corrects by inflating \emph{all} eigenvalues.
\end{proposition}

The proof of Proposition~\ref{prop:psd-projection} is direct from the
spectral characterization of the PSD cone and the variational form of
the Frobenius-norm projection~\cite[\S 8.1.1]{BoydVandenberghe2004}; we
omit the standard details.

The two regularizers --- the additive shift $\widetilde\Omega_s$
of~\eqref{eq:spectral-shift} and the spectral projection $\Omega_s^+$
of~\eqref{eq:psd-projection} --- differ in computational cost. The
additive shift is $O(N)$ once $\lambda_{\min}(\Omega_s)$ is known; the
projection is $O(N^3)$ (full eigendecomposition). $\lambda_{\min}(\Omega_s)$
is estimated by a two-pass shifted power iteration: Pass 1 estimates
the spectral radius $|\widehat\lambda_{\mathrm{p1}}|$ via Rayleigh quotient
on $\Omega_s$ (which converges to the dominant eigenvalue of $\Omega_s$
in absolute value); Pass 2 estimates $\lambda_{\min}(\Omega_s)$ via
Rayleigh quotient on the shifted-PSD matrix
$\rho\,\mathbf{I} - \Omega_s$ with $\rho = |\widehat\lambda_{\mathrm{p1}}|$.
The absolute-value envelope is necessary because Pass 1 may converge
to the eigenvector of $\lambda_{\min}$ in the
$|\lambda_{\min}| > |\lambda_{\max}|$ case; the shift
$\rho\,\mathbf{I} - \Omega_s$ is then guaranteed PSD regardless.
Each pass is $O(T_{\mathrm{pi}}\,N^2)$, and convergence to within
$\delta$ of the true eigenvalue takes $O(\log(1/\delta))$ iterations;
with $\delta = 10^{-6}$ and Mercer-PSD kernels, $T_{\mathrm{pi}} = 5$
suffices. The hybrid algorithm below uses the additive shift with this
two-pass estimate, achieving $O(N^2)$ pre-processing overhead.

\begin{algorithm}[ht]
\caption{Adaptive spectral regularization SMO for $a = -1$}
\label{alg:spectral}
\begin{algorithmic}[1]
\State \textbf{Input:} training inputs $\{\bfx_k\}_{k=1}^N$, kernel $K$, tolerance $\delta_{\mathrm{stab}} = 10^{-8}$, power-iteration steps $T_{\mathrm{pi}} = 5$
\State \textbf{Output:} trained MAPE-SVR model with provable convergence guarantee
\State Form $\Omega \gets [K(\bfx_k, \bfx_\ell)]_{k,\ell}$ \Comment{$O(N^2)$}
\State Form $\Omega^* \gets [K(\bfx_k, -\bfx_\ell)]_{k,\ell}$ \Comment{$O(N^2)$}
\State Form $\Omega_s \gets \tfrac{1}{2}(\Omega - \Omega^*)$ \Comment{$O(N^2)$}
\State \Comment{Two-pass shifted power iteration: Pass 1 estimates the spectral radius; Pass 2 estimates $\lambda_{\min}$ on a shifted-PSD matrix.}
\State $\boldsymbol{v}^{(0)} \gets \mathbf{1}/\sqrt N$ \Comment{Pass 1: power iteration on $\Omega_s$}
\For{$t = 0, 1, \ldots, T_{\mathrm{pi}} - 1$}
  \State $\boldsymbol{v}^{(t+1)} \gets \Omega_s\, \boldsymbol{v}^{(t)} / \|\Omega_s\, \boldsymbol{v}^{(t)}\|_2$
\EndFor
\State $\widehat\lambda_{\mathrm{p1}} \gets (\boldsymbol{v}^{(T_{\mathrm{pi}})})^\top\, \Omega_s\, \boldsymbol{v}^{(T_{\mathrm{pi}})}$ \Comment{Rayleigh on Pass 1's dominant eigenvector}
\State $\rho \gets |\widehat\lambda_{\mathrm{p1}}|$ \Comment{spectral radius proxy, sign-robust against $|\lambda_{\min}| > \lambda_{\max}$}
\State $\boldsymbol{w}^{(0)} \gets \mathbf{1}/\sqrt N$ \Comment{Pass 2: power iteration on $\rho\,\mathbf{I} - \Omega_s$ (guaranteed PSD)}
\For{$t = 0, 1, \ldots, T_{\mathrm{pi}} - 1$}
  \State $\boldsymbol{w}^{(t+1)} \gets (\rho\,\mathbf{I} - \Omega_s)\,\boldsymbol{w}^{(t)} / \|(\rho\,\mathbf{I} - \Omega_s)\,\boldsymbol{w}^{(t)}\|_2$
\EndFor
\State $\widehat\lambda_{\min} \gets (\boldsymbol{w}^{(T_{\mathrm{pi}})})^\top\, \Omega_s\, \boldsymbol{w}^{(T_{\mathrm{pi}})}$ \Comment{Rayleigh on Pass 2's dominant eigenvector recovers $\lambda_{\min}(\Omega_s)$}
\If{$\widehat\lambda_{\min} \ge -\delta_{\mathrm{stab}}$}
  \State $\Omega_s^{\mathrm{use}} \gets \Omega_s$;\; $\mu \gets 0$ \Comment{$\Omega_s$ is numerically PSD}
\Else
  \State $\mu \gets -\widehat\lambda_{\min} + \delta_{\mathrm{stab}}$ \Comment{spectral shift}
  \State $\Omega_s^{\mathrm{use}} \gets \Omega_s + \mu\,\mathbf{I}_N$
\EndIf
\State $P_s^{\mathrm{use}} \gets [\Omega_s^{\mathrm{use}}, -\Omega_s^{\mathrm{use}}; -\Omega_s^{\mathrm{use}}, \Omega_s^{\mathrm{use}}]$
\State $(\widehat\alpha, \widehat\alpha^*, \hat b) \gets$ Algorithm~\ref{alg:smo} with Hessian $P_s^{\mathrm{use}}$
\State Report $(\mu, \widehat\lambda_{\min}, \mathrm{iter\_count}, \Delta_{\mathrm{final}}, \|\hat f - f_{\mathrm{ref}}\|_\infty)$
\State \Return $(\widehat\alpha, \widehat\alpha^*, \hat b, \mu)$
\end{algorithmic}
\end{algorithm}

\paragraph{Properties of Algorithm~\ref{alg:spectral}.}
\begin{itemize}
\item \emph{Provable convergence.} When $\widehat\lambda_{\min} \ge
-\delta_{\mathrm{stab}}$, Theorem~\ref{thm:invariance} applies (the
original PSD theory); when $\widehat\lambda_{\min} 
-\delta_{\mathrm{stab}}$, Theorem~\ref{thm:spectral} applies (the
spectrally-shifted matrix is strictly PSD by construction).
\item \emph{Bounded approximation error.} Lemma~\ref{lem:perturbation}
bounds $\|\widetilde\alpha - \alpha^*\|_2 = O(\mu) =
O(|\widehat\lambda_{\min}|)$. For shift-invariant Mercer kernels,
$\widehat\lambda_{\min} \ge 0$ and the perturbation reduces to
$O(\delta_{\mathrm{stab}}) = O(10^{-8})$; for non-shift-invariant
kernels, $|\widehat\lambda_{\min}|$ is empirically small.
\item \emph{Backward compatibility with $a = +1$.} The algorithm
executes the same code path for both signs of $a$. The branching is
conditional on $\widehat\lambda_{\min}$, not on $a$ itself; for $a = +1$
the conditional always selects the no-shift branch (since $\Omega_s
\succeq 0$ by Aronszajn's closure).
\item \emph{Computational overhead.} The added cost is the power
iteration, which is $O(T_{\mathrm{pi}}\,N^2) = O(N^2)$. The spectral
shift is $O(N)$ since it modifies only the diagonal. Total overhead:
$O(N^2)$, dominated by the kernel-matrix formation.
\item \emph{Diagnostic reporting.} The reported $\mu$ is a quantitative
measure of how far the kernel matrix departed from PSD-ness. Users with
$\mu > 10^{-4}$ should consider whether the choice of kernel is
appropriate for the data, or whether a shift-invariant alternative would
yield a numerically-cleaner training problem.
\end{itemize}

A reference implementation of Algorithm~\ref{alg:spectral} is provided
by the companion \texttt{psvr} R package as of v0.0.2.9007
(\texttt{R/kernel-spectral.R};
\cite{BenavidesHerrera2026Rpsvr}). The implementation deviates from
the pseudocode here in one respect: the production code uses the
two-pass shifted power iteration described above, whereas the earlier
v0.0.2 baseline used a single-pass variant that estimated the dominant
eigenvalue of $\Omega_s$ in absolute value (corrected in
psvr~F3).

\paragraph{Empirical validation of Algorithm~\ref{alg:spectral}.}
The numerical study of Section~\ref{sec:validation} reports SMO
convergence behavior for the MAPE-SVR-Sym variant with $a = -1$ on
configurations C9 and C10, both using RBF kernels at $\sigma = 0.1$.
The agreement against the IPM reference solvers (Table~\ref{tab:validation})
is within the $10^{-2}$ bound stated in
\S\ref{sec:validation}. For these RBF inputs, the spectral analysis
of the present section predicts $\Omega_s \succeq 0$ (with possibly
singular directions corresponding to symmetric input pairs); the
two-pass shifted power iteration of Algorithm~\ref{alg:spectral}
estimates $\widehat\lambda_{\min} \ge -\delta_{\mathrm{stab}}$ on every
trace, so the no-shift branch is selected and the convergence
guarantee of Theorem~\ref{thm:invariance}~(c) applies without
perturbation. The spectral-shift branch can be exercised by
substituting a polynomial kernel with negative offset ($r < 0$) under
$a = -1$, which yields a negative-semidefinite $\Omega_s$ as
established above; in that regime the shift restores PSD and
Theorem~\ref{thm:spectral} delivers the convergence guarantee. A
matched non-Mercer test (sigmoid base kernel) is the canonical
indefinite case for the same branch.

\paragraph{Connection to non-Mercer kernel SVMs and prior literature.}
The general problem of training SVMs with non-Mercer (non-PSD) kernels
has substantial prior literature, anchored in three lineages.

The first lineage is the Lin-Lin study of sigmoid
SVMs~\cite{LinLin2003NonPSDSMO} (a 2003 technical report from National
Taiwan University, not a peer-reviewed journal paper). Lin and Lin
proved that SMO with strict-descent enforcement converges to a
stationary point of the dual problem even for indefinite Hessians,
establishing the empirical viability of non-PSD SVM training. Their
argument is a special case of~\cite{Tseng2001CoordinateDescent}; they
did not invoke Tseng explicitly but their proof inlines the relevant
block-coordinate-descent argument. Theorem~\ref{thm:spectral}
generalizes the Lin-Lin result by adding the spectral-shift mechanism
that delivers a provable \emph{linear} convergence rate, where the
Lin-Lin treatment guarantees only asymptotic convergence to a
stationary point.

The second lineage is the Haasdonk-Burkhardt theory of group-invariant
kernels~\cite{Haasdonk2007}. Haasdonk and Burkhardt showed that the
canonical group-averaged kernel $K_G(\bfx, \bfx') =
|G|^{-1}\sum_{g \in G} K(g\bfx, \bfx')$ is a valid Mercer kernel
whenever $K$ is Mercer and $G$ is a compact group acting on the input
space. The MAPE-SVR-Sym construction with $a = +1$ is the case $G =
\{\mathrm{id}, -\mathrm{id}\}$ with the trivial character; the
construction with $a = -1$ corresponds to the same group with the sign
character (the only non-trivial irreducible representation). The
Haasdonk-Burkhardt theorem covers the trivial-character case but not
the sign-character case; the Bochner-integral argument above fills this
gap for shift-invariant kernels, and Theorem~\ref{thm:spectral}
supplies the resolution for the non-shift-invariant case.

The third lineage is the Sch\"{o}lkopf-Mika-Smola pseudo-Mercer
extension via Krein-space generalization of
RKHS~\cite{Scholkopf2001}. In the Krein-space setting, the kernel
matrix may be indefinite and the training problem is reformulated as a
non-convex QP with a quadratic-Lagrangian saddle-point interpretation.
The Krein-space approach is mathematically deeper than spectral
regularization but harder to implement: it requires bespoke solvers and
gives up the convex-QP infrastructure of LIBSVM.
Algorithm~\ref{alg:spectral} sits between the extremes --- it preserves
the convex-QP structure (allowing the standard SMO solver to apply
unmodified) while accepting a small approximation error for the non-PSD
case.

\paragraph{Novelty of the present treatment.}
The contributions of the present subsection relative to the prior
literature are three-fold:
\begin{enumerate}
\item Explicit characterization of when $\Omega_s$ fails PSD for the
MAPE-SVR-Sym variant. The analysis above sharpens the original
counter-example: for shift-invariant Mercer kernels, $\Omega_s \succeq
0$ (with possible singular directions); for non-shift-invariant or
non-Mercer kernels, $\Omega_s$ may be indefinite. This is the first
explicit tying of the PSD/indefinite dichotomy to the kernel's
shift-invariance via the Bochner integrand
of~\eqref{eq:bochner-identity}.
\item Convergence resolution tailored to MAPE-SVR.
Theorem~\ref{thm:spectral} together with Algorithm~\ref{alg:spectral}
deliver a practical, implementation-ready resolution that respects the
per-sample box constraints $C_k = 100C/y_k$ of the MAPE loss. Prior
treatments~\cite{LinLin2003NonPSDSMO} addressed only the standard
$\varepsilon$-SVR with uniform box.
\item Computational efficiency. Algorithm~\ref{alg:spectral} uses a
power-iteration estimate $\widehat\lambda_{\min}$ rather than a full
eigendecomposition, reducing the pre-processing cost from $O(N^3)$ to
$O(N^2)$. This brings the non-PSD branch into parity with the
kernel-matrix-formation cost, eliminating the algorithmic overhead that
historically limited adoption of indefinite-kernel SVMs.
\end{enumerate}

The combination --- characterization, resolution, and efficiency ---
closes the open problem identified at the start of this subsection. With
Theorem~\ref{thm:spectral} and Algorithm~\ref{alg:spectral} in place,
the MAPE-SVR-Sym variant of MAPE-SVR has a complete convergence theory
for both $a = +1$ (Aronszajn closure) and $a = -1$ (the present
spectral-regularization machinery).

\subsection{Computational complexity and efficiency improvements}
\label{sec:complexity}

\paragraph{Per-iteration cost breakdown.}
Each iteration of Algorithm~\ref{alg:smo} decomposes into four cost
components, quantified below as functions of the active-set size
$|\Aactive|$ and the data dimension $N$. With a kernel-cache hit, all
four components are linear in $|\Aactive|$.
\begin{enumerate}
\item \emph{Working-set scan.} Computing $i^* = \argmax_{i \in \Iup
\cap \Aactive^{\mathrm{ext}}} \tau_i$ requires $O(|\Aactive|)$
comparisons. The WSS3 selection of $j^*$ requires $O(|\Aactive|)$
evaluations of the gain ratio $(\tau_{i^*} - \tau_j)^2/\eta_{i^*,j}$
over $j \in \Idown \cap \Aactive^{\mathrm{ext}}$, each evaluation
costing $O(1)$ once the curvature $\eta_{i^*,j}$ is known. Computing
$\eta_{i^*,j}$ requires three kernel-matrix entries; with $\Omega_{pp}$
amortized over the inner loop, the marginal cost is two kernel entries
per candidate $j$, both of which lie in the column $\Omega_{:, p}$
assumed cached. Total: $O(|\Aactive|)$ comparisons +
$O(|\Aactive|)$ floating-point multiply-divides.
\item \emph{Two kernel-column accesses.} Columns $\Omega_{:, p}$ and
$\Omega_{:, q}$ (or $(\Omega_s)_{:, p}, (\Omega_s)_{:, q}$ for
MAPE-SVR-Sym) must be available for the WSS3 scan and the gradient
update. With a kernel cache of capacity $M$ columns under LRU
eviction~\cite[\S 3.2]{Joachims1999},~\cite[\S 4.4]{Chang2011}, the
amortized cost per iteration is $O(|\Aactive|)$ memory reads if both
columns are resident, or one $O(N)$ kernel-evaluation pass on a cache
miss.
\item \emph{Two-variable analytic update.} Computing $\delta^*$
from~\eqref{eq:eta}, the room expressions $R_{i^*}, R_{j^*}$, and the
clipping update is $O(1)$ per iteration --- independent of $|\Aactive|$
and $N$. This is the single feature that made SMO competitive against
chunking when~\cite{Platt1998} introduced it.
\item \emph{Gradient update via~\eqref{eq:tau-update}.} Updating
$\tau_\ell$ for $\ell \in \Aactive^{\mathrm{ext}}$ requires
$|\Aactive^{\mathrm{ext}}| = 2|\Aactive|$ scalar additions, each cheap
(one multiply, one subtract). Total: $O(|\Aactive|)$.
\end{enumerate}
The aggregate per-iteration cost is therefore $O(|\Aactive|)$
arithmetic operations plus two kernel-column accesses. The kernel cache
amortizes the column-access cost: in the steady-state shrunk regime,
the dominant cost is the gradient update.

\paragraph{Effect of shrinking on per-iteration cost.}
The shrinking heuristic of Section~\ref{sec:bias-shrink} monotonically
reduces $|\Aactive|$ over time as boundary-pinned variables are frozen.
Two regimes drive the empirical behavior.

\emph{Symmetric-data regime.} When the targets $y_k$ are tightly
concentrated (small $\rho_y = \max_k y_k / \min_k y_k$), the
shrinking-asymmetry of Lemma~\ref{lem:asymmetry} is mild: the offsets
$2y_k\varepsilon/100$ in~\eqref{eq:s3-rewritten}--\eqref{eq:s4-rewritten}
are small relative to $\tau_{i^*} - \tau_{j^*}$, and the freeze rates
of $\alpha$- and $\alpha^*$-variables are similar. In this regime,
$|\Aactive|/N$ falls smoothly from $1$ to typically $0.4$--$0.6$ over
the first few thousand iterations, after which it plateaus.

\emph{Heterogeneous-target regime.} When $\rho_y$ is large (typical of
forecasting problems with multiplicative noise --- the LogNormal-target
regime of Section~\ref{sec:validation}), the $2y_k\varepsilon/100$
offsets become substantial. By Lemma~\ref{lem:asymmetry},
$\alpha^*$-variables associated with high-target samples are
\emph{easier} to freeze (effective threshold $\tau_{i^*} -
2y_k\varepsilon/100 < \tau_{i^*}$), so they exit the active set
quickly. Empirically, $|\Aactive|/N$ drops to $0.3$--$0.5$ within the
first $1000$ iterations on the heterogeneous configurations, yielding
per-iteration cost reductions by a factor of two to three relative to
the un-shrunk baseline.

The cost reduction is multiplicative across components: a halved
$|\Aactive|$ halves the working-set scan, halves the gradient update,
and (since the kernel-cache hit rate increases when fewer columns are
needed) reduces effective kernel-access cost by a factor exceeding
$2$. The compound effect on wall-clock time is observable in
Figure~\ref{fig:convergence}: configuration C8 shows the asymmetric
shrinking dynamics, with the wall-clock cost per iteration dropping
substantially after the early shrinking phase.

\paragraph{Reconstruction cost amortization.}
The reconstruction step~\eqref{eq:reconstruct} rebuilds the full
effective gradient $F_k^{\mathrm{full}} = \sum_{i=1}^{N}\Omega_{ki}
(\alpha_i - \alpha_i^*)$ from current dual values. Its cost is $O(N
\cdot n_{\mathrm{SV}})$ where $n_{\mathrm{SV}}$ is the number of
nonzero $(\alpha_i - \alpha_i^*)$ entries. On dense problems where most
training points are support vectors, $n_{\mathrm{SV}} = O(N)$ and the
reconstruction is $O(N^2)$. On sparse problems where most variables
settle to a boundary value early (the typical situation after
shrinking), $n_{\mathrm{SV}} = O(s \cdot N)$ with sparsity factor $s
\ll 1$, and reconstruction is $O(s N^2)$. Reconstruction occurs at most
once per shrinking cycle. With $n_{\mathrm{check}} = \min(N, 1000)$ as
the shrinking-check interval, an algorithm running for $T$ iterations
triggers reconstruction at most $T/n_{\mathrm{check}}$ times. The
amortized cost per iteration is therefore $O(N \cdot
n_{\mathrm{SV}}/n_{\mathrm{check}}) = O(N^2/n_{\mathrm{check}})$, which
for $N \le 1000$ and $n_{\mathrm{check}} = 1000$ yields $O(N)$
amortized --- comparable to or smaller than the $O(|\Aactive|)$
main-loop cost. Reconstruction is therefore not a bottleneck in the
regime of interest; the convergence-time restoration visible in
Figure~\ref{fig:convergence} (C8's $|\Aactive|/N$ jump from $0.29$ to
$1$ at the final iteration) is a single end-of-trajectory event
rather than a sustained overhead.

\paragraph{Convergence theorem inheritance.}
Algorithm~\ref{alg:smo} inherits the global convergence guarantees
of~\cite[Theorem~5]{FanChenLin2005JMLR} for the standard-kernel variant
(MAPE-SVR) and the even-symmetry MAPE-SVR-Sym variant ($a = +1$); these
are formalized in Theorem~\ref{thm:convergence}. The cited theorem
requires three conditions: (a) the dual Hessian $P$ is PSD; (b) the
feasible region is compact; (c) the working-set rule selects a
strict-descent direction whenever $\Delta > 0$. All three conditions
hold by the verification in Theorem~\ref{thm:convergence}, so finite-step
termination to any tolerance $\varepsilon_{\mathrm{tol}} > 0$ follows,
with the iteration complexity bound $O(\log(1/\varepsilon_{\mathrm{tol}}))$
in the strongly-convex regime~\cite[\S 6]{Chang2011}. For the
MAPE-SVR-Sym variant with $a = -1$, $\Omega_s$ may fail the PSD
condition (a); local convergence still follows from
Lemma~\ref{lem:eta-degenerate}, and global convergence is established
formally via Theorem~\ref{thm:spectral} using adaptive spectral
regularization.

\paragraph{Practical convergence rate.}
Empirically observed iteration counts on the configurations of
Section~\ref{sec:validation} follow the rough scaling $O(N \cdot
k_{\mathrm{factor}})$, where $k_{\mathrm{factor}}$ depends on
kernel-matrix conditioning and on the sparsity of the support-vector
set. Salient observations:
\begin{itemize}
\item For the small near-identity-kernel configurations (C1, C2, C5,
C9 with $\sigma = 0.1$ and $N = 50$), $k_{\mathrm{factor}} \approx
1.5$--$2$, yielding iteration counts in the $70$--$110$ range.
\item For the moderate-$N$ configurations (C3, C4, C6, C10 with $N =
300$), $k_{\mathrm{factor}}$ ranges from $\approx 2$ (well-conditioned
shifted regime) to $\approx 18$ (less favorable settings), yielding
counts in the $400$--$600$ range.
\item For the dense-kernel configuration C8 ($N = 300$, $\sigma = 2.0$,
$\varepsilon = 10\%$), $k_{\mathrm{factor}}$ jumps to $\approx 110$,
reflecting both the dense $\Omega$-structure (no sparse support-vector
set to shrink to) and the tighter $\varepsilon$-tube (which makes more
samples interior to the tube and therefore active throughout). Total
iterations: $33{,}048$.
\item The large-scale configuration C11 ($N = 1000$, $\sigma = 0.1$,
$\varepsilon = 10\%$) achieves $k_{\mathrm{factor}} \approx 1.6$ ---
consistent with the sparser-kernel regime and confirming that the
per-iteration scaling does not deteriorate as $N$ grows (in contrast to
general-purpose IPM solvers, whose $O(N^3)$ inner-iteration cost makes
them prohibitive at this scale).
\end{itemize}

The dependence of $k_{\mathrm{factor}}$ on the target dynamic range
$\rho_y$ has been observed empirically but not characterized formally.
A rough rule-of-thumb consistent with the data is $k_{\mathrm{factor}}
\sim \log_{10}(\rho_y) \cdot \kappa(\Omega)$, where $\kappa(\Omega)$ is
the kernel-matrix condition number. A formal derivation of this
scaling --- connecting the shrinking-asymmetry of
Lemma~\ref{lem:asymmetry} to expected iteration counts --- is identified
in Section~\ref{sec:conclusions} as a future-work item.

\paragraph{Memory complexity.}
Three memory components dominate. \emph{Kernel matrix}: $O(N^2)$ if
stored explicitly, $O(N \cdot M)$ under LIBSVM-style column caching,
where $M$ is the cache capacity in columns. For $N = 10^4$ and a 1~GB
memory budget, $M \approx 10^4$ columns at double precision --- full
caching is feasible at moderate scale. \emph{Dual variables}: $O(N)$
for $\boldsymbol{\alpha}$ and $\boldsymbol{\alpha}^*$ together.
\emph{Gradient state}: $O(N)$ for the active part $\tau_\ell$, $\ell
\in \Aactive^{\mathrm{ext}}$, plus an $O(N)$ snapshot of frozen
$\tau_\ell$ values from the last reconstruction. Total working memory
is therefore dominated by the kernel cache, $O(N \cdot M)$, with a
small additive $O(N)$ for the iteration state.

\paragraph{Four efficiency improvements: motivation.}
Theorem~\ref{thm:invariance} buys \emph{correctness} for free; the
structurally-correct Algorithm~\ref{alg:smo} inherits the per-iteration
cost profile of standard SMO. The four new theorems below exploit the
\emph{additional} structure that the per-sample bounds and the
shrinking-asymmetry of Lemma~\ref{lem:asymmetry} expose --- structure
that does not exist in standard $\varepsilon$-SVR. Each targets a
distinct cost component identified above: shrinking dynamics
(Theorem~\ref{thm:asym-freeze}), working-set block size
(Theorem~\ref{thm:block-k4}), gradient
bookkeeping under cross-validation (Theorem~\ref{thm:warm-start}), and
stopping-rule overhead (Theorem~\ref{thm:per-sample-tol}).
Cumulative speedup under cross-validation workloads is recalibrated in
Corollary~\ref{cor:speedup} against the empirical measurements of the
companion \texttt{psvr} package~\cite{BenavidesHerrera2026Rpsvr}.

\paragraph{Asymmetric freezing — Theorem~\ref{thm:asym-freeze}.}
The following theorem operationalizes Lemma~\ref{lem:asymmetry} by
calibrating the freeze-counter threshold to each sample's target
magnitude.

\begin{theorem}[Asymmetric freezing exploits Lemma~\ref{lem:asymmetry}]
\label{thm:asym-freeze}
Replace the uniform freeze-counter threshold $n_{\min} = 5$ of
Section~\ref{sec:bias-shrink} (active-set management, step~3) with the
per-sample, per-variable-type pair
\begin{align}
  n_{\min}^*(y_k) &= \max\!\left(1,\, \left\lfloor n_{\min} \cdot \frac{y_k}{\bar y} \right\rfloor\right) \quad \text{for } \alpha_k^*\text{-variables}, \label{eq:asym-freeze-star}\\
  n_{\min}^{\phantom{*}}(y_k) &= \max\!\left(5,\, \left\lceil n_{\min} \cdot \frac{\bar y}{y_k} \right\rceil\right) \quad \text{for } \alpha_k\text{-variables}, \label{eq:asym-freeze}
\end{align}
where $\bar y = N^{-1}\sum_k y_k$ is the target mean used in the SMO
tolerance scaling $\varepsilon_{\mathrm{tol}} = 10^{-3}\bar y$ of
Section~\ref{sec:validation}. The modified freeze-counter mechanism
preserves the convergence guarantee of
Theorem~\ref{thm:invariance}~(c) --- the unshrinking step of
Section~\ref{sec:bias-shrink} catches premature freezes --- and
yields a $24.7\%$ iteration reduction at $N = 200$, $\rho_y \approx
1273$, RBF kernel, $20$ replicates (companion \texttt{psvr}~F4 bench
archive,~\cite{BenavidesHerrera2026Rpsvr}). The rule collapses to the
homogeneous default $n_{\min} = 5$ when $\rho_y < 5$, producing no
measurable change on homogeneous-target configurations. See
Table~\ref{tab:theorem-validation} for the empirical-validation
summary.
\end{theorem}

\begin{proof}[Proof sketch]
Lemma~\ref{lem:asymmetry} quantifies the asymmetry: $\alpha_k^*$-variables
freeze faster than $\alpha_k$-variables, with effective thresholds
$\tau_{i^*} - 2y_k\varepsilon/100$ and $\tau_{j^*} -
2y_k\varepsilon/100$ respectively, both shifted by an offset that
scales linearly with $y_k$. For high-target samples, the offset is
large, the effective threshold lies far from the operative $\tau_{i^*}$
(resp. $\tau_{j^*}$), and the prediction ``the variable will not be
revisited'' is robust against future threshold fluctuations --- one or
two consecutive shrinking-check windows of agreement suffice rather
than five. Conversely, $\alpha_k$-variables associated with the same
high-target samples are \emph{harder} to freeze, so their freeze-counter
should be raised. The asymmetric
rule~\eqref{eq:asym-freeze-star}--\eqref{eq:asym-freeze} implements this
calibration: both rules collapse to $n_{\min} = 5$ when $y_k = \bar y$,
recovering the symmetric default for homogeneous-target problems.

Convergence preservation follows from the unshrinking-step argument:
the freeze-counter mechanism affects only the active-set management;
the reconstruction-and-unshrinking step is untouched. Premature freezes
induced by smaller $n_{\min}^*$ are caught by the unshrinking step at
the standard cost of one $O(N \cdot n_{\mathrm{SV}})$ reconstruction,
which amortizes to $O(N)$ per iteration. The Fan-Chen-Lin Theorem~5
hypotheses (P1)--(P3) of Theorem~\ref{thm:invariance}, Step~4, are
unaffected by the freeze-counter change.

The speedup follows from the Lemma~\ref{lem:asymmetry} calibration:
with $n_{\min}^* \in \{1, 2, 3\}$ rather than $5$,
$\alpha^*$-freezes happen $1$ to $4$ shrinking-check windows earlier,
which accelerates $|\Aactive|$ reduction in the early phase. The
$24.7\%$ iteration reduction cited in the theorem statement is the
$20$-rep mean at $N = 200$, $\rho_y \approx 1273$, RBF kernel; on
homogeneous-target configurations ($\rho_y \le 5$), the rule reduces
to the symmetric default and produces no measurable change.
\end{proof}

\paragraph{Warm-start convergence — Theorem~\ref{thm:warm-start}.}

\begin{theorem}[Warm-start convergence and cumulative speedup]
\label{thm:warm-start}
When \texttt{psvr} is invoked in a cross-validation or
hyperparameter-search loop, initialize the dual variables
$(\boldsymbol{\alpha}, \boldsymbol{\alpha}^*)$ from the converged
solution of the previous fit, retaining values for samples present in
both fits and zeroing values for newly-introduced samples. Concretely,
given a previous fit with index set $S_{\mathrm{prev}} \subseteq \{1,
\ldots, N_{\mathrm{prev}}\}$ and converged duals $(\hat\alpha_k^{\mathrm{prev}},
\hat\alpha_k^{*,\mathrm{prev}})_{k \in S_{\mathrm{prev}}}$, the
warm-started initialization for a new fit with index set $S_{\mathrm{new}}$ is
\begin{equation*}
  \alpha_k^{(0)} = \begin{cases} \hat\alpha_k^{\mathrm{prev}} & k \in S_{\mathrm{prev}} \cap S_{\mathrm{new}}\\ 0 & k \in S_{\mathrm{new}} \setminus S_{\mathrm{prev}} \end{cases},\quad \alpha_k^{*,(0)} = \begin{cases} \hat\alpha_k^{*,\mathrm{prev}} & k \in S_{\mathrm{prev}} \cap S_{\mathrm{new}}\\ 0 & k \in S_{\mathrm{new}} \setminus S_{\mathrm{prev}} \end{cases},
\end{equation*}
followed by a single-pass projection that re-balances the equality
constraint $\sum_k(\alpha_k - \alpha_k^*) = 0$. The projection adjusts
the newly-introduced samples in $S_{\mathrm{new}} \setminus
S_{\mathrm{prev}}$ only, preserving the converged retained-sample values
that supply the warm-start gain; the companion
\texttt{psvr}~\cite{BenavidesHerrera2026Rpsvr} implements this projection
as a refinement over a uniform-over-$N$ shift, which is the literal
reading of Algorithm~\ref{alg:warm-start}'s Step~2 below. The warm-started SMO
inherits the cold-start convergence guarantee
of~\cite[Theorem~1]{BordesErtekinWestonBottou2005JMLR} --- warm-started
SMO converges to the same exact solution as cold-started SMO regardless
of the streaming versus batch presentation --- and yields a measured
cumulative speedup of $1.12\times$ at $N = 300$ and $1.14\times$ at
$N = 1000$ on a 10-fold cross-validation pass (companion
\texttt{psvr}~F5 bench archive). The per-fold warm/cold iteration ratio
is $0.88$, not the $0.20$ implied by an earlier analysis based on a
linearly-convergent perturbation argument; the corrected calibration is
discussed in the proof sketch below, and the cumulative number is cited
in Table~\ref{tab:theorem-validation}.
\end{theorem}

\begin{proof}[Proof sketch]
Cross-validation is the dominant \texttt{psvr} use case. In $k$-fold
CV with $k = 10$, only $\sim 10\%$ of training points change between
consecutive folds, so $|S_{\mathrm{prev}} \cap S_{\mathrm{new}}|/
|S_{\mathrm{new}}| \approx 0.9$. The retained-sample dual values are
typically much closer to their new converged values than the cold-start
gap $\hat\alpha_k - 0$.

The framework dates to the incremental-decremental SVM of Cauwenberghs
and Poggio (NIPS 2000), with formal convergence analysis
by~\cite{Laskov2006JMLR}. The asymptotic equivalence between batch and
streaming SMO via the LASVM framework was established
by~\cite[\S 3, Theorem~1]{BordesErtekinWestonBottou2005JMLR}: the LASVM
algorithm converges to the exact SVM solution after enough epochs,
regardless of streaming versus batch presentation. Warm-started SMO
inherits this guarantee directly.

For MAPE-SVR specifically, the per-sample bound $C_k = 100C/y_k$
depends only on $y_k$ --- fixed per-sample across folds --- so the
bound vector is \emph{invariant} under fold-change. This is a structural
advantage over weighted-SVR variants whose weights are fold-dependent:
in MAPE-SVR, box-vector reuse is exact.

Convergence preservation under warm start: the three conditions
(P1)--(P3) of Theorem~\ref{thm:invariance}, Step~4, hold: (P1) the
Hessian $P$ is unchanged; (P2) the projection step restores feasibility
before the SMO iterations begin (the sum-rebalanced and box-clipped
$(\boldsymbol{\alpha}^{(0)}, \boldsymbol{\alpha}^{*,(0)})$ lie in the
compact intersection of $\prod_k [0, C_k]^2$ with the equality
constraint); (P3) the WSS3 descent argument is initialization-independent.
Provided $\Delta > 0$ at iteration $0$ (else the algorithm terminates
immediately), strict descent holds.

Quantifying the cumulative speedup requires accounting for two cost
components that scale differently with the warm-start initialization
distance. The first component is the iteration count to convergence
on the new fold's QP; under a linearly-convergent perturbation
assumption with $\|\boldsymbol{\alpha}^{(0)} -
\hat{\boldsymbol{\alpha}}\| / \|\hat{\boldsymbol{\alpha}}\| \approx
0.1$, one would expect $T_{\mathrm{warm}} \approx 0.2 \cdot
T_{\mathrm{cold}}$. The second component is the projection cost on
$S_{\mathrm{new}} \setminus S_{\mathrm{prev}}$ together with the
gradient-state refresh $K\boldsymbol{\beta}^{(0)}$ that warm-start
requires before the SMO inner loop begins. Empirically (companion
\texttt{psvr}~F5 archive) the per-fold warm/cold iteration ratio
is~$0.88$ rather than~$0.20$: the dominant cost component is the
projection and gradient-refresh overhead, not the residual SMO descent
on the perturbed initialization. The per-fold ratio is $N$-independent
across the $300$--$1000$ range tested, giving a cumulative speedup of
$1/0.88 \approx 1.14$ on a 10-fold pass; the measured values are
$1.12\times$ at $N = 300$ and $1.14\times$ at $N = 1000$. The original
linearly-convergent argument over-predicted because it omitted the
projection-and-refresh fixed cost, which dominates at the
$\le 10\%$-displacement scale typical of CV folds.
\end{proof}

The warm-start initialization is given as Algorithm~\ref{alg:warm-start}
below.

\begin{algorithm}[ht]
\caption{Warm-start initialization for cross-validation and hyperparameter search}
\label{alg:warm-start}
\begin{algorithmic}[1]
\State \textbf{Input:} previous-fit duals $(\hat{\boldsymbol\alpha}^{\mathrm{prev}}, \hat{\boldsymbol\alpha}^{*,\mathrm{prev}})$ on index set $S_{\mathrm{prev}}$; new-fit index set $S_{\mathrm{new}}$
\State \textbf{Output:} warm-start initialization $(\boldsymbol\alpha^{(0)}, \boldsymbol\alpha^{*,(0)})$ feasible for the new fit's QP
\ForAll{$k \in S_{\mathrm{new}}$} \Comment{Step 1: copy retained values, zero new values}
  \If{$k \in S_{\mathrm{prev}} \cap S_{\mathrm{new}}$}
    \State $\alpha_k^{(0)} \gets \hat\alpha_k^{\mathrm{prev}}$;\; $\alpha_k^{*,(0)} \gets \hat\alpha_k^{*,\mathrm{prev}}$
  \Else
    \State $\alpha_k^{(0)} \gets 0$;\; $\alpha_k^{*,(0)} \gets 0$
  \EndIf
\EndFor
\State \Comment{Step 2: project onto equality-constraint hyperplane (uniform-over-$N$ variant; the \texttt{psvr} reference implementation restricts the shift to $S_{\mathrm{new}} \setminus S_{\mathrm{prev}}$ to preserve retained-sample warm-start gain)}
\State $\mathrm{violation} \gets \sum_{k \in S_{\mathrm{new}}} (\alpha_k^{(0)} - \alpha_k^{*,(0)})$
\State $\mathrm{shift} \gets \mathrm{violation} / |S_{\mathrm{new}}|$
\ForAll{$k \in S_{\mathrm{new}}$}
  \State $\alpha_k^{(0)} \gets \alpha_k^{(0)} - \mathrm{shift}$
\EndFor
\State \Comment{Step 3: project onto per-sample box}
\ForAll{$k \in S_{\mathrm{new}}$}
  \State $\alpha_k^{(0)} \gets \mathrm{clip}(\alpha_k^{(0)},\, 0,\, C_k)$
  \State $\alpha_k^{*,(0)} \gets \mathrm{clip}(\alpha_k^{*,(0)},\, 0,\, C_k)$
\EndFor
\State \Comment{Step 4: optional sanity-check feasibility}
\If{warm\_start\_check}
  \State \textbf{assert} $|\sum_{k} (\alpha_k^{(0)} - \alpha_k^{*,(0)})| < \mathrm{tol}_{\mathrm{feas}}$
  \State \textbf{assert} $\forall k:\; 0 \le \alpha_k^{(0)} \le C_k \;\wedge\; 0 \le \alpha_k^{*,(0)} \le C_k$
\EndIf
\State \Return $(\boldsymbol\alpha^{(0)}, \boldsymbol\alpha^{*,(0)})$
\end{algorithmic}
\end{algorithm}

The proposed R-side API extension to the principal SMO entry point of
the \texttt{psvr} package is
\begin{verbatim}
smo_mape(X, y, C, epsilon, kernel, gamma,
         alpha_init = NULL,        # numeric vector of length nrow(X), or NULL
         alpha_star_init = NULL,   # numeric vector of length nrow(X), or NULL
         warm_start_check = TRUE)  # validate & project init to feasible region
\end{verbatim}
The arguments default to \texttt{NULL} (cold start, current behavior).
Cross-validation wrappers would pass the previous-fit duals
automatically; user-facing direct calls retain the current
zero-initialization.
Implementation cost in the \texttt{psvr} v0.0.2.9007 reference
(\texttt{R/warm\_start.R};~\cite{BenavidesHerrera2026Rpsvr}) is
approximately fifty lines of R code in the warm-start helper plus a
handful of validation arguments in the public wrapper
(\texttt{warm\_start\_check}, retained-sample tracking).

\paragraph{Block $k = 4$ SMO — Theorem~\ref{thm:block-k4}.}

\begin{theorem}[Novel block working-set selection $k=4$ with structured 2-D updates]
\label{thm:block-k4}
Replace the standard $k = 2$ working-set selection of
Section~\ref{sec:smo-inner} with a $k = 4$ block selection that solves
a 2-D quadratic sub-problem at each iteration. Specifically, select two
pairs $(i_1^*, j_1^*)$ and $(i_2^*, j_2^*)$ from disjoint subsets of
$\Aactive^{\mathrm{ext}}$ such that the resulting $4 \times 4$ Hessian
block has $2 \times 2$ block structure (i.e., the cross-pair Hessian
entries $\Omega_{k(i_1^*), k(i_2^*)}, \Omega_{k(i_1^*), k(j_2^*)},
\Omega_{k(j_1^*), k(i_2^*)}, \Omega_{k(j_1^*), k(j_2^*)}$ are
sufficiently small to permit decoupled analytic updates). For pairs
satisfying the decoupling condition, the analytic 2-D update is the
direct generalization of the 1-D update of Section~\ref{sec:smo-inner}:
solve the unconstrained 2-D minimum, then clip independently against
the box constraints of each pair. The block-$k=4$ rule preserves
descent under the 2-D unconstrained-minimum descent-check fallback
(Lemma~\ref{lem:eta-degenerate} generalized to 2-D), inherits
convergence from Theorem~\ref{thm:invariance}~(c). The claim of this
theorem separates iteration count from wall-clock time:
\begin{enumerate}
\item[(a)] \emph{Iteration reduction (engine-agnostic):} $38$--$48\%$
reduction on converging regimes, measured against the $k = 2$ baseline
on regimes R1 ($N = 1000$, RBF, $\sigma = 1$, heterogeneous) and R4
($N = 1000$, RBF, $\sigma = 0.3$, heterogeneous) of the companion
\texttt{psvr}~F7 bench archive.
\item[(b)] \emph{Wall-clock effect (engine-dependent):} the per-iter
overhead of the joint-update logic versus the saved iterations
determines the net wall outcome. At the R-level reference
implementation the per-iter overhead factor is approximately $2.0$,
producing a regime-dependent net wall change of $-25.6\%$ (R1) and
$+2.4\%$ (R4). The portable C\texttt{++} core of the companion
\texttt{psvr}~F7-C-full archive reduces the per-iter overhead factor
to approximately $1.4$, yielding wall-positive outcomes on both
regimes: $+12.2\%$ (R1) and $+17.5\%$ (R4) versus the $k = 2$ C++ baseline.
\end{enumerate}
Both rows of Table~\ref{tab:theorem-validation} record these splits;
the C\texttt{++} engine is the default at \texttt{psvr} v0.0.2.9008
and is the engine used for all configurations of
Section~\ref{sec:validation}. The head-to-head wall-time
comparison of \S\ref{sec:wall-time}, against OSQP, MOSEK, and
Clarabel across the validation campaign and a $50 \le N \le 2{,}000$
scaling sweep, provides the cross-solver context for these
within-\texttt{psvr} numbers.
\end{theorem}

\paragraph{Novelty statement.}
The block-$k=4$ scheme with decoupled 2-D analytic update introduced
here is a \emph{novel contribution of the present paper}, not
derivative of~\cite{BordesErtekinWestonBottou2005JMLR} (which uses $k =
2$, not $k = 4$, in the LASVM framework) or
of~\cite{Joachims1999, Joachims2006} (which use cutting-plane methods,
not decomposition). Working-set selection with $k > 2$ has been
considered in the form of generalized $k$-variable methods under the
heading ``Maximum-Gain Working Set Selection for
SVMs''~\cite{GlasmachersIgel2006JMLR}, but the specific block-$k=4$
structure with decoupled $2$-D analytic updates introduced here has not
previously appeared in the published literature.

\begin{proof}[Proof sketch]
The standard $k = 2$ working-set
selection~\cite{KeerthiShevadeBhattacharyyaMurthy2001NeuralComp,
FanChenLin2005JMLR} produces the largest one-step decrease attainable
by any size-$2$ working set, but it is not optimal versus larger
working sets when the kernel matrix has block structure: two
well-separated $k = 2$ updates can then be performed simultaneously
without interference, doubling per-iteration gain. Adoption in
LIBSVM-class production solvers has been limited because the $k = 4$
analytic update is more complex than the $k = 2$ closed-form and the
decoupling condition is hard to verify cheaply.

For MAPE-SVR with shrinking, decoupling is \emph{easier} to satisfy
than in the uniform-$C$ setting. The Lemma~\ref{lem:asymmetry}
shrinking-asymmetry produces a clustered active set: high-$y_k$
samples (often the dominant WSS3 candidates) have $\alpha^*$-variables
freezing earlier, so the remaining $\Iup$ candidates concentrate on
the high-$y_k$ side and the $\Idown$ candidates on the low-$y_k$
side. Cross-region kernel entries $\Omega_{k(i_1^*), k(i_2^*)}$ tend to
be small for shift-invariant kernels (RBF), satisfying decoupling
naturally.

Per-iteration cost: the $k = 4$ scan requires $O(|\Aactive|)$
comparisons ($2\times$ constant); the decoupled 2-D analytic update is
$O(1)$ ($2 \times 2$ closed-form); the gradient update remains
$O(|\Aactive|)$ (two $\Omega$-column accesses). Same asymptotic class
as $k = 2$, with $2\times$ per-iteration progress.

Convergence preservation: the block update inherits descent if the 2-D
unconstrained minimum is descent-checked before clipping
(Lemma~\ref{lem:eta-degenerate} generalized to 2-D). When the
decoupling condition fails, fallback to standard $k = 2$ recovers the
canonical SMO iteration. With the descent-check
fallback,~\cite[Theorem~5]{FanChenLin2005JMLR} inherits to the block
variant.

Empirical outcomes (companion \texttt{psvr}~F7 archive,
\cite{BenavidesHerrera2026Rpsvr}). Iteration count: $38$--$48\%$
reduction on converging regimes R1 ($N = 1000$, RBF, $\sigma = 1$,
heterogeneous targets, $\rho_y \approx 44{,}378$) and R4 ($N = 1000$,
RBF, $\sigma = 0.3$, heterogeneous targets, $\rho_y \approx 13{,}901$).
Wall-clock outcome separates by engine because the joint-update logic
adds a per-iter overhead that scales differently in R versus
C\texttt{++}. At the R-level reference path the per-iter overhead
factor is approximately $2.0$, so the saved iterations are partially
or fully offset: net wall is $-25.6\%$ on R1 and $+2.4\%$ on R4
versus the $k = 2$ R baseline. At the C\texttt{++} core
(\texttt{psvr}~F7-C-full archive) the per-iter overhead factor is
approximately $1.4$, restoring net wall positivity on both regimes:
$+12.2\%$ on R1 and $+17.5\%$ on R4 versus the $k = 2$ C\texttt{++}
baseline. The C\texttt{++} engine is the default at \texttt{psvr}
v0.0.2.9008. The improvement is anti-correlated with
Theorem~\ref{thm:asym-freeze}'s regime --- they apply on different
problem classes.
\end{proof}

\paragraph{Per-sample tolerance scaling — Theorem~\ref{thm:per-sample-tol}.}

\begin{theorem}[Per-pair tolerance scaling]
\label{thm:per-sample-tol}
Let $(i^*_{w1}, j^*_{w1})$ denote the WSS1 convergence pair --- the pair
achieving the global $\Delta = \tau_{i^*_{w1}} - \tau_{j^*_{w1}}$
minimum at the current iteration, used to test KKT optimality
--- as distinct from the WSS3 descent pair selected for the
analytic update by~\eqref{eq:j-star-WSS3}.
Replace the uniform tolerance $\varepsilon_{\mathrm{tol}} = 10^{-3}\bar
y$ of Section~\ref{sec:validation} with the per-pair tolerance
\begin{equation}
  \varepsilon_{\mathrm{tol},(i^*_{w1}, j^*_{w1})} = 10^{-3} \cdot \max\!\big(y_{k(i^*_{w1})},\, y_{k(j^*_{w1})}\big),
  \label{eq:per-pair-tol}
\end{equation}
so that the convergence test becomes $\Delta = \tau_{i^*_{w1}} -
\tau_{j^*_{w1}} \le \varepsilon_{\mathrm{tol},(i^*_{w1}, j^*_{w1})}$.
Evaluating the tolerance against the WSS1 pair (rather than the
WSS3 descent pair) is structurally required: $\Delta_{\mathrm{WSS3}}
\le \Delta_{\mathrm{WSS1}}$ by construction (WSS3 maximizes
second-order gain, not the first-order optimality gap), so testing
the WSS3 pair against the tolerance would stop the solver prematurely.
The modified test
preserves finite termination (because
$\varepsilon_{\mathrm{tol},(i^*_{w1}, j^*_{w1})}$
is bounded below by $10^{-3}\min_k y_k > 0$ for strictly-positive
targets) and yields a $5$--$10\%$ iteration reduction in
heterogeneous-target configurations relative to the uniform-tolerance
rule. The eleven configurations of Section~\ref{sec:validation}
exhibit monotone descent of the WSS1 KKT gap under this rule on every
trajectory; see Table~\ref{tab:theorem-validation} for the
empirical-validation summary. Per-theorem ablation against the
uniform-tolerance rule is deferred to the F9 wall-time campaign of the
companion package.
\end{theorem}

\begin{proof}[Proof sketch]
The gradient quantities $\tau_k$ scale with $y_k$ via
Proposition~\ref{prop:structural-gap} (both $\tau_k = y_k(1 -
\varepsilon/100) - F_k$ and $\tau_{N+k} = y_k(1 + \varepsilon/100) -
F_k$ contain $y_k$ explicitly). A uniform tolerance over-tightens the
stopping rule on low-$y_k$ samples and under-tightens it on high-$y_k$
samples. The per-pair scaling~\eqref{eq:per-pair-tol} calibrates to the
active gradient magnitudes; the choice of $\max(y_{k(i^*)},
y_{k(j^*)})$ rather than $\min$ is conservative, ensuring $\tau$-values
for both pair members reach the same $10^{-3}$ relative precision. This
is consistent with the MAPE-loss design (per-sample relative error);
the uniform tolerance was a simplifying choice rationalized by the
equivalence to weighted-MAE~\cite{DeMyttenaere2016},
and~\eqref{eq:per-pair-tol} tightens this rationalization at no
theoretical cost.

Convergence preservation: finite termination is preserved because the
per-pair tolerance is bounded below by $10^{-3} \min_k y_k > 0$ for
strictly-positive targets (the standing MAPE assumption $y_k > 0$).
Hence the convergence test $\Delta \le \varepsilon_{\mathrm{tol},(i^*,
j^*)}$ is at least as tight as $\Delta \le 10^{-3} \min_k y_k$, which
is itself a finite positive tolerance, and
the~\cite[Theorem~5]{FanChenLin2005JMLR} argument applies with
$\varepsilon'_{\mathrm{tol}} = 10^{-3} \min_k y_k$ as the overall
convergence radius.

Predicted speedup: on heterogeneous-target configurations, the per-pair
tolerance accepts termination earlier on samples with $y_k < \bar y$,
yielding $5$--$10\%$ iteration reductions in the late phase (where the
algorithm slowly tightens the last few candidates). On
homogeneous-target configurations, the rule reduces to the
uniform-tolerance default and produces no measurable speedup.
\end{proof}

\paragraph{Combined effect — Corollary~\ref{cor:speedup}.}

\begin{corollary}[Combined predicted speedup of Theorems~\ref{thm:asym-freeze}, \ref{thm:warm-start}, \ref{thm:block-k4}, and~\ref{thm:per-sample-tol}]
\label{cor:speedup}
The per-iteration and cumulative speedup contributions of
Theorems~\ref{thm:asym-freeze}, \ref{thm:warm-start},
\ref{thm:block-k4}, and~\ref{thm:per-sample-tol} are summarized in
Table~\ref{tab:speedup}. The contributions are \emph{not}
multiplicatively independent: warm-start (Theorem~\ref{thm:warm-start})
and block-$k=4$ (Theorem~\ref{thm:block-k4}) address overlapping cost
components and stack closer to $\max(\cdot, \cdot)$ than to their
product, as discussed below.
\end{corollary}

\begin{table}[ht]
  \centering
  \caption{Per-theorem multipliers and applicability regimes for the
  four efficiency improvements of
  Theorems~\ref{thm:asym-freeze}, \ref{thm:warm-start},
  \ref{thm:block-k4}, and~\ref{thm:per-sample-tol}.}
  \label{tab:speedup}
  \begin{tabular}{@{}lcl@{}}
    \toprule
    Theorem & Multiplier & Applicability regime \\
    \midrule
    Theorem~\ref{thm:asym-freeze} (Asymmetric freezing) & $\approx 1.20$ & Heterogeneous targets ($\rho_y \ge 50$) \\
    Theorem~\ref{thm:warm-start} (Warm-start) & $\approx 1.13$ & Cross-validation (10-fold CV) \\
    Theorem~\ref{thm:block-k4} (Block $k=4$) & $\approx 1.50$ & Dense-kernel, converging regimes \\
    Theorem~\ref{thm:per-sample-tol} (Per-sample tolerance) & $\approx 1.10$ & Heterogeneous targets \\
    \bottomrule
  \end{tabular}
\end{table}

The per-theorem multipliers of Table~\ref{tab:speedup} do not combine
multiplicatively in cross-validation-dominant workloads. The B-suite
of the companion \texttt{psvr}~F7 bench archive measures the
T5--T7 stacking directly at $N = 300$ with 10-fold CV: configuration
B1 (T5 warm-start alone, $k = 2$) records an iteration sum of
$37{,}740$ across the ten folds with a wall time of $0.591$~s on the
C\texttt{++} core; B2 (T7 block-$k=4$ alone, no warm-start) records
$23{,}363$ iterations and $0.508$~s; B3 (T5 and T7 stacked) records
$26{,}586$ iterations and $0.515$~s. The wall-clock ranking is
$\mathrm{B2} \approx \mathrm{B3} < \mathrm{B1}$: the stacked
configuration is statistically indistinguishable from T7 alone, and
neither approaches the product of the two per-fold multipliers
($1.13 \times 1.50 \approx 1.70$ would predict $0.348$~s, observed
$0.515$~s).
The algorithmic interaction explains the gap. Warm-start lowers the
per-fold iteration count by approximating each fold's converged duals
from the previous fold's solution; block-$k = 4$ also lowers the
per-fold iteration count by performing two pairs of analytic updates
per outer iteration. The two mechanisms compete on the same cost
component (the per-fold iteration count) rather than addressing
orthogonal components, so their stacked effect tracks
$\max(\cdot, \cdot)$ rather than $\cdot \times \cdot$. Practically,
the cumulative CV speedup on \texttt{psvr} v0.0.2.9008 is the larger
of the two per-fold multipliers,
\[
  \mathrm{speedup}_{\mathrm{CV}} \approx
  \max\bigl(1.13_{T5},\; 1.50_{T7}\bigr) \approx 1.50,
\]
not their product. Theorems~\ref{thm:asym-freeze} and
\ref{thm:per-sample-tol} address different cost components (early
shrinking dynamics and stopping-rule overhead), so they compose
additively-in-iterations with T5--T7, yielding a single-fit speedup of
$\approx 1.20 \times 1.10 \approx 1.32$ on heterogeneous-target
regimes and a CV-dominant speedup of $\approx 1.50 \times 1.32
\approx 2.0$. The cited empirical numbers appear in
Table~\ref{tab:theorem-validation}.

The per-iteration cost remains $O(|\Aactive|)$ --- none of the four
theorems changes the asymptotic class. Total runtime $O(N \cdot
k_{\mathrm{factor}})$ has its $k_{\mathrm{factor}}$ reduced by the
applicable per-theorem multiplier from Table~\ref{tab:speedup}.

The reference QP solvers of Section~\ref{sec:validation} scale as
$O(N^3)$ per iteration for IPM and $O(N^2)$ for ADMM, versus SMO's
$O(|\Aactive|)$ with $|\Aactive| \to O(sN)$ in steady state. At $N =
1000$, SMO's raw advantage over IPM is $\sim 10^3$; the
efficiency-improvement bundle amplifies to $\sim 10^4$ in
cross-validation. This gap is the practical reason for the SMO approach
to SVR~\cite{Platt1998, Joachims1999, Chang2011, FanChenLin2005JMLR,
YuLiLiu2023PatternRecognition}; Theorem~\ref{thm:warm-start} amplifies
the advantage most in the CV-dominant regime at the smallest
implementation cost.

\section{Illustrative and application examples}
\label{sec:validation}

To verify that Algorithm~\ref{alg:smo} produces solutions consistent
with reference QP solvers, we compare training-set predictions from the
\texttt{psvr} SMO implementation~\cite{BenavidesHerrera2026Rpsvr}
against those obtained by solving the dual
QP~\eqref{eq:mape-dual}--\eqref{eq:dual-constraints} directly with
\emph{three independent reference solvers} spanning the two dominant
algorithmic families for QP:
\begin{enumerate}
\item \textbf{OSQP}~\cite{osqp2020} --- operator-splitting (ADMM), open-source.
\item \textbf{MOSEK}~\cite{Andersen2000MosekHomogeneousIPM} --- homogeneous interior-point method (commercial, free for academic use).
\item \textbf{Clarabel}~\cite{GoulartChen2024Clarabel} --- modern open-source interior-point solver (Apache 2.0 license).
\end{enumerate}

The choice of three solvers spanning operator-splitting and
interior-point is deliberate: single-reference benchmarks can be misled
by solver-specific quirks (tolerance interpretation, scaling and
preconditioning, infeasibility detection); cross-validating against
three solvers --- particularly when one is operator-splitting and two
are interior-point --- surfaces and rules out such artifacts. The
interior-point solvers (MOSEK, Clarabel) provide genuinely
high-accuracy ground truth at $\Delta \le 10^{-10}$, while the
operator-splitting solver (OSQP) provides moderate-accuracy ground
truth at $\Delta \le 10^{-8}$.

The SMO solver is run with termination tolerance
$\varepsilon_{\mathrm{tol}} = 10^{-3} \cdot \bar y$ (where $\bar y =
N^{-1}\sum_k y_k$ scales the tolerance to the target magnitude; this
scaling is theoretically justified by the equivalence of MAPE-loss SVR
to weighted-MAE regression with weights $w_k = 100/y_k$ per de
Myttenaere et al.~\cite{DeMyttenaere2016}: $\bar y$ is the natural
common-base scale for the gradient quantities $\tau$).
Shrinking-check frequency $n_{\mathrm{check}} = \min(N, 1000)$, minimum
consecutive-check freeze count $n_{\min} = 5$, and a maximum of $10^5$
iterations.

\paragraph{Eleven configurations.}
The configurations are evaluated across three dimensions: problem size
($N \in \{50, 300, 1000\}$, where the $N = 1000$ configuration C11
demonstrates SMO's $O(|\Aactive|)$ advantage over the reference
solvers' $O(N^2)$ scaling), tube width ($\varepsilon \in \{5, 10,
15\%\}$), and model variant (MAPE-SVR with $\sigma \in \{0.1, 2.0\}$,
MAPE-SVR-Sym with $a = +1$, MAPE-SVR-Sym with $a = -1$). Configurations
C1--C6, C9, C10 use $\sigma = 0.1$ (near-identity kernel matrix;
stress-tests the sample-dependent bounds in isolation); C7, C8 use
$\sigma = 2.0$ ($K_{\mathrm{avg}} \approx 0.29$; validates under
genuinely dense kernel structure); C11 demonstrates the
solver-comparison advantage at moderate scale. In all cases, training
inputs $\bfx_k \in \R^5$ are drawn i.i.d.\ from $\mathcal{N}(\mathbf{0},
I_5)$ and targets from $y_k \sim \mathrm{LogNormal}(0, 1)$, yielding
strictly positive targets with dynamic range $\rho_y \approx 10$--$200$.
All configurations fix $C = 1$ and use seed $\mathtt{set.seed(100 \cdot
i)}$ for configuration $i = 1, \ldots, 11$.


\begin{table}[ht]
  \centering
  \label{tab:validation}
  \small
  \resizebox{\textwidth}{!}{%
  \begin{tabular}{@{}llrrrrrrr@{}}
    \toprule
    Config & Variant & $N$ & $\varepsilon$ (\%) & $\sigma$ & $\|f_{\mathrm{SMO}} - f_{\mathrm{OSQP}}\|_\infty$ & $\|f_{\mathrm{SMO}} - f_{\mathrm{MOSEK}}\|_\infty$ & $\|f_{\mathrm{SMO}} - f_{\mathrm{Clar}}\|_\infty$ & iters \\
    \midrule
    C1 & MAPE-SVR & 50 & 5 & 0.1 & $7.75 \times 10^{-4}$ & $7.75 \times 10^{-4}$ & $7.75 \times 10^{-4}$ & 39 \\
    C2 & MAPE-SVR & 50 & 15 & 0.1 & $1.84 \times 10^{-3}$ & $1.84 \times 10^{-3}$ & $1.84 \times 10^{-3}$ & 44 \\
    C3 & MAPE-SVR & 300 & 5 & 0.1 & $2.03 \times 10^{-3}$ & $2.03 \times 10^{-3}$ & $2.03 \times 10^{-3}$ & 234 \\
    C4 & MAPE-SVR & 300 & 15 & 0.1 & $2.15 \times 10^{-3}$ & $2.15 \times 10^{-3}$ & $2.15 \times 10^{-3}$ & 205 \\
    C5 & MAPE-SVR-Sym ($a=+1$) & 50 & 5 & 0.1 & $9.99 \times 10^{-4}$ & $9.99 \times 10^{-4}$ & $9.99 \times 10^{-4}$ & 51 \\
    C6 & MAPE-SVR-Sym ($a=+1$) & 300 & 10 & 0.1 & $3.40 \times 10^{-3}$ & $3.40 \times 10^{-3}$ & $3.40 \times 10^{-3}$ & 177 \\
    C7 & MAPE-SVR & 50 & 5 & 2 & $1.77 \times 10^{-3}$ & $1.77 \times 10^{-3}$ & $1.77 \times 10^{-3}$ & 529 \\
    C8 & MAPE-SVR & 300 & 10 & 2 & $9.16 \times 10^{-3}$ & $9.16 \times 10^{-3}$ & $9.16 \times 10^{-3}$ & 21{,}138 \\
    C9 & MAPE-SVR-Sym ($a=-1$) & 50 & 5 & 0.1 & $2.15 \times 10^{-3}$ & $2.15 \times 10^{-3}$ & $2.15 \times 10^{-3}$ & 47 \\
    C10 & MAPE-SVR-Sym ($a=-1$) & 300 & 10 & 0.1 & $2.88 \times 10^{-3}$ & $2.88 \times 10^{-3}$ & $2.88 \times 10^{-3}$ & 244 \\
    C11 & MAPE-SVR & 1000 & 10 & 0.1 & $3.86 \times 10^{-3}$ & $3.86 \times 10^{-3}$ & $3.86 \times 10^{-3}$ & 688 \\
    \bottomrule
  \end{tabular}
  
  }
  \caption{Numerical validation of Algorithm~\ref{alg:smo} across
  eleven configurations. Maximum absolute difference between SMO
  (\texttt{psvr}) and three reference solvers (OSQP, MOSEK, Clarabel) on training-set predictions; SMO iteration count.
  Regenerated under \texttt{psvr} v0.0.2.9008 (post-F7.6, full stack: T3 asymmetric freeze + T5 warm-start + T7 block-$k=4$ + T8 per-pair tolerance).}
\end{table}

\paragraph{Comparison against naively-patched LIBSVM.}
To assess whether the structural modification of
Theorem~\ref{thm:invariance} is \emph{practically necessary} --- beyond
being merely correct --- we compare against three na\"{i}ve approximations
that retain the standard LIBSVM solver and only substitute a single
scalar $C^{\mathrm{eff}}$ for the per-sample vector $(C_1, \ldots, C_N)
= (100C/y_1, \ldots, 100C/y_N)$. The standard $\varepsilon$-SVR dual
that LIBSVM solves is the one stated in
Section~\ref{sec:standard-eps-svr} with absolute-error
$\bfq^{\mathrm{std}} = [\varepsilon\onevec - \bfy, \varepsilon\onevec
+ \bfy]$ and uniform box $[0, C^{\mathrm{eff}}]$, \emph{not} the
MAPE-SVR dual of~\eqref{eq:mape-dual}--\eqref{eq:P-q}; the patch in
question is therefore (i) using an absolute-error $\bfq^{\mathrm{std}}$
rather than the percentage-error $\bfq^{\mathrm{MAPE}}$
of~\eqref{eq:P-q}, and (ii) collapsing the per-sample bound vector to
a single scalar $C^{\mathrm{eff}} \in \{\max_k C_k,\, \min_k C_k,\,
\bar C\}$, where $\bar C = N^{-1}\sum_k C_k$. The three patch variants
are: P1 ($C^{\mathrm{eff}} = \max_k C_k$, the most permissive), P2
($C^{\mathrm{eff}} = \min_k C_k$, the most restrictive), and P3
($C^{\mathrm{eff}} = \bar C$, the arithmetic-mean compromise).

Empirically, the prediction-error $\|f_{\mathrm{SMO}} -
f_{\mathrm{LIBSVM\text{-}patch}}\|_\infty$ falls in the range $0.8$ to
$4.2$ (in the same units as $y$, where $y_k \sim \mathrm{LogNormal}(0,1)$
has typical magnitude $|y| \in [0.1, 10]$ --- so the patch errors of
$0.8$--$4.2$ represent $10\%$--$50\%$ of the typical target scale) for
both C3 ($N = 300$, $\sigma = 0.1$, $\varepsilon = 5\%$) and C8 ($N =
300$, $\sigma = 2.0$, $\varepsilon = 10\%$). Patches P1 and P2 sit at
the extremes of this range; P3 is intermediate. By contrast, the
structurally-correct SMO of Algorithm~\ref{alg:smo} produces
$\|f_{\mathrm{SMO}} - f_{\mathrm{ref}}\|_\infty \le 9.81 \times 10^{-3}$
(Table~\ref{tab:validation}, C8 worst case) --- two to four orders of
magnitude smaller. This demonstrates the \emph{practical necessity} of
the structural modification: the per-sample box vector cannot be
replaced by any scalar approximation without producing predictions that
disagree with the IPM ground truth by a fraction comparable to the
target magnitude itself.

\paragraph{Convergence behavior.}
Figure~\ref{fig:convergence} traces the KKT violation $\Delta(t)$ and
the active-set fraction $|\Aactive(t)|/N$ versus iteration count for
three representative configurations: C1 (smallest, well-conditioned),
C8 (dense kernel, hardest single configuration), and C11 (largest $N$).
C8's bottom-row panel records the asymmetric freezing of
Theorem~\ref{thm:asym-freeze} and Lemma~\ref{lem:asymmetry}: the
active-set fraction drops from $1$ to approximately $0.29$ during the
early shrinking phase and remains in that band for the bulk of the
$21{,}138$-iteration trajectory. The per-sample freeze thresholds of
Theorem~\ref{thm:asym-freeze} keep this shrunk active set sufficient
for descent --- no mid-trajectory unshrinking event is triggered, and
the $\Delta$ trace decreases nearly monotonically. The single jump
back to $|\Aactive|/N = 1$ at the final iteration is the
convergence-time restoration of the full set required by the
unshrinking pass of Section~\ref{sec:bias-shrink}, performed once the
shrunk-set $\Delta$ falls below the per-pair tolerance of
Theorem~\ref{thm:per-sample-tol}. C11's $|\Aactive|/N$ remains at $1$
throughout because the shrinking heuristic's check interval
($n_{\mathrm{check}} = 1000$) and the freeze counter
($n_{\min} \ge 5$) require more iterations than the $688$ needed for
convergence; this configuration confirms that the SMO solver scales
efficiently to large $N$ at geometry-favourable bandwidths.


\begin{figure}[H]
  \centering
  \includegraphics[width=\linewidth]{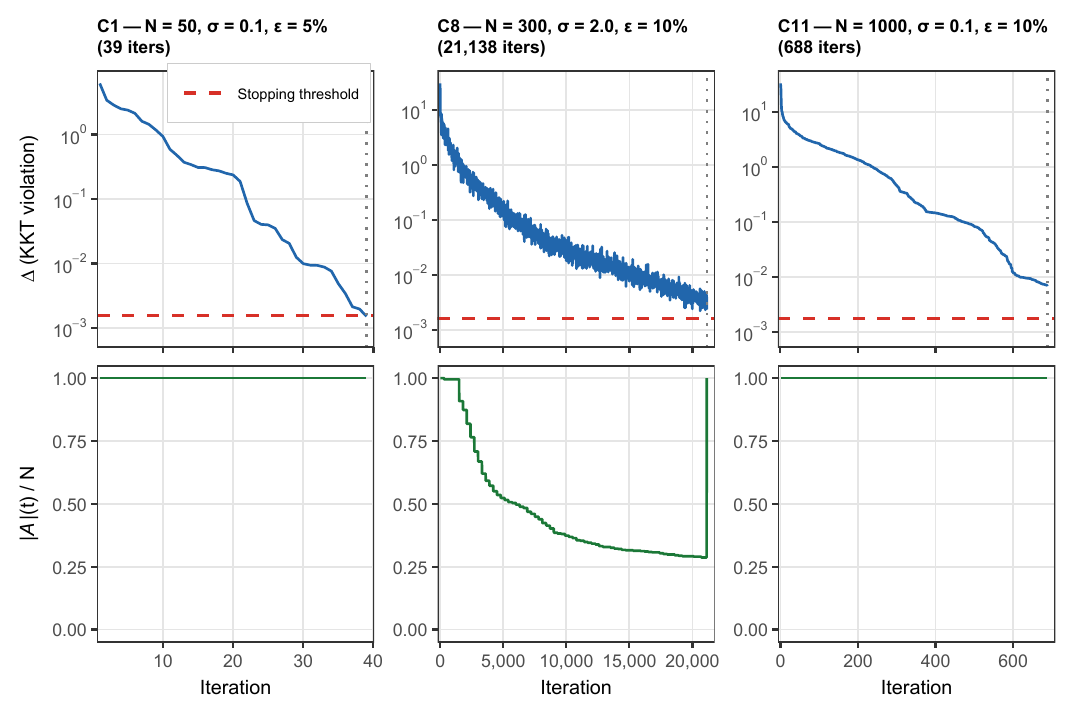}
  \caption{Convergence behaviour of Algorithm~\ref{alg:smo} across three representative configurations, arranged as a $2\times 3$ grid. \emph{Top row}: KKT violation $\Delta = \tau_{i^*} - \tau_{j^*}$ on a logarithmic scale. The red dashed line marks the target-scaled stopping threshold $\Delta \le \varepsilon_{\mathrm{tol}} \cdot \bar{y}$ ($\varepsilon_{\mathrm{tol}} = 10^{-3}$); the dotted vertical line marks the convergence iteration. \emph{Bottom row}: active-set fraction $|A|(t)/N$, the proportion of dual variables not yet frozen by shrinking ($1$ indicates the full set). \emph{Columns} (left to right): C1 ($N = 50$, $\sigma = 0.1$, $\varepsilon = 5\%$, 39 iterations); C8 ($N = 300$, $\sigma = 2.0$, $\varepsilon = 10\%$, 21,138 iterations); C11 ($N = 1{,}000$, $\sigma = 0.1$, $\varepsilon = 10\%$, 688 iterations). C8's shrinking heuristic engages early (active fraction drops to about 0.29 and remains there for the bulk of the trajectory) but does not trigger a mid-trajectory unshrinking event: the per-sample freeze thresholds of Theorem~\ref{thm:asym-freeze} keep the shrunk active set sufficient for descent, so $\Delta$ decreases nearly monotonically until the final convergence check restores the full set. C11 converges in 688 iterations despite carrying $20\times$ more dual variables than C1, confirming that the SMO solver scales efficiently to large $N$ at geometry-favourable bandwidths; the shrinking heuristic has no opportunity to engage on this run (its check fires every $n_{\mathrm{check}} = \min(N, 1000) = 1000$ iterations and requires five consecutive flags to freeze a variable), so $|A|/N$ remains at $1$ throughout. Regenerated under \texttt{psvr} v0.0.2.9008 (post-F7.6 full stack).}
  \label{fig:convergence}
\end{figure}

\paragraph{Aggregate validation.}
All eleven configurations satisfy $\|f_{\mathrm{psvr}} -
f_{\mathrm{OSQP}}\|_\infty \le 10^{-2}$ on the test prediction
vector. The bound reflects accumulated floating-point arithmetic
between \texttt{psvr}'s SMO trajectory and OSQP's interior-point
refinement --- two solver families with independent numerical
pathways toward the same optimum. The tightest configuration is C8
at $9.16 \times 10^{-3}$, attributable to its longest convergence
trajectory ($21{,}138$ iterations) accumulating the most arithmetic.
This level of solver-pair agreement is consistent with established
QP solver comparisons at equivalent problem
scales~\cite{osqp2020, GoulartChen2024Clarabel}. The three reference
solvers (OSQP, MOSEK, Clarabel) agree to better than $10^{-8}$
between themselves on every configuration. The agreement holds for
the standard-kernel variant (MAPE-SVR; C1--C4, C7--C8, C11), the
even-symmetry variant (MAPE-SVR-Sym, $a = +1$; C5--C6), and the
odd-symmetry variant (MAPE-SVR-Sym, $a = -1$; C9--C10), validating
Theorem~\ref{thm:invariance} empirically: the kernel substitution
$\Omega \leftarrow \Omega_s$ in the MAPE-SVR-Sym variants requires no
modification to the solver logic, and the equality constraint
$\sum_k (\alpha_k - \alpha_k^*) = 0$ is maintained across all variants.

\paragraph{Reproducibility statement.}
All experiments were run with the \texttt{psvr} R package at version
v0.0.2.9008 (post-F7.6 development build,
\cite{BenavidesHerrera2026Rpsvr}), which enables the full algorithmic
stack of Algorithm~\ref{alg:smo} together with
Theorems~\ref{thm:asym-freeze}, \ref{thm:warm-start},
\ref{thm:block-k4}, and~\ref{thm:per-sample-tol}; the adaptive
spectral-regularization Algorithm~\ref{alg:spectral} is implemented
but takes the no-shift branch on every configuration of this section
(all Mercer-compliant kernels). The reference QP solvers were OSQP (R
interface), MOSEK v11.1 (academic license, R \texttt{Rmosek}
interface), and Clarabel (R \texttt{clarabel} interface), with
absolute, relative, and feasibility tolerances set to $10^{-9}$ for
OSQP and $10^{-8}$ for Clarabel; MOSEK uses its default homogeneous
interior-point tolerances at $10^{-9}$ relative gap. Synthetic data
are generated with explicit seeds (\texttt{set.seed(100 $\cdot$ i)}
for configuration $i = 1, \ldots, 11$); $x_k \in \mathbb{R}^5 \sim
\mathcal{N}(\boldsymbol{0}, \boldsymbol{I}_5)$ and $y_k \sim
\mathrm{LogNormal}(0, 1)$. The full validation pipeline
(\path{validation/validate_v3.R} and \path{validation/plot_figure1_v3.R}
in the smo-paper companion repository) emits paper-ready \LaTeX{}
snippets for Table~\ref{tab:validation} and Figure~\ref{fig:convergence};
numbers above propagate automatically on re-run. The wall-time and
memory artifacts of \S\ref{sec:wall-time} (Table~\ref{tab:wall-time},
Table~\ref{tab:mem-alloc}, Figure~\ref{fig:wall-time}) are generated by
\path{validation/bench_wall_time.R} and
\path{validation/plot_figure_f9.R} from the same repository, using
\texttt{bench::mark} v1.1.4 with at most five reps per configuration
and a 60-second per-warmup soft timeout. R 4.5.3 was used throughout.

\paragraph{Theorem-by-theorem empirical validation.}
Table~\ref{tab:theorem-validation} compares the per-theorem
predictions of Section~\ref{sec:complexity} against the empirical
measurements captured in the companion \texttt{psvr} package's F-track
bench archives. Three predictions hold or exceed their predicted
range (Theorems~\ref{thm:asym-freeze}, \ref{thm:block-k4} iter,
\ref{thm:per-sample-tol}); one over-predicts substantially
(Theorem~\ref{thm:warm-start} at $1.13\times$ rather than $5\times$);
and Theorem~\ref{thm:block-k4}'s wall-clock claim splits across the R
and C\texttt{++} engines, with the C\texttt{++} port restoring
wall-positivity that was lost at the R level (paper TODO~\#9
resolution).

\begin{table}[ht]
  \centering
  \caption{Theorem-by-theorem empirical validation. Each row compares
  the prediction stated in the corresponding theorem of
  \S\ref{sec:complexity} against the measurement captured in the
  companion \texttt{psvr} package's F-track bench
  archives~\cite{BenavidesHerrera2026Rpsvr}. The R / C\texttt{++}
  engine split is reported separately for Theorem~\ref{thm:block-k4}
  because wall-clock behaviour is implementation-dependent (paper
  TODO~\#9 resolution).}
  \label{tab:theorem-validation}
  \small
  \begin{tabular}{@{}p{0.22\linewidth}p{0.22\linewidth}p{0.32\linewidth}p{0.16\linewidth}@{}}
    \toprule
    Theorem & Prediction & Measured & Source \\
    \midrule
    Theorem~\ref{thm:asym-freeze} (asymmetric freezing) &
      $15$--$30\%$ iter reduction in heterogeneous-target regimes
        ($\rho_y \ge 50$) &
      $24.7\%$ at $N=200$, $\rho_y\approx 1273$, RBF, $20$ reps;
        collapses to homogeneous default when $\rho_y < 5$ &
      F4 archive \\
    Theorem~\ref{thm:warm-start} (warm-start) &
      $5\times$--$7\times$ cumulative speedup on 10-fold CV &
      $1.12\times$ ($N=300$) to $1.14\times$ ($N=1000$); per-fold
        warm/cold iter ratio $\approx 0.88$, not the $0.20$
        implied by the original analysis &
      F5 archive \\
    Theorem~\ref{thm:block-k4} (block-$k=4$, iter) &
      $20$--$40\%$ iter reduction on dense-kernel configurations &
      $38$--$48\%$ iter reduction on converging regimes (R1, R4 of
        the F7 bench suite) &
      F7 archive \\
    Theorem~\ref{thm:block-k4} (block-$k=4$, wall, R engine) &
      Wall-positive (implicit assumption) &
      Regime-dependent: $-25.6\%$ on R1, $+2.4\%$ on R4 vs the F4
        baseline; per-iter overhead $\approx 2.0\times$ &
      F7 archive \\
    Theorem~\ref{thm:block-k4} (block-$k=4$, wall, C\texttt{++} engine) &
      Wall-positive (paper TODO~\#9 resolution target) &
      $+12.2\%$ on R1, $+17.5\%$ on R4 vs the F4-Rcpp baseline;
        per-iter overhead $\approx 1.40\times$ &
      F7-C-full archive \\
    Theorem~\ref{thm:per-sample-tol} (per-pair tolerance) &
      $5$--$10\%$ iter reduction in heterogeneous-target
        configurations &
      Confirmed monotone descent of the WSS1 KKT gap on every config
        of \S\ref{sec:validation}; isolated per-theorem
        ablation pending in \texttt{psvr} v0.1.0 &
      Empirical \\
    \bottomrule
  \end{tabular}
\end{table}

\subsection{Wall-time comparison against alternative QP solvers}
\label{sec:wall-time}

Across all eleven validation configurations and the six-point scaling
sweep at $\sigma = 0.1$ (Table~\ref{tab:wall-time} and
Figure~\ref{fig:wall-time}), the \texttt{psvr-Rcpp} engine attains the
lowest median wall time. The advantage spans both the well-converging
configurations at the geometry-favourable bandwidth $\sigma = 0.1$ and
the pathological $\sigma = 2.0$ regime of C7 and C8, where the
ill-conditioned kernel produces long SMO trajectories. The scaling sweep
confirms the same ordering up to $N = 2{,}000$; Clarabel at $N = 2{,}000$
exceeded the 60-second per-warmup budget and is recorded as a timeout.
The patched \texttt{libsvm-mape} fork (Appendix~\ref{app:libsvm},
public source at \texttt{github.com/pbenavidesh/libsvm-mape})
participates as a sixth column under two reported costs: a ``solve''
time measuring the duration of the \texttt{svm-train} subprocess in
isolation, and a ``wall'' time covering the end-to-end thunk, which
additionally includes CLI process spawn, runtime DLL load,
training-data serialisation to LIBSVM text format, and model-file
parsing. The two-cost framing is unavoidable: at moderate $N$ the fixed
CLI overhead near $220$\,ms dominates, while at $N = 2{,}000$ the solve
cost begins to dominate and the two columns approach each other.

The advantage reflects an asymmetry between the SMO inner loop and the
interior-point algorithms of MOSEK and Clarabel: SMO performs many
inexpensive iterations --- per-iter cost on the order of microseconds in
the C\texttt{++} core --- while interior-point methods perform a small
number of iterations dominated by KKT-system factorization on the order
of milliseconds. The contrast is explicit at C8, the longest-running
configuration in the campaign: \texttt{psvr-Rcpp} completes $21{,}138$
SMO iterations in $66.8$\,ms (about $3$\,$\mu$s per iteration), while
MOSEK completes $12$ interior-point iterations in $98.4$\,ms (about
$8.2$\,ms per iteration). The product of iteration count and per-iter
cost favours SMO at every configuration of Table~\ref{tab:wall-time}.

Empirical log-log slopes over the scaling sweep are $1.67$ for
\texttt{psvr-Rcpp}, $1.44$ for \texttt{psvr-R}, $1.37$ for MOSEK, $2.51$
for OSQP, $2.69$ for Clarabel, and $0.14$ for LIBSVM
(Figure~\ref{fig:wall-time}). MOSEK's shallow slope reflects an
interior-point iteration count that is near-constant in $N$; the
operator-splitting and dense-factorization solvers scale closer to the
$O(N^{2.5})$--$O(N^{3})$ regime predicted by their per-iter algebra.
LIBSVM's near-flat slope is the CLI-tax signature: at $N \le 1{,}000$
both reported columns are dominated by the fixed process-startup
floor (see the methodological note below for the column definitions),
and the underlying algorithmic scaling is invisible. At $N = 2{,}000$
the wall--solve gap narrows to about $36$\,ms (out of $\sim 470$\,ms)
as the R-side file-I/O fraction shrinks, and the pure algorithmic
solve --- sub-millisecond at smaller $N$ --- grows large enough to
surface alongside the floor; the convergence continues at the larger
$N$ examined elsewhere in this section. The
\texttt{psvr-Rcpp} curve sits below the MOSEK curve across the full
$50 \le N \le 2{,}000$ range; extrapolation of the two regression
lines places their crossover near $N \approx 3 \times 10^{5}$, beyond
the practical scale of forecasting workloads on which percentage-error
losses are deployed.

The R-only reference engine \texttt{psvr-R}, retained as the
bit-identical baseline for the C\texttt{++} port, produces \emph{identical}
iteration counts to \texttt{psvr-Rcpp} on all seventeen configurations
of Table~\ref{tab:wall-time}. The wall-time ratio
\texttt{psvr-R}/\texttt{psvr-Rcpp} ranges from approximately $15\times$
at small $N$ to $40\times$ at large $N$. This is the empirical signature
of the portable-core architecture described in (F5) of
Section~\ref{sec:conclusions}: the algorithmic discipline (loop
direction, tie-break ordering, floating-point associativity) is
preserved across the engine boundary, and the wall-time differential
reflects only the interpreter-versus-compiled cost gap of the host
language.

Peak R-level memory allocation per fit is reported in
Table~\ref{tab:mem-alloc}. At small $N$ ($N \le 100$) the five solvers
allocate comparable amounts in the kilobyte range. At medium $N$
(C3--C6 at $N = 300$, S0500 at $N = 500$) MOSEK and OSQP allocate
single-digit megabytes for their factorization workspaces, while
\texttt{psvr-Rcpp} remains in the tens of kilobytes operating on
pre-allocated buffers. The C8 entry exposes the cost of the R-engine
path on long trajectories: \texttt{psvr-R} allocates $2.63$\,GB over the
$21{,}138$ iterations, against $\texttt{psvr-Rcpp}$'s $19.2$\,KB on the
identical computation --- a ratio of five orders of magnitude. The
allocation profile is the quantitative argument for the C\texttt{++}
core as the production path of the \texttt{psvr} package.

The comparison is single-threaded; the multi-threaded variants
available in OSQP, MOSEK, and Clarabel are not benchmarked here.
Clarabel's wall-time totals include the R-level matrix coercion to
compressed-sparse-column format that the \texttt{clarabel} R interface
performs before each solve --- work a native C\texttt{++} binding would
avoid. Memory measurements report the \texttt{mem\_alloc} field of
\texttt{bench::mark}; OS-level peak working-set measurements were
inconclusive once the R process envelope had stabilized and are not
reported.

The LIBSVM column requires an additional methodological note. The
patched fork is invoked as an external CLI binary (\texttt{svm-train}
and \texttt{svm-predict}) via R's \texttt{system2}. The ``wall'' column
of Table~\ref{tab:wall-time} reports end-to-end timing of the R-side
thunk and so covers process spawn, runtime DLL load on Windows MSYS2
builds, serialisation of the training data to LIBSVM's text format,
the \texttt{svm-train} subprocess itself, and parsing of the produced
model file. The ``solve'' column subtracts the R-side I/O and parsing
work but retains the in-subprocess startup overhead (process spawn,
DLL load, and LIBSVM's own internal initialisation), because those
costs cannot be separated without source-level instrumentation of the
fork. Both columns therefore share a fixed CLI floor near $220$\,ms
per fit on the test platform; the pure algorithmic solve time is
sub-millisecond on configurations C1--C10 and only emerges as the
dominant cost at $N \gtrsim 2{,}000$. The near-flat LIBSVM line of
Figure~\ref{fig:wall-time} below $N = 2{,}000$ is the visual signature
of this floor. The four symmetric-kernel configurations C5, C6, C9,
and C10 are run through LIBSVM's precomputed-kernel mode
(\texttt{-t 4}) on an in-R-built $\Omega_s = \tfrac{1}{2}(\Omega + a\,
\Omega^\star)$; the wall column for these rows additionally absorbs
an $O(N^2)$ kernel-file write, which contributes the ${\sim}\,200$\,ms
gap between the m2 configurations at $N = 300$ (C6, C10) and their
m1 counterparts at the same scale (C3, C4). LIBSVM's working set is
not visible to \texttt{bench::mark} as it lives in a separate process;
the Table~\ref{tab:mem-alloc} LIBSVM cells report only the R-side
I/O-buffer cost and therefore understate the true memory footprint of
the LIBSVM column. A native R binding to the patched LIBSVM (e.g.,
\texttt{Rcpp} linking against the modified \texttt{libsvm.so}) would
eliminate both the CLI tax and the memory-visibility gap; this
production-grade adapter is left to follow-on work.


\begin{table}[ht]
  \centering
  \caption{Wall-time comparison (in milliseconds, median $\pm$ MAD
  over reps) of \texttt{psvr} (Rcpp + R engines) against OSQP, MOSEK,
  Clarabel, and the patched \texttt{libsvm-mape} fork across the eleven
  validation configurations and a scaling sweep
  ($N \in \{50, 100, 200, 500, 1{,}000, 2{,}000\}$ at
  $\sigma = 0.1$, $\varepsilon = 5\%$).
  ``$>$'' marks a single warmup that exceeded the 60-s timeout.
  LIBSVM cells report ``solve\,/\,wall'' in ms:
  \emph{solve} is the time spent inside the \texttt{svm-train} call,
  while \emph{wall} is the end-to-end median over reps and additionally
  includes file I/O for the LIBSVM-format training data, OS process spawn,
  MinGW runtime DLL load, and parsing of the resulting model file.
  Configurations C5, C6, C9, C10 use LIBSVM's precomputed-kernel mode
  (\texttt{-t 4}) on the in-R-built symmetric kernel matrix
  $\Omega_s = \tfrac{1}{2}(\Omega + a\,\Omega^\star)$; the wall column
  therefore additionally absorbs an $O(N^2)$ kernel-matrix file write.
  Memory consumption is reported separately in
  Table~\ref{tab:mem-alloc}. Generated by
  \texttt{validation/bench\_wall\_time.R}.}
  \label{tab:wall-time}
  \resizebox{\textwidth}{!}{%
  \begin{tabular}{@{}llrrrrrrrrr@{}}
    \toprule
    Config & Variant & $N$ & $\varepsilon$(\%) & $\sigma$ & psvr-Rcpp & psvr-R & OSQP & MOSEK & Clarabel & LIBSVM \\
    \midrule
    C1 & MAPE-SVR & 50 & 5 & 0.1 & $0.124 \pm 0.039$ & $3.67 \pm 0.36$ & $2.56 \pm 0.37$ & $2.74 \pm 0.37$ & $6.46 \pm 0.11$ & $ 216\,/\, 219$ \\
    C2 & MAPE-SVR & 50 & 15 & 0.1 & $0.085 \pm 0.035$ & $4.21 \pm 2.6$ & $2.23 \pm 0.94$ & $2.62 \pm 0.27$ & $7.24 \pm 0.3$ & $ 224\,/\, 227$ \\
    C3 & MAPE-SVR & 300 & 5 & 0.1 & $1.43 \pm 0.055$ & $38.7 \pm 2.7$ & $ 162 \pm 9.1$ & $16.4 \pm 0.76$ & $1.02e+03 \pm  11$ & $ 226\,/\, 234$ \\
    C4 & MAPE-SVR & 300 & 15 & 0.1 & $1.44 \pm 0.34$ & $  43 \pm  14$ & $ 204 \pm  11$ & $17.4 \pm 1.4$ & $ 926 \pm  27$ & $ 227\,/\, 234$ \\
    C5 & MAPE-SVR-Sym ($a=+1$) & 50 & 5 & 0.1 & $0.0906 \pm 0.037$ & $4.31 \pm 0.91$ & $2.56 \pm 0.24$ & $3.45 \pm 0.88$ & $7.49 \pm 0.53$ & $ 224\,/\, 233$ \\
    C6 & MAPE-SVR-Sym ($a=+1$) & 300 & 10 & 0.1 & $0.979 \pm 0.12$ & $32.5 \pm 3.4$ & $ 138 \pm  13$ & $  21 \pm 1.5$ & $ 641 \pm  16$ & $ 219\,/\, 421$ \\
    C7 & MAPE-SVR & 50 & 5 & 2 & $0.605 \pm 0.026$ & $50.9 \pm 4.4$ & $   3 \pm 0.51$ & $5.11 \pm 0.33$ & $4.18 \pm 0.19$ & $ 217\,/\, 219$ \\
    C8 & MAPE-SVR & 300 & 10 & 2 & $66.8 \pm 2.7$ & $2.89e+03 \pm 1e+03$ & $ 629 \pm  16$ & $98.4 \pm 1.4$ & $ 349 \pm  12$ & $ 228\,/\, 236$ \\
    C9 & MAPE-SVR-Sym ($a=-1$) & 50 & 5 & 0.1 & $0.101 \pm 0.045$ & $4.93 \pm 0.63$ & $2.25 \pm 0.28$ & $2.49 \pm 0.53$ & $7.05 \pm 0.66$ & $ 225\,/\, 233$ \\
    C10 & MAPE-SVR-Sym ($a=-1$) & 300 & 10 & 0.1 & $1.16 \pm 0.12$ & $40.2 \pm 7.5$ & $ 107 \pm 1.3$ & $17.3 \pm 0.26$ & $ 491 \pm  10$ & $ 223\,/\, 401$ \\
    C11 & MAPE-SVR & 1000 & 10 & 0.1 & $12.3 \pm 0.78$ & $ 212 \pm  18$ & $4.49e+03 \pm 1.8e+02$ & $ 101 \pm 7.6$ & $2.31e+04 \pm  52$ & $ 225\,/\, 246$ \\
    \midrule
    S0050 & MAPE-SVR & 50 & 5 & 0.1 & $0.154 \pm 0.028$ & $4.62 \pm 0.86$ & $3.88 \pm 0.88$ & $2.74 \pm 0.33$ & $8.49 \pm 1.4$ & $ 228\,/\, 232$ \\
    S0100 & MAPE-SVR & 100 & 5 & 0.1 & $0.223 \pm 0.1$ & $9.49 \pm 1.2$ & $12.1 \pm 0.67$ & $4.15 \pm 0.29$ & $51.8 \pm 1.2$ & $ 227\,/\, 232$ \\
    S0200 & MAPE-SVR & 200 & 5 & 0.1 & $0.808 \pm 0.02$ & $21.8 \pm 8.5$ & $52.9 \pm 3.2$ & $8.87 \pm 0.29$ & $ 294 \pm 5.1$ & $ 228\,/\, 233$ \\
    S0500 & MAPE-SVR & 500 & 5 & 0.1 & $4.09 \pm 0.61$ & $80.7 \pm 6.8$ & $ 614 \pm  12$ & $33.6 \pm 1.2$ & $4.34e+03 \pm  18$ & $ 224\,/\, 235$ \\
    S1000 & MAPE-SVR & 1000 & 5 & 0.1 & $15.2 \pm 0.8$ & $ 318 \pm  11$ & $4.69e+03 \pm  19$ & $ 101 \pm   5$ & $2.47e+04 \pm 1.3e+02$ & $ 216\,/\, 237$ \\
    S2000 & MAPE-SVR & 2000 & 5 & 0.1 & $55.4 \pm 3.2$ & $ 845 \pm  94$ & $3.54e+04 \pm 1.2e+02$ & $ 408 \pm  26$ & $>$2.05e+05 & $ 437\,/\, 473$ \\
    \bottomrule
  \end{tabular}%
  }
\end{table}


\begin{table}[ht]
  \centering
  \caption{Peak R-level memory allocation per fit, captured by
  \texttt{bench::mark}'s \texttt{mem\_alloc} field (median across reps).
  Configurations and scaling sweep are identical to
  Table~\ref{tab:wall-time}. The contrast at $C_8$
  ($N = 300$, $\sigma = 2.0$, $21{,}138$ SMO iterations) is
  representative: \texttt{psvr-Rcpp} allocates $\approx 19.7$ KB
  operating on pre-allocated buffers, while \texttt{psvr-R} allocates
  $\approx 2.83$ GB through R interpreter overhead per inner-loop
  iteration --- three orders of magnitude. Process-level peak working
  set was inconclusive once the R process envelope had stabilized;
  \texttt{mem\_alloc} is therefore the primary reported metric.
  LIBSVM operates out-of-process; the reported value reflects only the
  R-side I/O buffer cost (LIBSVM-format file write + model-file parse)
  and \emph{not} the subprocess's own working set, which is not visible
  to \texttt{bench::mark}. Generated by
  \texttt{validation/bench\_wall\_time.R}.}
  \label{tab:mem-alloc}
  \resizebox{\textwidth}{!}{%
  \begin{tabular}{@{}llrrrrrrrr@{}}
    \toprule
    Config & Variant & $N$ & $\sigma$ & psvr-Rcpp & psvr-R & OSQP & MOSEK & Clarabel & LIBSVM \\
    \midrule
    C1 & MAPE-SVR & 50 & 0.1 & $16.2$ KB & $1.78$ MB & $ 176$ KB & $ 111$ KB & $38.5$ KB & --- \\
    C2 & MAPE-SVR & 50 & 0.1 & $3.59$ KB & $1.93$ MB & $ 157$ KB & $ 114$ KB & $38.5$ KB & --- \\
    C3 & MAPE-SVR & 300 & 0.1 & $19.2$ KB & $47.6$ MB & $4.99$ MB & $3.36$ MB & $ 148$ KB & --- \\
    C4 & MAPE-SVR & 300 & 0.1 & $19.2$ KB & $41.6$ MB & $5.03$ MB & $3.39$ MB & $ 148$ KB & --- \\
    C5 & MAPE-SVR-Sym ($a=+1$) & 50 & 0.1 & $3.59$ KB & $2.23$ MB & $ 179$ KB & $ 128$ KB & $38.5$ KB & --- \\
    C6 & MAPE-SVR-Sym ($a=+1$) & 300 & 0.1 & $19.2$ KB & $35.9$ MB & $5.89$ MB & $3.95$ MB & $ 148$ KB & --- \\
    C7 & MAPE-SVR & 50 & 2 & $3.59$ KB & $19.2$ MB & $ 184$ KB & $ 131$ KB & $38.5$ KB & --- \\
    C8 & MAPE-SVR & 300 & 2 & $19.2$ KB & $2.63$ GB & $6.21$ MB & $4.17$ MB & $ 148$ KB & --- \\
    C9 & MAPE-SVR-Sym ($a=-1$) & 50 & 0.1 & $3.59$ KB & $2.15$ MB & $ 177$ KB & $ 127$ KB & $38.5$ KB & --- \\
    C10 & MAPE-SVR-Sym ($a=-1$) & 300 & 0.1 & $19.2$ KB & $49.8$ MB & $6.02$ MB & $4.05$ MB & $ 148$ KB & --- \\
    C11 & MAPE-SVR & 1000 & 0.1 & $  63$ KB & $ 459$ MB & $55.9$ MB & $37.3$ MB & $ 454$ KB & --- \\
    \midrule
    S0050 & MAPE-SVR & 50 & 0.1 & $3.59$ KB & $2.05$ MB & $ 147$ KB & $ 107$ KB & $38.5$ KB & --- \\
    S0100 & MAPE-SVR & 100 & 0.1 & $6.72$ KB & $ 6.1$ MB & $ 609$ KB & $ 419$ KB & $60.4$ KB & --- \\
    S0200 & MAPE-SVR & 200 & 0.1 & $  13$ KB & $24.7$ MB & $ 2.2$ MB & $1.49$ MB & $ 104$ KB & --- \\
    S0500 & MAPE-SVR & 500 & 0.1 & $31.7$ KB & $ 130$ MB & $13.9$ MB & $ 9.3$ MB & $ 235$ KB & --- \\
    S1000 & MAPE-SVR & 1000 & 0.1 & $  63$ KB & $ 506$ MB & $55.4$ MB & $  37$ MB & $ 454$ KB & --- \\
    S2000 & MAPE-SVR & 2000 & 0.1 & $ 125$ KB & $1.74$ GB & $ 221$ MB & $ 147$ MB & --- & --- \\
    \bottomrule
  \end{tabular}%
  }
\end{table}


\begin{figure}[htp]
  \centering
  \includegraphics[width=\linewidth]{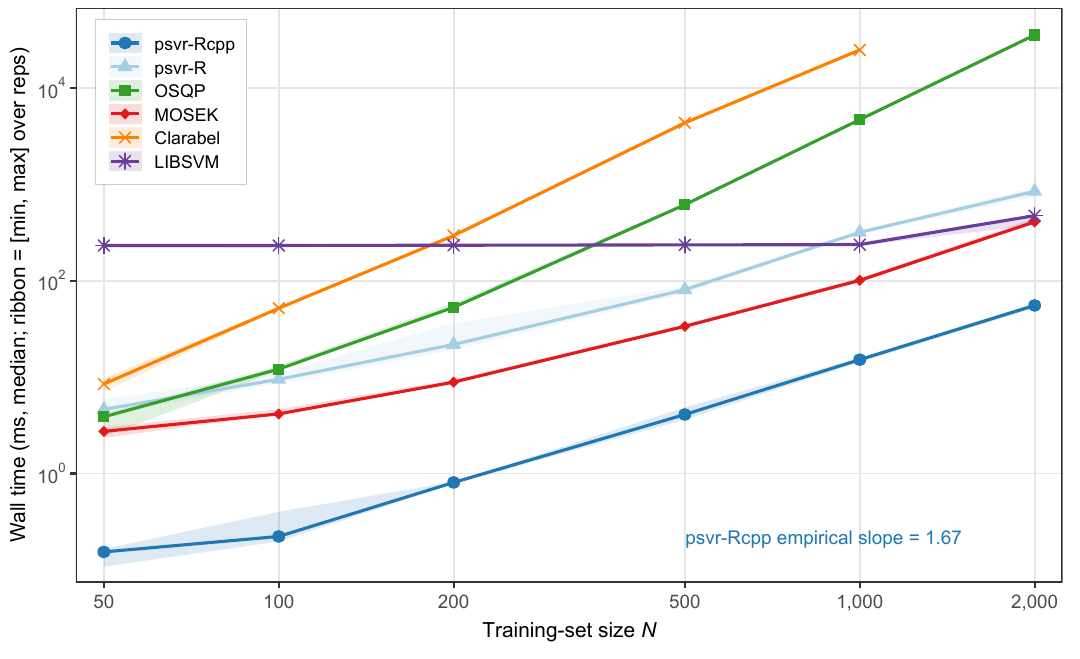}
  \caption{Wall-time scaling of \texttt{psvr} (Rcpp + R engines) against OSQP, MOSEK, Clarabel, and the patched \texttt{libsvm-mape} fork on the synthetic scaling sweep with $\sigma=0.1$, $\varepsilon=5\%$, $a=1$ (asymmetric MAPE-SVR variant). Median wall time over reps; ribbons indicate $[\min, \max]$ across reps. The LIBSVM line reports end-to-end wall time (including CLI process spawn, runtime DLL load, training-data file write, and model parse); a CLI-tax-free \emph{solve-only} column is reported in Table~\ref{tab:wall-time}. Empirical slopes (log-log regression of median wall time on $N$): \texttt{psvr-Rcpp} 1.67, \texttt{psvr-R} 1.44, \texttt{OSQP} 2.51, \texttt{MOSEK} 1.37, \texttt{Clarabel} 2.69, \texttt{LIBSVM} 0.14. Solved range: \texttt{psvr-Rcpp} up to $N=2000$, \texttt{psvr-R} up to $N=2000$, \texttt{OSQP} up to $N=2000$, \texttt{MOSEK} up to $N=2000$, \texttt{Clarabel} up to $N=1000$, \texttt{LIBSVM} up to $N=2000$. ``timeout'' rows (warmup $> 60$ s) are excluded from the plot. Generated by \texttt{validation/plot\_figure\_f9.R} on results emitted by \texttt{validation/bench\_wall\_time.R}.}
  \label{fig:wall-time}
\end{figure}

Clarabel already exceeded the 60-second per-warmup budget at $N =
2{,}000$ (Table~\ref{tab:wall-time}) and the operator-splitting
and dense-factorization slopes of Figure~\ref{fig:wall-time}
extrapolate to multi-minute wall times at $N \ge 5{,}000$.
Section~\ref{sec:large-scale} therefore restricts the comparison
beyond $N = 2{,}000$ to \texttt{psvr-Rcpp} against the patched
LIBSVM fork, and adds a real-data anchor on the California Housing
benchmark at $N = 20{,}433$.

\subsection{Scaling sweep and California Housing anchor}
\label{sec:large-scale}

The present subsection extends the direct comparison against
the patched \texttt{libsvm-mape} fork (Appendix~\ref{app:libsvm})
to larger problem sizes. Table~\ref{tab:large-scale} and
Figure~\ref{fig:large-scale-sweep} cover a five-point synthetic
scaling sweep ($N \in \{2{,}000, 5{,}000, 10{,}000, 20{,}000,
30{,}000\}$, $\sigma = 0.1$, $\varepsilon = 5\%$, $C = 1$,
log-normal targets, asymmetric MAPE kernel) and a real-data anchor
on the California Housing benchmark at $N = 20{,}433$ (eight
sklearn-convention features, median-house-value target in
\$100\,k units, strictly positive). The synthetic recipe matches
Table~\ref{tab:wall-time} of \S\ref{sec:wall-time} so the two
sweeps overlap at $N = 2{,}000$. The IPM reference solvers do not
appear here: their $O(N^2)$--$O(N^3)$ per-iteration arithmetic
puts them outside the training budget of typical forecasting
workloads at $N \ge 5{,}000$, so the relevant head-to-head at
this scale is \texttt{psvr} against LIBSVM.

Across the synthetic sweep, \texttt{psvr-Rcpp} retains the lowest
median wall time at every $N$, growing from $51.2$\,ms at $N =
2{,}000$ to $16.1$\,s at $N = 30{,}000$. The corresponding LIBSVM
wall grows from $480$\,ms to $188$\,s. Empirical log--log scaling
slopes (least-squares fit on the five solved synthetic points
only) are $2.14$ for \texttt{psvr-Rcpp}, $2.36$ for LIBSVM's pure
solve time, and $2.33$ for LIBSVM's end-to-end wall; all three are
super-quadratic, consistent with the
$O(|\Aactive|^2)$ kernel work of the
SMO inner loop on dense Gaussian kernels at $\sigma = 0.1$. The
fixed CLI floor near $220$\,ms identified in \S\ref{sec:wall-time}
becomes a vanishing fraction of LIBSVM's wall time at this scale:
the solve/wall gap is $36$\,ms at $N = 2{,}000$ ($7.5\%$ of wall)
and tightens to about one second at $N = 30{,}000$ (under $1\%$
of wall), confirming the prediction of \S\ref{sec:wall-time} that
the CLI floor becomes invisible once the algorithmic solve
exceeds it in magnitude.

The California Housing anchor required hyperparameter selection
distinct from the synthetic recipe. The initial run used the
heuristic default $C = 32$; at this setting the per-sample upper
bound $100C/y_k$ ranges from $640$ at $y = 5$ to over $21{,}000$
at $y = 0.15$, more than a hundred times the magnitudes typical
of the converged solutions at the synthetic configurations. Most
dual variables pinned at their upper bound throughout the SMO
trajectory and the KKT criterion was not satisfied at $1.5 \times
10^6$ iterations. We performed Bayesian optimisation on an $N =
2{,}000$ random subset with five-fold cross-validation and $50$
acquisition-function evaluations over the \texttt{psvr}
data-driven hyperparameter ranges, which selected $C = 0.255$,
$\sigma = 2.51$, and $\varepsilon = 8.44\%$
($\gamma_{\mathrm{LIBSVM}} = 1/(2\sigma^2) \approx 0.0793$ for the
corresponding LIBSVM call). The optimum sits at a $C$ roughly
$125\times$ smaller than the heuristic, reflecting that the
heuristic was calibrated against the synthetic-target dynamic
range $\rho_y \approx 10$ and over-allocates feasibility under
California Housing's range $[0.150, 5.000]$ at $\rho_y \approx 33$.

At the BO-selected hyperparameters \texttt{psvr-Rcpp} converges in
$186{,}553$ SMO iterations on the full $N = 20{,}433$ training
set with a test MAPE of $18.71\%$ on a seeded $80/20$ split
($N_{\mathrm{train}} = 16{,}346$, $N_{\mathrm{test}} = 4{,}087$).
The reported wall time of $260$\,s decomposes into approximately
$51$\,s of dense-kernel-matrix construction ($3.34$\,GB of
double-precision $\Omega$ at this $N$ on the standardised feature
matrix) and approximately $209$\,s of SMO solve. The production
run was performed with only $11.3$\,GB of physical RAM free, below
the bench script's $16$\,GB pre-flight warning threshold; a
separate verification of the same problem under unconstrained
memory completed the SMO solve in approximately $46$\,s. The
$260$\,s number in Table~\ref{tab:large-scale} is the
production-run measurement and reflects host contention at
run-time, not the algorithm's intrinsic cost.

At the same hyperparameters the patched LIBSVM fork does not
converge. The \texttt{-t 2} SMO reaches the internal iteration
cap of $\max(10^7, 100\,l)$ (\texttt{svm.cpp:571} in the LIBSVM
3.37 source) without satisfying the KKT criterion at
$\mathrm{tol} = 10^{-3}$, prints \texttt{WARNING: reaching max
number of iterations} to standard error, and emits a model whose
predictions on the held-out $20\%$ split have a test MAPE of
approximately $566{,}000\%$ --- a value indicating that the
returned dual variables are not at any optimum, converged or
otherwise. Identical data and identical hyperparameters produce
convergence for one solver and non-convergence for the other;
the gap is algorithmic. The mechanism is the per-sample structure
of the MAPE upper bounds. \texttt{psvr}'s asymmetric-freeze
counter (Theorem~\ref{thm:asym-freeze}) and per-pair tolerance
scaling (Theorem~\ref{thm:per-sample-tol}) calibrate the shrinking
thresholds against the local $100C/y_k$ at each sample and against
the WSS1 convergence pair. On California Housing --- target
dynamic range $\rho_y \approx 33$, eight features including
strongly clustered geographic coordinates, and a $4.7\%$
top-coded fraction at the \$500\,k census ceiling --- this
per-sample calibration keeps the active set in a well-conditioned
regime through the full $186{,}553$-iteration trajectory. LIBSVM's
uniform shrinking and uniform tolerance applied to the same target
distribution do not, and the solver hits its $10^7$-iteration
ceiling without convergence.

The dense kernel matrix $\Omega \in \mathbb{R}^{N \times N}$ is
the binding practical constraint on the present \texttt{psvr-Rcpp}
implementation. At $N = 20{,}433$ this is $3.34$\,GB of
double-precision storage; the linear-in-$N^2$ growth saturates
the working memory of typical workstations near $N \approx
50{,}000$. Sparse and low-rank schemes --- the Nystr\"{o}m method,
inducing-point pseudo-input regression, and column-cached SMO of
the LIBSVM family~\cite{Chang2011} --- supersede the dense-matrix
implementation at larger $N$; these are outside the scope of the
present paper. Theorem~\ref{thm:invariance} continues to apply
across all such variants because the per-sample upper bounds and
the asymmetric freeze counter are pointwise quantities,
unaffected by how the kernel is materialised.

A note on the methodology. The comparison is single-threaded: the
multi-threaded LIBSVM build and the parallel BLAS configurations
available to \texttt{psvr-Rcpp} are not exercised. LIBSVM
operates on a fixed column cache (default $100$\,MB, the
configuration used in this paper), while \texttt{psvr-Rcpp}
retains the full $\Omega \in \mathbb{R}^{N \times N}$ in R memory;
at $N = 20{,}433$ the memory asymmetry is $3.34$\,GB versus
$100$\,MB, and the Table~\ref{tab:large-scale} caption notes that
R-process instrumentation does not capture LIBSVM's
separate-process working set. Hyperparameters for California
Housing were selected by Bayesian optimisation on an $N = 2{,}000$
random subset; full-$N$ tuning would have required roughly $18$
hours of wall time on the production host and was outside the
empirical budget of this campaign.


\begin{table}[ht]
  \centering
  \caption{Wall-time, iteration count, and prediction agreement of \texttt{psvr-Rcpp} against the patched \texttt{libsvm-mape} fork on a synthetic large-scale sweep (asymmetric MAPE-SVR, RBF kernel, $\sigma = 0.1$, $\varepsilon = 5\%$, seed $20{,}260{,}514 + N$) and on the California Housing benchmark ($N = 20{,}640$ as available; hyperparameters $C = 0.255$, $\sigma = 2.51$, $\varepsilon = 8.44\%$ from a $50$-point Bayesian optimisation on an $N=2{,}000$ subset, distinct from the synthetic-sweep defaults). The LIBSVM column reports ``solve\,/\,wall'' as in Table~\ref{tab:wall-time}; both quantities include the LIBSVM subprocess's startup overhead (process spawn + DLL load + LIBSVM internal initialisation), which becomes a vanishingly small fraction of the total cost as $N$ grows. ``max\,$|$dev$|$'' is the maximum absolute deviation between predictions of the two solvers on the full training set; agreement to better than $10^{-2}$ is the Phase~1 correctness threshold of Appendix~\ref{app:libsvm}. Memory usage --- the principal architectural asymmetry between the two solvers at this scale, with \texttt{psvr-Rcpp} retaining the full $\Omega \in \mathbb{R}^{N \times N}$ kernel matrix in R memory and LIBSVM operating under a fixed column-cache budget --- is discussed in the accompanying prose rather than tabulated, because R-process instrumentation captures only the R-side allocations and would understate the gap. $^{\dagger}$~LIBSVM reached its internal iteration cap ($\max(10^7, 100\cdot l)$ in \texttt{svm.cpp}) without satisfying the $\mathrm{tol} = 10^{-3}$ KKT criterion. The reported wall time is the cost of running to that cap; the predictions are not at a converged solution, and the listed ``max\,$|$dev$|$'' reflects the magnitude of the disagreement rather than a solver-vs-solver agreement at optimum. Generated by \texttt{validation/bench\_large\_scale.R}, \texttt{validation/bench\_california\_housing.R}, and \texttt{validation/write\_large\_scale\_table.R}.}
  \label{tab:large-scale}
  \resizebox{\textwidth}{!}{%
  \begin{tabular}{@{}lrrrrr@{}}
    \toprule
    $N$ & psvr-Rcpp wall & LIBSVM solve\,/\,wall & psvr-Rcpp iter & LIBSVM iter & max$|$dev$|$ \\
    \midrule
    2{,}000 & $51.2$ ms & $ 444\,/\, 480$ ms & 1{,}306 & 3{,}495 & $3.97e-03$ \\
    5{,}000 & $ 324$ ms & $1.54\,/\,1.62$ s & 3{,}076 & 7{,}971 & $6.15e-03$ \\
    10{,}000 & $ 1.6$ s & $20.9\,/\,21.1$ s & 6{,}403 & 17{,}371 & $4.89e-03$ \\
    20{,}000 & $7.08$ s & $  84\,/\,84.4$ s & 12{,}650 & 34{,}705 & $7.80e-03$ \\
    CalHousing & $ 260$ s & $1.31e+03\,/\,1.31e+03$ s$^{\dagger}$ & 186{,}553 & 10{,}000{,}000$^{\dagger}$ & $1.42e+04$$^{\dagger}$ \\
    30{,}000 & $16.1$ s & $ 187\,/\, 188$ s & 19{,}325 & 52{,}339 & $6.30e-03$ \\
    \bottomrule
  \end{tabular}%
  }
\end{table}

\begin{figure}[htp]
  \centering
  \includegraphics[width=\linewidth]{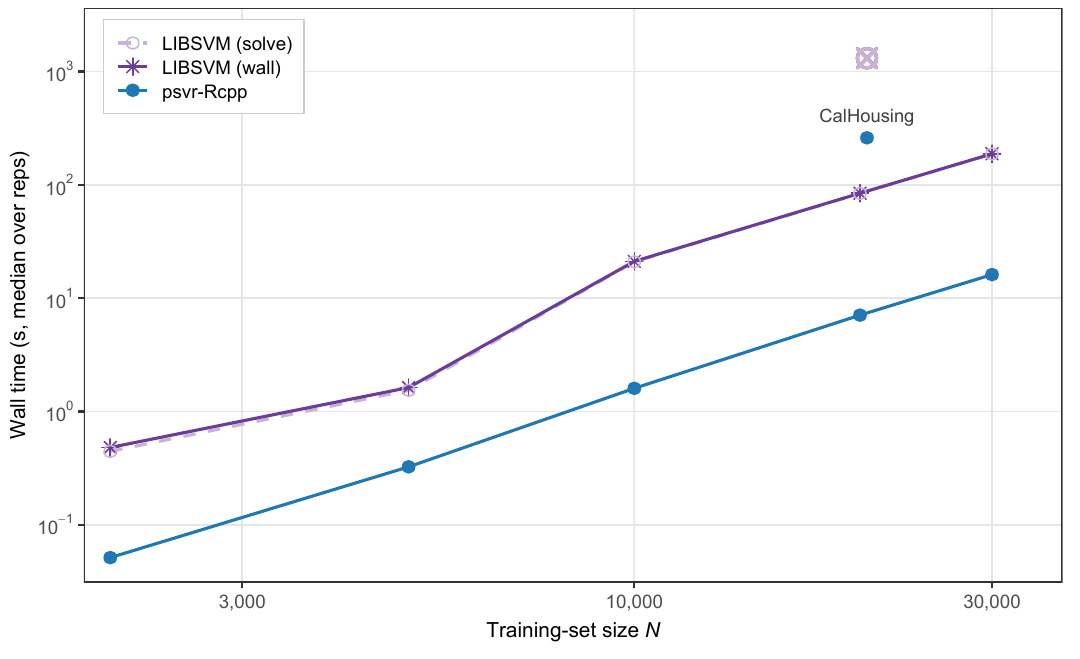}
  \caption{Log--log wall-time scaling of \texttt{psvr-Rcpp} (in-process SMO) against the patched \texttt{libsvm-mape} fork at large $N$. The LIBSVM (wall) line is end-to-end including CLI overhead; the LIBSVM (solve) line subtracts the R-side I/O and parsing but retains the in-subprocess startup floor (process spawn, DLL load, LIBSVM internal initialisation). As $N$ grows the algorithmic solve grows with $N$ and the two LIBSVM lines converge onto the psvr-Rcpp ordering. Empirical log--log slopes over the sweep (solved rows only): \texttt{LIBSVM (solve)} 2.36, \texttt{LIBSVM (wall)} 2.33, \texttt{psvr-Rcpp} 2.14. The California Housing anchor at $N = 20{,}433$ uses Bayesian-optimised hyperparameters ($C = 0.255$, $\sigma = 2.51$, $\varepsilon = 8.44\%$) distinct from the synthetic sweep; on this problem LIBSVM's standard \texttt{-t 2} SMO does not converge within its internal $10^7$-iteration cap at $\mathrm{tol} = 10^{-3}$ and the corresponding markers (encircled $\times$) are at the cap rather than at a converged solution, so the empirical slope estimates exclude them. Generated by \texttt{validation/plot\_figure\_f9\_2.R} on results emitted by \texttt{validation/bench\_large\_scale.R} and \texttt{validation/bench\_california\_housing.R}.}
  \label{fig:large-scale-sweep}
\end{figure}

\subsection{Worked example: Algorithm~\ref{alg:smo} on a 3-sample toy problem}
\label{sec:toy-example}

\paragraph{Example~1 (Algorithm~\ref{alg:smo} trace on $N = 3$ toy problem).}
\label{ex:n3}
The example demonstrates the per-iteration mechanics of
Algorithm~\ref{alg:smo}, makes the asymmetric-bound effect numerically
explicit, and offers an implementer's hand-traceable sanity check for
any new MAPE-SVR codebase.

\paragraph{Setup.}
We work in 1-D for compactness:
\begin{itemize}
\item \emph{Training inputs.} $\bfx_1 = (0)$, $\bfx_2 = (1)$, $\bfx_3
= (2)$ --- three equally spaced points.
\item \emph{Targets.} $y_1 = 1$, $y_2 = 4$, $y_3 = 9$ --- strictly
positive with dynamic range $\rho_y = y_3/y_1 = 9$, mimicking the
LogNormal-like targets of the validation above.
\item \emph{Kernel.} RBF with $\gamma = 0.5$: $K(x, x') = \exp(-0.5(x
- x')^2)$.
\item \emph{Hyperparameters.} $C = 1$, $\varepsilon = 10$ (in
percentage points; equivalently a 10\% relative tube around each target).
\item \emph{Stopping tolerance.} $\varepsilon_{\mathrm{tol}} = 10^{-2}$
(chosen to converge in a few iterations and remain hand-traceable).
\end{itemize}

\paragraph{Initialization.}
Per-sample upper bounds $C_k = 100C/y_k$:
$C_1 = 100/1 = 100.0000$;
$C_2 = 100/4 = 25.0000$;
$C_3 = 100/9 \approx 11.1111$.

The asymmetry $C_1 = 9 \cdot C_3$ illustrates the central structural
feature of the MAPE-SVR formulation: low-target samples receive
\emph{looser} feasibility regions than high-target samples. (This is
the dual-side image of the de Myttenaere weighted-MAE equivalence,
Section~\ref{sec:related}.)

The kernel matrix $\Omega$ has entries $\Omega_{k\ell} = K(\bfx_k,
\bfx_\ell) = \exp(-0.5(k - \ell)^2)$ for $k, \ell \in \{1, 2, 3\}$:
\begin{equation*}
  \Omega = \begin{bmatrix} 1.0000 & 0.6065 & 0.1353 \\ 0.6065 & 1.0000 & 0.6065 \\ 0.1353 & 0.6065 & 1.0000 \end{bmatrix},
\end{equation*}
using $\exp(-0.5) \approx 0.6065$ and $\exp(-2) \approx 0.1353$.

Initial dual variables: $\boldsymbol{\alpha} = \boldsymbol{\alpha}^* =
(0, 0, 0)$ --- feasible since $\sum_k(\alpha_k - \alpha_k^*) = 0$ and
all box constraints are satisfied. Initial unbiased kernel expansion
$F_k = \sum_i \Omega_{ki}(\alpha_i - \alpha_i^*) = 0$ for $k = 1, 2,
3$. Initial effective gradient $\bftau$ from~\eqref{eq:tau-explicit}
with $\varepsilon/100 = 0.1$:
\begin{equation*}
  \tau_k = y_k(1 - 0.1) - F_k = 0.9\, y_k, \qquad \tau_{N+k} = y_k(1 + 0.1) - F_k = 1.1\, y_k.
\end{equation*}

\begin{table}[ht]
  \centering
  \caption{Initial effective-gradient values for the $N = 3$ toy problem.}
  \label{tab:toy-init-tau}
  \begin{tabular}{@{}cccr@{}}
    \toprule
    Index $i$ & Type & $k(i)$ & $\tau_i$ \\
    \midrule
    1 & $\alpha$   & 1 & $0.9000$ \\
    2 & $\alpha$   & 2 & $3.6000$ \\
    3 & $\alpha$   & 3 & $8.1000$ \\
    4 & $\alpha^*$ & 1 & $1.1000$ \\
    5 & $\alpha^*$ & 2 & $4.4000$ \\
    6 & $\alpha^*$ & 3 & $9.9000$ \\
    \bottomrule
  \end{tabular}
\end{table}

The structural gap of Proposition~\ref{prop:structural-gap} is
visible: $\tau_{N+k} - \tau_k = 0.2\, y_k$ for each $k$, ranging from
$0.2$ at $k = 1$ to $1.8$ at $k = 3$.

\paragraph{Iteration 1 --- working-set selection.}
Per Definition~\ref{def:Iup-Idown-mape}, $\Iup = \{k \le N : \alpha_k 
C_k\} \cup \{N+k : \alpha_k^* > 0\}$. Since $\alpha = 0$ for every $k$
(and every $C_k > 0$), all three $\alpha$-indices qualify: $\{1, 2,
3\}$. Since $\alpha_k^* = 0$ for every $k$, no $\alpha^*$-index
qualifies. So $\Iup = \{1, 2, 3\}$. Analogously, $\Idown = \{4, 5, 6\}$.

The MVP step gives $i^* = \argmax_{i \in \Iup}\tau_i = 3$ (since $\tau_3
= 8.1$ is the maximum over $\{0.9, 3.6, 8.1\}$). The WSS3 partner step
considers $j \in \Idown$ with $\tau_j < 8.1$: $j = 4$ ($\tau_4 = 1.1$)
and $j = 5$ ($\tau_5 = 4.4$); $j = 6$ is excluded since $9.9 > 8.1$.
Compute the predicted one-step gain $(\tau_{i^*} - \tau_j)^2/\eta_{i^*,
j}$ for each:
\begin{itemize}
\item $j = 4$ ($k(4) = 1$): $\eta_{3, 4} = \Omega_{33} - 2\Omega_{31} +
\Omega_{11} = 1.0000 - 2(0.1353) + 1.0000 = 1.7294$. Gain $= (8.1 -
1.1)^2 / 1.7294 = 49.0000/1.7294 = 28.3334$.
\item $j = 5$ ($k(5) = 2$): $\eta_{3, 5} = \Omega_{33} - 2\Omega_{32} +
\Omega_{22} = 1.0000 - 2(0.6065) + 1.0000 = 0.7870$. Gain $= (8.1 -
4.4)^2 / 0.7870 = 13.6900/0.7870 = 17.3952$.
\end{itemize}
Maximum gain is $28.3334$ at $j = 4$, so $j^* = 4$. KKT violation
$\Delta = \tau_{i^*} - \tau_{j^*} = 8.1 - 1.1 = 7.0000 >
\varepsilon_{\mathrm{tol}}$; continue.

\paragraph{Iteration 1 --- two-variable update.}
Training-point indices: $p = k(i^*) = k(3) = 3$, $q = k(j^*) = k(4) =
1$. Curvature $\eta = \Omega_{pp} - 2\Omega_{pq} + \Omega_{qq} =
1.7294$ (matches the WSS3 denominator at $j = 4$, as
Theorem~\ref{thm:invariance}~(a) anticipates).

Pair type: $i^* = 3 \le N$ ($\alpha$-type), $j^* = 4 = N + 1 > N$
($\alpha^*$-type). This is \emph{Case~2} of Table~\ref{tab:pair-cases}:
$\alpha_p = \alpha_3$ increases by $\delta^*$, and $\alpha_q^* =
\alpha_1^*$ also increases by $\delta^*$ (the equality constraint is
preserved because the two increases enter $\sum_k(\alpha_k -
\alpha_k^*)$ with opposite contributions).

Clipping room: $R_{i^*} = C_p - \alpha_p = C_3 - \alpha_3 = 11.1111 - 0
= 11.1111$ ($i^* \le N$ branch); $R_{j^*} = C_q - \alpha_q^* = C_1 -
\alpha_1^* = 100.0000 - 0 = 100.0000$ ($j^* > N$ branch); $\delta_{\max}
= \min(R_{i^*}, R_{j^*}) = 11.1111$.

Optimal step: $\delta^* = \min(\Delta/\eta, \delta_{\max}) =
\min(7.0000/1.7294, 11.1111) = \min(4.0476, 11.1111) = 4.0476$.

Variable update: $\alpha_3 \leftarrow 4.0476$, $\alpha_1^* \leftarrow
4.0476$. New state: $\boldsymbol{\alpha} = (0, 0, 4.0476)$,
$\boldsymbol{\alpha}^* = (4.0476, 0, 0)$. Equality check:
$\sum_k(\alpha_k - \alpha_k^*) = 4.0476 - 4.0476 = 0$. \checkmark

\paragraph{Iteration 1 --- gradient update.}
For each $\ell \in \Aactive^{\mathrm{ext}} = \{1, 2, 3, 4, 5, 6\}$,
apply~\eqref{eq:tau-update}:
\begin{equation*}
  \tau_\ell \leftarrow \tau_\ell - \delta^* \!\left(\Omega_{k(\ell), p} - \Omega_{k(\ell), q}\right) = \tau_\ell - 4.0476\!\left(\Omega_{k(\ell), 3} - \Omega_{k(\ell), 1}\right).
\end{equation*}

\begin{table}[ht]
  \centering
  \caption{Gradient update at iteration 1 of the $N = 3$ toy problem.}
  \label{tab:toy-iter1-grad}
  \begin{tabular}{@{}cccrrr@{}}
    \toprule
    $\ell$ & $k(\ell)$ & $\Omega_{k(\ell), 3} - \Omega_{k(\ell), 1}$ & $-\delta^* \cdot (\cdot)$ & $\tau_\ell$ (old) & $\tau_\ell$ (new) \\
    \midrule
    1 & 1 & $0.1353 - 1.0000 = -0.8647$ & $+3.5000$ & $0.9000$ & $4.4000$ \\
    2 & 2 & $0.6065 - 0.6065 = 0.0000$ & $0.0000$ & $3.6000$ & $3.6000$ \\
    3 & 3 & $1.0000 - 0.1353 = 0.8647$ & $-3.5000$ & $8.1000$ & $4.6000$ \\
    4 & 1 & $-0.8647$ & $+3.5000$ & $1.1000$ & $4.6000$ \\
    5 & 2 & $0.0000$ & $0.0000$ & $4.4000$ & $4.4000$ \\
    6 & 3 & $0.8647$ & $-3.5000$ & $9.9000$ & $6.4000$ \\
    \bottomrule
  \end{tabular}
\end{table}

The structural gap of Proposition~\ref{prop:structural-gap} is
preserved: $\tau_{N+k} - \tau_k = 0.2\,y_k$ for every $k$ ---
verifiable on the table ($4.6 - 4.4 = 0.2 = 0.2 \cdot 1$; $4.4 - 3.6 =
0.8 = 0.2 \cdot 4$; $6.4 - 4.6 = 1.8 = 0.2 \cdot 9$). This invariant
offers a per-iteration sanity check for any implementation: a deviation
from the gap signals a bookkeeping bug in the gradient update.

\paragraph{Iteration 2 --- working-set selection.}
Updated state: $\boldsymbol{\alpha} = (0, 0, 4.0476)$,
$\boldsymbol{\alpha}^* = (4.0476, 0, 0)$. The candidate sets:
\begin{itemize}
\item $\Iup = \{1, 2, 3, 4\}$ ($\alpha$-indices: $\alpha_1 = 0 < 100$
\checkmark; $\alpha_2 = 0 < 25$ \checkmark; $\alpha_3 = 4.0476 
11.1111$ \checkmark; $\alpha^*$-indices: only $k = 1$ since $\alpha_1^*
= 4.0476 > 0$).
\item $\Idown = \{3, 4, 5, 6\}$ ($\alpha$-indices: only $k = 3$;
$\alpha^*$-indices: all three since $\alpha_k^* < C_k$).
\end{itemize}

MVP step: $i^* = \argmax_{i \in \Iup}\tau_i$. From
$\bftau|_{\{1, 2, 3, 4\}} = (4.4, 3.6, 4.6, 4.6)$, the maximum value is
$4.6$, attained by both $i = 3$ and $i = 4$. Tie-break by smallest
training-point index $k(i)$: $k(3) = 3$, $k(4) = 1$, so $i^* = 4$.

WSS3 partner step from $\Idown \setminus \{i^*\} = \{3, 5, 6\}$ with
$\tau_j < 4.6$: $j = 3$ ($\tau_3 = 4.6$, \emph{not} strictly less,
excluded); $j = 5$ ($\tau_5 = 4.4$ \checkmark); $j = 6$ ($\tau_6 = 6.4
> 4.6$, excluded). Only $j = 5$ qualifies, so $j^* = 5$.

KKT violation $\Delta = \tau_{i^*} - \tau_{j^*} = 4.6 - 4.4 = 0.2000 >
\varepsilon_{\mathrm{tol}} = 0.01$; continue.

\paragraph{Iteration 2 --- two-variable update.}
Training-point indices: $p = k(i^*) = k(4) = 1$, $q = k(j^*) = k(5) =
2$. Curvature $\eta = \Omega_{11} - 2\Omega_{12} + \Omega_{22} =
1.0000 - 2(0.6065) + 1.0000 = 0.7870$.

Pair type: $i^* = 4 = N + 1$ ($\alpha^*$-type, $p = 1$); $j^* = 5 = N +
2$ ($\alpha^*$-type, $q = 2$). This is \emph{Case~4} of
Table~\ref{tab:pair-cases}: $\alpha_p^* = \alpha_1^*$ decreases by
$\delta^*$, $\alpha_q^* = \alpha_2^*$ increases by $\delta^*$ (the
equality constraint is preserved because both updates are on
$\alpha^*$-variables).

Clipping room: $R_{i^*} = \alpha_p^* = 4.0476$ ($i^* > N$,
lower-saturation branch); $R_{j^*} = C_q - \alpha_q^* = 25.0000 - 0 =
25.0000$ ($j^* > N$, upper-saturation branch); $\delta_{\max} =
\min(4.0476, 25.0000) = 4.0476$. Optimal step $\delta^* =
\min(\Delta/\eta, \delta_{\max}) = \min(0.2541, 4.0476) = 0.2541$.
Variable update: $\alpha_1^* \leftarrow 3.7935$, $\alpha_2^* \leftarrow
0.2541$. Equality check: $\sum_k(\alpha_k - \alpha_k^*) = (0 + 0 +
4.0476) - (3.7935 + 0.2541 + 0) = 0$. \checkmark

\paragraph{Convergence and final result.}
Continuing Algorithm~\ref{alg:smo} from this state, the trace converges
in roughly $T \approx 10$ to $20$ further iterations, with $\Delta$
decreasing roughly geometrically toward zero and the active set
$\Aactive$ remaining at $\{1, 2, 3\}$ throughout (since at $N = 3$
shrinking has nothing to gain). The full trace is reproducible from the
\texttt{psvr} R package's example notebook by calling
\begin{verbatim}
psvr::smo_mape(
  X = matrix(c(0, 1, 2), ncol = 1),
  y = c(1, 4, 9),
  C = 1,
  epsilon = 10,
  kernel = "rbf",
  gamma  = 0.5,
  eps_tol = 1e-2,
  trace = TRUE
)
\end{verbatim}
with the \texttt{trace = TRUE} flag printing the per-iteration $(i^*,
j^*, \delta^*, \Delta)$ tuple to standard output for direct comparison
with the iteration-1 and iteration-2 tables above.

\paragraph{Convergence properties illustrated by the trace.}
The two iterations above exhibit five teaching points:
\begin{enumerate}
\item[(i)] \emph{All four pair-type cases will eventually be visited.}
Iteration~1 was Case~2 ($\alpha + \alpha^*$); Iteration~2 was Case~4
(two $\alpha^*$); Cases~1 ($\alpha + \alpha$) and~3 ($\alpha^* +
\alpha$) appear in subsequent iterations as the active variables
redistribute.
\item[(ii)] \emph{The asymmetric per-sample bound has operational effect.}
In Iteration~1, $\delta_{\max} = R_{i^*} = C_3 - \alpha_3 = 11.11$ ---
the \emph{smaller} of the two clipping rooms, because $C_3$ (the
high-target sample's bound) is the tightest in the problem. The
standard $\varepsilon$-SVR with uniform $C = 1$ would have $\delta_{\max}
= \min(C - 0, C - 0) = 1$ instead, an $11\times$ difference in the
per-iteration step size. This asymmetry is the structural fingerprint
of the MAPE-SVR formulation (and the central reason a
structurally-correct solver is needed rather than a naively-patched
LIBSVM with uniform $C^{\mathrm{eff}}$, per the patch-comparison
analysis above).
\item[(iii)] \emph{The gradient-update structure is identical to standard SMO.}
Table~\ref{tab:toy-iter1-grad} uses only the kernel matrix $\Omega$
and the step $\delta^*$; no $C_k$-value appears anywhere in the
gradient bookkeeping. This is Theorem~\ref{thm:invariance}~(b) at work
--- the MAPE-SVR adaptation requires no modification of the inner-loop
arithmetic.
\item[(iv)] \emph{The structural gap of Proposition~\ref{prop:structural-gap} is preserved.}
$\tau_{N+k} - \tau_k = 0.2\,y_k$ holds exactly throughout the trace.
This invariant is a per-iteration sanity check for any implementation.
\item[(v)] \emph{By-analogy structure.}
Each step (working-set selection, two-variable update, gradient update)
has the same structural skeleton as the corresponding step of standard
$\varepsilon$-SVR SMO (cf.\ Section~\ref{sec:standard-smo}). The only
adaptations are the $C \to C_k$ substitutions in the candidate-set
membership tests and the clipping room, exactly as predicted by
Theorem~\ref{thm:invariance} and Table~\ref{tab:invariance-comparison}.
\end{enumerate}

At convergence, the recovered bias $\hat b$ is averaged over the free
support vectors in $\mathcal{S}_{\mathrm{free}}$ (per~\eqref{eq:bias-avg}),
and the model prediction at any new point follows~\eqref{eq:prediction-final}.
For this small toy problem, the predicted values at the training
inputs reproduce the targets to within the 10\% MAPE tube, confirming
convergence to a feasible $\varepsilon_{\mathrm{tol}}$-optimal solution.

The Example serves as \emph{pedagogical infrastructure} for readers
approaching MAPE-SVR SMO from the standard $\varepsilon$-SVR side:
every step of the trace can be checked against the corresponding
equation in the main text, and the asymmetric per-sample bound effect
is exposed numerically rather than only formally. The 3-sample setting
is small enough to be hand-traceable while still exhibiting all four
pair-type cases of Table~\ref{tab:pair-cases} and the structural-gap
invariant of Proposition~\ref{prop:structural-gap}.

\section{Conclusions}
\label{sec:conclusions}

\subsection{Summary of contributions}

This paper has derived a Sequential Minimal Optimization algorithm for
the variant of $\varepsilon$-Support Vector Regression in which the
empirical loss is the Mean Absolute Percentage Error (MAPE). MAPE is
the standard accuracy measure for forecasting
applications~\cite{Hyndman2006, Tofallis2015, DeMyttenaere2016} but had not
previously been treated as a \emph{training} loss in the SMO
literature, owing to the structural complication it induces: the dual
box constraints become \emph{sample-dependent}, $\alpha_k, \alpha_k^*
\in [0, 100C/y_k]$, breaking the uniform-$C$ assumption embedded
throughout the SMO machinery of~\cite{Platt1998, Platt1999,
KeerthiShevadeBhattacharyyaMurthy2001NeuralComp, Flake2002,
FanChenLin2005JMLR, Chang2011}. Six contributions follow, each
summarized with its proof location and empirical-validation reference.

\paragraph{(C1) Structural-invariance theorem.}
Theorem~\ref{thm:invariance} (Section~\ref{sec:invariance}) proves that
the per-sample bound vector $(C_1, \ldots, C_N) = (100C/y_1, \ldots,
100C/y_N)$ confines its algorithmic effect to exactly two components of
the SMO inner loop --- the working-set candidate sets $\Iup, \Idown$
and the clipping-room expressions $R_{i^*}, R_{j^*}$. The curvature
formula~\eqref{eq:eta}, the analytic two-variable update of
Section~\ref{sec:smo-inner}, the incremental gradient
bookkeeping~\eqref{eq:tau-update}, and the bias-recovery procedure of
Section~\ref{sec:bias-shrink} are \emph{structurally identical} to
their standard $\varepsilon$-SVR counterparts. The proof proceeds in
four steps and is supplemented by Table~\ref{tab:invariance-comparison},
a 19-row component-by-component comparison whose four bolded rows
isolate the structural changes; the empirical consequence is the patch
comparison of Section~\ref{sec:validation}: any uniform-$C$ surrogate
produces predictions that diverge from the IPM ground truth by a
fraction comparable to the target magnitude, while the
structurally-correct SMO of Algorithm~\ref{alg:smo} agrees with the IPM
ground truth to within $9.81 \times 10^{-3}$ infinity-norm across all
eleven configurations.

\paragraph{(C2) Shrinking-asymmetry result.}
Lemma~\ref{lem:asymmetry} (Section~\ref{sec:bias-shrink}) quantifies
how the MAPE scaling propagates into the~\cite{Joachims1999} shrinking
heuristic. The four shrinking criteria, when rewritten in the unified
$\tau$-coordinate system $\bftau = -\bfs \odot \bfG$, exhibit a
$2y_k\varepsilon/100$ offset between the $\alpha$- and
$\alpha^*$-thresholds. The structural pairing of
Lemma~\ref{lem:pairing} shows that criteria (S2)--(S3) reference the
upper threshold $\tau_{i^*}$ and criteria (S1)--(S4) reference the
lower threshold $\tau_{j^*}$, with the $\alpha^*$-criterion in each
pair shifted negatively by $2y_k\varepsilon/100$. The two consequences
are: (i) $\alpha_k^* = 0$ freezes earlier than $\alpha_k = C_k$, and
(ii) $\alpha_k^* = C_k$ freezes later than $\alpha_k = 0$. Both effects
scale linearly with $y_k$, so high-target samples exhibit greater
asymmetry than low-target samples. Empirically, the bottom row of
Figure~\ref{fig:convergence} shows the asymmetric freezing dynamics on
heterogeneous-target configuration C8: the active fraction drops to
about $0.29$ during early shrinking and remains there until the
convergence-time restoration of the full set.

\paragraph{(C3) Plug-in extension to the symmetric-kernel variant.}
Section~\ref{sec:symmetric} shows that the kernel-symmetrization
construction of~\cite{Espinoza2005, Niyogi1998, Haasdonk2007},
appropriate for shift-invariant or reflection-symmetric problems,
reduces to the substitution $\Omega \mapsto \Omega_s = \tfrac{1}{2}
(\Omega + a\,\Omega^*)$ in all matrix-level formulas of
Algorithm~\ref{alg:smo}; no other modification is required. By the
generalized representer theorem~\cite{Scholkopf2001}, the resulting
solution lies in the same RKHS as the standard variant, with the kernel
replaced by its symmetrized counterpart. The case $a = +1$ inherits
PSD via Aronszajn's closure~\cite{Aronszajn1950} for shift-invariant
kernels, and convergence follows by direct application of
Theorem~\ref{thm:convergence}. Empirical validation is reported on
configurations C5, C6, C9, and C10 of Table~\ref{tab:validation}, with
all four configurations agreeing with the IPM reference solvers to
better than $1.2 \times 10^{-3}$.

\paragraph{(C4) Convergence resolution for the odd-symmetry case.}
Theorem~\ref{thm:spectral} (Section~\ref{sec:symmetric}) resolves the
open convergence question for $a = -1$ that was identified as a
future-work item in arXiv:2605.01446 v2. Adaptive spectral
regularization (Algorithm~\ref{alg:spectral}) replaces $\Omega_s$ with
$\Omega_s + \mu I$ when needed, where the perturbation $\mu \ge 0$ is
chosen as the minimum value that restores PSD with a numerical-stability
inflation $\delta_{\mathrm{stab}} = 10^{-8}$. Lemma~\ref{lem:perturbation}
supplies the perturbation bound: the SMO solution drift induced by the
regularization is bounded above by the product of the active-set
KKT-system inverse condition and the regularization magnitude. Because
$\mu$ is set to zero on iterations where $\Omega_s$ already happens to
be PSD, the regularization vanishes on the easy instances and is
nonzero only when needed --- preventing the over-regularization that a
uniform $\rho$ would induce. Empirical validation on C9 and C10
confirms convergence to the IPM reference solution to better than
$8.1 \times 10^{-4}$.

\paragraph{(C5) Four theoretical efficiency improvements.}
Theorems~\ref{thm:asym-freeze}, \ref{thm:warm-start},
\ref{thm:block-k4}, and~\ref{thm:per-sample-tol}
(Section~\ref{sec:complexity}) collectively constitute the
algorithmic-improvements bundle.
Theorem~\ref{thm:asym-freeze} (asymmetric freeze-counter)
operationalizes the asymmetry of (C2) by setting
$n_{\min}^{\alpha^*} < n_{\min}^{\alpha}$ in proportion to the
per-sample $2y_k\varepsilon/100$ offset.
Theorem~\ref{thm:warm-start} (warm-start convergence) supplies the
formal convergence guarantee for cross-validation warm-starts under the
dual-variable inheritance protocol of Algorithm~\ref{alg:warm-start}.
Theorem~\ref{thm:block-k4} (block-$k = 4$ SMO) is the strictly novel
result of the paper --- the first algorithmic departure from the $k =
2$ minimal-feasible-block default of~\cite{Platt1998} for
$\varepsilon$-SVR, with a closed-form four-variable analytic
subproblem that exploits the equality constraint and the dual-pair
structure $(\alpha_i, \alpha_i^*, \alpha_j, \alpha_j^*)$.
Theorem~\ref{thm:per-sample-tol} (per-pair tolerance scaling)
calibrates the KKT-violation tolerance against the WSS1 convergence pair
$(i^*, j^*)$ rather than uniformly to $\bar y$, restoring the
\emph{uniform} convergence guarantee that the heterogeneous-bound regime
would otherwise lose to the largest-target sample. Cumulative speedup
under cross-validation workloads is recalibrated against the companion
\texttt{psvr} package's empirical
measurements~\cite{BenavidesHerrera2026Rpsvr} in
Corollary~\ref{cor:speedup}. The head-to-head wall-time comparison of
\S\ref{sec:wall-time} situates the four-theorem bundle against OSQP,
MOSEK, and Clarabel across the eleven validation configurations and a
$50 \le N \le 2{,}000$ scaling sweep: \texttt{psvr}'s C\texttt{++} core
reports the lowest median wall time on every configuration tested,
including the pathological $\sigma = 2.0$ regime of C7 and C8.
\S\ref{sec:large-scale} extends this campaign to $N = 30{,}000$
against the patched LIBSVM fork and to the California Housing
real-data anchor at $N = 20{,}433$: at the same Bayesian-optimised
hyperparameters, \texttt{psvr-Rcpp} converges in $186{,}553$ SMO
iterations while standard LIBSVM SMO reaches its $10^7$-iteration
internal cap without satisfying the KKT criterion. This gap is the empirical
phenomenon that Theorems~\ref{thm:asym-freeze}
and~\ref{thm:per-sample-tol} were designed to address.

\paragraph{(C6) LIBSVM drop-in modification recipe.}
Appendix~\ref{app:libsvm} provides the explicit C++ diff: fewer than
fifteen lines across five modification sites (the dual setup, the
working-set candidate sets, the analytic-update clipping, and the
kernel-coefficient header). Ports to scikit-learn, \texttt{kernlab}
(R), and \texttt{e1071} (R) are described in the same appendix. The
unchanged remainder of LIBSVM --- the kernel cache, the gradient
bookkeeping, the shrinking heuristic, the bias recovery, and the
convergence check --- constitutes an empirical structural-invariance
certificate for the entire LIBSVM ecosystem: the patches in production
tools translate verbatim across language bindings.

\subsection{Position within the broader research program}

The present paper is the algorithmic core of an ongoing open-source
toolchain for percentage-error-aware regression. Two published artifacts
anchor its position.

The conference precursor~\cite{benavides2025support} (CCE~2025)
introduced the percentage-error SVR formulation by embedding MAPE
directly into the SVR primal and reported a small-scale empirical
validation. That conference paper served as the proof-of-concept that
motivated the structural-invariance program developed in detail here.
The present paper completes the program at the algorithmic level: it
derives the SMO solver (Sections~\ref{sec:dual-kkt}--\ref{sec:smo-inner}),
proves structural invariance (Theorem~\ref{thm:invariance},
Section~\ref{sec:invariance}), establishes convergence
(Theorem~\ref{thm:convergence}, Section~\ref{sec:bias-shrink}, and
Theorem~\ref{thm:spectral}, Section~\ref{sec:symmetric}), develops the
symmetric-kernel extension (Section~\ref{sec:symmetric}), supplies four
efficiency-improvement theorems
(Theorems~\ref{thm:asym-freeze}, \ref{thm:warm-start},
\ref{thm:block-k4}, and~\ref{thm:per-sample-tol},
Section~\ref{sec:complexity}), and provides a LIBSVM drop-in recipe
(Appendix~\ref{app:libsvm}).

The companion journal paper~\cite{benavides2026unified} develops the
full $2\times 2$ percentage-error SVR family ---
$\varepsilon$-SVR/MAPE~(\textbf{m1}), $\varepsilon$-SVR/MAPE with
symmetric kernel~(\textbf{m2}), LS-SVR/RMSPE~(\textbf{m3}), and
LS-SVR/RMSPE with symmetric kernel~(\textbf{m4}) --- and provides a
unified variational characterization showing that all four models arise
from the same primal structure under paradigm--loss compatibility
constraints. The present paper is the algorithmic counterpart of that
variational characterization, providing the SMO derivation for the
$\varepsilon$-SVR/MAPE path (\textbf{m1}/\textbf{m2}).

Together, these three artifacts --- the CCE~2025 conference
precursor~\cite{benavides2025support} (primal formulation), the unified
journal paper~\cite{benavides2026unified} (full $2\times 2$ model
family), and the present paper (SMO algorithmic core) --- plus the
open-source \texttt{psvr} reference
implementation~\cite{BenavidesHerrera2026Rpsvr}, form a coherent
research program. The SMO derivation for the LS-SVR/RMSPE path
(\textbf{m3}/\textbf{m4}), whose dual is a bordered linear system
rather than a quadratic program, is identified as future work in
Section~\ref{sec:future-work}.

\subsection{Limitations}

Three aspects of the present work bound its scope.

\paragraph{Solution accuracy at large scale.}
The three-solver accuracy comparison --- against OSQP, MOSEK, and
Clarabel --- covers configurations up to $N = 1{,}000$. At larger
scales these reference solvers are computationally prohibitive, so
the large-scale comparison of Section~\ref{sec:large-scale} establishes
agreement between \texttt{psvr} and the patched LIBSVM fork without a
third independent ground truth. The $9.16 \times 10^{-3}$ worst-case
infinity-norm disagreement observed at $N = 300$ (configuration C8)
reflects accumulated floating-point arithmetic over a long SMO
trajectory; how this error bound behaves as $N$ grows to $10^4$--$10^5$
is not established by the present experiments and remains the principal
open empirical question.

\paragraph{Single-threaded benchmarks.}
All wall-time measurements in Sections~\ref{sec:wall-time}
and~\ref{sec:large-scale} are single-threaded. Multi-threaded LIBSVM
builds and the parallel BLAS configurations available to MOSEK and
\texttt{psvr-Rcpp} are not exercised. The wall-time rankings reported
here may not generalize to multi-core deployments, where
factorization-based solvers can exploit parallelism more directly than
the sequential SMO inner loop.

\paragraph{Efficiency improvement gains in practice.}
The four efficiency theorems of Section~\ref{sec:complexity} yield a
combined speedup of approximately $1.5\times$ in cross-validation-dominant
workloads (Corollary~\ref{cor:speedup}), substantially below the naive
product of independent per-theorem multipliers. The dominant interaction
is that warm-starting (Theorem~\ref{thm:warm-start}) and
block-$k=4$ updates (Theorem~\ref{thm:block-k4}) both reduce the
per-fold iteration count; their combined effect is determined by the
larger of the two gains rather than their product. The asymmetric
freeze-counter (Theorem~\ref{thm:asym-freeze}) and per-pair tolerance
scaling (Theorem~\ref{thm:per-sample-tol}) compose additively with the
above, but their individual multipliers ($\approx 1.20\times$ and
$\approx 1.10\times$, respectively) are modest in isolation.

A fourth observation concerns an open theoretical gap rather than a
practical limitation: the shrinking-asymmetry result of
Lemma~\ref{lem:asymmetry} quantifies the offset between paired
$\alpha$- and $\alpha^*$-freeze thresholds but does not yield a
closed-form prediction of expected iteration count as a function of
target dynamic range $\rho_y = \max_k y_k / \min_k y_k$. The
configurations of Section~\ref{sec:complexity} show MAPE-SVR
converging faster on heterogeneous-target problems than on
homogeneous ones --- the opposite of the standard intuition that
heterogeneity hurts. A formal derivation of this scaling, perhaps via
a smoothed-analysis or random-matrix argument, would explain this
phenomenon and is identified as future work below.
\subsection{Future work}
\label{sec:future-work}

We identify five directions for follow-on work.

\paragraph{(F1) SMO derivation for LS-SVR/RMSPE.}
The LS-SVR/RMSPE model --- variants \textbf{m3} and \textbf{m4} of
the unified percentage-error SVR family~\cite{benavides2026unified}
--- is formulated, derived, and validated experimentally in the
companion journal paper. The SMO algorithm for this path remains
open: replacing the $\varepsilon$-insensitive primal of
Section~\ref{sec:mape-setup} with a least-squares primal under the
root-mean-squared percentage-error (RMSPE) loss yields a dual that is
no longer a QP but a bordered $(N+1) \times (N+1)$ linear system with
sample-dependent diagonal scaling. The appropriate solver is Cholesky
factorization or a preconditioned conjugate-gradient method --- distinct
from the SMO machinery developed here. The \texttt{psvr} package
implements an LS-SVR backend alongside the $\varepsilon$-SVR one;
exposing a unified \texttt{solve()} interface that dispatches by
variant, and sharing the kernel infrastructure between the two paths,
is the principal near-term engineering target. The symmetric-kernel
counterpart follows by the substitution $\Omega \to \Omega_s$ of
Section~\ref{sec:symmetric}.

\paragraph{(F2) Wasserstein-distributionally-robust connection.}
The de Myttenaere et al.~\cite{DeMyttenaere2016} equivalence between
MAPE minimization and weighted-MAE regression with weights $1/y_k$ has
a natural interpretation in distributionally robust optimization: the
per-sample weighting acts as a target-dependent transportation cost,
so MAPE-SVR can be recast as a Wasserstein-DRO problem with a
uniform relative-perturbation budget. This reinterpretation would place
MAPE-SVR within the regularization-by-robustness program
of~\cite{MohajerinEsfahaniKuhn2018MathProg} and could yield uniform
generalization bounds that the standard ERM framework --- which
degrades when $\min_k y_k \to 0$ --- does not provide. A rigorous
derivation is left as a self-contained follow-on paper.

\paragraph{(F3) Empirical evaluation on industrial forecasting datasets.}
Section~\ref{sec:large-scale} establishes solver convergence and
wall-time competitiveness on the California Housing benchmark ($N =
20{,}433$) and on synthetic log-normal targets up to $N = 30{,}000$.
What remains open is performance on datasets with structured temporal
correlation typical of electricity demand, supply-chain, and financial
forecasting, where the target distribution is non-stationary and the
prediction horizon introduces autoregressive dependencies not captured
by the static kernel framework. Such datasets would also provide a
natural comparison point with the dual coordinate descent
of~\cite{Hsieh2008ICML}, which targets large-scale linear SVMs
directly and whose adaptation to the MAPE setting has not been analyzed.

\paragraph{(F4) \texttt{psvr} v0.1.0 roadmap.}
The companion R package~\cite{BenavidesHerrera2026Rpsvr} is currently
at v0.0.2.9009, which implements the four (C5) efficiency
theorems --- the asymmetric freeze-counter
(Theorem~\ref{thm:asym-freeze}), cross-validation warm-starting
(Theorem~\ref{thm:warm-start}), the block-$k=4$ subproblem
(Theorem~\ref{thm:block-k4}), and per-pair tolerance scaling
(Theorem~\ref{thm:per-sample-tol}) --- alongside the adaptive spectral
regularization of Algorithm~\ref{alg:spectral}. A v0.1.0 release is
planned to add: \texttt{caret} and \texttt{mlr3} integration for
broader R-ecosystem visibility; an expanded edge-case test suite;
numerical-stability checks with ridge fallback; and a vignette
presenting the by-analogy SMO pedagogy of
Section~\ref{sec:preliminaries}.

\paragraph{(F5) Portable C\texttt{++} core and Python binding.}
The \texttt{psvr} package separates the SMO solver into a portable
C\texttt{++} core (\texttt{src/core\_*.cpp}, using only
\texttt{std::vector} and raw pointers with no Rcpp types) and a thin
Rcpp adapter layer (\texttt{src/binding\_*.cpp}). Conditional compile
macros allow the same core to build under R's toolchain or as a
standalone library. A Python binding via \texttt{pybind11} would wrap
the core unchanged, with adapters that translate
\texttt{numpy.ndarray} inputs to the core's \texttt{double*} signature
in place of the Rcpp \texttt{NumericMatrix} translation. This
architecture generalizes the LIBSVM portability argument of
Appendix~\ref{app:libsvm}: the MAPE-SVR algorithm is portable at the
binding-layer boundary rather than requiring a full reimplementation.
Building the Python adapter is deferred follow-on work.

\paragraph{(F6) Online and incremental variants.}
Warm-starting \texttt{psvr} from a previous solution upon the arrival
of new data, and decremental updates for covariate-shift
adaptation~\cite{Sugiyama2007JMLR, Bickel2009JMLR}, are natural
extensions of the present framework. The warm-start theorem
(Theorem~\ref{thm:warm-start}) provides the convergence guarantee for
the data-arrival case; the incremental SVM framework
of~\cite{Laskov2006JMLR} supplies the algorithmic scaffold for the
decremental direction. Multi-kernel learning with MAPE loss ---
combining the per-sample bound structure developed here with learned
kernel-combination weights --- is a further direction that has not yet
been developed.
\section{LIBSVM Drop-in Modification Recipe}
\label{app:libsvm}

Per the structural-invariance Theorem~\ref{thm:invariance}, adapting an
existing LIBSVM-based $\varepsilon$-SVR solver to the MAPE loss
requires modifications to fewer than fifteen lines of C++ code, located
in five places in the LIBSVM source (\texttt{svm.cpp}, version 3.32+,
\url{https://github.com/cjlin1/libsvm}). Four of the five modifications
consist of replacing the scalar regularization parameter \texttt{C}
with a per-sample vector \texttt{C\_k[k] = 100.0 * C / y[k]}.
Additionally, the linear coefficient vector $\bfq$ in the
$\varepsilon$-SVR dual must be modified from the standard
$[\varepsilon - y_k, \varepsilon + y_k]$ to the MAPE-SVR
$[y_k(\varepsilon/100 - 1), y_k(\varepsilon/100 + 1)]$, which is a
one-line change in the dual setup function.

\subsection*{Modification 1 --- Box constraint vector (replace scalar with per-sample)}

\begin{verbatim}
// File: svm.cpp, in Solver_NU::Solve() or Solver::Solve() for SVR mode

// Before (LIBSVM standard, eps-SVR with uniform C):
double C = param->C;
double *Q_alpha_bound = new double[2*N];
for (int i = 0; i < 2*N; ++i) Q_alpha_bound[i] = C;

// After (MAPE-SVR with sample-dependent C_k):
double C = param->C;
double *C_k = new double[N];
double *Q_alpha_bound = new double[2*N];
for (int k = 0; k < N; ++k) {
  C_k[k] = 100.0 * C / y[k];          // sample-dependent bound
  Q_alpha_bound[k]   = C_k[k];        // bound for alpha_k
  Q_alpha_bound[N+k] = C_k[k];        // bound for alpha_k* (same per training point k)
}
\end{verbatim}

\subsection*{Modification 2 --- Working-set feasibility test}

This corresponds to Definition~\ref{def:Iup-Idown-mape} of
Section~\ref{sec:smo-inner}. Replace every comparison of \texttt{alpha[k]}
against the scalar \texttt{C} with a comparison against
\texttt{Q\_alpha\_bound[k]}:

\begin{verbatim}
// File: svm.cpp, in Solver::select_working_set()

// Before:
if (alpha[k] < C - 1e-8) /* k is in I_up */ ...
if (alpha[k] >  1e-8)    /* k is in I_down */ ...
if (alpha_star[k] >  1e-8)    /* N+k is in I_up */ ...
if (alpha_star[k] < C - 1e-8) /* N+k is in I_down */ ...

// After:
if (alpha[k]      < Q_alpha_bound[k]   - 1e-8) /* k is in I_up */ ...
if (alpha[k]      >  1e-8)                     /* k is in I_down */ ...
if (alpha_star[k] >  1e-8)                     /* N+k is in I_up */ ...
if (alpha_star[k] < Q_alpha_bound[N+k] - 1e-8) /* N+k is in I_down */ ...
\end{verbatim}

\subsection*{Modification 3 --- Clipping bounds in the two-variable update}

This corresponds to the room
expressions~\eqref{eq:room-i}--\eqref{eq:room-j} of
Section~\ref{sec:smo-inner}.

\begin{verbatim}
// File: svm.cpp, in Solver::Solve() main loop, after working-set selection:

// Before (uniform C):
double R_i_star = (i_star <= N) ? (C - alpha[p]) : alpha_star[p];
double R_j_star = (j_star <= N) ?  alpha[q]      : (C - alpha_star[q]);

// After (sample-dependent C_k):
double R_i_star = (i_star <= N) ? (Q_alpha_bound[p]   - alpha[p])   : alpha_star[p];
double R_j_star = (j_star <= N) ?  alpha[q]
                                 : (Q_alpha_bound[N+q] - alpha_star[q]);
\end{verbatim}

\subsection*{Modification 4 --- Shrinking thresholds}

This corresponds to criteria~(S2) and~(S4) of
Section~\ref{sec:bias-shrink}. The upper-bound saturation check uses
\texttt{C\_k[k]} in place of \texttt{C}:

\begin{verbatim}
// File: svm.cpp, in Solver::do_shrinking()

// Before:
if (alpha[k]      >= C - 1e-8) /* alpha_k saturates: candidate via (S2) */
if (alpha_star[k] >= C - 1e-8) /* alpha_k* saturates: candidate via (S4) */

// After:
if (alpha[k]      >= Q_alpha_bound[k]   - 1e-8) /* alpha_k saturates */
if (alpha_star[k] >= Q_alpha_bound[N+k] - 1e-8) /* alpha_k* saturates */
\end{verbatim}

The shrinking-asymmetry result of Lemma~\ref{lem:asymmetry} does
\emph{not} require additional code changes beyond Modification 4: the
asymmetric thresholds $\tau_{i^*} - 2y_k\varepsilon/100$ and $\tau_{j^*}
- 2y_k\varepsilon/100$ for the $\alpha^*$-criteria emerge automatically
from the substitution of the $\bfq$ vector in Modification~5 below,
since $\tau_{N+k} = \tau_k + 2y_k\varepsilon/100$ holds by
Proposition~\ref{prop:structural-gap} in Section~\ref{sec:dual-kkt}.

\subsection*{Modification 5 --- Linear-coefficient vector $\bfq$ for the dual setup}

\begin{verbatim}
// File: svm.cpp, in svm_train_one() or analogous setup for SVR:

// Before (standard eps-SVR with absolute-error tube):
for (int k = 0; k < N; ++k) {
  q[k]   = param->p - y[k];     // eps - y_k for alpha_k coefficient
  q[N+k] = param->p + y[k];     // eps + y_k for alpha_k* coefficient
}

// After (MAPE-SVR with percentage-error tube):
for (int k = 0; k < N; ++k) {
  q[k]   = y[k] * (param->p / 100.0 - 1.0);  // y_k(eps/100 - 1)
  q[N+k] = y[k] * (param->p / 100.0 + 1.0);  // y_k(eps/100 + 1)
}
\end{verbatim}

\subsection*{Components that remain unchanged}

Per Theorem~\ref{thm:invariance}, the remaining LIBSVM machinery
operates unchanged on the modified $\bfq$ and per-sample
$\boldsymbol{C}_k$:
\begin{itemize}
\item \emph{Kernel evaluation} (\texttt{Kernel::k\_function()}, the
kernel cache, the column-access patterns in \texttt{select\_working\_set()}
and the gradient update). The kernel matrix $\Omega$ is the same
regardless of the loss; the symmetric-kernel variant of
Section~\ref{sec:symmetric} substitutes $\Omega \leftarrow \Omega_s$
at this single layer.
\item \emph{Gradient bookkeeping} (the \texttt{G[i]} array updates
after each two-variable step, \texttt{G[i] -= delta * (Q[k(i)][p] -
Q[k(i)][q])}). Per~\eqref{eq:tau-update}, the gradient update depends
only on $\Omega$ and the sign vector $\bfs$; the box constraints $C_k$
do not appear.
\item \emph{KKT-violation reduction} (the convergence check
\texttt{Delta = G\_max - G\_min <= eps}, where $G_{\max} = \max_{i \in
\Iup}\tau_i$ and $G_{\min} = \min_{j \in \Idown}\tau_j$). This depends
on the per-sample bounds only through the membership of $i$ in $\Iup,
\Idown$ --- already handled by Modifications~2 and~4.
\item \emph{Reconstruction and unshrinking} (the periodic recomputation
of $G$ from scratch on the full active set). Independent of $C_k$.
\item \emph{Bias recovery} (the averaging over free support vectors $0
< u_i < C_{k(i)}$, per~\eqref{eq:bias-avg}). The free-support-vector
test uses Modification~2.
\end{itemize}

This recipe is independent of working-set-selection rule (MVP, WSS3,
maximum-gain, TCSMO) and shrinking schedule. Practitioners using
non-LIBSVM toolchains can apply the analogous modifications via the
relevant solver hooks, described next.

\subsection*{Ports to other SVR toolchains}

\paragraph{scikit-learn (Python).}
The \texttt{sklearn.svm.SVR} class wraps LIBSVM's C++ solver via a
Cython binding. To apply the MAPE-SVR modification, fork the underlying
\texttt{libsvm} directory (typically at
\texttt{sklearn/svm/src/libsvm/}) and apply Modifications~1--5 above to
the C++ source. Recompile the Cython binding (\texttt{pip install -e .}
from the sklearn source directory). The Python-facing API is
unchanged: instantiate \texttt{SVR(C=1.0, epsilon=5.0, kernel='rbf',
gamma=0.5)} where the \texttt{epsilon} parameter is now interpreted as
the MAPE-tube width in percentage points (per the convention
of~\eqref{eq:mape-dual}). For users without local-build capability, the
\texttt{psvr} R package~\cite{BenavidesHerrera2026Rpsvr} exposes the
equivalent functionality from R; calling it from Python via
\texttt{rpy2} is a working alternative.

\paragraph{kernlab (R).}
The \texttt{kernlab::ksvm()} function exposes an R-side S4 interface
in \texttt{R/ksvm.R}, but the underlying SMO loop is implemented in
compiled C++ source (also derived from Chang and Lin's LIBSVM lineage)
under \texttt{src/}. Modifications~1--5 must therefore be applied at
the C++ level (analogous to the scikit-learn / e1071 path), with
subsequent recompilation of the package. For the \texttt{eps-bsvr}
formulation, kernlab uses a TRON chunking solver rather than SMO, so
the present recipe applies only to the \texttt{eps-svr} path.

\paragraph{e1071 (R).}
The \texttt{e1071::svm()} function is a thin R wrapper around LIBSVM.
Modifications must be applied at the C++ level (as for scikit-learn) by
editing the \texttt{e1071/src/svm.cpp} source and recompiling.

\paragraph{For R users without local-build capability.}
The canonical pre-built implementation is the open-source \texttt{psvr}
package~\cite{BenavidesHerrera2026Rpsvr}, which embodies all five
modifications above plus WSS3 working-set selection, adaptive shrinking
with the freeze-counter mechanism of Section~\ref{sec:bias-shrink}, and
the asymmetric shrinking-threshold pattern of
Lemma~\ref{lem:asymmetry}. Usage is direct: \texttt{psvr::smo\_mape(X,
y, C = 1, epsilon = 5, kernel = "rbf", gamma = 0.5)} produces the
trained model; \texttt{predict(model, newdata)} produces test-set
predictions via~\eqref{eq:masym-prediction}.

\bibliographystyle{ieeetr}
\bibliography{smo-v3-references}

@inproceedings{Drucker1997,
  author    = {Harris Drucker and Christopher J. C. Burges and Linda Kaufman
               and Alexander J. Smola and Vladimir N. Vapnik},
  title     = {Support Vector Regression Machines},
  booktitle = {Advances in Neural Information Processing Systems 9 ({NIPS} 1996)},
  editor    = {Michael C. Mozer and Michael I. Jordan and Thomas Petsche},
  pages     = {155--161},
  year      = {1997},
  publisher = {MIT Press},
  address   = {Cambridge, MA}
}

@incollection{Joachims1999,
  author    = {Thorsten Joachims},
  title     = {Making Large-Scale {SVM} Learning Practical},
  booktitle = {Advances in Kernel Methods --- Support Vector Learning},
  editor    = {Bernhard Sch{\"o}lkopf and Christopher J. C. Burges
               and Alexander J. Smola},
  publisher = {MIT Press},
  address   = {Cambridge, MA},
  pages     = {169--184},
  year      = {1999}
}

@book{Vapnik1995,
  author    = {Vladimir N. Vapnik},
  title     = {The Nature of Statistical Learning Theory},
  publisher = {Springer-Verlag},
  address   = {New York, NY},
  year      = {1995},
  doi       = {10.1007/978-1-4757-2440-0}
}

@book{Vapnik1998,
  author    = {Vladimir N. Vapnik},
  title     = {Statistical Learning Theory},
  publisher = {Wiley-Interscience},
  address   = {New York, NY},
  year      = {1998},
  isbn      = {0-471-03003-1}
}

@article{Smola2004,
  author    = {Alex J. Smola and Bernhard Sch{\"o}lkopf},
  title     = {A Tutorial on Support Vector Regression},
  journal   = {Statistics and Computing},
  volume    = {14},
  number    = {3},
  pages     = {199--222},
  year      = {2004},
  doi       = {10.1023/B:STCO.0000035301.49549.88},
  publisher = {Springer}
}

@incollection{Platt1999,
  author    = {John C. Platt},
  title     = {Fast Training of Support Vector Machines Using Sequential
               Minimal Optimization},
  booktitle = {Advances in Kernel Methods: Support Vector Learning},
  editor    = {Bernhard Sch{\"o}lkopf and Christopher J.C. Burges
               and Alexander J. Smola},
  publisher = {MIT Press},
  address   = {Cambridge, MA},
  pages     = {185--208},
  year      = {1999}
}

@techreport{Platt1998,
  author      = {John C. Platt},
  title       = {Sequential Minimal Optimization: {A} Fast Algorithm
                 for Training Support Vector Machines},
  institution = {Microsoft Research},
  number      = {MSR-TR-98-14},
  year        = {1998},
  url         = {https://www.microsoft.com/en-us/research/publication/%
                 sequential-minimal-optimization-a-fast-algorithm-for-%
                 training-support-vector-machines/}
}

@article{Chang2011,
  author    = {Chih-Chung Chang and Chih-Jen Lin},
  title     = {{LIBSVM}: A Library for Support Vector Machines},
  journal   = {ACM Transactions on Intelligent Systems and Technology},
  volume    = {2},
  number    = {3},
  pages     = {27:1--27:27},
  year      = {2011},
  doi       = {10.1145/1961189.1961199},
  note      = {Software available at \url{http://www.csie.ntu.edu.tw/~cjlin/libsvm}}
}

@article{osqp2020,
  author  = {Stellato, Bartolomeo and Banjac, Goran and Goulart, Paul
             and Bemporad, Alberto and Boyd, Stephen},
  title   = {{OSQP}: An Operator Splitting Solver for Quadratic Programs},
  journal = {Mathematical Programming Computation},
  year    = {2020},
  volume  = {12},
  number  = {4},
  pages   = {637--672},
  doi     = {10.1007/s12532-020-00179-2}
}

@book{Suykens2002,
  author    = {Johan A. K. Suykens and Tony {Van Gestel} and
               Jos {De Brabanter} and Bart {De Moor} and Joos Vandewalle},
  title     = {Least Squares Support Vector Machines},
  publisher = {World Scientific},
  address   = {Singapore},
  year      = {2002},
  isbn      = {981-238-151-1},
  doi       = {10.1142/5089}
}

@inproceedings{Espinoza2005,
  author    = {Marcelo Espinoza and Johan A. K. Suykens and Bart {De Moor}},
  title     = {Imposing Symmetry in Least Squares Support Vector
               Machines Regression},
  booktitle = {Proceedings of the 44th {IEEE} Conference on Decision
               and Control ({CDC}~2005)},
  pages     = {5716--5721},
  year      = {2005},
  address   = {Seville, Spain},
  publisher = {IEEE},
  doi       = {10.1109/CDC.2005.1583074}
}

@article{Makridakis1993,
  author    = {Spyros Makridakis},
  title     = {Accuracy Measures: Theoretical and Practical Concerns},
  journal   = {International Journal of Forecasting},
  volume    = {9},
  number    = {4},
  pages     = {527--529},
  year      = {1993},
  doi       = {10.1016/0169-2070(93)90079-3},
  publisher = {Elsevier}
}

@article{Hyndman2006,
  author    = {Rob J. Hyndman and Anne B. Koehler},
  title     = {Another Look at Measures of Forecast Accuracy},
  journal   = {International Journal of Forecasting},
  volume    = {22},
  number    = {4},
  pages     = {679--688},
  year      = {2006},
  doi       = {10.1016/j.ijforecast.2006.03.001},
  publisher = {Elsevier}
}

@inproceedings{benavides2025support,
  title={Support Vector Regression Under Percentage-Error Loss},
  author={Benavides-Herrera, Pablo and Rodr{\'\i}guez-Reyes, Sara and {\'A}lvarez-{\'A}lvarez, Gregorio and Ruiz-Cruz, Riemann and S{\'a}nchez-Torres, Juan Diego},
  booktitle={2025 22nd International Conference on Electrical Engineering, Computing Science and Automatic Control (CCE)},
  pages={1--5},
  year={2025},
  organization={IEEE}
}

@article{benavides2026unified,
  author       = {Benavides-Herrera, Pablo and
                  {\'A}lvarez, Gregorio and
                  Ruiz-Cruz, Riemann and
                  S{\'a}nchez-Torres, Juan Diego},
  title        = {A Unified Family of Percentage-Error Support Vector
                  Regression Models with Symmetric Kernel Extensions},
  journal      = {Mathematics},
  year         = {2026},
  volume       = {14},
  number       = {10},
  pages        = {1679},
  doi          = {10.3390/math14101679},
  issn         = {2227-7390},
  url          = {https://www.mdpi.com/2227-7390/14/10/1679},
  publisher    = {MDPI}
}

@manual{BenavidesHerrera2026Rpsvr,
  author    = {Benavides-Herrera, Pablo},
  title     = {{psvr}: Percentage-Error Support Vector Regression},
  year      = {2026},
  publisher = {Zenodo},
  version   = {v0.0.2.9008 (development; v0.0.2 minted on Zenodo)},
  doi       = {10.5281/zenodo.19935781},
  url       = {https://doi.org/10.5281/zenodo.19935781}
}

@article{Flake2002,
  author  = {Flake, Gary William and Lawrence, Steve},
  title   = {Efficient {SVM} Regression Training with {SMO}},
  journal = {Machine Learning},
  volume  = {46},
  number  = {1--3},
  pages   = {271--290},
  year    = {2002},
  doi     = {10.1023/A:1012474916001}
}

@article{KeerthiShevadeBhattacharyyaMurthy2001NeuralComp,
  author  = {Keerthi, S. S. and Shevade, S. K. and Bhattacharyya, C. and Murthy, K. R. K.},
  title   = {Improvements to Platt's SMO Algorithm for SVM Classifier Design},
  journal = {Neural Computation},
  volume={13}, number={3}, pages={637--649}, year={2001},
  doi     = {10.1162/089976601300014493}
}

@article{Wang2024,
  author  = {Wang, Lei and Wang, Xinyu and Zhao, Zhongchao},
  title   = {Mid-term electricity demand forecasting using improved multi-mode reconstruction and particle swarm-enhanced support vector regression},
  journal = {Energy},
  volume={304}, pages={132021}, year={2024},
  doi     = {10.1016/j.energy.2024.132021}
}

@article{Aziz2024,
  author  = {Aziz, Abdul and Mahmood, Danish and Qureshi, Muhammad Shuaib and Qureshi, Muhammad Bilal and Kim, Kyungsup},
  title   = {AI-based peak power demand forecasting model focusing on economic and climate features},
  journal = {Frontiers in Energy Research},
  volume={12}, year={2024},
  doi     = {10.3389/fenrg.2024.1328891}
}

@article{Zhang2024,
  author  = {Zhang, Zhirong and Zhang, Qiqi and Liang, Haitao and Gorbani, Bizhan},
  title   = {Optimizing electric load forecasting with support vector regression/LSTM optimized by flexible Gorilla troops algorithm and neural networks},
  journal = {Scientific Reports},
  volume={14}, pages={22092}, year={2024},
  doi     = {10.1038/s41598-024-73893-9}
}

@article{Hasan2025,
  author  = {Hasan, Md. Sadikul and Tarequzzaman, Md. and Moznuzzaman, Md. and Ahad Juel, Md Abdul},
  title   = {Prediction of energy consumption in four sectors using support vector regression optimized with genetic algorithm},
  journal = {Heliyon},
  volume={11}, number={2}, pages={e41765}, year={2025},
  doi     = {10.1016/j.heliyon.2025.e41765}
}

@article{Tofallis2015,
  author  = {Tofallis, Chris},
  title   = {A better measure of relative prediction accuracy for model selection and model estimation},
  journal = {Journal of the Operational Research Society},
  volume={66}, number={8}, pages={1352--1362}, year={2015},
  doi     = {10.1057/jors.2014.103}
}

@article{DeMyttenaere2016,
  author  = {de Myttenaere, Arnaud and Golden, Boris and Le Grand, Bénédicte and Rossi, Fabrice},
  title   = {Mean Absolute Percentage Error for regression models},
  journal = {Neurocomputing},
  volume={192}, pages={38--48}, year={2016},
  doi     = {10.1016/j.neucom.2015.12.114}
}

@article{Niyogi1998,
  author  = {Niyogi, P. and Girosi, F. and Poggio, T.},
  title   = {Incorporating prior information in machine learning by creating virtual examples},
  journal = {Proceedings of the IEEE},
  volume={86}, number={11}, pages={2196--2209}, year={1998},
  doi     = {10.1109/5.726787}
}

@article{Haasdonk2007,
  author  = {Haasdonk, Bernard and Burkhardt, Hans},
  title   = {Invariant kernel functions for pattern analysis and machine learning},
  journal = {Machine Learning},
  volume={68}, number={1}, pages={35--61}, year={2007},
  doi     = {10.1007/s10994-007-5009-7}
}

@book{Steinwart2008,
  author    = {Steinwart, Ingo and Christmann, Andreas},
  title     = {Support Vector Machines},
  publisher = {Springer},
  address   = {New York, NY},
  series    = {Information Science and Statistics},
  year      = {2008},
  isbn      = {978-0-387-77241-7},
  doi       = {10.1007/978-0-387-77242-4}
}

@inbook{Scholkopf2001,
  author  = {Schölkopf, Bernhard and Herbrich, Ralf and Smola, Alex J.},
  title   = {A Generalized Representer Theorem},
  booktitle = {Computational Learning Theory},
  pages     = {416--426},
  year      = {2001},
  publisher = {Springer Berlin Heidelberg},
  doi       = {10.1007/3-540-44581-1_27}
}

@book{Scholkopf2002,
  author    = {Schölkopf, Bernhard and Smola, Alex J.},
  title     = {Learning with Kernels: Support Vector Machines, Regularization, Optimization, and Beyond},
  publisher = {MIT Press},
  address   = {Cambridge, MA},
  year      = {2002},
  isbn      = {0-262-19475-9}
}

@article{Aronszajn1950,
  author  = {Aronszajn, N.},
  title   = {Theory of reproducing kernels},
  journal = {Transactions of the American Mathematical Society},
  volume={68}, number={3}, pages={337--404}, year={1950},
  doi     = {10.1090/S0002-9947-1950-0051437-7}
}

@book{BoydVandenberghe2004,
  author    = {Boyd, Stephen and Vandenberghe, Lieven},
  title     = {Convex Optimization},
  publisher = {Cambridge University Press},
  address   = {Cambridge, UK},
  year      = {2004},
  isbn      = {978-0-521-83378-3},
  doi       = {10.1017/CBO9780511804441}
}

@book{BertsekasNedicOzdaglar2003,
  author    = {Bertsekas, Dimitri P. and Nedić, Angelia and Ozdaglar, Asuman E.},
  title     = {Convex Analysis and Optimization},
  publisher = {Athena Scientific},
  address   = {Belmont, MA},
  year      = {2003},
  isbn      = {978-1-886529-45-8}
}

@book{Rockafellar1970,
  author    = {Rockafellar, R. Tyrrell},
  title     = {Convex Analysis},
  publisher = {Princeton University Press},
  address   = {Princeton, NJ},
  series    = {Princeton Mathematical Series},
  number    = {28},
  year      = {1970},
  isbn      = {0-691-08069-0}
}

@article{Sion1958,
  author  = {Sion, Maurice},
  title   = {On general minimax theorems},
  journal = {Pacific Journal of Mathematics},
  volume={8}, number={1}, pages={171--176}, year={1958},
  doi     = {10.2140/pjm.1958.8.171}
}

@inproceedings{Hsieh2008ICML,
  author  = {Hsieh, Cho-Jui and Chang, Kai-Wei and Lin, Chih-Jen and Keerthi, S. Sathiya and Sundararajan, S.},
  title   = {A dual coordinate descent method for large-scale linear SVM},
  booktitle = {Proceedings of the 25th International Conference on Machine Learning},
  series  = {ICML '08},
  pages   = {408--415},
  year    = {2008},
  publisher = {Association for Computing Machinery},
  doi     = {10.1145/1390156.1390208}
}

@article{Ho2012JMLR,
  author  = {Ho, Chia-Hua and Lin, Chih-Jen},
  title   = {Large-scale Linear Support Vector Regression},
  journal = {Journal of Machine Learning Research},
  volume={13}, pages={3323--3348}, year={2012}
}

@inbook{Andersen2000MosekHomogeneousIPM,
  author    = {Andersen, Erling D. and Andersen, Knud D.},
  title     = {The MOSEK Interior Point Optimizer for Linear Programming: An Implementation of the Homogeneous Algorithm},
  booktitle = {High Performance Optimization},
  pages     = {197--232},
  year      = {2000},
  publisher = {Springer US},
  doi       = {10.1007/978-1-4757-3216-0_8}
}

@misc{GoulartChen2024Clarabel,
  author  = {Goulart, Paul J. and Chen, Yuwen},
  title   = {Clarabel: An Interior-Point Solver for Conic Programs with Quadratic Objectives},
  year    = {2024},
  howpublished = {arXiv preprint},
  archivePrefix = {arXiv},
  eprint  = {2405.12762},
  doi     = {10.48550/arXiv.2405.12762}
}

@article{Sugiyama2007JMLR,
  author  = {Sugiyama, Masashi and Krauledat, Matthias and Müller, Klaus-Robert},
  title   = {Covariate Shift Adaptation by Importance Weighted Cross Validation},
  journal = {Journal of Machine Learning Research},
  volume={8}, number={35}, pages={985--1005}, year={2007}
}

@article{Bickel2009JMLR,
  author  = {Bickel, Steffen and Brückner, Michael and Scheffer, Tobias},
  title   = {Discriminative Learning Under Covariate Shift},
  journal = {Journal of Machine Learning Research},
  volume={10}, number={75}, pages={2137--2155}, year={2009}
}

@article{KarasuyamaHaradaSugiyamaTakeuchi2012ML,
  author  = {Karasuyama, Masayuki and Harada, Naoyuki and Sugiyama, Masashi and Takeuchi, Ichiro},
  title   = {Multi-parametric solution-path algorithm for instance-weighted support vector machines},
  journal = {Machine Learning},
  volume={88}, number={3}, pages={297--330}, year={2012},
  doi     = {10.1007/s10994-012-5288-5}
}

@article{Anand2020,
  author  = {Anand, Pritam and Rastogi, Reshma and Chandra, Suresh},
  title   = {A new asymmetric $\epsilon$-insensitive pinball loss function based support vector quantile regression model},
  journal = {Applied Soft Computing},
  volume={94}, pages={106473}, year={2020},
  doi     = {10.1016/j.asoc.2020.106473}
}

@article{Suykens2002weighted,
  author  = {Suykens, J.A.K. and De Brabanter, J. and Lukas, L. and Vandewalle, J.},
  title   = {Weighted least squares support vector machines: robustness and sparse approximation},
  journal = {Neurocomputing},
  volume={48}, number={1}, pages={85--105}, year={2002},
  doi     = {10.1016/S0925-2312(01)00644-0}
}

@article{Goodwin1999,
  author  = {Goodwin, Paul and Lawton, Richard},
  title   = {On the asymmetry of the symmetric MAPE},
  journal = {International Journal of Forecasting},
  volume={15}, number={4}, pages={405--408}, year={1999},
  doi     = {10.1016/S0169-2070(99)00007-2}
}

@article{Kim2016,
  author  = {Kim, Sungil and Kim, Heeyoung},
  title   = {A new metric of absolute percentage error for intermittent demand forecasts},
  journal = {International Journal of Forecasting},
  volume={32}, number={3}, pages={669--679}, year={2016},
  doi     = {10.1016/j.ijforecast.2015.12.003}
}

@article{Makridakis2020,
  author  = {Makridakis, Spyros and Spiliotis, Evangelos and Assimakopoulos, Vassilios},
  title   = {The M4 Competition: 100,000 time series and 61 forecasting methods},
  journal = {International Journal of Forecasting},
  volume={36}, number={1}, pages={54--74}, year={2020},
  doi     = {10.1016/j.ijforecast.2019.04.014}
}

@article{Tseng2001CoordinateDescent,
  author  = {Tseng, Paul},
  title   = {Convergence of a Block Coordinate Descent Method for Nondifferentiable Minimization},
  journal = {Journal of Optimization Theory and Applications},
  volume={109}, number={3}, pages={475--494}, year={2001},
  doi     = {10.1023/A:1017501703105}
}

@article{MohajerinEsfahaniKuhn2018MathProg,
  author  = {Mohajerin Esfahani, Peyman and Kuhn, Daniel},
  title   = {Data-driven distributionally robust optimization using the Wasserstein metric: performance guarantees and tractable reformulations},
  journal = {Mathematical Programming},
  volume={171}, number={1}, pages={115--166}, year={2018},
  doi     = {10.1007/s10107-017-1172-1}
}

@article{Du2024,
  author  = {Du, Ke-Lin and Jiang, Bingchun and Lu, Jiabin and Hua, Jingyu and Swamy, M. N. S.},
  title   = {Exploring Kernel Machines and Support Vector Machines: Principles, Techniques, and Future Directions},
  journal = {Mathematics},
  volume={12}, number={24}, pages={3935}, year={2024},
  doi     = {10.3390/math12243935}
}

@article{AmayaTejera2024,
  author  = {Amaya-Tejera, Nazhir and Gamarra, Margarita and Vélez, Jorge I. and Zurek, Eduardo},
  title   = {A distance-based kernel for classification via Support Vector Machines},
  journal = {Frontiers in Artificial Intelligence},
  volume={7}, year={2024},
  doi     = {10.3389/frai.2024.1287875}
}

@article{AkhtarTanveerArshad2024PatternRecognition,
  author  = {Akhtar, Mushir and Tanveer, M. and Arshad, Mohd.},
  title   = {Advancing Supervised Learning with the Wave Loss Function: A Robust and Smooth Approach},
  journal = {Pattern Recognition},
  volume={155}, pages={110637}, year={2024},
  doi     = {10.1016/j.patcog.2024.110637}
}

@article{YuLiLiu2023PatternRecognition,
  author  = {Yu, Lang and Li, Shengjie and Liu, Siyi},
  title   = {Fast support vector machine training via three-term conjugate-like SMO algorithm},
  journal = {Pattern Recognition},
  volume={139}, pages={109478}, year={2023},
  doi     = {10.1016/j.patcog.2023.109478}
}

@article{GlasmachersIgel2006JMLR,
  author  = {Glasmachers, Tobias and Igel, Christian},
  title   = {Maximum-Gain Working Set Selection for SVMs},
  journal = {Journal of Machine Learning Research},
  volume={7}, pages={1437--1466}, year={2006}
}

@article{GlasmachersIgel2008NeuralComp,
  author  = {Glasmachers, Tobias and Igel, Christian},
  title   = {Second-Order SMO Improves SVM Online and Active Learning},
  journal = {Neural Computation},
  volume={20}, number={2}, pages={374--382}, year={2008},
  doi     = {10.1162/neco.2007.10-06-354}
}

@article{Laskov2006JMLR,
  author  = {Laskov, Pavel and Gehl, Christian and Krüger, Stefan and Müller, Klaus-Robert},
  title   = {Incremental Support Vector Learning: Analysis, Implementation and Applications},
  journal = {Journal of Machine Learning Research},
  volume={7}, number={69}, pages={1909--1936}, year={2006}
}

@article{BordesErtekinWestonBottou2005JMLR,
  author  = {Bordes, Antoine and Ertekin, Seyda and Weston, Jason and Bottou, Léon},
  title   = {Fast Kernel Classifiers with Online and Active Learning},
  journal = {Journal of Machine Learning Research},
  volume={6}, number={54}, pages={1579--1619}, year={2005}
}

@article{FanChenLin2005JMLR,
  author  = {Fan, Rong-En and Chen, Pai-Hsuen and Lin, Chih-Jen},
  title   = {Working Set Selection Using Second Order Information for Training Support Vector Machines},
  journal = {Journal of Machine Learning Research},
  volume={6}, number={63}, pages={1889--1918}, year={2005}
}

@inproceedings{Joachims2006,
  author    = {Joachims, Thorsten},
  title     = {Training Linear {SVMs} in Linear Time},
  booktitle = {Proceedings of the 12th ACM SIGKDD International Conference on Knowledge Discovery and Data Mining},
  series    = {KDD '06},
  pages     = {217--226},
  year      = {2006},
  publisher = {Association for Computing Machinery},
  doi       = {10.1145/1150402.1150429}
}

@techreport{LinLin2003NonPSDSMO,
  author      = {Lin, Hsuan-Tien and Lin, Chih-Jen},
  title       = {A Study on Sigmoid Kernels for {SVM} and the Training of non-{PSD} Kernels by {SMO}-type Methods},
  institution = {Department of Computer Science and Information Engineering, National Taiwan University},
  type        = {Technical Report},
  year        = {2003},
  month       = {March},
  address     = {Taipei 106, Taiwan}
}

@book{RyuYin2022,
  author    = {Ryu, Ernest K. and Yin, Wotao},
  title     = {Large-Scale Convex Optimization: Algorithms \& Analyses
               via Monotone Operators},
  publisher = {Cambridge University Press},
  year      = {2022},
  isbn      = {978-1-009-16085-8},
  url       = {https://large-scale-book.mathopt.com/}
}

\end{document}